\documentclass[reqno,11pt,a4paper]{amsart}
\usepackage{graphicx} 

\usepackage[toc]{appendix}
\usepackage[T1]{fontenc}
\usepackage{enumerate} 
\usepackage{amsmath,amsfonts,amssymb,mathrsfs }
\usepackage[latin1]{inputenc}
\usepackage[english]{babel}
\usepackage{mathtools}
\usepackage{braket}
\usepackage{color}
\usepackage{hyperref}
\usepackage{xfrac, nicefrac}
\usepackage{csquotes}
\usepackage{esint}
\usepackage{orcidlink}

\newtheorem{theorem}{Theorem}[section]
\newtheorem{lemma}[theorem]{Lemma}
\newtheorem{proposition}[theorem]{Proposition}

\newtheorem{corollary}[theorem]{Corollary}
\newtheorem{definition}[theorem]{Definition}
\newtheorem{remark}[theorem]{Remark}
\newtheorem*{remark*}{Remark}
\newtheorem*{definition*}{Definition}

\numberwithin{equation}{section}

\usepackage{hyperref}

\newcommand{\R}{\mathbb{R}}
\newcommand{\N}{\mathbb{N}}
\newcommand{\rn}{\mathbb{R}^{n}}

\newcommand{\mint}{\mathop{\int\hskip -1,05em -\, \!\!\!}\nolimits}

\def\A{\mathcal A}
\def\l{\lambda}
\def\L{\Lambda}
\def\d{\delta}
\def \diver{{\rm div}}
\def\wB{\widetilde{B}}
\def\e{\varepsilon}
\def\Ae{\A_\e}
\def\vrho{\varrho}
\def\ve{v_\e}

\def\we{w_\e}

\title[Gradient regularity of quasilinear elliptic operators]{Local and global $C^{1,\beta}$-regularity for uniformly elliptic quasilinear equations of $p$-Laplace and Orlicz-Laplace type}

\begin{document}

\begin{abstract}
We establish gradient H\"older continuity for solutions to quasilinear, uniformly elliptic equations, including  $p$-Laplace and Orlicz-Laplace type operators.
We provide quantitative gradient estimates both in the interior and up to the boundary, under Dirichlet or Neumann boundary conditions.
 \end{abstract}
 

\author[Carlo Alberto Antonini]{Carlo Alberto Antonini \orcidlink{0000-0002-7663-1090}}  \address{Carlo Alberto Antonini \\ 
Dipartimento di Matematica ``Federico Enriques'', 
Universit\`a degli studi di Milano, 
Viale Cesare Saldini 50, 20133,
Milan,
Italy\\ ORCID ID: 0000-0002-7663-1090}
\email{\url{carlo.antonini@unimi.it}\\ \url{antonini@altamatematica.it}}

\maketitle

\section{Introduction}
In this paper we investigate the regularity of the gradient of scalar-valued solutions
$u : \Omega \to \mathbb{R}$ to a class of quasilinear elliptic equations. A prototypical
example is the Orlicz--Laplace equation
\begin{equation}\label{eq:example}
-\Delta_B u := -\operatorname{div}\!\left( \frac{B'(|Du|)}{|Du|}\, Du \right) = f
\quad \text{in } \Omega.
\end{equation}
where
\begin{itemize}
\item $\Omega \subset \mathbb{R}^n$ is an open set, with $n\in \N$, $n \geq 2$;
\item $f : \Omega \to \mathbb{R}$ is a prescribed function;
\item $B : [0,\infty) \to [0,\infty)$ is a convex function vanishing only at $0$.
\end{itemize}

Any such function $B$ is called a Young (or Orlicz) function and admits the representation
\begin{equation}\label{Bb}
B(t) = \int_0^t b(s)\,ds,
\end{equation}
for some non-decreasing function $b : (0,\infty) \to (0,\infty)$.

A classical example is given by the power-type function
$B(t) = \frac{t^p}{p}$ for some $p>1$, in which case \eqref{eq:example}
reduces to the $p$-Laplace equation
\begin{equation}\label{eq:plapl}
-\Delta_p u := -\operatorname{div}\!\left( |Du|^{p-2} Du \right) = f.
\end{equation}

Throughout the paper we assume additional regularity on $B$, namely
\begin{equation}\label{reg:Byoung}
B \in C^2(0,\infty) \cap C^1([0,\infty)).
\end{equation}
Then for $b \in C^0([0,\infty)) \cap C^1(0,\infty)$ as in \eqref{Bb} we define
\begin{equation}\label{def:a}
a(t) := \frac{b(t)}{t}, \qquad t>0.
\end{equation}
so that $a \in C^1(0,\infty)$, and we further assume that there exist constants
$i_a \le s_a$ such that
\begin{equation}\label{iasa}
\begin{split}
-1 < i_a \le \inf_{t>0} \frac{t\,a'(t)}{a(t)}
\le \sup_{t>0} \frac{t\,a'(t)}{a(t)} \le s_a < \infty.
\end{split}
\end{equation}
\vspace{0.1cm}

In this paper, we shall consider general quasilinear elliptic equations of the form
\begin{equation}\label{eq1}
-\operatorname{div}\big(\mathcal{A}(x,Du)\big) = f \quad \text{in } \Omega,
\end{equation}
where the vector field $\A:\overline{\Omega}\times \R^n\to \rn$ is such that 
\begin{equation}\label{reg:A}
    \A\in C^0(\overline{\Omega}\times \R^n),\quad \A(x,\cdot)\in C^1(\R^n\setminus\{0\})\quad\text{for all $x\in \overline{\Omega}$}.
\end{equation}
Here $\overline\Omega$ denotes the closure of $\Omega$. We assume that $\A(x,\xi)=\{\A^i(x,\xi)\}_{i=1,\dots,n}$ satisfies Orlicz type growth and coercivity assumptions; namely,  there exist  two constants $0<\l\leq \L$ such that
\begin{equation}\label{ass:A}
    \begin{split}
&\sum_{i,j=1}^n\left|\frac{\partial \A^i }{\partial \xi_j}(x,\xi)\right|\leq \L\,a\left(|\xi|\right)
\\ &\sum_{i,j=1}^n \frac{\partial \A^i }{\partial \xi_j}(x,\xi)\,\eta_i\,\eta_j\geq \l\,a(|\xi|)\,|\eta|^2
    \end{split}
\end{equation}
for all  $x\in \overline{\Omega}$, for all  $\xi\in \R^n\setminus \{0\}$ and $\eta\in \R^n$.

We also suppose that there exist  $\alpha\in (0,1]$ and $\L_\mathrm h>0$ \footnote{The constant $\Lambda_{\mathrm{h}}$ need not coincide with the constant $\Lambda$ appearing in \eqref{ass:A} and \eqref{A:hold}$_2$. We distinguish the notation because most constants arising below may depend on $\Lambda$, but not on $\Lambda_{\mathrm{h}}$. This is the case, e.g., of the H\"older's exponent $\beta$ in Theorems \ref{thm:interior}- \ref{thm:neu}.} such that
\begin{equation}\label{A:hold}
    |\A(x,\xi)-\A(y,\xi)|\leq \L_{\mathrm{h}}\,\big(1+b(|\xi|)\big)\,|x-y|^{\alpha}\quad \text{and}\quad |\A(x,0)|\leq \L\,,
\end{equation}
for all $x,y\in \overline{\Omega}$, and for all $\xi \in \R^n$. By adjusting the constants $\l,\L$, we may also assume that
\begin{equation}\label{ab=1}
    a(1)=1\quad\text{hence}\quad b(1)=1.
\end{equation}

Given the assumptions above, let us briefly recall the definition of weak solution. For the definition of the relevant Sobolev spaces, we refer to Section \ref{sec:spaces}
below.
\vspace{0.1cm}

Let $f\in L^n_{loc}(\Omega)$; we say that $u\in W^{1,B}_{loc}(\Omega)$ is a local weak solution to \eqref{eq1} if  it satisfies
\begin{equation}\label{def:weak}
    \int_\Omega \A(x,Du)\cdot D\varphi\,dx=\int_\Omega f\,\varphi\,dx
\end{equation}
for all test functions $\varphi\in C^\infty_c(\Omega)$. By a standard density argument-- see \cite[Lemma 2.1]{DT71} and \cite[Theorem 4.4.7]{HHbook}-- equation \eqref{def:weak} extends to all test functions $\varphi\in W^{1,B}_c(\Omega)$.

The first main result of this paper concerns the interior gradient H\"older regularity of solutions to \eqref{eq1}. Here and in what follows, $\mathrm{diam}(A)$ will denote the diameter of a set $A$, and $\mathrm{dist}(A,B)$ the distance between two sets $A,B$.
\begin{theorem}[Local $C^{1,\beta}$ regularity]\label{thm:interior}
 Let $u\in W^{1,B}_{loc}(\Omega)$ be a local weak solution to \eqref{eq1} under the assumptions \eqref{iasa}, \eqref{reg:A}-\eqref{A:hold}, \eqref{ab=1} and  
    $$f\in L^d_{loc}(\Omega),\quad d>n.$$
    Then there exists $\beta\in (0,1)$ depending on  $n,\l,\L,i_a,s_a,\alpha,d$ such that
    \begin{equation*}
        u\in C^{1,\beta}_{loc}(\Omega),
    \end{equation*}
    and for every $\Omega''\Subset \Omega'\Subset \Omega$, we have the quantitative estimate
\begin{equation}\label{stimaDuint}
                \begin{split}\|u\|_{C^{1,\beta}(\Omega'')} \leq C\bigg(n&,\l,\L,\L_{\mathrm{h}},i_a,s_a,\alpha,d,\|f\|_{L^d(\Omega')},
                \\
                &\mathrm{dist}(\Omega'',\partial \Omega'), \mathrm{diam}\,\Omega'',\int_{\Omega'}|u|\,dx+ \int_{\Omega'}B(|Du|)\,dx\bigg)\,.
       \end{split}
    \end{equation}
\end{theorem}
 We remark that in \eqref{stimaDuint} the dependence on
$\|f\|_{L^d(\Omega')}$ and
$\int_{\Omega'} |u|\,dx + \int_{\Omega'} B(|Du|)\,dx$
is only through an upper bound, as it will be clear from the proof.
\vspace{0.2cm}

We now move onto boundary regularity, and we first consider the Dirichlet boundary value problem. Our results are local in nature, so we consider a bounded domain $\mathcal{U}\subset \R^n$ such that $\partial \Omega\cap \mathcal{U}$ is relatively open in $\partial \Omega$; we study weak solutions to the boundary value problem
\begin{equation}\label{eq:dir1}
\begin{cases}
    -\diver\big(\A(x,Du) \big)=f\quad &\text{in  $\Omega\cap\mathcal{U}$}
    \\
    u=g\quad & \text{on $ \partial\Omega\cap\mathcal{U}$.}
    \end{cases}
\end{equation}
We say that a function $u\in W^{1,B}(\Omega\cap \mathcal{U})$ is a weak solution to \eqref{eq:dir1} if $u=g$ on $\partial \Omega\cap \mathcal{U}$ in the sense of traces, and 
\begin{equation}\label{weak:dir}
    \int_{\Omega\cap\,\mathcal{U}}\A(x,Du)\cdot D\varphi\,dx=\int_{ \Omega\cap\,\mathcal{U}}f\,\varphi\,dx,
\end{equation}
for all test functions $\varphi\in C^\infty(\mathcal{U})$ such that $\varphi=0$ on $\partial \Omega\cap \mathcal{U}$. Again, via a density argument, equation \eqref{weak:dir} extends to test functions $\varphi\in W^{1,B}(\mathcal{U})$ with zero trace on $\partial \Omega\cap \mathcal{U}$.
\vspace{0.2cm}

We now state our boundary  gradient regularity result for the Dirichlet problem \eqref{eq:dir1}.
In what follows, we indicate by $\mathcal{L}_\Omega=(L_\Omega,R_\Omega)$ the Lipschitz characteristic of the Lipschitz domain $\Omega$, which depends on the Lipschitz
constant $L_\Omega$ of the functions which locally describe $\partial \Omega$, and on the radius $R_\Omega$ of their ball domains. For the precise definition, we refer to Section \ref{subsec:domain}. 

Moreover, if $\partial\Omega\cap \mathcal{U}$ is of class $C^{1,\alpha}$, and $\mathcal{U}'\Subset\mathcal{U}$, we denote by $\|\partial \Omega \cap \mathcal{U}\|_{C^{1,\alpha}(\mathcal{U}')}$ the
$C^{1,\alpha}$-norm of the boundary functions locally representing $\partial \Omega\cap \mathcal{U}$, and whose associated coordinate cylinders cover $\partial\Omega\cap \mathcal{U}'$.
Accordingly,  $\|g\|_{C^{1,\alpha}(\partial\Omega\cap \bar{\mathcal{U}}')}$ stands for the $C^{1,\alpha}$-norm of $g$ when restricted to $\partial\Omega\cap \bar{\mathcal{U}}'$. The precise definitions and notations are given in Section \ref{subsec:domain}--see in particular \eqref{normadeom} and \eqref{norm:hg}.

\begin{theorem}[$C^{1,\beta}$ regularity, Dirichlet problems]\label{thm:dir}
    Let $\Omega\subset \R^n$ be a Lipschitz domain with Lipschitz characteristic $\mathcal{L}_\Omega=(L_\Omega,R_\Omega)$, and let $\,\mathcal{U}$ be a bounded domain of $\R^n$. Let  $u\in W^{1,B}(\Omega\cap \mathcal{U})$ be a weak solution to the Dirichlet problem \eqref{eq:dir1},  under the assumptions \eqref{iasa}, \eqref{reg:A}-\eqref{A:hold} and \eqref{ab=1}. Suppose that
    \begin{equation*}
        f\in L^d(\Omega\cap\mathcal{U})\,,\quad d>n,
    \end{equation*}
that
    \begin{equation*}
    \partial \Omega\cap \mathcal{U}\quad \text{is of class $C^{1,\alpha},$}
    \end{equation*}
    and that
    \begin{equation*}
        g\in C^{1,\alpha}\big(\partial \Omega\cap \mathcal{U}\big).
    \end{equation*}
Then there exists $\beta\in (0,1)$ determined by $n,\l,\L,i_a,s_a,\alpha,d,L_\Omega$, such that
\begin{equation*}
    u\in C^{1,\beta}( \overline{\Omega}\cap \mathcal{U}')\,,
\end{equation*}
 for every $ \mathcal{U}'\Subset \mathcal{U}$, with quantitative estimate
\begin{equation}\label{stimafin:Dir}
        \begin{split}
            \|u\|_{C^{1,\beta}(\bar{\Omega}\cap\,\mathcal{U}')}\leq C\bigg(n&,\l,\L,\L_{\mathrm{h}},i_a,s_a,\alpha,d,\mathcal{L}_\Omega,\|f\|_{L^d(\Omega\cap\,\mathcal{U})},\|g\|_{C^{1,\alpha}(\partial\Omega\cap \bar{\mathcal{U}}')},\mathrm{diam}\,(\mathcal{U}),
           \\
           &\mathrm{dist}(\mathcal{U}',\partial \mathcal{U}), \|\partial\Omega \cap \mathcal{U} \|_{C^{1,\alpha}(\mathcal{U}')}, \int_{\Omega\cap\, \mathcal{U}}|u|\,dx+\int_{\Omega\cap\, \mathcal{U}}B\big( |Du|\big)\,dx\bigg).
        \end{split}
    \end{equation}
\end{theorem}

We now move onto the Neumann (or co-normal) boundary value problem
\begin{equation}\label{eq:neu1}
    \begin{cases}
        -\diver\big(\A(x,Du) \big)=f\quad & \text{in $ \Omega\cap \mathcal{U}$}
        \\
        \A(x,Du)\cdot \nu=h\quad & \text{on $\partial\Omega\cap\mathcal{U}$,}
    \end{cases}
\end{equation}
where $\nu$ denotes the outer normal of $\partial \Omega$.

We say that   $u\in W^{1,B}(\mathcal{U}\cap \Omega)$ is a weak solution to  \eqref{eq:neu1} if
\begin{equation}
    \int_{ \Omega\cap \mathcal{U}}\A(x,Du)\cdot D\varphi\,dx=\int_{\Omega\cap \mathcal{U}}f\,\varphi\,dx+\int_{ \partial\Omega\cap \mathcal{U}} h\,\varphi\,d\mathcal{H}^{n-1}\,,
\end{equation}
for all test functions $\varphi\in W^{1,B}(\mathcal{U})$. Here $\mathcal{H}^{n-1}$ stands for the $(n-1)$ dimensional Hausdorff measure.  Our gradient regularity result for the co-normal problem  \eqref{eq:neu1} is the following.

\begin{theorem}[$C^{1,\beta}$ regularity, Neumann problems]\label{thm:neu}
    Let $\Omega\subset \R^n$ be a Lipschitz domain, with Lipschitz characteristic $\mathcal{L}_\Omega=(L_\Omega,R_\Omega)$, and let $\mathcal{U}\subset \R^n$ be a bounded domain.
    
    Suppose that $u\in W^{1,B}(\Omega\cap \mathcal{U})$ is a weak solution to the Neumann problem \eqref{eq:neu1} under the assumptions \eqref{iasa}, \eqref{reg:A}-\eqref{A:hold} and \eqref{ab=1}. Assume that 
\[
f\in L^d(\Omega\cap \mathcal{U}),\,\,d>n,
\]
that
 \begin{equation*}
    \partial \Omega\cap \mathcal{U}\quad \text{is of class $C^{1,\alpha},$}
    \end{equation*}
    and that 
    \[
    h\in C^{0,\alpha}(\partial\Omega\cap \mathcal{U}).
    \]
Then there exists $\beta\in (0,1)$ depending on $n,\l,\L,i_a,s_a,\alpha,d,L_\Omega$, such that
\begin{equation*}
    u\in C^{1,\beta}(\overline{\Omega}\cap \mathcal{U'})\,,
\end{equation*}
     for every $\mathcal{U}'\Subset \mathcal{U}$,  with quantitative estimate
\begin{equation}\label{stimafin:Neu}
        \begin{split}
            \|u\|_{C^{1,\beta}(\bar{\Omega}\cap\,\mathcal{U}')}\leq C\bigg(n&,\l,\L,\L_{\mathrm{h}},i_a,s_a,\alpha,d,\mathcal{L}_\Omega,\|f\|_{L^d(\Omega\cap\,\mathcal{U})},\|h\|_{C^{0,\alpha}(\partial\Omega\cap \bar{\mathcal{U}}')},\mathrm{diam}\,(\mathcal{U}),
           \\
           &\mathrm{dist}(\mathcal{U}',\partial \mathcal{U}), \|\partial\Omega \cap \mathcal{U} \|_{C^{1,\alpha}(\mathcal{U}')}, \int_{\Omega\cap\, \mathcal{U}}|u|\,dx+\int_{\Omega\cap\, \mathcal{U}}B\big( |Du|\big)\,dx\bigg).
        \end{split}
    \end{equation}
\end{theorem}

We  emphasize that in estimates \eqref{stimafin:Dir} and \eqref{stimafin:Neu}, the dependence on
\[
\|f\|_{L^d(\Omega\cap \mathcal{U})},\,\|\partial\Omega \cap \mathcal{U} \|_{C^{1,\alpha}(\mathcal{U}')},\, 
\|g\|_{C^{1,\alpha}(\partial\Omega\cap \overline{\mathcal{U}}')}, \,\|h\|_{C^{0,\alpha}(\partial\Omega\cap \overline{\mathcal{U}}')},
\, \text{ and }\,  
\int_{\Omega\cap\, \mathcal{U}}\Big\{|u|+B\big( |Du|\big)\Big\}\,dx
\]
enters only through \emph{upper bounds} on these quantities, as it will be clear from the proofs. 
As for the dependence on the geometric parameters 
\(\mathcal{L}_\Omega=(L_\Omega,R_\Omega)\), this enters through an \emph{upper bound on \(L_\Omega\)}
and a \emph{lower bound on \(R_\Omega\)}, in view of the proof
and Definition~\ref{def:Lom},
\vspace{0.2cm}

Finally, by combining the results of Theorems \ref{thm:interior}-\ref{thm:neu},  we will immediately infer the following two corollaries.

\begin{corollary}[Global $C^{1,\beta}$ regularity, Dirichlet problems]\label{cor:dirichlet}
    Let $\Omega$ be a bounded domain of class $C^{1,\alpha}$, with Lipschitz characteristic $\mathcal{L}_\Omega=(L_\Omega,R_\Omega)$. Let $u\in W^{1,B}(\Omega)$ be the weak solution to the Dirichlet problem
    \begin{equation}\label{eq:dir2}
        \begin{cases}
            -\diver\big(\A(x,Du) \big)=f\quad & \text{in $\Omega$}
            \\
            u=g\quad&\text{on $\partial \Omega$,}
        \end{cases}
    \end{equation}
 under the assumptions  \eqref{iasa}, \eqref{reg:A}-\eqref{A:hold}, \eqref{ab=1}, and suppose that 
 \[
 f\in L^d(\Omega),\, d>n\quad\text{and}\quad g\in C^{1,\alpha}(\R^n).
 \]
Then there exists $\beta\in (0,1)$ depending on $n,\l,\L,i_a,s_a,\alpha, d, L_\Omega $ such that $u\in C^{1,\beta}(\overline{\Omega})$, with quantitative estimate
\begin{equation}\label{est:globdirichlet}
   \|u\|_{C^{1,\beta}(\overline{\Omega})}\leq C\bigg(n,\l,\L,\L_\mathrm{h},i_a,s_a,\alpha,d,\mathcal{L}_\Omega,\|\partial\Omega\|_{C^{1,\alpha}}, \|g\|_{C^{1,\alpha}(\rn)}, \|f\|_{L^d(\Omega)} \bigg)\,.
\end{equation}
\end{corollary}

\begin{corollary}[Global $C^{1,\beta}$ regularity, Neumann problems]\label{cor:neu}
    Let $\Omega$ be a bounded domain of class $C^{1,\alpha}$, with Lipschitz characteristic $\mathcal{L}_\Omega=(L_\Omega,R_\Omega)$. Suppose that  assumptions \eqref{iasa}, \eqref{reg:A}-\eqref{A:hold}, \eqref{ab=1} are in force, and let $u\in W^{1,B}(\Omega)$ be the weak solution to the co-normal problem 
    \begin{equation}\label{eq:neu2}
        \begin{cases}
            -\diver\big(\A(x,Du) \big)=f\quad & \text{in $\Omega$}
\\
            \A(x,Du)\cdot \nu=h\quad & \text{on $\partial\Omega$.}
        \end{cases}
    \end{equation}
satisfying
\begin{equation}\label{zero:mean}
    \int_\Omega u\,dx=0.
\end{equation}
    Assume that $h\in C^{0,\alpha}(\partial \Omega)$, and that $f\in L^d(\Omega)$, $d>n$, with compatibility condition
\begin{equation*}
    \int_\Omega f\,dx+\int_{\partial \Omega} h\,d\mathcal{H}^{n-1}=0.
\end{equation*} 
 Then there exists $\beta\in (0,1)$ depending on $n,\l,\L,i_a,s_a,\alpha,d,L_\Omega$, such that  $u\in C^{1,\beta}(\overline{\Omega})$, with quantitative estimate
\begin{equation}\label{est:neucoroll}
   \|u\|_{C^{1,\beta}(\overline{\Omega})}\leq C\bigg(n,\l,\L,\L_\mathrm{h},i_a,s_a,\alpha,d,\mathcal{L}_\Omega,\|\partial\Omega\|_{C^{1,\alpha}}, \|h\|_{C^{0,\alpha}(\partial \Omega)}, \|f\|_{L^d(\Omega)} \bigg)\,.
\end{equation}
\end{corollary}

\subsection*{Key features of the operator}
One of the main properties of the operator \eqref{eq1} is 
the so-called \emph{uniform ellipticity} in the sense of Ladyzhenskaya--Ural'tseva, which is ensured by assumption~\eqref{ass:A}.
 More precisely, setting
\begin{equation}\label{def:nxiA}
    \nabla_\xi \A(x,\xi)
    = \big(\nabla_\xi \A(x,\xi)\big)_{i,j=1}^n
    := \frac{\partial \A^i}{\partial \xi_j}(x,\xi),
\end{equation}
and assuming for simplicity that $\nabla_\xi \A(x,\xi)$ is a symmetric matrix for every $x \in \overline{\Omega}$ and $\xi\neq 0$, uniform ellipticity can be expressed as
\begin{equation}\label{def:unifell}
   \sup_{|\xi|\geq 1} \Bigg(\frac{
    \sup_{x \in \overline{\Omega}}
    \bigl(\text{largest eigenvalue of } \nabla_\xi \A(x,\xi)\bigr)
    }{
    \inf_{x \in \overline{\Omega}}
    \bigl(\text{smallest eigenvalue of } \nabla_\xi \A(x,\xi)\bigr)
    }\Bigg)
    < \infty.
\end{equation}

Another important feature is the power-type control of the operator provided by assumption~\eqref{iasa}. In particular, the rightmost inequality in~\eqref{iasa} ensures that the operator grows at most polynomially, whereas the leftmost inequality guarantees that the operator stays away from the $1$-Laplacian, whose analytical behavior is markedly different.

Let us also remark that by now there exists a vast literature on \emph{nonuniformly elliptic operators}, namely operators for which condition~\eqref{def:unifell} fails. Providing a comprehensive list of contributions would be far beyond the scope of this paper; we therefore refer the reader to \cite{simon76, M89, M91, M93, CM15, BCM18, BS24, S24, DfM23, M23, CMMP23, CMM24, DfM25, MNP26, SY26, SYZ26} and the references therein. For problems exhibiting the so-called \emph{nearly linear growth}, corresponding to the case $i_a=-1$ (for instance, the logarithmic Young function $B(t)=t\log(1+t)$), we refer to \cite{DfM231, DfDfP25, DfP24} and the references therein.
\vspace{0.1cm}

We conclude by highlighting a further key feature of our problem, namely that the degeneracy of the operator is confined to the set of critical points $\{Du=0\}$. To illustrate this fact, we observe that the first step of the proof consists in differentiating equation~\eqref{eq1} in the autonomous case $\A(x,\xi)=\A(\xi)$ and $f=0$. In this setting, each partial derivative $D_k u$ is a weak solution of
\begin{equation*}
    -\mathrm{div}\bigl(\nabla_\xi \A(Du)\,D(D_k u)\bigr)=0,
    \qquad k=1,\dots,n.
\end{equation*}
By assumption~\eqref{ass:A}, the coefficient matrix satisfies $\nabla_\xi \A(Du)\approx a(|Du|)\,\mathrm{Id}$, and hence the equation degenerates exclusively at points where $|Du|=0$.

This behavior is in sharp contrast with that of the so-called \emph{orthotropic $p$-Laplace operator}, whose degeneracy occurs separately along each coordinate direction. Indeed, for orthotropic operators the degeneracy takes place independently on each set $\{D_k u = 0\}$, $k=1,\dots,n$. For results in this direction, we refer to \cite{BB18, BB20, BBCV18, BBC24}.

\subsection*{Examples of Young functions and admissible operators} As discussed earlier, power-type functions of the form
 \begin{equation*}
     B(t)=\frac{t^p}{p},\quad p>1
 \end{equation*}
satisfy condition \eqref{iasa}; indeed, in this case one has
\(i_a=s_a=p-2>-1\).
Another Young function is given by
\begin{equation}\label{Bmu}
    B_\mu(t)=\frac{1}{p}(\mu^2+t^2)^{\frac{p}{2}},\quad \text{so that}\quad a_\mu(t)=(\mu^2+t^2)^{\frac{p-2}{2}}
\end{equation}
for $\mu\neq 0$ and $p>1$. In this case, a simple computation shows that 
\[
-1<\min\{p-2,0\}\leq i_{a_\mu}\leq s_{a_\mu}\leq \max\{p-2,0\}<\infty,
\]
and the corresponding operator \eqref{eq:example} is the so-called \textit{nondegenerate
  $p$-Laplacian}
\begin{equation}
    -\mathrm{div}\big((\mu^2+|Du|^2)^{\frac{p-2}{2}}Du \big).
\end{equation}

Other admissible functions are obtained by multiplying with powers of logarithms, that is
\begin{equation*}
    B(t)=t^p\log^q (c+t),\quad p>1,\,q\in \R,
\end{equation*}
with $c\geq 1$ large enough for $B(t)$ to be convex (this can be checked via elementary computations; in particular, one may take $c=1$ if $q\geq 1$).
More elaborated instances, borrowed from \cite{M93, Ta91}, are
\[
\begin{split}
    &B(t)=t^3\,\big(1+(\log t)^2\big)^{-\frac{1}{2}} \exp\big(\log t\,\arctan(\log t) \big);
    \\
    &B(t)=t^{4+\sin\sqrt{1+(\log t)^2}},
\end{split}
\]
and we also refer to \cite[Eq. (1.9)]{M93},   \cite[Section 2.3]{BL17} and \cite[pp. 314]{L91} for a wide class of examples of Young functions fulfilling \eqref{iasa}. Moreover, as in \eqref{Bmu}, for any function $a(t)$ satisfying \eqref{iasa}, one can consider its nondegenerate counterpart
\[
a_\mu(t)=a\Big(\sqrt{\mu^2+t^2} \Big),\quad B_\mu(t)=\int_0^t a_\mu(s)\,s\,ds,\quad\mu\neq 0
\]
in which case, via a simple computation, we have
\[
\frac{a'_\mu(t)\,t}{a_\mu(t)}=\frac{a'\Big(\sqrt{\mu^2+t^2} \Big)\,\sqrt{\mu^2+t^2} }{a\Big(\sqrt{\mu^2+t^2} \Big)}\Big(\frac{t^2}{\mu^2+t^2}\Big)
\]
so that, owing to \eqref{iasa}, it immediately follows that 
\[
-1<\min\{i_a,0\}\leq i_{a_\mu}\leq s_{a_\mu}\leq \max\{s_a,0\}<\infty.
\]
The corresponding operator in \eqref{eq:example} $-\Delta_{B_\mu}u$ is called \textit{nondegenerate Orlicz-Laplacian}.
\vspace{0.1cm}

As a last example of admissible Young function one may consider the \textit{frozen double phase function}
\[
B(t)=\frac{t^p}{p}+a_0\,\frac{t^q}{q},\quad\text{with $a_0\geq 0$ and $1<p\leq q<\infty$ fixed, }
\]
whose corresponding operator \eqref{eq:example}
\[-\Delta_B u=-\mathrm{div}\big(|Du|^{p-2}Du+a_0\,|Du|^{q-2}Du \big)\]
is usually called \textit{frozen double phase operator.}

In this case, recalling \eqref{Bb}, a simple computation shows that the function 
\[
h(t):=\frac{b'(t)t}{b(t)}=\frac{(p-1)+a_0\,(q-1)\,t^{q-p}}{1+a_0\,t^{q-p}}
\]
satisfies $h'(t)\geq 0$ for any $a_0>0$. Consequently $(p-1)\leq h(t)\leq (q-1)$ for all $t\geq0$,  which in turn implies   
(see \eqref{def:isb}-\eqref{ibsb} below)
$i_a=(p-2),\, s_a=(q-2)$; in particular, these quantities are independent of $a_0\geq0$.
\vspace{0.2cm}

Regarding admissible operators, the prototypical example is the
Orlicz-Laplace operator \eqref{eq:example}, which arises as the Euler-Lagrange
equation associated with the functional
\begin{equation}\label{F:euclid}
\mathcal{F}(v)
=
\int_\Omega B\bigl(|Dv|\bigr)\,dx
-
\int_\Omega f\,v\,dx,
\qquad v\in W^{1,B}(\Omega),
\end{equation}
and which features rotational invariance.

More generally, one may consider anisotropic operators, which lack rotational
invariance. Specifically, let \(H=H(\xi)\colon \mathbb{R}^n\to[0,\infty)\) be an
anisotropy, i.e., a norm on $\R^n$ of class
\(C^2(\mathbb{R}^n\setminus\{0\})\). Given such a function \(H\), we consider
vector fields of the form\footnote{Here,
\(\nabla_\xi F(\xi)\) denotes the differentiation of \(F(\xi)\) with respect to the
variable \(\xi\). The notation $\nabla_\xi B\big(H(Dv) \big)$ stands for the evaluation $\nabla_\xi B\big(H(\xi) \big)|_{\xi=Dv}$.}
\begin{equation}\label{A:t4mpv}
\mathcal{A}(\xi)
=
\nabla_\xi B\bigl(H(\xi)\bigr)
=
a\bigl(H(\xi)\bigr)\,\frac{1}{2}\nabla_\xi H^2(\xi),
\end{equation}
which lead to the so-called \emph{anisotropic (or Finsler) Orlicz--Laplace}
equation
\begin{equation}\label{anis:Orlicz}
-\Delta^H_B u
:=
-\mathrm{div}\Bigl(\nabla_\xi B\bigl(H(Du)\bigr)\Bigr)
=
f.
\end{equation}

In order for the operator to be admissible, we impose natural 
ellipticity assumptions on the anisotropy \(H\). Specifically, we assume that
\begin{equation}\label{H:elliptic}
\lambda\,|\eta|^2
\le
\frac{1}{2}\nabla_\xi^2 H^2(\xi)\,\eta\cdot\eta
\le
\Lambda\,|\eta|^2,
\qquad
\text{for all }\xi\in\mathbb{R}^n\setminus\{0\},\ \eta\in\mathbb{R}^n,
\end{equation}
for some constants \(0<\lambda\le\Lambda<\infty\).
Under this assumption, the structural condition \eqref{ass:A} is satisfied-- see
\cite[Equation~(3.8)]{ACCFM25}. We also note that equation \eqref{anis:Orlicz} arises as the Euler--Lagrange equation
associated with the variational functional
\begin{equation*}
\mathcal{F}_H(v)
=
\int_\Omega B\bigl(H(Dv)\bigr)\,dx
-
\int_\Omega f\,v\,dx,
\qquad v\in W^{1,B}(\Omega),
\end{equation*}
which reduces to \eqref{F:euclid} when \(H\) coincides with the Euclidean norm.

Anisotropic-type operators of the form~\eqref{anis:Orlicz} have recently attracted considerable attention, motivated by their wide range of applications and by their appearance in physical models. For results in this direction, we refer to \cite{CS09, CFV16, CRS19, CFR20, DPV22, CL22, ACF23, CL24, EMSV24, FSV24, BMV25, BERV25, V25}; see also \cite{GM23, MM24} for a wider class of equations.

Further admissible operators can be obtained by multiplying the vector
field \eqref{A:t4mpv} by a strictly positive H\"older continuous coefficient.
That is, one may consider equations of the form
\begin{equation*}
-\mathrm{div}\bigl(c(x)\,\nabla_\xi B\bigl(H(Du)\bigr)\bigr)
=
f,
\end{equation*}
where \(c(x)\in C^{0,\alpha}(\overline{\Omega})\) satisfies
\(0<\lambda\le c(x)\le\Lambda\). Another admissible class of nonautonomous operators is given by
\begin{equation*}
    -\mathrm{div}\Big(\big(M(x)\,Du\cdot Du\big)^{\frac{p-2}{2}}\,M(x)Du \Big)=f,\quad p>1,
\end{equation*}
or, more generally, by its Orlicz-type counterpart
\begin{equation*}
     -\mathrm{div}\Big(\psi'\big(M(x)Du\cdot Du \big)\,M(x)Du \Big)=f
\end{equation*}
where $\psi$ is defined by \eqref{def:psi}. Above, the matrix  $M=(M_{ij})\in C^{0,\alpha}(\overline{\Omega})$ satisfies the ellipticity and growth conditions
\begin{equation}\label{M:elliptic}
    M(x)\,\eta\cdot\eta\geq \l\,|\eta|^2,\quad\text{and}\quad \sum_{i,j=1}^n|M_{ij}(x)|\leq \L\quad\text{for all $x\in \overline{\Omega}$, and all $\eta\in \R^n$.}
\end{equation}
\vspace{0.1cm}

We conclude this short section by noting that our results also apply to $u$-dependent operators of the form
\begin{equation}\label{general}
    -\mathrm{div}\big(\A(x,u,Du) \big)=f
\end{equation}
where $\A \colon \overline{\Omega}\times\R\times\R^n \to \R^n$ is continuous, $\A(x,u,\cdot)\in C^1(\R^n\setminus\{0\})$ satisfying the growth and coercivity conditions
\begin{equation*}
    \sum_{i,j=1}^n\left|\frac{\partial \A^i}{\partial \xi_j}(x,u,\xi)\right|
    \le \L\, a(|\xi|),
    \qquad
    \sum_{i,j=1}^n \frac{\partial \A^i}{\partial \xi_j}(x,u,\xi)\,\eta_i\,\eta_j
    \ge \l\, a(|\xi|)\,|\eta|^2,
\end{equation*}
for all $\xi\in\R^n\setminus\{0\}$, $\eta\in\R^n$, for all $(x,u)\in\overline{\Omega}\times\R$, and enjoying the H\"older continuity 
\begin{equation*}
    |\A(x,u,\xi)-\A(y,v,\xi)|
    \le \L_{\mathrm{h}}\bigl(1+b(|\xi|)\bigr)
    \bigl(|x-y|^{\alpha}+|u-v|^{\alpha}\bigr),
    \qquad
    |\A(x,u,0)|\le \L,
\end{equation*}
for all $x,y\in\overline{\Omega}$ and $u,v\in\R$.

Under these assumptions, one first establishes the H\"older continuity of weak solutions $u$. As a consequence, the problem reduces to the case~\eqref{eq1}, allowing us to apply our results. Given the already considerable length of this paper, we decided to omit the proof, and we refer to \cite{S64, T67, L91, G03} for classical results on H\"older regularity in this general setting.

\subsection*{Overview of the literature}
Regularity theory for quasilinear elliptic equations is by now a classical subject. 
Without any claim of exhaustiveness, we briefly outline some of its historical development. 
The modern theory traces back to the seminal works of De~Giorgi, Nash, and Moser 
\cite{DG56,N58,M60}, who established interior H\"older regularity for solutions to the 
linear equation
\begin{equation}\label{divAxdu}
    -\mathrm{div}\big(A(x)\,Du \big)=0
\end{equation}
with bounded and uniformly elliptic, measurable coefficients $A(x)=\{A_{ij}(x)\}$.

Building on these ideas, Ladyzhenskaya and Ural'tseva established a comprehensive
regularity theory for a wide class of quasilinear elliptic equations--see
\cite{LU68}. In particular, they proved interior and boundary gradient regularity
results for nondegenerate quasilinear equations of the type \eqref{general}
with vector field $\A$ differentiable
in the coefficients \((x,u)\), and satisfying coercivity and growth assumptions modeled upon the nondegenerate $p$-Laplace equation:
\begin{equation}\label{nondegenerate}
    -\mathrm{div}\Big(\big(\mu^2+|Du|^2\big)^{\frac{p-2}{2}}Du \Big)=f,\quad \text{for $p>1$ and $\mu\neq 0$.}
\end{equation}
Note that the parameter $\mu\neq 0$ guarantees the non-degeneracy of the operator.

For equations with nondifferentiable, but merely H\"older continuous coefficients,  gradient regularity in the
quadratic case \(p=2\) was later obtained by Giaquinta and Giusti
\cite{GG83,GG84}, using the so-called \emph{perturbation argument}.

Turning to genuinely degenerate equations, interior \(C^{1,\beta}\)-regularity for
solutions to \(p\)-Laplace type equations  was established by
Ural'tseva \cite{U68} and Uhlenbeck \cite{U77} for (vectorial) solutions of
\eqref{eq:plapl} with vanishing right-hand side and superquadratic growth
\(p\ge 2\).

A different proof, based on the so-called \emph{fundamental alternative}, was later
provided by Evans \cite{E82}--see also \cite{L83, L88} for alternative proofs. Building on all these ideas, local \(C^{1,\beta}\)-regularity
was subsequently extended to general \(p\)-Laplace type equations, $p>1$, of the form
\eqref{general} by DiBenedetto, Tolksdorf, Manfredi and Lieberman
\cite{DB83,T84,M86,M88, L93}. Moreover, this regularity is, in general, optimal even in the homogeneous case \(f \equiv 0\)--see \cite{IM89,BBDS26}. We also refer to the recent papers \cite{ATU17,ATU18,WY25} which address the  H\"older exponent in the inhomogeneous case $f\not\equiv 0$.

Later developments, based on potential-theoretic techniques, led to sharp
assumptions on the right-hand side \(f\) and on the coefficients ensuring interior
gradient estimates and continuity of solutions to \eqref{eq:plapl}. This line of
research was pursued in a series of papers by Duzaar, Mingione, and Kuusi
\cite{DM10,DM100,M11,KM12,KM13,KM14,KM18}-- see also the survey \cite{KM141}. We further
refer to \cite{Ok16,Ok17,BK17,B25} and the references therein for related results on the
\(p(x)\)-Laplace equation.

Concerning quasilinear equations with Orlicz growth, namely those satisfying
\eqref{iasa} and \eqref{ass:A}, interior H\"older continuity of the gradient was
proved by Lieberman \cite{L91}-- see also the paper of Simon \cite[Section 4, pp 844-847]{simon76}, and Marcellini's paper \cite[Section 7]{M93} in which Lipschitz continuity of solutions is established for a general class of nonuniformly elliptic equation as well. 

Local potential estimates in the case \(i_a\ge 0\)
were obtained by Baroni \cite{B15}-- see also \cite{CKW23} for the
vectorial setting. For equations with generalized Orlicz growth and finer assumptions on the coefficients, we refer to the recent works of
H{\"a}st{\"o}, Lee, and Ok \cite{HO221,HO22,HO23,HLO25}-- see also \cite{BaB25} for results concerning minima of Orlicz multi-phase type functionals.

Overall, while the literature on interior gradient regularity is extensive,
boundary gradient regularity to \eqref{eq1} has received comparatively less attention. H\"older
continuity of the gradient up to the boundary was established in \cite{GG84} for quadratic operators, and in \cite{L88} for general
\(p\)-Laplace type equations-- see also \cite{F07} for the \(p(x)\)-Laplacian, and
\cite{L90,L931, DB94, BO15,BDNS22} and references therein concerning the
parabolic $p$-Laplace operator. The result was also established by Lin-Li \cite{LL91} in the context of the obstacle problem for $p$-harmonic functions. Recently, a viscosity approach was used by \cite{GYZ26} to establish gradient boundary regularity for the Neumann problem (see \cite[Theorem 4.2]{GYZ26}).
For other results concerning viscosity solutions to $p$-Laplacian free boundary problems we refer to \cite{FL21,FL24, FL26} and references therein.

For isotropic operators \eqref{eq:example} 
and zero boundary datum, Cianchi and Maz'ya \cite{CM11,CM14,CM141, CVM15} established 
global Lipschitz regularity under sharp assumptions on $\partial \Omega$ and 
the right-hand side, both for the Dirichlet and Neumann problems. Similar results have been recently obtained in \cite{AC251} for anisotropic 
operators \eqref{anis:Orlicz}, and in \cite{BDMS22,DfP23} for $(p,q)$-growth operators  with isotropic structure.

We also point out  the very recent paper \cite{H26} dealing with Musielak-Orlicz operators more general than those considered here. There, analogous results are obtained under the additional assumption that solutions are bounded, and making use of the estimates for homogeneous equations  of \cite{L91} and those developed here in Sections \ref{sec:int0}, \ref{sec:dirhomog}, and \ref{sec:neuhom}.

\subsection*{Main novelties and ideas of the proofs}
Amongst the main contributions of this paper is the  global gradient 
regularity to general quasilinear equations with Orlicz growth, both for the Dirichlet and the Neumann boundary value problems. This is the natural generalization of the results in \cite{L88} to the Orlicz setting. 
Our estimates \eqref{stimafin:Dir}, \eqref{stimafin:Neu} also provide explicit dependence of the quantitative constants with respect to the structural parameters of the operator and the geometric properties of the reference domain $\Omega$. This feature makes the results particularly well-suited for compactness arguments involving this class of operators.

Moreover, in the homogeneous case $f \equiv 0$, we establish explicit $L^\infty$--$L^1$ and oscillation decay estimates for the gradient of solutions as described below; this is the content of Theorems \ref{thm:inthom}, \ref{thm:gradddir} and \ref{thm:homneu}. These estimates are of independent interest since they can be leveraged to obtain potential-type bounds for boundary value problems as well.

\vspace{0.1cm}

Let us briefly outline the main steps and novelties of the proofs; a substantial portion of the paper is devoted to establishing \(C^{1,\alpha}\)-regularity for solutions \(v\) of the homogeneous, autonomous problem
\begin{equation}\label{www}
    -\mathrm {div}\big(\A(Dv)\big)=0 \quad \text{in } B_R.
\end{equation}
For what concerns the interior case (see Theorem \ref{thm:inthom}), we first prove boundedness of $Dv$ via the so-called \textit{Bernstein method}. Namely, we show   that the function $B\big(|Dv|\big)$ is a nonnegative subsolution of a strongly elliptic equation \eqref{divAxdu}, from which the $L^\infty$-estimate of $B(|Dv|)$ follows via the weak Harnack inequality. Combining this with a simple interpolation argument we obtain the $L^\infty$-$L^1$ estimate \eqref{inf:hom} for $Dv$-- see Proposition \ref{prop:bernstein}.

As a next step we prove the $C^{1,\alpha}$-regularity via the aforementioned fundamental alternative. We refer to Theorem ~\ref{thm:alternative} and Lemmas~\ref{lem:1alt}--\ref{lemma:altfin} for the precise statements and details. 

We stress that our proof of the alternative exploits \textit{new De Giorgi-type inequalities} (see Lemma \ref{lemma:levquadratic}), which greatly simplify the classical arguments of Evans, DiBenedetto and Tolksdorf \cite{E82, DB83,T84}, all of which heavily relied on the monotonicity of \(a(t)=t^{p-2}\), an assumption that may fail in our more general setting. Additionally, inspired by \cite{L88}, we provide another proof of the fundamental alternative via  Moser-type iterations. We briefly note that, to justify all the computations, an additional regularization procedure will be required--see Proposition~\ref{prop:veregular}  and Remark~\ref{rem:approx}.

Once the fundamental alternative is proven, the  H\"older continuity of $Dv$ follows in a standard way--see the discussion preceding Proposition \ref{exs:Dve}. Not only that, but by following ideas from \cite{L91,DM10}, in Theorem \ref{thm:inthom} we also establish the $L^1$-excess decay estimate
\begin{equation*}
      \mint_{B_r}|Dv-(Dv)_{B_r}|\,dx\lesssim  \left( \frac{r}{R}\right)^{\alpha} \mint_{B_R}|Dv-(Dv)_{B_R}|\,dx,\quad 0<r\leq R,
\end{equation*}
which is of independent interest as mentioned above. We remark that this inequality was already proven in \cite{B15} in the case $i_a\geq 0$.

Having the $C^{1,\alpha}$- estimates of solutions to \eqref{www} at our disposal, the interior regularity \eqref{stimaDuint} is then obtained via the perturbation method. The argument is nowadays standard, so we refer to the discussion preceding Proposition \ref{prop:ex} and to Section \ref{sec:thmint} for the details.
\vspace{0.1cm}

Moving to the proof of the global regularity Theorems \ref{thm:dir}-\ref{thm:neu},  we first reduce the problems \eqref{eq:dir1} and \eqref{eq:neu1} to the half ball $B_R^+$ via a flattening argument-see Section \ref{subsec:domain}. 
Then, as in the interior case, we establish $C^{1,\alpha}$-regularity  of solutions to the homogeneous problem
\begin{equation}\label{bdry:eqproblems}
    \begin{cases}
        -\mathrm{div}\big(\A(Dv)\big)=0\quad&\text{in $B^+_R$}
        \\
        v=g\quad &\text{on $B^0_R$}
    \end{cases}
\end{equation}
in the case of Dirichlet problems, or
\begin{equation}\label{bdry:eqneupn}
            \begin{cases}
        -\mathrm{div}\big(\A(Dv)\big)=0\quad&\text{in $B^+_R$}
        \\
        \A(Dv)\cdot e_n+h_0=0\quad &\text{on $B^0_R$}
    \end{cases}
\end{equation}
for Neumann boundary value problems, where $h_0$ is a real constant.

The global regularity of \eqref{bdry:eqproblems}, which is the content of Theorem \ref{thm:gradddir}, is obtained via a careful barrier argument, aimed at establishing bounds and oscillation estimates for the normal derivative \(D_n v\). Once these estimates are proven, the desired results  follow from the interior regularity, together with tangential control of the derivatives provided by the Dirichlet datum \(g\), and a suitable interpolation argument.  

Regarding the global $C^{1,\alpha}$-regularity  of the conormal problem \eqref{bdry:eqneupn}, we just remark that it is based on suitable modifications of the Bernstein method and of the fundamental alternative. We also establish the $L^1$-excess decay estimate for the Neumann problem~\eqref{bdry:eqneupn}--see Theorem \ref{thm:homneu}-- which, to the best of our knowledge, is completely new for this type of boundary value problems. We refer to Section~\ref{sec:neuhom} for details and the proof.

Finally, having established \(C^{1,\alpha}\)-regularity for solutions to \eqref{bdry:eqproblems}-\eqref{bdry:eqneupn}, the gradient H\"older continuity of solutions to \eqref{eq:dir1} and \eqref{eq:neu1} will follow once again via the perturbation argument.

\subsection*{Plan of the paper.}
The rest of the paper is organized as follows: in Section \ref{sec:prel} we introduce auxiliary results concerning Young functions, Orlicz Lebesgue and Sobolev spaces, the vector field $\A$, classes of regular domains, and provide auxiliary lemmas.

 Section \ref{sec:bounded} is devoted to the local and global boundedness of solutions.
 
 In Section \ref{sec:int0} we study the interior gradient regularity of solutions to homogeneous, autonomous equations \eqref{www}.
 
 In Section \ref{sec:trace} we establish bounds and oscillations estimates for $u/x_n$, where $u$ is solution to a uniformly elliptic linear equation in trace form. 
These results will be instrumental to the $C^{1,\alpha}$ proof of the homogeneous Dirichlet problem \eqref{bdry:eqproblems} in Section \ref{sec:dirhomog}. 

Section \ref{sec:neuhom} then deals with the global gradient regularity of the homogeneous Neumann problem \eqref{bdry:eqneupn}. In  the last three sections we provide the proof of the main results. Specifically, in Section \ref{sec:thmint} we prove Theorem \ref{thm:interior}, in Section \ref{sec:pfdir} we provide the proof of Theorem \ref{thm:dir} and Corollary \ref{cor:dirichlet}, and in Section \ref{sec:finNeum} we give the proof of Theorem \ref{thm:neu} and of Corollary \ref{cor:neu}.
Finally, in Appendix \ref{appendixA} we prove the interpolation  Lemma \ref{lemma:interpol}.

\subsection*{Notation} 

We denote points in $\R^n$ by $x=(x',x_n)$, where $x'\in \R^{n-1}$ and $x_n\in \R$. When the context is clear, the notation $x_0'$ will be used either to refer to points in $\R^{n-1}$ or to points in $\R^n$ lying on the hyperplane $\{x_n=0\}$, i.e., we identify $x_0' \equiv (x_0',0)$.  We write $\R^n_+=\{(x',x_n)\in \R^n:\,x_n>0\}$ for the upper half space. We denote by $x\cdot y$ the standard scalar product in $\R^n$, and $|x|=\sqrt{x\cdot x}$ the Euclidean norm of $x$. We denote by $e_i=(0,\dots,1,0,\dots,0)$ the $i$-th canonical unit vector of $\R^n$.
Also, for a given matrix $M=\{M_{ij}\}_{\substack{i=1,\dots,d\\
j=1,\dots,n}}$, we denote by $|M|=\sqrt{\sum_{i,j}|M_{i,j}|^2}$ the Frobenius norm of $M$.

We write $B_R(x_0)$ for the $n$-dimensional ball of radius $R>0$ centered at $x_0\in \R^n$, and when $x_0=0$ we simply write $B_R:=B_R(0)$. Similarly, $B'_R(x'_0)$ denotes the $(n-1)$-dimensional ball of radius $R$ centered at $x'_0\in \R^{n-1}$. We define the upper half-ball
\[
B^+_R(x'_0):=\{(x',x_n)\in \R^n : |x-x'_0|<R, \ x_n>0\},
\]
centered at $x'_0\in \R^{n-1}\times \{0\}$. When $x'_0=0$, we simply write $B^+_R:=B^+_R(0)$.  We also denote the flat part of its boundary by 
\[
B^0_R(x'_0):=\{(x',0)\in \R^n : |x-x'_0|<R\} \subset \R^{n-1}\times \{0\},
\]
so that $B^0_R(x'_0)\cong B'_R(x'_0)$. Accordingly, $B^0_R:=B^0_R(0)\cong B'_R$ when the center is the origin.
Furthermore, since all our estimates are local, we shall always assume that the radii satisfy $R\leq 1$.

Given a measurable set $A\subset \R^n$, we denote by $\partial A$ its boundary, $\bar A$ its closure, and $|A|$ its Lebesgue measure. Also, we write $\chi_A$ for its characteristic function:
\[
\chi_A(x) :=
\begin{cases}
1, & x\in A,\\
0, & x\notin A.
\end{cases}
\]

For a nonnegative Borel measure $\mu$ on $\R^n$, and $0<p\leq \infty$, the space $L^p(A;d\mu)$ denotes the set of $p$-integrable functions $F:A\to \R^N$  with respect to $\mu$, endowed with norm
\[
\|F\|_{L^p(A;d\mu)} := \Big( \int_A |F|^p\,d\mu\Big)^{1/p}.
\]
When $\mu$ is the Lebesgue measure, we simply write $L^p(A)$ and $\|f\|_{L^p(A)}$.
For $\beta\in (0,1)$, the H\"older norm is defined by
\[
\|F\|_{C^{0,\beta}(A)}=\sup_A|F|+\sup_{x\neq y}\frac{|F(x)-F(y)|}{|x-y|^\beta}.
\]
If $A\subset \R^n$ has positive measure and $F\in L^1(A)$, we define its average over $A$ by 
\[
(F)_A := \frac{1}{|A|}\int_A F\,dx = \mint_A F\,dx.
\]
For $a\in \R^d$, we write $\{F=a\} := \{x\in A : F(x)=a\}$ for the level set of $F$.
Similarly, if $F$ is real-valued, the sub- and superlevel sets are denoted by $\{F<a\}$, $\{F>a\}$, etc. We also write $\mathrm{spt}\,F$ for the support of $F$.  If $F = \{F_k\}_{k=1}^N$, we define its oscillation over $A$ by
\[
\operatorname*{osc}_A F := \max_{k=1,\dots,N} \Big( \sup_A F_k - \inf_A F_k \Big).
\]
We let $\rho$ denote a standard, radially symmetric mollifier, and for $\delta>0$ define the scaled kernel $\rho_\delta(x) := \delta^{-n} \rho(x/\delta)$. For a measurable function $F$, the notation $F*\rho_\delta$ denotes the convolution of $F$ with $\rho_\delta$. For a real measurable function $v$, we denote its positive and negative parts by
\[
v_+(x) := \max\{v(x),0\}, \qquad v_-(x) := -\min\{v(x),0\}.
\]
We write the gradient as $Dv = (D_1 v, \dots, D_n v)$, $D'v = (D_1 v, \dots, D_{n-1}v)$ the tangential gradient, and $D^2v$ the hessian matrix.  If $V = \{V^i\}_{i=1}^N : A \to \R^d$, we write its derivative as $DV=(DV)_{ij}=D_j V^i$.

\section{Preliminary results}\label{sec:prel}

\subsection{Young functions} \label{sec:young}
We consider a Young function $B$ satisfying \eqref{Bb} and \eqref{reg:Byoung}, and the function $a$ defined by \eqref{def:a} such that $a\in C^1(0,\infty)$ and \eqref{iasa}, \eqref{ab=1} hold.

Here we derive some elementary yet very useful properties of the functions
$a$, $b$, and $B$. Although these properties are well known to experts, we
include a brief proof for the reader's convenience.
 First, we have 
\begin{equation}\label{mon:iasa}
    \begin{split}
         &t\mapsto \frac{a(t)}{t^\mathfrak{i}} \quad\text{is a nondecreasing function for $\mathfrak{i}\leq i_a$}
        \\
        &t\mapsto \frac{a(t)}{t^\mathfrak{s}} \quad\text{is a nonincreasing function for $\mathfrak{s}\geq s_a$}.
    \end{split}
\end{equation}
In fact, by the leftmost inequality in \eqref{iasa}, for all $t>0$ we have
\begin{equation*}
        \frac{a'(t)}{a(t)} = (\log a(t))' \;\geq\; \frac{\mathfrak{i}}{t} = (\log t^{\mathfrak{i}})',
\end{equation*}
and therefore, integrating this inequality yields \eqref{mon:iasa}$_1$. 
The proof of \eqref{mon:iasa}$_2$ proceeds in the same way by using the rightmost inequality of \eqref{iasa}.
\vspace{0.2cm}

In particular, given two fixed constants $m_0,M_0>0$, from \eqref{mon:iasa} one immediately deduces the following very useful estimate\footnote{We anticipate that inequality \eqref{ultra:utile} will be repeatedly used along the proof of the fundamental alternative; see Theorem \ref{thm:alternative}, Lemma \ref{lemma:levquadratic}, and Lemmas  \ref{lem:alttang}-\ref{lemma:alt2Dn}. This is the key estimate that allows us to prove the gradient H\"older continuity of  solutions, rather than  the monotonicity of $a(\cdot)$ which may fail in this setting. }:
\begin{equation}\label{ultra:utile}
    c_0\leq\frac{a(t_2)}{a(t_1)}\leq C_0\,,\quad \text{for all $t_1,t_2>0$ such that $m_0\leq\frac{t_2}{t_1}\leq M_0$}
\end{equation}
where $c_0,C_0>0$ depend on $i_a,s_a,m_0,M_0$.
\vspace{0.2cm}

Next, by setting
\begin{equation}\label{def:isb}
    i_b=i_a+1,\quad\text{and}\quad s_b=s_a+1\,,
\end{equation}
and since $  \frac{b'(t)\,t}{b(t)}=\frac{a'(t)\,t}{a(t)}+1$, from \eqref{iasa} we deduce
\begin{equation}\label{ibsb}
    \begin{split}
            0<i_b\leq\inf_{t>0}\frac{b'(t)\,t}{b(t)}\leq \sup_{t>0} \frac{b'(t)\,t}{b(t)}\leq s_b<\infty\,,
    \end{split}
\end{equation}
Therefore, as in \eqref{mon:iasa}, we infer
\begin{equation}\label{mon:ibsb}
    \begin{split}
        &t\mapsto \frac{b(t)}{t^{\mathfrak{i}_b}} \quad\text{is a nondecreasing (increasing) function for $\mathfrak{i}_b\leq i_b$ ($\mathfrak{i}_b< i_b$)}
        \\
        &t\mapsto \frac{b(t)}{t^{\mathfrak{s}_b}} \quad\text{is a nonincreasing (decreasing) function for $\mathfrak{s}_b\geq s_b$ ($\mathfrak{s}_b> s_b$)}.
    \end{split}
    \end{equation}
In particular, as $i_b>0$, we have that 
\begin{equation*}
    t\mapsto b(t)\quad\text{is strictly increasing, and}\quad\lim_{t\to +\infty}b(t)=+\infty
\end{equation*}

From Equations \eqref{mon:ibsb} we immediately deduce
\begin{equation}\label{b2t}
\begin{split}
    C^{i_b}b(t)\leq b(C\,t)\leq C^{s_b}b(t)\,,\quad &\text{for all $C>1$}
    \\
    c^{s_b}b(t)\leq b(c\,t)\leq c^{i_b}b(t)\quad &\text{for all $c<1$,}
    \end{split}
\end{equation}
which implies the (quasi-)triangle inequality for $b(t)$:
\begin{equation}\label{triangleb}
    b(t+s)\leq 2^{s_b}\left\{b(t)+b(s)\right\},\quad t,s\geq 0.
\end{equation}
In fact, assuming without loss of generality that $t\geq s$ and using the monotonicity of $b(t)$, we get 
\begin{equation*}
b(t+s)\leq b(2t)\stackrel{\eqref{b2t}}{\leq} 2^{s_b} b(t)\leq 2^{s_b}\left\{b(t)+b(s)\right\}    
\end{equation*}
that is \eqref{triangleb}. We now claim that the functions $B(t)$ and $b(t)\,t$ are equivalent, that is
\begin{equation}\label{Bcomeb}
    \left(\frac{1}{1+s_b}\right)b(t)\,t\leq B(t)\leq \left(\frac{1}{1+i_b}\right)b(t)\,t,\quad \text{for $t>0$.}
\end{equation}
Indeed, via integration by parts and since $b(0)=0$, we have
\begin{equation*}
    B(t)=\int_0^tb(s)\,ds=t\,b(t)-\int_0^ts\,b'(s)\,ds\stackrel{\eqref{ibsb}_2}{\geq} t\,b(t)-s_b\,B(t),
\end{equation*}
hence the first inequality in \eqref{Bcomeb} follows. The second inequality of \eqref{Bcomeb} follows similarly via \eqref{ibsb}$_1$. From \eqref{Bcomeb} and \eqref{ab=1}, we also get
\begin{equation}\label{B=1}
   \frac{1}{2+s_a}\leq  B(1)\leq \frac{1}{2+i_a}.
\end{equation}
Let us now set
\begin{equation}\label{def:isB}
    i_B=i_b+1,\quad\text{and}\quad s_B=s_b+1\,.
\end{equation}
Taking advantage of $B'(t)=b(t)$, $i_b>0$ and \eqref{Bcomeb}, it is easy to see that
\begin{equation}\label{iBsB}
    1<i_B\leq \inf_{t>0}\frac{B'(t)t}{B(t)}\leq \sup_{t>0}\frac{B'(t)t}{B(t)}\leq s_B<\infty\,,
\end{equation}
hence, in the same way as \eqref{mon:iasa}, we get
\begin{equation}\label{mon:iBsB}
    \begin{split}
        &t\mapsto \frac{B(t)}{t^{\mathfrak i_B}} \quad\text{is a nondecreasing  function for all $\mathfrak i_B\leq i_B$}
        \\
        &t\mapsto \frac{B(t)}{t^{\mathfrak s_B}} \quad\text{is a nonincreasing function for all $\mathfrak{s}_B\geq s_B$}.
    \end{split}
\end{equation}
As in \eqref{b2t}, from Equations \eqref{mon:iBsB} we immediately infer the so-called $\Delta_2$ and $\nabla_2$-properties: 
\begin{equation}\label{B2t}
\begin{split}
    &C^{i_B}B(t)\leq B(C\,t)\leq C^{s_B}B(t)\,,\quad \text{for all $C>1$}
    \\
    &c^{s_B}B(t)\leq B(c\,t)\leq c^{i_B}B(t)\quad \text{for all $c<1$.}
    \end{split}
\end{equation}
We refer to \cite[Chapter 4]{funcsp} for further details concerning these two properties. From \eqref{B2t}, with the very same proof of \eqref{triangleb}, we obtain the (quasi-)triangle inequality for $B(t)$:
\begin{equation}\label{triangle}
    B(t+s)\leq 2^{s_B}\left\{B(t)+B(s)\right\},\quad t,s>0.
\end{equation}

\noindent Another simple inequality is 
\begin{equation}\label{simple}
    t\leq C(i_a,s_a)\,B(t)+1 ,\quad t>0.
\end{equation}
Indeed, since $i_B=i_a+2>1$ and recalling \eqref{B=1}, from \eqref{mon:iBsB} we have  $c(i_a,s_a)\,t\leq B(1)\,t\leq B(t)$ for $t\geq 1$, while for $t<1$ Equation \eqref{simple} is trivial.
\vspace{0.3cm}

Another elementary, yet very useful inequality,\footnote{Inequality
\eqref{el:comparison} is a versatile tool for proving the so-called
comparison estimates; see
Propositions~\ref{prop:comparison}, \ref{prop:comparisondir}, and
\ref{prop:comparisonNEU} below. It is a generalization of the classical $p$-coercivity estimates for power-type nonlinearities
\cite[Lemma~1]{T84}.} is the following:
\begin{equation}\label{el:comparison}
    B(|\xi-\eta|)\leq C\,\delta\big[B(|\xi|)+B(|\eta|) \big]+C\,\delta^{-1}\,a\big(|\xi|+|\eta| \big)\,|\xi-\eta|^2,
\end{equation}
valid for all $\delta \in (0,1]$, for all $\xi,\eta\in \R^n$, with $C=C(i_a,s_a)$. 
\vspace{0.1cm}

To prove it, we observe that by Young's inequality, we have 
\begin{equation*}
    |\xi-\eta|\leq \frac{\delta}{2}\,\big(|\xi|+|\eta|\big)+\frac{\delta^{-1}}{2}\big(|\xi|+|\eta|\big)^{-1}\, |\xi-\eta|^2.
\end{equation*}
Therefore, from the above inequality and the monotonicity of $b(t)$, we infer
\begin{equation*}
    \begin{split}
        B(|\xi-\eta|)&\stackrel{\eqref{Bcomeb}}{\leq } C\,b(|\xi-\eta|)\,|\xi-\eta|\leq C'\,\delta \,b\big( |\xi|+|\eta|\big)\,\big(|\xi|+|\eta| \big)+C'\,\d^{-1} \frac{b\big( |\xi|+|\eta|\big)}{\big(|\xi|+|\eta|\big)}\,|\xi-\eta|^2
        \\
        &\stackrel{\eqref{Bcomeb},\eqref{triangle}}{\leq} C''\,\delta\,\big( B(|\xi|)+B(|\eta|)\big)\,+C'\,\d^{-1}\,a\big(|\xi|+|\eta|\big)\,|\xi-\eta|^2\,,
    \end{split}
\end{equation*}
where $C,C',C''>0$ depend on $i_a,s_a$, and \eqref{el:comparison} is proven.
\vspace{0.2cm}

Let us now introduce the \emph{complementary Young function} (also called conjugate Young function)
\begin{equation}\label{def:wB}
    \widetilde{B}(t)=\sup\left\{st-B(s):\,s>0 \right\}\,.
\end{equation}
Since $b\in C^0\big([0,\infty)\big)\cap C^1(0,\infty)$ is strictly monotone, $b(0)=0$ and $\lim_{t\to \infty}b(t)=+\infty$, we may write
\begin{equation}\label{wB}
    \wB(t)=\int_0^t b^{-1}(s)\,ds
\end{equation}
where $b^{-1}(s)$ is the inverse function of $b$--see \cite[pp. 10-11]{RR91}.
Observe that
\begin{equation*}
    \frac{(b^{-1})'(s)s}{b^{-1}(s)}\stackrel{b(t)=s}{=}\frac{b(t)}{b'(t)t}
\end{equation*}
so that
\begin{equation*}
    0<\frac{1}{s_b}\leq \inf_{s>0}\frac{(b^{-1})'(s)\,s}{b^{-1}(s)}\leq \sup_{s>0}\frac{(b^{-1})'(s)\,s}{b^{-1}(s)}\leq \frac{1}{i_b}<\infty
\end{equation*}
Coupling this piece of information with the representation formula \eqref{wB}, and arguing as in \eqref{Bcomeb}-\eqref{mon:iBsB}, we deduce 
\begin{equation}\label{wBcome}
    \left(\frac{1}{1+i_B'}\right)b^{-1}(s)\,s\leq \wB(s)\leq \left(\frac{1}{1+s_B'}\right)b^{-1}(s)\,s,\quad s>0,
\end{equation}
and
\begin{equation}
    1<s'_B\leq\inf_{t>0}\frac{\wB'(s)s}{\wB(s)}\leq \sup_{t>0}\frac{\wB'(s)s}{\wB(s)} \leq i'_B<\infty\,,
\end{equation}
where 
\[
i_B'=\frac{i_B}{i_B-1}\quad \text{and}\quad s_B'=\frac{s_B}{s_B-1}
\]
are the H\"older's conjugates of $i_B,s_B$, respectively.

Correspondingly, we may obtain the  same monotonicity properties for $\wB$ as in \eqref{mon:iBsB}, with $s_B',i_B'$ replacing $i_B,s_B$, respectively. In particular, we have
\begin{equation}\label{tB:usef}
\begin{split}
    &c^{i_B'}\wB(t)\leq \wB(c t)\leq c^{s_B'}\wB(t),\quad\text{for all } c\in (0,1)\text{ and $t>0$}
    \\
    &C^{s_B'}\wB(t)\leq \wB(C t)\leq C^{i_B'}\wB(t),\quad\text{for all } C\geq 1 \text{ and $t>0$},
    \end{split}
\end{equation}
and
\begin{equation}\label{tB=1}
    c(i_a,s_a)\leq \widetilde{B}(1)\leq C(i_a,s_a).
\end{equation}
Moreover, by \eqref{Bcomeb} and \eqref{wBcome}, we have that
\begin{equation}\label{come:BwB}
  c(i_a,s_a)\,\wB(t)  \leq B\big(b^{-1}(t) \big)\leq C(i_a,s_a)\,\wB(t),\quad t>0.
\end{equation}

Next, we will usually exploit Young's inequality for Orlicz-functions
\begin{equation}\label{Young}
\begin{split}
st& \leq\delta^{i_B} B(s)+\delta^{-i_B'}\wB(t)\quad\text{and}
\\
    st&\leq \delta^{s_B'}\wB(s)+\delta^{-s_B}\,B\left( t\right),\quad \text{for any $\delta\in (0,1]$,}
    \end{split}
\end{equation}
which readily follow by the definition of $\wB$ in \eqref{def:wB}, together with \eqref{B2t} and \eqref{tB:usef}.

Moreover, by \eqref{wBcome} and \eqref{Bcomeb}, we have
\begin{equation}\label{wBbt}
    c(i_b,s_b)\,B(t)\leq \wB\big(b(t)\big)\leq C(i_b,s_b)\,B(t)\,,\quad t>0.
\end{equation}
We will also repeatedly use the following version of Young's inequality
\begin{equation}\label{young1}
    b(t)\,s\leq \delta\,B(t)+C(i_a,s_a)\,\delta^{-(s_B-1)}\,B(s)\,
\end{equation}
which is a simple consequence of \eqref{Young} and \eqref{wBbt}.
\vspace{0.2cm}

We shall also use the auxiliary function
\begin{equation}\label{def:psi}
    \psi(t)=\int_0^t a(s)^{1/2}\,ds.
\end{equation}
We have that
\begin{equation}\label{ultima}
    \left(\frac{1}{1+s_a/2}\right)\,a(t)^{\frac12}t\leq \psi(t)\leq \left(\frac{1}{1+i_a/2}\right)\,a(t)^{\frac1 2}t.
\end{equation}
 Indeed, integrating by parts, and using \eqref{iasa} and $a(s)^{1/2}s|_{s=0}=b(s)^{1/2}s^{1/2}|_{s=0}=0$, we get
\begin{equation*}
    \begin{split}\psi(t)&=a(s)^{1/2}s\big|_{s=0}^{s=t}-\frac{1}{2}\int_0^t\frac{a'(s)s}{a(s)^{1/2}}\,dx
    \\
    &=a(t)^{1/2}t-\frac{1}{2}\int_0^ta(s)^{1/2}\,\frac{a'(s)s}{a(s)}\,ds
    \\
    &\leq a(t)^{1/2}t-\frac{i_a}{2}\,\psi(t)\quad\text{and}\quad \geq  a(t)^{1/2}t-\frac{s_a}{2}\,\psi(t)\,,
    \end{split}
\end{equation*}
from which \eqref{ultima} follows
\vspace{0.2cm}

Later on, we will need to approximate the functions $a,b,B$ via a sequence of smoother functions. To this end, we exploit the following lemma inspired by \cite[Lemma 3.3]{CM11} and \cite[Lemma 4.5]{CM14}

\begin{lemma}\label{lemma:ae}
    Let $a\in C^1((0,\infty))$ be a function satisfying \eqref{iasa}, and let $b,B$ be as in \eqref{def:a},\eqref{Bb}. Then there exists a sequence of functions $\{a_\e\}_{\e>0}$, $a_\e:[0,\infty)\to [0,\infty)$ such that, for all $\e>0$,
    \begin{equation}\label{ae:nonzero}
        \begin{split}
            &a_\e \in C^\infty([0,\infty))
        \\
    \e\leq & \,a_\e(t)\leq \e^{-1}\quad\text{for all $t\geq 0$,}
        \end{split}
    \end{equation}   
and
    \begin{equation}\label{iaesae}
    \begin{cases}
    \displaystyle
        \inf_{t>0}\frac{a_\e'(t)\,t}{a_\e(t)}\geq \min\{i_a,0\}>-1\\[1ex]
        \\
        \displaystyle
        \sup_{t>0}\frac{a_\e'(t)\,t}{a_\e(t)}\leq \max\{s_a,0\}<\infty,
        \end{cases}
    \end{equation}
Moreover, by setting
\begin{equation}\label{def:bBe}
    b_\e(t)=a_\e(t)\,t,\quad B_\e(t)=\int_0^t b_\e(s)\,ds
\end{equation}
  we have
    \begin{align}
        &\lim_{\e\to 0} a_\e=a\quad\text{uniformly in  $[L,M]$ for every $0<L\leq M$}\label{ae:unif}
        \\
       & \lim_{\e\to 0} b_\e=b\quad\text{uniformly in  $ [0,M]$ for every $M>0$}\label{be:unif}
    \end{align}
    and hence
    \begin{align}
    &\lim_{\e\to 0} a_\e(|\xi|)\,\xi=a(|\xi|)\,\xi\quad\text{uniformly in  $|\xi|\leq M$ for every $M>0$.}\label{aexi:unif}
    \\
        &\lim_{\e\to 0} B_\e=B\quad\text{uniformly in $ [0,M]$ for every $M>0$.}\label{BE:unif}
        \end{align}
Furthermore
\begin{equation}\label{aexi:Cinf}
 \text{the map}\quad \xi\mapsto a_\e(|\xi|)\quad\text{is of class $C^\infty(\R^n)$.}  
\end{equation}
\end{lemma}

\begin{proof}
Let $A:\R\to [0,\infty)$ be the function defined by
\begin{equation*}
    A(s)=a(\mathrm{e}^s)\quad s\in \R,
\end{equation*}
and note that, by \eqref{iasa}, we have
\begin{equation}\label{IASA}
    i_a\,A(s)\leq A'(s)\leq s_a\,A(s)\,.
\end{equation}
For $\e>0$, we consider the convolution  $A_\e(s)\coloneqq A\ast \rho_\e(s)$, so that \eqref{IASA} and standard properties of convolution yield
\begin{equation}\label{IAESAE}
    i_a\,A_\e(s)\leq A'_\e(s)\leq s_a\,A_\e(s)\,,
\end{equation}
and $A_\e\in C^\infty(\R)$ for all $\e>0$.  Next, define $\hat{a}_\e:(0,\infty)\to [0,\infty)$ as
\begin{equation}\label{def:hatae}
    \hat{a}_\e(t)=A_\e(\log t)\,.
\end{equation}
Then $\hat{a}_\e\in C^\infty((0,\infty))$, and by \eqref{IAESAE},
\begin{equation}\label{temp:iaesae}
    i_a\leq \inf_{t>0}\frac{\hat{a}_\e'(t)\,t}{\hat{a}_\e(t)}\leq \sup_{t>0}\frac{\hat{a}_\e'(t)\,t}{\hat{a}_\e(t)}\leq s_a.
\end{equation}
Additionally, since $A_\e\to A$ locally uniformly in $\R$, we have that $\hat a_\e\to a$ locally uniformly in $(0,\infty)$. Next, for $t\geq 0$ we define
    \begin{equation}\label{def:ae}
        a_\e(t)\coloneqq \frac{\hat{a}_\e\big(\sqrt{\e+t^2}\big)+\e}{1+\e\,\hat{a}_\e\big(\sqrt{\e+t^2}\big)}\,,
    \end{equation}
Clearly $a_\e\in C^\infty([0,\infty))$, and \eqref{ae:nonzero} follows from the fact that the function $[0,\infty)\ni s\mapsto \frac{s+\e}{1+\e\,s}$ is increasing for every $\e\in (0,1)$. We then have
\begin{equation*}
    a_\e'(t)=\frac{(1-\e^2)\,\hat{a}_\e'\big(\sqrt{\e+t^2}\big)\,t}{\big(1+\e\,\hat{a}_\e\big(\sqrt{\e+t^2}\big) \big)^2\sqrt{\e+t^2}}\,,
\end{equation*}
and a straighforward computation shows that
\begin{equation*}
    \frac{a_\e'(t)\,t}{a_\e(t)}=\left(\frac{\hat{a}_\e'\big(\sqrt{\e+t^2}\big) \,\big(\sqrt{\e+t^2}\big) }{\hat{a}_\e\big(\sqrt{\e+t^2}\big)}\right)\,\left[\frac{(1-\e^2)\,\hat{a}_\e\big(\sqrt{\e+t^2}\big)}{\big(1+\e\,\hat{a}_\e\big(\sqrt{\e+t^2}\big) \big)\,(\hat{a}_\e\big(\sqrt{\e+t^2}\big)+\e)}\,\frac{t^2}{\e+t^2}\right]\,,
\end{equation*}
hence \eqref{iaesae} immediately follows from \eqref{temp:iaesae} and the fact that the term in the square bracket above is  nonnegative and smaller than 1 for every $\e\in (0,1)$.

Next, \eqref{ae:unif} is a consequence of the local uniform convergence of  $\hat a_\e\to a$ in $(0,\infty)$, so by \eqref{def:bBe} we also have
\begin{equation}\label{p:beunifo}
    \lim_{\e\to 0} b_\e=b\quad\text{uniformly in  $[L,M]$ for every $0<L\leq M$.}
\end{equation}

Then, owing to \eqref{iaesae}, we may exploit \eqref{mon:iasa} (with $a_\e$ in place of $a$), and get
\begin{equation*}
    0\leq b_\e(t)=a_\e(t)\,t\leq  a_\e(1)\,t^{1+\min\{i_a,0\}}\leq 2\,a(1)\,t^{1+\min\{i_a,0\}},\quad t\in (0,1)\,,
\end{equation*}
where in the last inequality we used $a_\e(1)\leq 2\,a(1)$ for $\e>0$ small enough as a consequence of \eqref{ae:unif}. Therefore, recalling that $i_a>-1$, we have $\lim_{t\to 0} b_\e(t)=0$ uniformly in $\e\in (0,1)$, and this piece of information and \eqref{p:beunifo} yield \eqref{be:unif}.

Finally, the map $\xi \mapsto a_\varepsilon(|\xi|)$ is smooth in a neighborhood of the origin thanks to the definition of $a_\varepsilon$ in \eqref{def:ae} and the smoothness of $\hat a_\varepsilon$. This also implies the regularity property \eqref{aexi:Cinf}. The proof is thus complete.
\end{proof}

\begin{remark}[Uniformity in $\e$]\label{remark:importante}
\rm{ By the lower and upper bounds in \eqref{iaesae}, the functions 
$a_\e$, $b_\e$, the Young function $B_\e$, its Young conjugate $\widetilde B_\e$, and 
\(\psi_\e(t) = \int_0^t a_\e(s)^{1/2} \, ds\) 
still satisfy the properties \eqref{mon:iasa}--\eqref{ultima}, 
with the constants 
\[
\text{$i_a$ and $s_a$ replaced by $\min\{i_a,0\}$ and $\max\{s_a,0\}$, }
\]
respectively. Accordingly, in view of \eqref{def:isb} and \eqref{def:isB}, we have that
\[
i_b\,,s_b,\,i_B,\,s_B\,\text{ are replaced by }\min\{i_b,1\},\,\max\{s_b,1\},\,\min\{i_B,2\},\, \max\{s_B,2\},
\]
respectively. The key point is that all the estimates are uniform in $\e \in (0,1)$.
}
\end{remark}

\subsection{Orlicz Lebesgue and Sobolev spaces}\label{sec:spaces}
Let $\Omega$ be an open set of $\R^n$. The Lebesgue-Orlicz space is defined as
\begin{equation*}
    L^B(\Omega)\coloneqq\left\{u:\Omega\to \R\, \text{ measurable:}\, \int_\Omega B(|u|)\,dx<\infty \right\}\,.
\end{equation*}
We endow space with the so-called \textit{Luxemburg norm}
\begin{equation}\label{lux}
    \|u\|_{L^B(\Omega)}=\inf\bigg\{k>0:\int_\Omega B\left( \frac{|u(x)|}{k}\right)\,dx\leq 1 \bigg\}\,.
\end{equation}
In particular, a sequence $u_k\to u$ in $L^B(\Omega)$ if $\lim_{k\to \infty}\|u_k-u\|_{L^B(\Omega)}$.
We also have H\"older's inequality in Orlicz spaces \cite[Theorem 4.7.8]{funcsp}
\begin{equation}\label{holder:orlicz}
    \bigg|\int_\Omega uv\,dx\bigg|\leq \|u\|_{L^B(\Omega)}\|v\|_{L^{\widetilde{B}}(\Omega)}.
\end{equation}
where $\widetilde{B}$ is the Young conjugate. Since by \eqref{iBsB} and \eqref{B2t}, the function $B$ satisfies the $\Delta_2$ condition, when $\Omega$ is bounded, convergence in $L^B(\Omega)$ is equivalent to the so-called \textit{modular convergence}, i.e.,
\begin{equation}\label{equiv:modular}
    u_k\to u\quad\text{in $L^B(\Omega)$} \iff \lim_{k\to\infty}\int_\Omega B(|u_k-u|)\,dx=0.
\end{equation}
See, for instance, \cite[Theorem 4.10.6]{funcsp}. As $B$ satisfies the $\Delta_2$ and $\nabla_2$-conditions by \eqref{B2t}, it is well known that the space $L^{B}(\Omega)$ is a reflexive Banach space--see \cite[Theorem 3.6.6]{HHbook}.  Also, via convolution, it is possible to show that smooth functions are dense in $L^{B}(\Omega)$--see \cite[Lemma 2.1]{DT71} or \cite[Theorem 4.4.7]{HHbook}.

Next, we have
\begin{equation}\label{modular1}
    \int_\Omega B\bigg(\frac{|v|}{\|v\|_{L^B(\Omega)}} \bigg)\,dx=1.
\end{equation}
Indeed, by definition of Luxemburg norm \eqref{lux}  and the continuity of $B(t)$, there holds
\begin{equation*}
     \int_\Omega B\bigg(\frac{|v|}{\|v\|_{L^B(\Omega)}} \bigg)\,dx\leq 1.
\end{equation*}
On the other hand, by \eqref{lux} and \eqref{B2t}, for every $\delta\in (0,1)$
\begin{equation*}
    1\leq \int_\Omega B\bigg(\frac{v}{\|v\|_{L^B(\Omega)}(1-\delta)} \bigg)\,dx\leq \frac{1}{(1-\delta)^{s_B}}\int_\Omega B\bigg(\frac{|v|}{\|v\|_{L^B(\Omega)}} \bigg)\,dx,
\end{equation*}
hence by letting $\delta\to 0$ we deduce \eqref{modular1}.

Next, thanks to \eqref{simple}, for any bounded measurable set $U\subset \R^n$, we have
\begin{equation}\label{int:simple0}
    \int_U |v|\,dx\leq C(i_a,s_a)\int_UB(|v|)\,dx+|U|,
\end{equation}
and taking into account Remark \ref{remark:importante}, the same inequality holds for $B_\e$, that is

\begin{equation}\label{int:simple}
    \int_U |v|\,dx\leq C(i_a,s_a)\,\int_UB_\e(|v|)\,dx+|U|,
\end{equation}
for all $0<\e<\e_0$ small enough. More generally, since \[
B(t)\geq c(i_a,s_a)\,t^{i_B},\quad B_\e(t)\geq c(i_a,s_a)\,t^{\min\{i_a+2,2\}}\quad\text{ for $t\geq 1$}
\] 
thanks to \eqref{B2t}, Remark \ref{remark:importante}, \eqref{B=1},  and \eqref{BE:unif}, we also have
\begin{equation}\label{unif:refl}
\begin{split}
&\int_U|v|^{i_B}\,dx\leq C\,\int_U B\big(|v| \big)\,dx+C\,|U|,
\\
  &  \int_{U} |v|^{\min\{i_B,2\}}\,dx\leq C\,\int_U B_\e(|v|)\,dx+C\,|U|\,,\quad \text{for all $\e\in(0,\e_0)$.}
    \end{split}
\end{equation}
with $C=C(i_a,s_a)>0$.
Next, the Orlicz-Sobolev space is defined as
\begin{equation*}
    W^{1,B}(\Omega)\coloneqq\left\{u\in W^{1,1}(\Omega):\,Du\in L^B(\Omega)\right\}\,,
\end{equation*}
where $Du$ denotes the distributional gradient of $u$.
Accordingly we define the spaces
\[
W^{1,B}_{loc}(\Omega)=\left\{u\in W^{1,1}(K):\,Du\in L^B(K)\quad\text{for all $K\Subset \Omega$}\right\},
\] 
\[
W^{1,B}_c(\Omega)\coloneqq\left\{u\in W^{1,B}(\Omega):\text{$u$ is compactly supported in $\Omega$}\right\},
\]
and
\[
W^{1,B}_0(\Omega)\coloneqq\left\{u\in W^{1,B}(\Omega):\text{$u$ can be extended to a function in $W^{1,B}_c(\R^n)$}\right\}.
\]
By classical extension theorems, such as \cite[Theorem 13.17]{Leo17}, if $\Omega$ is a bounded Lipschitz domain, then the space  $W^{1,B}_0(\Omega)$ is equivalent to the space of functions $u\in W^{1,B}(\Omega)$ such that $u=0$ on $\partial \Omega$ in the sense of traces.

When $B(t)=t^p$, $p\geq 1$, we will simply denote by $W^{1,B}(\Omega)=W^{1,p}(\Omega)$, and analogous definitions hold for $W^{1,p}_{loc}(\Omega)$, $W^{1,p}_c(\Omega)$ and $W^{1,p}_0(\Omega)$.

\begin{remark}\rm{
    Observe that, by \eqref{Bcomeb} and \eqref{ae:nonzero}, we have
    \begin{equation}\label{bBquadratic}
        c_\e\,t\leq b_\e(t)\leq C_\e\,t\quad \text{and}\quad  c_\e\,t^2\leq B_\e(t)\leq C_\e\,t^2
    \end{equation}
    for some constants $c_\e,C_\e>0$ depending on $\e$ as well. Thus, for any bounded open set $\Omega\subset \R^n$, we have $W^{1,B_\e}(\Omega)=W^{1,2}(\Omega)$, and similarly $W^{1,B_\e}_0(\Omega)=W^{1,2}_0(\Omega)$ and $W^{1,B_\e}_{loc}(\Omega)=W^{1,2}_{loc}(\Omega)$.}
\end{remark}

\subsection{The stress field \texorpdfstring{$\mathcal{A}$}{A}}\label{sec:stress}
In this subsection, we collect some properties of the vector field $\A$, 
often referred to as the \emph{stress field}.  

We start with the following elementary lemma, establishing natural coercivity and growth properties of $\A$. These are quite standard in the literature for $p$-Laplace type operators-- see \cite[Section 2]{D98}.

\begin{lemma}\label{lemma:Agrco}
  Assume that $a(\cdot)$ satisfies \eqref{iasa}, \eqref{ab=1}, and that $\A$ fulfills \eqref{reg:A}-\eqref{ass:A}. Then we have
\begin{equation}\label{co:gr}
    \begin{split}
        &\sum_{i=1}^n |\A^i(x,\xi)-\A^i(x,0)|\leq C_1\,b(|\xi|)
        \\
        &\big(\A(x,\xi)-\A(x,0)\big)\cdot \xi\geq c_1\,B(|\xi|)\,,\quad \text{for all $x\in \overline{\Omega}$,  for all $\xi\in \rn$,}
    \end{split}
\end{equation}
where $c_1=c_1(n,\l,i_a, s_a)$ and $C_1=C_1(n,\L,i_a)$ are positive constants. Moreover,
\begin{equation}\label{strong:coer}
    \big(\A(x,\xi)-\A(x,\eta)\big)\cdot (\xi-\eta)\geq c_0\,a\big(|\xi|+|\eta|\big)\,|\xi-\eta|^2.
\end{equation}
with $c_0=c_0(n,\l,\L,i_a,s_a)>0$.
\end{lemma}

\begin{proof}
   By the fundamental theorem of calculus
    \begin{equation*}
    \begin{split}
       \big( \A(x,\xi)-\A(x,0)\big)\cdot \xi&=\sum_{i,j=1}^n\int_0^1 \frac{\partial \A^i(t\,\xi)}{\partial \xi_j}dt\,\xi_j\,\xi_i\stackrel{\eqref{ass:Aaut}}{\geq}\l\,\int_0^1 a\big( t|\xi|\big)\,dt\,|\xi|^2 
        \\
        &\stackrel{\eqref{mon:iasa}}{\geq}\l\,\left(\int_0^1  t^{s_a}\,dt\right)\,a(|\xi|)\,|\xi|^2\stackrel{\eqref{Bcomeb}}{\geq} \l\,\left(\frac{1+i_b}{1+s_a}\right) B(|\xi|)\,,
     \end{split}
    \end{equation*}
   for all $x\in \overline{\Omega}$. We also have
    \begin{equation*}
        \begin{split}
            |\A^i(x,\xi)-\A^i(x,0)| & \leq \int_0^1\left| \frac{\partial \A^i(t\,\xi)}{\partial \xi_j}\right|dt\,|\xi|\stackrel{\eqref{ass:A}}{\leq} \L\,\int_0^1 a\big(t\,|\xi| \big)dt\,|\xi|
            \\
            &\stackrel{\eqref{mon:iasa}}{\leq}\L\,\left(\int_0^1 t^{i_a}dt\right)\,a(|\xi|)\,|\xi|=\frac{\L}{1+i_a} b(|\xi|).
        \end{split}
    \end{equation*}
Then we compute
    \begin{equation}\label{temp:coe}
    \begin{split}
        \big(\A(x,\xi)-\A(x,\eta)\big)\cdot (\xi-\eta)&=\sum_{i,j=1}^n\int_0^1 \frac{\partial \A^i}{\partial \xi_j}\big(t\xi+(1-t)\,\eta \big)\,dt\,(\xi_j-\eta_j)(\xi_i-\eta_i)
        \\
        &\geq \l\,\bigg(\int_0^1 a\big(|t\xi+(1-t)\eta| \big)\,dt\bigg)\,|\xi-\eta|^2.
        \end{split}
    \end{equation}
We now claim 
\begin{equation}\label{intadt}
    \int_0^1 a\big(|t\xi+(1-t)\eta| \big)\,dt\geq c\,a\big(|\xi|+|\eta| \big)
\end{equation}
    for some $c=c(n,i_a,s_a)>0$. Without loss of generality, we may assume that $|\eta|\geq |\xi|$, $|\eta|>0$. Let us first consider the case $|\xi-\eta|\leq |\eta|/2$; in this case, we have
    \begin{equation*}
        |\xi|+|\eta|\geq |t\xi+(1-t)\eta|\geq |\eta|-|\xi-\eta|\geq\frac{|\eta|}{2}\geq \frac{|\eta|+|\xi|}{4},
    \end{equation*}
    where in the last inequality we used that $|\eta|\geq |\xi|$. Thus, from \eqref{ultra:utile} we deduce 
    \[
    a\big(|t\xi+(1-t)\eta| \big)\geq c(i_a,s_a)\,a\big(|\xi|+|\eta| \big)\quad\text{for all $t\in [0,1]$,}
    \]
and \eqref{intadt} follows. In the case $|\xi-\eta|> |\eta|/2>0$, we put $t_0=\frac{|\eta|}{|\eta-\xi|}$, so that $t_0\in (0,2)$. Then
\begin{equation*}
\begin{split}
    |t\xi+(1-t)\eta|&\geq \big||\eta|-t|\eta-\xi| \big|=|t_0-t|\,|\eta-\xi|
    \\
    &\geq |t_0-t|\frac{|\eta|}{2}\geq |t_0-t|\frac{|\eta|+|\xi|}{4}.
    \end{split}
\end{equation*}
Therefore, from \eqref{mon:iasa} and since $|t-t_0|\leq 3$ and $i_a>-1$, we deduce
\begin{equation*}
     \int_0^1 a\big(|t\xi+(1-t)\eta| \big)\,dt\geq a\Big( 3\big(|\xi|+|\eta|\big)\Big)\,\int_0^1\frac{|t-t_0|^{i_a}}{3^{i_a}}\,dt\geq c\,a\big( |\xi|+|\eta|\big),
\end{equation*}
    and \eqref{intadt} is proven. From \eqref{temp:coe} and \eqref{intadt}, Equation \eqref{strong:coer} follows.
\end{proof}
We will also need to regularize the stress field $\A$. 
To this end, we state and prove an approximation lemma, 
which combines ideas from \cite[pp.~342]{L91} with those of 
\cite{CM11,CM14}, already used in Lemma~\ref{lemma:ae}. 
The key idea is to construct, through the use of a suitable cut-off function, an approximating vector field $\A_\e(\xi)$ 
that coincides with $a_\e(|\xi|)\,\xi$ for very small and very large values of $|\xi|$, while agreeing with (a regularization of) $\A(\xi)$ elsewhere (see equation \eqref{def:Aedelta}).
\vspace{0.2cm}

For simplicity, we focus on autonomous stress fields $\A(x,\xi)\equiv \A(\xi)$ such that $\A(0)=0$, and satisfying \eqref{ass:A}, that is 
\begin{equation}\label{ass:Aaut}
        \sum_{i,j=1}^n\left|\frac{\partial \A^i(\xi) }{\partial \xi_j}\right|\leq \L\,a\left(|\xi|\right), \quad\text{and}\quad
\sum_{i,j=1}^n \frac{\partial \A^i (\xi) }{\partial \xi_j}\,\eta_i\,\eta_j\geq \l\,a(|\xi|)\,|\eta|^2,
\end{equation}
for all $\eta\in \rn$ and for all $\xi \in \rn\setminus \{0\}$.

\begin{lemma}\label{lemma:Ae}
    Suppose $\A\in C^0(\R^n)\cap C^1(\R^n\setminus\{0\})$ is a vector field satisfying \eqref{ass:Aaut} and $\A(0)=0$, and let $a_\e$ be the function given by Lemma \ref{lemma:ae}. Then there exists a sequence of vector fields $\A_\e:\rn\to \rn$ such that
    \begin{align}
       & \A_\e\in C^\infty(\rn)\quad\text{and} \quad\Ae(0)=0,\nonumber
        \\
       & \A_\e\to \A\quad\text{locally uniformly in $\rn$,}\label{Aeunif}
    \end{align}
    and for all $\xi,\eta\in \rn$ and $\e\in (0,1)$, they fulfill
    \begin{equation}\label{coer:Ae}
    \begin{split}
        \sum_{i,j=1}^n\frac{\partial \Ae^i(\xi)}{\partial \xi_j}\eta_i\,\eta_j\geq c\,a_\e(|\xi|)\,|\eta|^2,
        \qquad\sum_{i,j=1}^n\left|\frac{\partial \Ae^i(\xi)}{\partial \xi_j}\right|\leq C\,a_\e(|\xi|)\,,
        \end{split}
    \end{equation}
    for some positive constants $c,C$ depending on $n,\l,\L, i_a,s_a$. 
    
    In particular, by \eqref{ae:nonzero}, \eqref{iaesae} and Lemma \ref{lemma:Agrco}, for all $\e>0$, and for all $\xi,\eta\in \rn$ they satisfy
    \begin{equation}\label{nonzero:Ae}
    \begin{split}
        \sum_{i,j=1}^n\frac{\partial \Ae^i(\xi)}{\partial \xi_j}\eta_i\,\eta_j\geq c\,\e\,|\eta|^2,\qquad \sum_{i,j=1}^n\left|\frac{\partial \Ae^i(\xi)}{\partial \xi_j}\right|\leq C\,\e^{-1}\,,
        \end{split}
    \end{equation}
\begin{equation}\label{coAe:gr}
    \begin{split}
        \sum_{i=1}^n |\Ae^i(\xi)|\leq C_0\,b_\e(|\xi|),
        \qquad\Ae(\xi)\cdot \xi\geq c_0\,B_\e(|\xi|)\,,
    \end{split}
\end{equation}
where $c_0,C_0>0$ depend on $n,i_a,s_a,\l,\L$.
\end{lemma}

\begin{proof}
    Let us fix \begin{equation}\label{cstar}
        C_\star=\frac{2(2\,C_1+1)}{\min\{\l/4,1,1+i_a\}}\quad\text{and a parameter}\quad 0<\delta_0<(2\rm{e}^{C_\star})^{-1} 
    \end{equation}
 where $C_1=C_1(n,\L,i_a)$ is the constant appearing in \eqref{co:gr}.
 
 For $\delta\in (0,\delta_0)$  we consider a family of functions $\{\eta_\delta\}_{\delta\in (0,\delta_0)}\subset C^\infty([0,\infty))$ such that
    \begin{equation}\label{eta:d1}
            0\leq \eta_\delta(t)\leq 1\,,\quad
            |\eta'_\delta(t)|\leq \frac{1}{C_\star \,t}\quad t>0\,,
    \end{equation}
    and
    \begin{equation}\label{eta:d2}
    \eta_\delta(t)\equiv
        \begin{cases}
            1\quad\text{for $t\in [0,\delta]\cup \big[\rm{e}^{C_\star}\delta^{-1},+\infty\big)$}
            \\
            \\
            0\quad\text{for $t\in \big[\rm{e}^{C_\star}\delta,\delta^{-1}\big]$.}
        \end{cases}
    \end{equation}
We will show later how to construct such function. Thanks to \eqref{ae:unif}-\eqref{be:unif} and the fact that $a,a_\e$ are strictly positive in $(0,\infty)$, for every $\delta\in (0,\delta_0)$, we may find $\e_\delta>0$ such that
\begin{equation}\label{bounds:delta}
    \frac{1}{2}a_\e(t)\leq a(t)\leq 2\,a_\e(t)\quad\text{for all $t\in \big[\delta,\,\rm{e}^{C_\star}\delta^{-1}\big]$,}
\end{equation}
for every $0<\e<\e_\delta$. Next, consider  
$$\widehat{\A}_\e(\xi)=\A\ast \rho_\e(\xi).$$
Thus $\widehat{\A}_\e\in C^\infty(\rn)$, and $\widehat\A_\e\to \A$  in $C^1_{loc}(\rn\setminus \{0\})\cap C^0_{loc}(\R^n)$. From this piece of information together with \eqref{ass:Aaut}, \eqref{co:gr} with $\A(x,0)= \A(0)= 0$, and the positivity of $a(t),b(t)$ in $(0,\infty)$, we can find a (possibly smaller) $\e_\delta>0$ such that
\begin{equation}\label{temp:cogr}
\sum_{i,j=1}^n\left|\frac{\partial \widehat\A_\e^i(\xi) }{\partial \xi_j}\right|  \leq 2\L\,a\left(|\xi|\right), \quad\sum_{i,j=1}^n \frac{\partial \widehat\A_\e^i (\xi) }{\partial \xi_j}\,\eta_i\,\eta_j\geq \frac{\l}{2}\,a(|\xi|)\,|\eta|^2,\quad\text{for all  $\,|\xi|\in [\delta,\rm{e}^{C_\star}\delta^{-1}],$}
\end{equation}
and
\begin{equation}\label{temp:cogr1}
\sum_{i=1}^n |\widehat\A_\e^i(\xi)|\leq 2\,C_1\,b(|\xi|),\qquad\text{for all $\xi$ such that $\,|\xi|\in [\delta,\rm{e}^{C_\star}\delta^{-1}],$}
\end{equation}
and for  all $0<\varepsilon<\varepsilon_\delta$. We now define 
\begin{equation}\label{def:Aedelta}
    \A_{\e,\delta}(\xi)=\big(1-\eta_\delta(|\xi|) \big)\,\widehat\A_\e(\xi)+\eta_\delta(|\xi|)\,a_\e(|\xi|)\,\xi\,,\quad \xi\in \rn.
\end{equation}
By the properties of convolution and \eqref{aexi:Cinf}, we have that $\A_{\e,\delta}\in C^\infty(\rn)$. Also, since $b_\e(0)=0$ and $\A\in C^1(\R^n\setminus\{0\})$, by \eqref{eta:d2} and the properties of convolution it follows that
\begin{equation*}
    |\A_{\e,\delta}(0)|=\lim_{|\xi|\to 0}a_\e(|\xi|)|\xi|=\lim_{|\xi|\to 0}b_\e(|\xi|)=0,
\end{equation*}
and
\begin{equation*}
    \lim_{\delta\to 0}\lim_{\e\to 0}\A_{\e,\delta}=\A\quad\text{locally uniformly in $\R^n$.}
\end{equation*}
Let us now verify \eqref{coer:Ae}.
A straighforward computation shows that, for all $i,j=1,\dots,n$, we have
\begin{equation}\label{d:Aed}
\begin{split}
    \frac{\partial \A^i_{\e,\delta}(\xi)}{\partial \xi_j}= & \,\big(1-\eta_\delta(|\xi|)\big)\, \frac{\partial \widehat\A_\e^i(\xi)}{\partial \xi_j}+\eta_\delta(|\xi|)\,a_\e(|\xi|)\,\Bigg\{\frac{a_\e'(|\xi|)\,|\xi|}{a_\e(|\xi|)}\,\frac{\xi_i\,\xi_j}{|\xi|^2}+\delta_{ij} \Bigg\}
    \\
    &+\eta_\delta'(|\xi|)\,\bigg\{-\widehat\A_\e^i(\xi)\,\frac{\xi_j}{|\xi|}+b_\e(|\xi|)\,\frac{\xi_i\,\xi_j}{|\xi|^2}\bigg\}\,,
  \end{split}  
\end{equation}
where $\delta_{ij}$ is the Kronecker delta.
Now observe that, by \eqref{eta:d2}, we have $\A_{\e,\delta}(\xi)\equiv \widehat\A_\e(\xi)$ if $|\xi|\in \big[\rm{e}^{C_\star}\delta,\delta^{-1}\big]$, hence by \eqref{temp:cogr} and \eqref{bounds:delta} we have
\begin{equation*}
\begin{split}
     &\sum_{i,j=1}^n\frac{\partial \A^i_{\e,\delta}(\xi)}{\partial \xi_j}\,\eta_i\,\eta_j\geq \frac{\l}{4} a_\e(|\xi|)\,|\eta|^2,\quad \eta\in \rn
     \\
     &\sum_{i,j=1}^n\left|\frac{\partial \A^i_{\e,\delta}(\xi)}{\partial \xi_j}\right|\leq 4\,\L\, a_\e(|\xi|),\quad\text{for all $\xi$ such that $|\xi|\in \big[\rm{e}^{C_\star}\delta,\delta^{-1}\big]$,}    
     \end{split}
\end{equation*}
and for all $0<\e<\e_\delta$. On the other hand, if $|\xi|\leq \delta$ or $|\xi|\geq \rm{e}^{C_\star}\delta^{-1}$, then by \eqref{eta:d2} we have $\A_{\e,\delta}(\xi)=a_\e(|\xi|)\,\xi$, and thus
\begin{equation*}
    \begin{split}
        &\sum_{i,j=1}^n\frac{\partial \A^i_{\e,\delta}(\xi)}{\partial \xi_j}\,\eta_i\,\eta_j=a_\e(|\xi|)\,\bigg\{\frac{a_\e'(|\xi|)\,|\xi|}{a_\e(|\xi|)}\,\frac{|\xi\cdot \eta|^2}{|\xi|^2}+|\eta|^2 \bigg\}\stackrel{\eqref{iaesae}}{\geq} \min\{1,i_a+1\}\,a_\e(|\xi|)\,|\eta|^2
        \\
         &\sum_{i,j=1}^n\left|\frac{\partial \A^i_{\e,\delta}(\xi)}{\partial \xi_j}\right|=\sum_{i,j=1}^n a_\e(|\xi|)\left|\frac{a_\e'(|\xi|)\,|\xi|}{a_\e(|\xi|)}\,\frac{\xi_i\,\xi_j}{|\xi|^2}+\delta_{ij} \right|\stackrel{\eqref{iaesae}}{\leq} n^2\,\max\{1,1+s_a\}\,a_\e(|\xi|).
    \end{split}
\end{equation*}
for all $0<\e<\e_\delta$. We are left to consider the case $|\xi|\in \big[\delta,\rm{e}^{C_\star}\delta\big]\cup \big[\delta^{-1},\rm{e}^{C_\star}\delta^{-1}\big]$. From \eqref{d:Aed} we compute
\begin{equation*}
    \begin{split}
         \frac{\partial \A^i_{\e,\delta}(\xi)}{\partial \xi_j}\eta_i\eta_j \stackrel{\eqref{temp:cogr},\eqref{iaesae}}{\geq}  &\big(1-\eta_\delta(|\xi|) \big)\,\frac{\l}{2}\,a(|\xi|)\,|\eta|^2+\eta_\delta(|\xi|)\,\min\{1,1+i_a\}\,a_\e(|\xi|)\,|\eta|^2
         \\
         &+\eta'_\delta(|\xi|)\left\{ -(\widehat\A_\e(\xi)\cdot \eta)\,\frac{(\xi\cdot \eta)}{|\xi|}+b_\e(|\xi|)\,\frac{|\xi\cdot \eta|^2}{|\xi|^2}\right\}
         \\
         \stackrel{\eqref{eta:d1}\text{-}\eqref{bounds:delta},\eqref{temp:cogr1} }{\geq}  &\min\{\l/4,1,1+i_a\} \,a_\e(|\xi|)\,|\eta|^2-\frac{1}{C_\star\,|\xi|}(2\,C_1+1)\,b_\e(|\xi|)\,|\eta|^2
         \\
         \stackrel{\eqref{def:bBe},\eqref{cstar}}{=} &\frac{1}{2}\min\{\l/4,1,1+i_a\}\,a_\e(|\xi|)\,|\eta|^2\,.
    \end{split}
\end{equation*}
 Analogously, from \eqref{d:Aed} we obtain
\begin{equation*}
    \begin{split}
         \left|\frac{\partial \A^i_{\e,\delta}(\xi)}{\partial \xi_j}\right|\leq & \big(1-\eta_\delta(|\xi|) \big)\, \left|\frac{\partial \widehat\A_\e^i(\xi)}{\partial \xi_j}\right|+\eta_\delta(|\xi|)\,a_\e(|\xi|)\,\left|\frac{a_\e'(|\xi|)\,|\xi|}{a_\e(|\xi|)}\,\frac{\xi_i\,\xi_j}{|\xi|^2}+\delta_{ij}\right|
         \\
         &+|\eta'_\delta(|\xi|)|\,\Big(|\widehat\A_\e^i(\xi)|+b_\e(|\xi|) \Big)
         \\
         \stackrel{\eqref{temp:cogr},\eqref{iaesae},\eqref{eta:d1},\eqref{temp:cogr1}}{\leq} &\big(1-\eta_\delta(|\xi|)\big)\,2\,\L\,a(|\xi|)+\eta_\delta(|\xi|)\,a_\e(|\xi|)\,\max\{1,1+s_a\}
         \\
         &+\frac{1}{C_\star\,|\xi|}(2\,C_1+1)\,b_\e(|\xi|)
         \\
         \stackrel{\eqref{eta:d2},\eqref{bounds:delta},\eqref{def:bBe}}{\leq} &\big[\max\{4\L,1,1+s_a\}+C_\star^{-1}(2\,C_1+1)\big]\,a_\e(|\xi|)\,.
    \end{split}
\end{equation*}
Therefore, we have shown that $\A_{\e,\delta}$ satisfies the coercivity 
and growth conditions \eqref{coer:Ae} for all $\delta \in (0,\delta_0)$ and 
all $0 < \e < \e_\delta$. The proof is completed by choosing a sequence 
$\delta_k \to 0$ and, via a diagonal argument, a sequence $\e_k \to 0$, hence  
the vector fields $\A_k \equiv \A_{\e_k,\delta_k}$ will satisfy the 
desired properties, up to relabeling the sequence.
\end{proof}

\begin{remark*}
    \rm{Here we construct the function $\eta_\delta$ fulfilling \eqref{eta:d1}-\eqref{eta:d2}. First, we define the function
    \begin{equation*}\widehat{\eta}_\delta(t)=
        \begin{cases}
            1\quad &t\in (-\infty,\,2\delta ]
            \\
            1-\frac{1}{2C_\star}\ln\left( \frac{t}{2\d}\right)\quad &t\in \big[2\d,\,2\rm{e}^{2C_\star}\d \big]
            \\
            0\quad & t\in \big[2\rm{e}^{2C_\star}\delta,\,2\d^{-1} \big]
            \\
            \frac{1}{2\,C_\star}\ln\left( \frac{\delta\,t}{2}\right)\quad & t\in \big[2\d^{-1},\,2\rm{e}^{2C_\star}\d^{-1} \big]
            \\
            1\quad & t\in \big[2\rm{e}^{2C_\star}\d^{-1}, \,+\infty \big)\,.
        \end{cases}
    \end{equation*}
    Then $\widehat{\eta}_\delta$ is Lipschitz continuous, and $|\widehat{\eta}'_\delta(t)|\leq 1/(2C_\star t)$ for all $t>0$, hence the desired function $\eta_\delta$ can be obtained via convolution. For instance, set
    \begin{equation*}
        \eta_\delta(t)=\widehat{\eta}_\delta\ast \rho_{\d^2}(t)\,.
    \end{equation*}
    By standard properties of convolution we have
    \[
    \eta_\delta \in C^{\infty}([0,\infty)),\quad 0\leq \eta_\d\leq 1,
    \]
    \[
    \eta_\d\equiv 1\,\text{ in $[0,2\delta-\delta^2]\cup \big[ 2\rm{e}^{2C_\star}\d^{-1}+\delta^2,+\infty)$,}\quad\text{and }\,\eta_\d\equiv 0\,\text{ in $\big[2\rm{e}^{2C_\star}\d+\delta^2,\,2\d^{-1}-\delta^2\big]$. }
    \]
In particular \eqref{eta:d2} is satisfied since $\delta<1$. Finally, we estimate $|\eta_\d'(t)|$. Clearly $\eta_\delta'\equiv 0$ in $[0,\delta]$, while for $t>\delta$  we obtain
    \begin{equation*}
        \begin{split}
            |\eta_\d'(t)| &\leq \int_{t-\d^2}^{t+\d^2} |\widehat{\eta}_\d'(s)|\,\rho_{\d^2}(s-t)\,ds\leq \frac{1}{2C_\star}\int_{t-\d^2}^{t+\d^2}  \frac{\rho_{\d^2}(s)}{s}\,ds
            \\
            &\leq \frac{1}{2C_\star(t-\delta^2)}\int_{t-\d^2}^{t+\d^2} \rho_{\d^2}(s)\,ds \leq \frac{1}{C_\star\,t}\,, 
        \end{split}
    \end{equation*}
   where, in the last inequality, we used that $t-\delta^2\geq t/2$ since $t>\delta$ and $\delta<1/2$. Equation \eqref{eta:d1} is thus proven.
    }
\end{remark*}

\subsection{Classes of domains and boundary flattening}\label{subsec:domain}
In what follows, we will consider $\Omega\subset \R^n$ a (not necessarily bounded) domain, i.e., an open connected set. We start with the following definitions.
\vspace{0.2cm}

\begin{definition}\label{def:Lom}\rm{ Let $\mathcal{U}$ be a bounded domain of $\R^n$. We say that $\Omega$ is a \textit{Lipschitz domain relatively to $\mathcal{U}$} if  there exist  constants $L_\Omega>0$ and $R_\Omega \in (0, 1)$ 
such that, for every $x_0\in \partial \Omega\cap \mathcal{U}$ and $R\in (0, R_\Omega]$ there exist an isometry $T=T_{x_0}$ of $\R^n$ such that $Tx_0=0$, an $L_\Omega$-Lipschitz continuous function 
$\phi=\phi_{x_0} : B'_{R}\to (-\ell, \ell)$, where $B'_{R}$ denotes the ball in $\mathbb R^{n-1}$, centered at $0'\in \R^{n-1}$ and with radius $R$, and
\[
\ell = R (1+L_\Omega),
\]
satisfying $\phi(0')=0$, and
\begin{equation}\label{may100}
\begin{split}
    &T(\partial \Omega\cap\mathcal{U}) \cap \big(B'_{R}\times (-\ell,\ell)\big)=\{(x', \phi (x'))\,:\,x'\in B'_{R}\},
    \\
    & T(\Omega\cap\mathcal{U}) \cap \big(B'_{R}\times (-\ell,\ell)\big)=\{(x',x_n)\,:\,x'\in B'_{R}\,,\,\phi (x')<x_n<\ell\}.
\end{split}
\end{equation}
The function $\phi$ is usually called local boundary chart. We set
\begin{equation}\label{may101}
\mathfrak L_\Omega = (L_\Omega, R_\Omega),
\end{equation}
and call $\mathfrak L_\Omega$ a \textit{Lipschitz characteristic} of $\Omega$ (relatively to $\mathcal{U}$). We remark that the Lipschitz characteristic is not unique. since we may always reduce the characteristic radius \(R_\Omega\).
\vspace{0.1cm}

 Given $\alpha\in (0,1]$, we say that $\partial\Omega\cap \mathcal{U}$ is of class $C^{1,\alpha}$  if the function $\phi$ satisfying \eqref{may100} belongs to $C^{1,\alpha}(B'_R)$, and we write $\partial \Omega\cap \mathcal{U}\in C^{1,\alpha}$. 
}
\end{definition}

Next, given $x_0\in \partial \Omega\cap \mathcal{U}\in C^{1,\alpha}$, and $T=T_{x_0},\phi=\phi_{x_0}$ fulfilling \eqref{may100}, we may define the local $C^{1,\alpha}$-diffeomorphism
\begin{equation}\label{diffeom}
    \begin{split}
    &\Phi:T^{-1}\big(B'_{R_\Omega}\times (-\ell,\ell) \big)\to \R^n,
    \\
    &(x',x_n)\mapsto \Phi(x',x_n)=(y',y_n-\phi(y')),\quad y=Tx,
    \end{split}
\end{equation}
 which satisfies, for all $0<R\leq R_\Omega$,
 \begin{equation}\label{flatten}
 \Phi\big(\Omega \cap T^{-1}\big(B'_{R}\times (-\ell,\ell)\big)\big)\subset B'_{R}\times (0,+\infty),\quad 
 \Phi\big(\partial \Omega \cap T^{-1}\big(B'_{R}\times (-\ell,\ell) \big)\big)= B^0_{R},  
 \end{equation}
 and whose inverse function is given by
 \begin{equation*}
      \Phi^{-1}(y',y_n)=T^{-1}(y',x_n+\varphi(y')).
 \end{equation*}
Note that the second equation in \eqref{flatten} tells that $\Phi$  is a flattening diffeomorphism, i.e., it (locally) maps the boundary $\partial\Omega$ onto $B^0_R$. Note also that 
\begin{equation}\label{graddifeom}
  \mathrm{det}\,\nabla \Phi\equiv 1\quad\text{and}\quad  \|\nabla \Phi\|_{\infty}+\|\nabla\Phi^{-1}\|_{\infty}\leq C(n)\,(1+L_\Omega),
\end{equation}
which implies
\begin{equation}\label{okii}
    \frac{1}{C(n)(1+L_\Omega)}\,|\xi|\leq|\nabla\Phi^\top(x)\,\xi|\leq C(n)(1+L_\Omega)\,|\xi|\quad\text{for all $\xi\in \R^n$.}
\end{equation}

Now let $x_0\in \partial \Omega\cap \mathcal{U}$, and  denote by 
\begin{equation}\label{cilindr}
\begin{split}
\mathcal{Q}_{\Omega,x_0}&:=T_{x_0}^{-1}(B'_{R_\Omega}\times (-\ell,\ell))
\\
   \mathcal{C}_\Omega:=\Omega\cap T^{-1}(B'_{R_\Omega}\times (-\ell,\ell))\quad &\text{and}\quad\mathcal{C}_{\partial\Omega}:=\partial\Omega\cap T^{-1}(B'_{R_\Omega}\times (-\ell,\ell)).
   \end{split}
\end{equation}
Suppose that $u\in W^{1,B}(\mathcal{C}_\Omega)$ is a weak solution to either the Dirichlet or Neumann problem
\begin{equation}\label{pb:C}
    \begin{cases}
        -\mathrm{div}\big( \A(x,Du)\big)=f\quad&\text{in $\mathcal{C}_\Omega$}
        \\
        u=g\quad &\text{on $\mathcal{C}_{\partial \Omega}$}
    \end{cases}
    \quad\text{or}\quad
     \begin{cases}
        -\mathrm{div}\big( \A(x,Du)\big)=f\quad&\text{in $\mathcal{C}_\Omega$}
        \\
        \A(x,Du)\cdot \nu=h\quad &\text{on $\mathcal{C}_{\partial \Omega}$}.
    \end{cases}
\end{equation}
Then by setting 
\[
\hat{u}(y)=u\circ \Phi^{-1}(y),\quad y\in  \Phi(\mathcal{C}_\Omega),
\]
a simple change of variables shows that $\hat{u}\in W^{1,B}(\Phi(\mathcal{C}_\Omega))$ is a weak solution to either
\begin{equation}\label{change:var}
        \begin{cases}
        -\mathrm{div}\big( \hat\A(y,D\hat u)\big)=\hat f\quad&\text{in $\Phi(\mathcal{C}_\Omega)$}
        \\
       \hat u=\hat g\quad &\text{on $\Phi(\mathcal{C}_{\partial\Omega})$}
    \end{cases}
    \quad\text{or}\quad
     \begin{cases}
        -\mathrm{div}\big( \hat\A(y,D\hat u)\big)=\hat f\quad&\text{in $\Phi(\mathcal{C}_\Omega)$}
        \\
        \hat \A(y,D\hat u)\cdot e_n+\hat h=0\quad &\text{on $\Phi(\mathcal{C}_{\partial\Omega})$}.
    \end{cases}
\end{equation}
respectively, with $\Phi(\mathcal{C}_{\partial\Omega})=B^0_{R_{\Omega}}$ by \eqref{flatten}. Above, we set
\begin{equation}\label{def:hat}
\hat \A(y,\xi)=\big[\nabla \Phi(\Phi^{-1}(y))\big]\,\A\Big(\Phi^{-1}(y),\,\big[\nabla \Phi\big(\Phi^{-1}(y)\big)\big]^{\top}\,\xi\Big),\quad \hat f(y)=f(\Phi^{-1}(y))
\end{equation}
and
\begin{equation}\label{def:hat1}
g(y')=g(\Phi^{-1}(x')),\quad \hat h(y')=h(\Phi^{-1}(y'))\,\sqrt{1+|\nabla \phi(y')|^2},\quad y'\in B
^0_{R_\Omega}.
\end{equation}
Let us now show that equation \eqref{eq1} is invariant under the change of
coordinates $\Phi$; namely, the transformed problem \eqref{change:var} still
satisfies assumptions \eqref{ass:A} and \eqref{A:hold}.

Clearly we have $\hat f\in L^d(\Phi(\mathcal{C}_\Omega))$, and the functions $\hat{g},\hat{h}$ are of class $C^{1,\alpha}(B^0_{R_0}),C^{0,\alpha}(B^0_{R_0})$, respectively.
Moreover, by setting $M(y):=[\nabla\Phi(\Phi^{-1}(y))]^\top$, a simple computation shows 
\begin{equation*}
   \sum_{i,j=1}^n \frac{\partial\hat{\A}^i}{\partial\xi_j}(y,\xi)\,\eta_i\cdot\eta_j= \sum_{i,j=1}^n \frac{\partial\A^i}{\partial\xi_j}\big(\Phi^{-1}(y),M(y)\xi\big)\,\big( M(y)\,\eta\big)\cdot\big( M(y)\,\eta\big)
\end{equation*}
and thanks to \eqref{ultra:utile} and \eqref{okii}, we also have 
\[
c(n,i_a,s_a,L_\Omega)\,a(|\xi|)\leq a\big(|M(y)\xi|\big)\leq C(n,i_a,s_a,L_\Omega)\,a(|\xi|).
\]
Thus, from \eqref{ass:A},\eqref{okii} and the two estimates above, we deduce that $\hat \A$ still satisfies the coercivity and growth conditions
\begin{equation}\label{new:coer}
\begin{split}
&\sum_{i,j=1}^n\left|\frac{\partial \hat\A^i }{\partial \xi_j}(x,\xi)\right|\leq C\big(n,\l,\L,i_a,s_a,L_\Omega\big)\,a(|\xi|)
\\
 &  \sum_{i,j=1}^n \frac{\partial\hat\A^i}{\partial\xi_j}(y,\xi)\,\eta_i\,\eta_j\geq c\big(n,\l,\L,i_a,s_a,L_\Omega\big)\,a(|\xi|)\,|\eta|^2
   \end{split}
\end{equation}
for all $y\in \Phi(\mathcal{C}_\Omega)$, $\xi\in \R^n\setminus\{0\}$ and $\eta\in \R^n$. Additionally, $\hat{\A}$ satisfies 
\begin{equation}\label{A:hold1}
\begin{split}
   | \hat\A(y,\xi)-\hat\A(z,\xi)|&\leq C\big(n,\l,\L,\L_{\mathrm{h}},\alpha, i_a,s_a,L_\Omega\big)\,\big(1+b(|\xi|)\big)\,\big(1+\|\phi\|_{C^{1,\alpha}}\big)\,|y-z|^\alpha
   \\
   &|\hat\A(y,0)|\leq C(n)\,(1+L_\Omega)\,\L.
   \end{split}
\end{equation}
for all $y,z\in \Phi(\mathcal{C}_\Omega)$. 

Equation \eqref{A:hold1}$_2$ immediately follows from \eqref{A:hold}$_2$, \eqref{def:hat} and \eqref{graddifeom}. Then, to prove \eqref{A:hold1}$_1$, it clearly suffices to study the quantity
\begin{equation}\label{lkodevo}
\begin{split}
    \big|\A\big(y,M(y)\,\xi\big) &-\A(y,M(z)\,\xi)\big|
    \\
    &\leq \sum_{j=1}^n\bigg|\int_0^1 \frac{\partial\A}{\partial \xi_j}\Big(tM(y)\xi+(1-t)M(z)\xi\Big)\,dt\bigg|\,|M(y)-M(z)|\,|\xi|
    \\
    &\leq C(n)\,\bigg(\int_0^1 a\Big(\big|tM(y)\xi+(1-t)M(z)\xi \big| \Big)\,dt\bigg)\,\|\phi\|_{C^{1,\alpha}}|y-z|^{\alpha}\,|\xi|\,,
    \end{split}
\end{equation}
where we used the fundamental theorem of calculus, \eqref{ass:A} and $\|M\|_{C^{0,\alpha}}\leq C(n)\,\|\phi\|_{C^{1,\alpha}}$.
Now assume for the moment that $|y-z|\leq \big(2\,C(n)\,(1+L_\Omega)\,\|\phi\|_{C^{1,\alpha}}\big)^{-1/\alpha}$, where $C(n)$ is the constant appearing in \eqref{okii}. From said inequality, we also have
{\small
\begin{equation*}
    \begin{split}
         C(n)\,(1+L_\Omega)\,|\xi|\,\geq\big|tM(y)\xi+(1-t)M(z)\xi\big|&\geq |M(y)\xi|-\|\phi\|_{C^{0,\alpha}}\,|\xi|\,|y-z|^\alpha\geq \frac{|\xi|}{2\,C(n)(1+L_\Omega)}.
    \end{split}
\end{equation*}
}
Coupling this information with \eqref{ultra:utile}, we deduce \[
a\big(\big|tM(y)\xi+(1-t)M(z)\xi\big| \big)\leq C(n,i_a,s_a,L_\Omega)\,a(|\xi|),
\] 
and this estimate together with \eqref{lkodevo} and \eqref{def:a} yields \eqref{A:hold1}. On the other hand, when $|y-z|\geq \big(C(n)\,(1+L_\Omega)\|\phi\|_{C^{1,\alpha}}\big)^{-1/\alpha}$, from \eqref{co:gr}, \eqref{A:hold}$_2$, \eqref{okii} and the monotonicity of $b(t)$, we get
{\small
\begin{equation*}
\begin{split}
    \big|\A\big(y,M(y)\,\xi\big) -\A(y,M(z)\,\xi)\big|&\leq C\,\Big(b(\max\{ |M(y)\,\xi|,|M(z)\,\xi|\})+1\Big)\big(C(n)\,(1+L_\Omega)\|\phi\|_{C^{1,\alpha}}\big)\, |y-z|^\alpha
    \\
    & \leq C'\big(1+b(|\xi|)\big)\,\|\phi\|_{C^{1,\alpha}}\,|y-z|^\alpha,
    \end{split}
\end{equation*}}
with $C,C'=C,C'(n,\l,\L,\L_\mathrm{h},i_a,s_a,L_\Omega)$, so that \eqref{A:hold1} is proven in this case as well.
\vspace{0.2cm}

\noindent
We conclude this subsection by introducing some notation.
Let $\mathcal{U} \subset \mathbb{R}^n$ be a bounded open set such that
$\partial\Omega \cap \mathcal{U}$ is of class $C^{1,\alpha}$, and let
$\mathcal{U}' \Subset \mathcal{U}$ be open and such that $\partial\Omega\cap \mathcal{U}'\neq \varnothing$. In particular, by Definition \ref{def:Lom}, this implies that  $\partial\Omega \cap \mathcal{U}'$ is of class $C^{1,\alpha}$  as well.

Let $\{x_i\}_{i=1}^N \subset \partial\Omega \cap \overline{\mathcal{U}'}$
and let $\{\phi_i\}_{i=1}^N=\{\phi_{x_i}\}_{i=1}^N$ be coordinate charts satisfying \eqref{may100},
with $\Phi_i$ denoting the associated diffeomorphisms defined in \eqref{diffeom}. Assume that the corresponding coordinate cylinders 
$\{\mathcal{Q}_{\Omega,x_i}\}_{i=1}^N$ form an open cover of
$\partial\Omega \cap \overline{\mathcal{U}'}$. We then define
\begin{equation}\label{normadeom}
\|\partial\Omega \cap \mathcal{U}\|_{C^{1,\alpha}(\mathcal{U}')}
:= \sup_{i=1,\dots,N} \|\phi^i\|_{C^{1,\alpha}}.
\end{equation}
If $\Omega$ is a bounded domain of class $C^{1,\alpha}$, the notation
$\|\partial\Omega\|_{C^{1,\alpha}}$ is self-explanatory.

We also denote by
\begin{equation}\label{norm:hg}
     \|h\|_{C^{0,\alpha}(\partial\Omega\cap \bar{\mathcal{U}}')}:=\max_{i=1,\dots,N}\|h\circ \Phi_i^{-1}\|_{C^{0,\alpha}(B^0_{R_\Omega})},\quad\|g\|_{C^{1,\alpha}(\partial\Omega\cap \bar{\mathcal{U}}')}:=\max_{i=1,\dots,N}\|g\circ \Phi_i^{-1}\|_{C^{1,\alpha}(B^0_{R_\Omega})}.
\end{equation}

\subsection{Auxiliary results and lemmas }
In this final subsection, we collect some useful results and lemmas for later use. We start with an elementary property of averages:  for any measurable set $U\subset \R^n$ with $0<|U|<\infty$, we have
\begin{equation}\label{medie}
   \mint_{U}\big|V(x)-(V)_{U}\big|^q\,dx\leq C(n,N,q) \,\mint_{U}\big|V(x)-V_0\big|^q\,dx
\end{equation}
for all $q\geq 1$, $V\in L^q(U;\R^N)$, $N\in\N\setminus\{0\}$, and for any constant vector $V_0\in \R^N$.

For any $\tau\in(0,1) $, we also have
\begin{equation}\label{added}
   \mint_{B_r}B\big(|V-(V)_{B_r} |\big) \,dx\leq C_\tau\,\mint_{B_R}B\big(|V-(V)_{B_R} |\big) \,dx\,,\quad\text{for all } \tau R\leq r\leq R,
\end{equation}
where $B_r\subset B_R$  are concentric balls, and $C_\tau>0$ depends on $i_a,s_a,\tau$.  To prove it, we use \eqref{triangle}, \eqref{B2t} and Jensen inequality, and get
\begin{equation*}
    \begin{split}
         \mint_{B_r}B\big(|V-(V)_{B_r} |\big) \,dx & \leq C\, \mint_{B_r}B\big(|V-(V)_{B_R} |\big) \,dx+C\,B\big(|(V)_{B_r}-(V)_{B_R} |\big)
         \\
         &\leq C\,\tau^{-n}\mint_{B_R}B\big(|V-(V)_{B_R} |\big) \,dx+C\,B\Big(\mint_{B_r}|V-(V)_{B_R}|\,dx \Big)
         \\
         &\leq C\,\tau^{-n}\mint_{B_R}B\big(|V-(V)_{B_R} |\big) \,dx+C\,\tau^{-ns_B}B\Big(\mint_{B_R}|V-(V)_{B_R}|\,dx \Big)
         \\
         &\leq C\,\tau^{-n}\mint_{B_R}B\big(|V-(V)_{B_R} |\big) \,dx+C\,\tau^{-ns_B}\mint_{B_R}B\big(|V-(V)_{B_R}|\big)\,dx,
    \end{split}
\end{equation*}
with $C=C(i_a,s_a)$, so \eqref{added} is proven.
\vspace{0.1cm}

The next, standard lemma can be found in  \cite[Lemma 4.3]{HL11} or \cite[Chapter 5, Lemma 3.1]{G83}.

\begin{lemma}\label{lemma:uss}
    Let $Z(t)\geq 0$ be a bounded function in $[\tau_0,\tau_1]$, $\tau_0\geq 0$. Suppose that
\begin{equation*}
    Z(t)\leq \theta\,Z(s)+\frac{A}{(s-t)^\alpha}+B\,,\quad\text{for all $\tau_0\leq t<s\leq \tau_1$,}
\end{equation*}
    for some $\theta\in [0,1)$, and $A,B,\alpha\geq 0$. Then
    \begin{equation*}
        Z(t)\leq C(\alpha,\theta)\,\left\{ \frac{A}{(s-t)^\alpha}+B\right\},
    \end{equation*}
    for all $\tau_0\leq t<s\leq \tau_1$.
\end{lemma}

In the following, we state three iterative lemmas, which are all essentially equivalent.
The first one can be found in \cite[Lemma 8.23]{GT}, and it is particularly useful to quantify the oscillation of a function.

\begin{lemma}\label{lemma:osc}
    Let $\omega$ be a non-decreasing non-negative function on $(0,R_0]$ such that
    \begin{equation*}
        \omega(\tau R)\leq \theta\,\omega(R)+C_0\,R^\alpha,\quad R\leq R_0
    \end{equation*}
    for $\tau,\theta_0\in (0,1)$ and $\alpha, C_0>0$. Then there exists $C_1=C_1(\tau,\theta)>0$ and $\beta=\beta(\theta,\tau)>0$ such that
    \begin{equation*}
        \omega(R)\leq C_1\,\left( \frac{R}{R_0}\right)^{\beta}\,\omega(R_0)+C_1\,C_0\,R^\alpha.
    \end{equation*}  
\end{lemma}

For the next simple, yet fundamental iterative lemma we refer to \cite[Lemma 5.13]{GM12}.

\begin{lemma}\label{lem:iteration}
Consider a non-decreasing function $\phi: (0,R_0] \to [0,\infty)$
which satisfies, for some constants \(A>0\), \(B\ge 0\) and exponents \(\alpha>\beta\), the inequality
\begin{equation*}
\phi(r)\leq
A\left[\left(\frac{r}{R}\right)^{\alpha} + \varepsilon\right] \phi(R)
+ B R^{\beta},
\quad
\text{for all }\, 0<r \le R \le R_0 .
\end{equation*}
Then there exists $C=C(A,\alpha,\beta)>0$ such that, if \[0\leq \e< \e_0=\Big(\frac{1}{2A} \Big)^{\frac{2\alpha}{\alpha-\beta}},\]
then
\begin{equation*}
\phi(r)
\leq
C\,\Big[
\frac{\phi(R)}{R^\beta} 
+ B
\Big]\,r^{\beta},
\quad
\text{for all }\, 0<r \leq R \leq R_0.
\end{equation*}
\end{lemma}

The next and final iterative lemma is useful for handling functions that are not necessarily monotone. The proof can be found in \cite[Lemma~7.3, p.~229]{G03}.

\begin{lemma}\label{lem:ultiter}
Let \(\varphi(t)\) be a positive function, and assume that there exist a constant \(q>0\) and a number \(\tau\in(0,1)\)  such that for every \(R \le R_0\)
\begin{equation*}
\varphi(\tau R) \le \tau^{\delta}\varphi(R) + B R^{\beta},
\end{equation*}
with \(0<\beta<\delta\), and
\begin{equation*}
\varphi(t) \le q\,\varphi(\tau^{k}R)
\end{equation*}
for every \(t\) in the interval \((\tau^{k+1}R,\tau^{k}R)\). Then, for every \(0<r < R \le R_0\), we have
\begin{equation*}
\varphi(r)
\leq
C\left\{
\left(\frac{r}{R}\right)^{\beta}\varphi(R)
+ B\,r^{\beta}
\right\},
\end{equation*}
where \(C\) is a constant depending only on \(q\), \(\tau\), \(\delta\), and \(\beta\).
\end{lemma}

The next result is De Giorgi's hypergeometric lemma (see \cite[Lemma 7.1, pp. 220]{G03}).
\begin{lemma}\label{lem:hyp}
    Let $\{Z_m\}$, $m=0,1,2,\dots$, be a sequence of positive numbers satisfying the recursive inequality
    \begin{equation*}
        Z_{m+1}\leq C_0\,b^{m}\,Z_m^{1+\alpha}\,,
    \end{equation*}
    where $C_0,b>1$ and $\alpha>0$ are given numbers. If $Z_0\leq C_0^{-1/\alpha}b^{-1/\alpha^2}$, then
        \begin{equation*}
        \lim_{m\to \infty} Z_m=0\,.
        \end{equation*}
\end{lemma}

Next, for a given function $u \in W^{1,1}(B_\varrho)$ and $\kappa\in \R$, we define the super-level set
\begin{equation}\label{def:Akr}
    A(\kappa,\varrho) \coloneqq 
    \bigl\{x \in B_\varrho : u(x) > \kappa \bigr\}.
\end{equation}
If instead $u \in W^{1,1}(B^+_\varrho)$, the same definition applies 
with $B_\varrho$ replaced by $B^+_\varrho$.

We have the following lemma, which is a slight modification of a lemma by De Giorgi \cite{DG56}-- see also \cite[Lemma 3.5, Chapter 3]{LU68}. It is often referred to as \emph{the discrete isoperimetric inequality}.

\begin{lemma}\label{lemma:levels}
For any function $u \in W^{1,1}(B_\varrho)$ and for all $\kappa < \ell$, one has
\begin{equation}\label{levels}
    (\ell-\kappa)\,|A(\ell,\varrho)|^{1-\frac{1}{n}}
    \leq c(n)\,
    \frac{|B_\varrho|}{\,|B_\varrho \setminus A(\kappa,\varrho)|}\,
    \int_{A(\kappa,\varrho)\setminus A(\ell,\varrho)} |Du|\,dx .
\end{equation}
The same estimate holds if $u\in W^{1,1}(B^+_\varrho(x_0))$, upon replacing 
$B_\varrho(x_0)$ with $B^+_\varrho(x_0)$ in \eqref{levels}.
\end{lemma}

Lemma \ref{lemma:levels} is an easy consequence of the following Poincar\'e type inequality. A similar inequality can be found in \cite[Chapter 1, Proposition 2.1]{DB94}, \cite[Proposition 5.2]{DB10}.

\begin{lemma}\label{lemma:poin}
    Let  $\mathcal E$  be a bounded open convex subset of $\R^n$, let $v\in W^{1,1}(\mathcal{E})$ be an arbitrary function, and suppose that the set
    \begin{equation*}
        A_0=\bigl\{x\in \mathcal E:v(x)=0 \bigr\}
    \end{equation*}
    has positive measure. Then, for any measurable set $A\subset \mathcal{E}$, the following inequality is valid:
    \begin{equation}\label{eq:levels}
        \int_A |v|\,dx\leq c(n)\,\frac{\big(\mathrm{diam}(\mathcal{E})\big)^n}{|A_0|}\,|A|^{\frac{1}{n}}\,\int_\mathcal{E}|Dv|\,dx\,.
    \end{equation}
\end{lemma}

\begin{proof}
   We may assume that $|A|>0$, for otherwise there is nothing to prove.
   For almost every $x\in \Omega$ and $z\in A_0$, we have
    \begin{equation*}
    \begin{split}
        |v(x)|=|v(x)-v(z)| & =\Bigg| \int_0^{|x-z|}\frac{\partial }{\partial r} v\Big( z+\frac{x-z}{|x-z|}r\Big)\,dr\Bigg|
        \\
        &\leq \int_0^{|x-z|}\left|Dv\left(x+\frac{x-z}{|x-z|}r\right)\right|dr
        \end{split}
    \end{equation*}
We integrate the above identity in $z\in A_0$, and use the polar coordinates with pole at $x$ and radial variable $\vrho=|x-z|$. We also let $\omega=\frac{x-z}{|x-z|}$ be the angular variable, and denote by $\mathcal{R}(\omega)$ the polar representation of $\partial \mathcal{E}$ with pole at $x$. We thus get
\begin{equation*}
    \begin{split}
        |A_0|\,|v(x)|&\leq\int_{A_0}\int_0^{|x-z|}\left|Dv\left(x+\frac{x-z}{|x-z|}r\right)\right|dr\,dz
        \\
        &\leq\int_{\mathbb{S}^{n-1}}\int_0^{\mathcal{R}(\omega)}\vrho^{n-1}\left(\int_0^{\vrho}|Dv(x+r\omega)|\,dr\right)\,d\vrho\,d\omega
        \\
        &\leq \int_{\mathbb{S}^{n-1}}\int_0^{\mathcal{R}(\omega)}\vrho^{N-1}\left(\int_0^{\mathcal{R}(\omega)}|Dv(x+r\omega)|\,dr\right)\,d\vrho\,d\omega
        \\
        &\leq \left(\int_0^{\mathrm{diam}(\mathcal{E})}\vrho^{n-1}\,d\vrho \right)\,\int_{\mathbb{S}^{n-1}}\int_0^{\mathcal{R}(\omega)}|Dv(x+r\omega)|\,dr\,d\omega
        \\
        &=\frac{\big(\mathrm{diam}(\mathcal{E})\big)^n}{n}\,\int_\Omega \frac{|Dv(y)|}{|x-y|^{n-1}}dy
    \end{split}
\end{equation*}
where $\mathbb{S}^{n-1}=\{x\in \R^n: |x|=1\}$ is the unit sphere of $\R^n$. 

We integrate both sides of this inequality over the set $A$, and applying Fubini-Tonelli theorem, for any $\delta>0$ we get
\begin{equation*}
    \begin{split}
        |A_0|&\,\int_{A}|v(x)|\,dx\leq \frac{\big(\mathrm{diam}(\mathcal{E})\big)^n}{n}\,\int_A\int_\mathcal{E} \frac{|Dv(y)|}{|x-y|^{n-1}}dy\,dx
        \\
        &=\frac{\big(\mathrm{diam}(\mathcal{E})\big)^n}{n}\int_\mathcal{E} |Dv(y)|\,\bigg\{\int_{A\cap\big\{|x-y|\geq \delta\big\}}\frac{dx}{|x-y|^{n-1}}+ \int_{A\cap\big\{|x-y|< \delta\big\}}\frac{dx}{|x-y|^{n-1}}\bigg\}\,dy
        \\
        &\leq \frac{\big(\mathrm{diam}(\mathcal{E})\big)^n}{n}\int_\mathcal{E} |Dv(y)|\,dy\,\Big[\delta^{1-n}|A|+\mathcal{H}^{n-1}(\mathbb{S}^{n-1})\,\delta \Big]\,.
    \end{split}
\end{equation*}
By choosing $\delta=|A|^{\frac{1}{n}}$, Equation \eqref{eq:levels} follows.
\end{proof}

\begin{proof}[Proof of Lemma \ref{lemma:levels}]
    It suffices to apply Lemma \ref{lemma:poin}  with either $\mathcal{E}= B_\vrho$ or $\mathcal{E}=B_\vrho^+$, with  function 
    \begin{equation*}
         v(x)=\begin{cases}
            \min\{u(x),\ell\}-\kappa\quad & u(x)>\kappa
            \\
            0\quad & u(x)\leq \kappa\,,
        \end{cases}
    \end{equation*}
  and set $A=A(\ell,\vrho)$, so that
\begin{equation*}
\begin{split}
   \int_A v\,dx=(\ell-\kappa)|A(\ell&,\vrho)|\,,\quad \int_\mathcal{E} |Dv|\,dx=\int_{A(\kappa,\vrho)\setminus A(\ell,\vrho)}|Du|\,dx
   \\
   & A_0=\mathcal{E}\setminus A(\kappa,\vrho).
   \end{split}
   \end{equation*}
\end{proof}
We shall often use Sobolev inequality on half balls $B^+_r$, $r\leq 1$:
\begin{equation}\label{half:sobol}
    \Big(\int_{B^+_r}|w|^{\kappa_p}\,dx\Big)^{1/\kappa_p}\leq C(n)\,\Big(\int_{B^+_r}|Dw|^p\,dx\Big)^{1/p},
\end{equation}
for all $w\in W^{1,p}(B^+_r)$ such that $w=0$ on $\partial B^+_r\setminus B^0_r$, where we set
\begin{equation*}
    \kappa_p=\begin{cases}
        \frac{np}{n-p}\quad &1\leq p<n
        \\
        \text{any number}>1\quad &p\geq n.
    \end{cases}
\end{equation*}
To prove it, it suffices to apply Sobolev inequality to the even extension
\begin{equation}\label{def:even}
    w^{e}(x',x_n) =
    \begin{cases}
        w(x',x_n), & x_n > 0,\\[4pt]
        w(x',-x_n), & x_n \leq 0,
    \end{cases}
\end{equation}
since this satisfies $w^e=0$ on $\partial B_r$. Similarly, one can prove Poincar\'e inequality on half balls
\begin{equation}\label{half:poincare}
    \int_{B^+_r}|w|^p\,dx\leq C(n)\,r^p\,\int_{B^+_r}|Dw|^p\,dx,
\end{equation}
for all $w\in W^{1,p}(B_r^+)$, with $w=0$ on $\partial B^+_r\setminus B^0_r$. We remark that the same inequality holds true if $w \in W^{1,p}(B_r^+)$ is such that $w=0$ on $B^0_r$.

We shall also use the following trace inequality
\begin{equation}\label{in:trace0}
    \int_{B^0_r}|w|d\mathcal{H}^{n-1}\leq C(n)\,\int_{B^+_r}|Dw|\,dx,
\end{equation}
for all $w\in W^{1,1}(B_r^+)$, with $w=0$ on $\partial B^+_r\setminus B^0_r$. where the integrand on the left-hand side has to be interpreted in the sense of traces.
\vspace{0.2cm}

We conclude this section with a simple lemma, which allows us to reduce the right-hand side $f$ in divergence form. It is a simple consequence of Calderon-Zygmund theory for the Laplacian.

\begin{lemma}\label{lemma:diver}
    Let $U\subset \R^n$ be a bounded, open set, and let $f\in L^{d}(U)$, with $d>n$. Then there exists $F\in C^{0,1-\frac{n}{d}}(\bar{U};\R^n)$ such that
    \begin{equation}\label{divFf}
        \mathrm{div} F=f\quad\text{and}\quad   \|F\|_{C^{0,1-\frac{n}{d}}(\bar U)}\leq C(n,d)\,\|f\|_{L^d(U)}.
    \end{equation}
\end{lemma}

\begin{proof}
    Extend $f\equiv 0$ in $\R^n\setminus U$, and let $B_R$ be a ball so large  that $U\Subset B_R$. Consider $w\in W^{1,2}(B_{R})$ solution to
    \begin{equation*}
        \begin{cases}
            -\Delta w=f\quad &\text{in $B_R$}
            \\
            w=0\quad &\text{on $\partial B_{R}$.}
        \end{cases}
    \end{equation*}
Then by standard elliptic regularity theory we have that $w\in W^{2,d}(B_{R})$ with $\|w\|_{W^{2,d}(B_R)}\leq C(n,d)\,\|f\|_{L^d(B_R)}$. By taking $F=-\nabla w$, the thesis follows by Morrey's embeddings
\[
W^{1,d}_0(B_R)\hookrightarrow C_0^{0,1-\frac{d}{n}}\big(\overline{B}_R\big),\quad d>n.
\]
\end{proof}

\section{Boundedness of solutions}\label{sec:bounded}
In this section, we provide local $L^\infty$-bounds for solutions to \eqref{eq1}. Since we are dealing with zero order regularity, here we require weaker assumptions on the stress field $\A$. 

Specifically, we only need to assume that $\A:\overline{\Omega}\times \R^n\to \R^n$ is a continuous function, and it fulfills
\begin{equation}\label{weaker:A}
\begin{cases}
    \A(x,\xi)\cdot \xi \geq \l\,B\big(|\xi| \big)-\L\,B(K_1)\,
\\[1ex]
|\A(x,\xi)|\leq \L\,b\big(|\xi| \big)+\L\,b(K_1),
    \end{cases}
\end{equation}
for all $x\in \overline{\Omega}$, for all $\xi \in \R^n$, for some constants $0<\l\leq \L$, and $K_1>0$.

\begin{theorem}[Local boundedness of solutions]\label{thm:bdd}
  Let $B$ be a Young function fulfilling \eqref{iasa}, and let $u\in W^{1,B}(B_{R})$, be a weak solution to
    \begin{equation}\label{again:eq}
        -\diver\big(\A(x,Dw) \big)=0\quad\text{in $B_{R}$,}
    \end{equation}
with $\A(x,\xi)$ satisfying \eqref{weaker:A}. Then $u\in L^\infty_{loc}(B_{R})$, and there exists a constant $C=C(n,\l,\L,i_a,s_a)$ such that Caccioppoli inequality holds:
\begin{equation}\label{caccioppoli}
    \mint_{B_{R/2}}B\big(|Du| \big)\,dx\leq C\,\mint_{B_{R}} B\left( \frac{|u|}{R}\right)\,dx+C\,B(K_1),
\end{equation}
and the following $L^\infty$-bound is valid:
\begin{equation}\label{sol:Linf}
    \sup_{B_{R/2}}|u|\leq C\,\mint_{B_{R}}|u|\,dx+C\,K_1\,R.
\end{equation}
\end{theorem}

\begin{remark}[Scaling argument]\label{rem:scaling}
    \rm{Due to the lack of homogeneity of the Young function $B(t)$, it is often convenient to reduce our problem to the case $R=1$ via a scaling argument. 
    Suppose $u\in W^{1,B}(\mathcal{B}_R)$, with either $\mathcal{B}_R=B_R$ or $\mathcal{B}_R=B_R^+$. Setting
    \begin{equation}\label{uR:scale}
     u_R(y)=\frac{1}{R}u(Ry)\quad\ y\in \mathcal{B}_1,\quad\text{then}\quad  Du_R(y)=Du(x),\quad x=Ry
    \end{equation}
 so that  $u_R\in W^{1,B}(\mathcal{B}_1)$ is solution to 
    \begin{equation*}
        -\mathrm{div}\big( \A_R(y,Du_R)\big)=0\quad\text{in $\mathcal{B}_1$},\quad\text{where }\quad \A_R(y,\xi)=\A(Ry,\xi).
    \end{equation*}
 In particular, if $\A(x,\xi)$ fulfills either \eqref{weaker:A} or \eqref{ass:A}, so does $\A_R$.

Also, if $u=g$ in the sense of traces on $B_R^0$, then $u_R=g_R$ on $B_1^0$, with $g_R(y)=\tfrac{1}{R}g(Ry)$.

As for the conormal boundary condition, if
\[
\mathcal{A}(x,Du)\cdot e_n + h=0 \quad \text{on } B_R^0,
\]
then the rescaled function $u_R$ satisfies
\[
\mathcal{A}_R(y,Du_R)\cdot e_n + h_R=0 \quad \text{on } B_1^0,
\]
where $h_R(y') := h(Ry')$, $y' \in B_1^0$.

 }
\end{remark}

\begin{proof}[Proof of Theorem \ref{thm:bdd}]
We prove the theorem in the case $R=1$, as the general argument follows from the scaling argument of Remark \ref{rem:scaling}.

    Let $\eta\in C^\infty_c(B_{1})$ be a cut-off function, such that $0\leq \eta\leq 1$, $\eta\equiv 1$ in $B_{1/2}$, and $|D\eta|\leq C(n)$. We test the weak formulation of \eqref{again:eq} with function $u\,\eta^{s_B}$, and get
    \begin{equation*}
        \int_{B_{1}}\A(x,Du)\cdot Du\,\eta^{s_B}\,dx+s_B\,\int_{B_{1}}\A(x,Du)\cdot D\eta\,u\,\eta^{s_B-1}\,\,dx=0
    \end{equation*}
By means of \eqref{weaker:A} and of the properties of $\eta$, we obtain
\begin{equation}\label{esempoio:cac}
      \begin{split}\int_{B_{1}}\A(x,Du)\cdot Du\,\eta^{s_B}\,dx\geq  \l\,\int_{B_{1}} &B\big( |Du|\big)\,\eta^{s_B}\,dx-\L\,B(K_1)\int_{B_1}\eta^{s_B}\,dx
\\
      \bigg|\int_{B_{1}}\A(x,Du)\cdot D\eta\,u\,\eta^{s_B-1}\,\,dx\bigg|\leq &\,\L\,\int_{B_{1}} b\big( |Du|\big)\,\eta^{s_B-1}\,|u|\,|D\eta|\,dx 
      \\
      &+\L\,b(K_1)\,\int_{B_{1}}|u|\,|D\eta|\,\eta^{s_B-1}\,dx
      \\
      \stackrel{\eqref{Young}}{\leq} &\delta\,\int_{B_{1}} \widetilde{B}\big(b(|Du|)\,\eta^{s_B-1}\big)\,dx+C_\delta\int_{B_{1}} B(|u|)\,dx
      \\
      &+C\,\int_{B_{1}}\wB\big( b(K_1)\,\eta^{s_B-1}\big)\,dx+C\,\int_{B_{1}}B\big(|u|\big)\,dx
            \\
      \stackrel{\eqref{tB:usef},\eqref{wBbt}}{\leq} &C'\delta\,\int_{B_{1}}  B\big(|Du| \big)\,\eta^{s_B}\,dx+C'_\delta\int_{B_{1}} B(|u|)\,dx
      \\
      &+C'\,B(K_1)\,\int_{B_{1}}\,\eta^{s_B}\,dx,
      \end{split}
\end{equation}
    for all $\delta\in (0,1)$, where $C,C'>0$ depend on $n,\l,\L,i_a,s_a$, and $C_{\delta}, C'_\delta>0$ on $\delta$ as well. 
Choosing $\delta=\delta(n,\l,\L,i_a,s_a)\in (0,1)$ small enough to reabsorb terms, and using the properties of $\eta$, we get the desired Caccioppoli inequality
\begin{equation*}
   \int_{B_{1/2}} B(|Du|)\,dx\leq C\,\int_{B_1}B(|u|)\,dx+C\,B(K_1)\,.
\end{equation*}

We now move onto the proof of \eqref{sol:Linf}. To this end, we need to obtain a Caccioppoli inequality for $(u-\kappa)_+$. So, for $0<r\leq 1$ and $\kappa \geq 0$, we set
\begin{equation*}
    A(\kappa,r)=\{x\in B_r:\,u(x)>\kappa\}.
\end{equation*}
Let us consider $0<\sigma<\tau\leq 1$, and a cut-off function $0\leq\phi\leq 1$ such that
\begin{equation}\label{proprieta:phi}
    \phi\equiv 1\quad\text{in $B_{\sigma}$},\quad \phi\in C^\infty_c\big( B_{\tau}\big),\quad |D\phi|\leq C(n)/(\tau-\sigma)\,.
\end{equation}
Testing \eqref{again:eq} with $(u-\kappa)_+\,\phi^{s_B}$, we get
\begin{equation*}
    \int_{B_{1}}\A(x,Du)\cdot D(u-\kappa)_+\,\phi^{s_B}\,dx+s_B\,\int_{B_{1}}\A(x,Du)\cdot D\phi\,(u-\kappa)_+\,\phi^{s_B-1}\,dx=0.
\end{equation*}
First notice that all the integrals are evaluated in $A(\kappa,1)$, and in such a set $Du=D(u-\kappa)_+$. So, from \eqref{weaker:A} and \eqref{proprieta:phi}, we get
\begin{equation}\label{esempio:cac0}
\begin{split}
\int_{B_{1}}\A(x,Du)\cdot D(u-\kappa)_+\,\phi^{s_B}\,dx \geq &  \,\l\,\int_{B_{1}} B\big( |D(u-\kappa)_+|\big)\,\phi^{s_B}\,dx-\L\,B(K_1)\int_{A(\kappa,1)}\phi^{s_B}\,dx
\\
\geq &\, \l\,\int_{B_{1}} B\big( |D(u-\kappa)_+|\big)\,\phi^{s_B}\,dx-\L\,B(K_1)|A(\kappa,\tau)|
\end{split}
\end{equation}
and by \eqref{Young}, \eqref{proprieta:phi} and the monotonicity of $B$, we have
\begin{equation}\label{esempio:cac1}
    \begin{split}
\bigg|\int_{B_{1}}\A(x,Du)\cdot D\phi& \, (u-\kappa)_+ \,\phi^{s_B-1}\,\,dx\bigg| 
\\
\leq &\, \L\,\int_{B_{1}}  b\big(|D(u-\kappa)_+| \big)\,(u-\kappa)_+\,\phi^{s_B-1}|D\phi|\,dx
\\
&+\L\,b(K_1)\,\int_{A(\kappa,\tau)}|D\phi|\,(u-\kappa)_+\,\phi^{s_B-1}\,dx
\\
\stackrel{\eqref{young1}}{\leq}& \,\delta\,\int_{B_1} \wB\big(\,b(|D(u-\kappa)_+|)\,\phi^{s_B-1}\big)\,dx
\\
&+C_\delta\,\int_{B_1}B\big((u-\kappa)_+\,|D\phi| \big)\,dx+ C\,B(K_1)\,|A(\kappa,\tau)|
\\
\stackrel{\eqref{tB:usef}}{\leq} & \,\delta\,\int_{B_1} \wB\big(\,b(|D(u-\kappa)_+|)\big)\,\phi^{s_B}\,dx
\\
&+C_\delta\,\int_{A(\kappa,\tau)}B\bigg(C(n)\,\frac{(u-\kappa)_+}{(\tau-\sigma)}\bigg)\,dx+ C\,B(K_1)\,|A(\kappa,\tau)|
\\
\stackrel{\eqref{B2t},\eqref{wBbt}}{\leq} & \,C'\,\delta\,\int_{B_1} B\big(|D(u-\kappa)_+|\big)\,\phi^{s_B}\,dx
\\
&+\frac{C'_\delta}{(\tau-\sigma)^{s_B}}\,\int_{A(\kappa,\tau)}B((u-\kappa)_+)\,dx+ C\,B(K_1)\,|A(\kappa,\tau)|\,.
    \end{split}
\end{equation}

for all $\delta\in (0,1)$, with $C,C'=C,C'(n,\l,\L,i_a,s_a)$ and $C_\delta,C'_\delta$ also depending on $\delta$. Taking $\delta$ small enough depending on the data, and reabsorbing terms, we arrive at the De Giorgi's type inequality
\begin{equation}\label{cacc:level}
    \begin{split}
        \int_{A(\kappa,\sigma)
        } B\big( |D(u-\kappa)_+|\big)\,dx\leq \frac{C}{(\tau-\sigma)^{s_B}} &\,\int_{A(\kappa,\tau)} B\big((u-\kappa)_+\big)\,dx +C\,B(K_1)\,|A(\kappa,\tau)|,
    \end{split}
\end{equation}
for all $0<\sigma<\tau\leq 1$, with $C=C(n,\l,\L,i_a,s_a)$, where we also used \eqref{proprieta:phi}.

Now let $\eta$ be another cut-off function such that $0\leq \eta\leq 1$,
\begin{equation*}
    \eta\in C^{\infty}\big(B_{\tfrac{\tau+\sigma}{2}}\big),\quad\eta\equiv 1\text{ in $B_\sigma$,}\quad\text{and }\,\,|D\eta|\leq C(n)/(\tau-\sigma)\,.
\end{equation*}
By H\"older and Sobolev inequalities, we get
\begin{equation*}
    \begin{split}
        \int_{A(\kappa,\sigma)} B\big((u-\kappa)_+\big)\,dx\leq &\,\int_{A(\kappa,\sigma)} B\big((u-\kappa)_+\,\eta \big)\,dx
        \\
        \leq  &\,|A(\kappa,\sigma)|^{\tfrac{1}{n}}\,\left(\int_{B_1} B\big((u-\kappa)_+\,\eta \big)^{\tfrac{n}{n-1}}\,dx \right)^{\tfrac{n-1}{n}}
        \\
        \leq &\,C(n)\,|A(\kappa,\sigma)|^{\tfrac{1}{n}}\,\int_{A(\kappa,\tfrac{\tau+\sigma}{2})} \big|DB\big((u-\kappa)_+\,\eta \big)\big|\,dx.
    \end{split}
\end{equation*}
Now observe that, by the monotonicity of $b(t)$, the properties of $\eta$, \eqref{Bcomeb} and \eqref{young1}, we have
\begin{equation*}
\begin{split}
    \big|DB\big((u-\kappa)_+\,\eta \big)\big|= &\,   b\big((u-\kappa)_+\,\eta  \big)\,\big|D(u-\kappa)_+\,\eta+(u-\kappa)_+\,D\eta \big|
    \\
    \leq &\,b\big((u-\kappa)_+\big)\,|D(u-\kappa)_+|+\frac{C(n)}{(\tau-\sigma)}b\big((u-\kappa)_+\big)\,(u-\kappa)_+
    \\
    \leq &\,C\,B\big(|D(u-\kappa)_+|\big)+\frac{C}{(\tau-\sigma)}B((u-\kappa)_+),\qquad\text{in $A(\kappa,1)$.}
    \end{split}
 \end{equation*}
Connecting the two inequalities above, we arrive at
\begin{equation*}
\begin{split}
    \int_{A(\kappa,\sigma)} B\big((u-\kappa)_+\big)\,dx \leq &\,  C\,|A(\kappa,\sigma)|^{\tfrac{1}{n}}\,\int_{A(\kappa,\tfrac{\tau+\sigma}{2})}B\big(|D(u-\kappa)_+|\big)\,dx
    \\
    &+C\,\frac{|A(\kappa,\sigma)|^{\tfrac{1}{n}}}{(\tau-\sigma)^{s_B}}\int_{A(\kappa,\tfrac{\tau+\sigma}{2})}B\big((u-\kappa)_+\big)\,dx,
    \end{split}
\end{equation*}
and coupling this inequality with \eqref{cacc:level} (with $(\tau+\sigma)/2$ in place of $\sigma$), we obtain
\begin{equation}\label{cacc:lev2}
    \int_{A(\kappa,\sigma)}B\big((u-\kappa)_+\big)\,dx\leq C\,\frac{|A(\kappa,\tau)|^{\frac{1}{n}}}{(\tau-\sigma)^{s_B}}\,\int_{A(\kappa,\tau)} B\big((u-\kappa)_+\big)\,dx+C\,
    B(K_1)\,|A(\kappa,\tau)|^{1+\frac{1}{n}},
\end{equation}
for all $0<\sigma<\tau\leq 1$, and all $\kappa>0$, with constant $C=C(n,\l,\L,i_a,s_a)$. Next, for $0<h<\kappa$, we have
\begin{equation*}
    |A(\kappa,\tau)|\,B(\kappa-h)\leq \int_{A(\kappa,\tau)} B(u-h)\,dx\leq \int_{A(h,\tau)} B\big((u-h)_+\big)\,dx.
\end{equation*}
Using this information with \eqref{cacc:lev2}, and majorizing $B(K_1)\leq B(\kappa)$ for $K_1<h<\kappa$, we deduce
\begin{equation}\label{cacc:lev3}
\begin{split}
    \int_{A(\kappa,\sigma)} &B\big((u-\kappa)_+\big)\,dx
    \\
    &\leq \frac{C}{B(\kappa-h)^{\frac{1}{n}}}\left(\int_{A(h,\tau)}B\big((u-h)_+ \big) \right)^{1+\frac{1}{n}}\bigg[\frac{1}{(\tau-\sigma)^{s_B}}+\frac{B(\kappa)}{B(\kappa-h)} \bigg]\,,
    \end{split}
\end{equation}
for all $0<\sigma<\tau\leq 1$, and for all $K_1<h<\kappa$.
Now fix $0<\sigma_0<\tau_0\leq 1$, and a constant $d\geq K_1$ to be determined. For $i\geq 0$, set
\begin{equation*}
    \kappa_0=d,\quad \kappa_{i+1}=\kappa_i+B^{-1}\left(\frac{B(d)}{2^{(i+1)s_B}} \right),\quad \sigma_i=\sigma_0+\frac{\tau_0-\sigma_0}{2^i}\,,
\end{equation*}
and
\begin{equation*}
    \Phi_{i}=\frac{1}{B(d)}\int_{A(\kappa_i,\sigma_i)} B\big( (u-\kappa_i)_+\big)\,dx.
\end{equation*}
Observe that $ \sigma_i-\sigma_{i+1}=\frac{\tau_0-\sigma_0}{2^{i+1}}$, and since by \eqref{B2t} there holds $B(d)/2^{(i+1)s_B}\leq B(d/2^{i+1}) $, we have
\begin{equation*}
\kappa_i=\kappa_0+\sum_{j=1}^{i} B^{-1}\bigg(\frac{B(d)}{2^{js_B}}\bigg)\leq \kappa_0+\sum_{j=1}^\infty\frac{d}{2^{j}}\leq 2d\,.
\end{equation*}
In particular, by \eqref{B2t},  $B(\kappa_i)\leq 2^{s_B}\,B(d)$ for all $i\geq 0$. Therefore, applying \eqref{cacc:lev3} with $h=\kappa_i$, $\kappa=\kappa_{i+1}$, $\sigma=\sigma_{i+1}$ and $\tau=\sigma_{i}$, we arrive at
\begin{equation*}
    \Phi_{i+1}\leq \frac{C}{(\tau_0-\sigma_0)^{s_B}}\,\big(2^{(1+\frac{1}{n})s_B}\big)^i\,\Phi_i^{1+\frac{1}{n}}\quad\text{for all $i\geq 0$.}
\end{equation*}
Therefore, by Lemma \ref{lem:hyp}, if
\begin{equation}\label{to:satisfy}
    \Phi_0=\frac{1}{B(d)}\int_{A(d,\tau_0)}B((u-d)_+)\,dx\leq C\,(\tau_0-\sigma_0)^{ns_B}\,,
\end{equation}
for some constant $C=C(n,\l,\L,i_a,s_a)$, then $\lim_{i\to \infty} \Phi_i=0$, and thus $u\leq 2d$ in $B_{\sigma_0}$. As Equation \eqref{to:satisfy} is certainly fulfilled if $d$ is such that
\[B(d)=B(K_1)+\frac{C}{(\tau_0-\sigma_0)^{ns_B}}\,\int_{B_{\tau_0}}B(|u|)\,dx,\]
we have thus proven
\begin{equation*}
    \sup_{B_{\sigma_0}} B(u_+)\leq C\,B(K_1)+\frac{C}{(\tau_0-\sigma_0)^{ns_B}}\,\int_{B_{\tau_0}}B(|u|)\,dx\,. 
\end{equation*}
Since $-u$ is solution to $-\mathrm{div}\big(\bar{\A}(x,Dv) \big)=f$ in $B_1$, with $\bar{\A}(x,\xi)=\A(x,-\xi)$ still satisfying \eqref{weaker:A}, we deduce that the same estimate holds for $u_-$, hence
\begin{equation}\label{firs:Linf}
    \sup_{B_{\sigma_0}} B(|u|)\leq C\,B(K_1)+\frac{C}{(\tau_0-\sigma_0)^{ns_B}}\,\int_{B_{\tau_0}}B(|u|)\,dx,
\end{equation}
for all $0<\sigma_0<\tau_0\leq 1$, with $C=C(n,\l,\L,i_a,s_a)$.
We now convert this inequality to one involving the $L^1$-norm. 

By  \eqref{Bcomeb}, Young's inequality \eqref{young1} and \eqref{B2t}, we deduce
\begin{equation*}
\begin{split}
    \frac{C}{(\tau_0-\sigma_0)^{ns_B}}\,\int_{B_{\tau_0}}B(|u|)\,dx & \leq \frac{C'}{(\tau_0-\sigma_0)^{ns_B}}\,b\big( \| u\|_{L^\infty(B_{\tau_0})}\big)\,\int_{B_{\tau_0}}|u|\,dx
    \\
    &\leq \frac{1}{2}B(\|u\|_{L^\infty(B_{\tau_0})})+C''\,B\left(\frac{1}{(\tau_0-\sigma_0)^{ns_B}}\int_{B_{\tau_0}}|u|\,dx  \right)
    \\
    &\leq \frac{1}{2}B(\|u\|_{L^\infty(B_{\tau_0})})+\frac{C'''}{(\tau_0-\sigma_0)^{ns_B^2}}\,B\left(\int_{B_{1}}|u|\,dx  \right)\,.
    \end{split}
\end{equation*}
 Setting $Z(t)=B(\|u\|_{L^\infty(B_t)})$, from \eqref{firs:Linf} and the latter inequality we have
\begin{equation*}
    Z(\sigma_0)\leq \frac{1}{2} Z(\tau_0)+\frac{C}{(\tau_0-\sigma_0)^{ns_B^2}}B\left(\int_{B_{1}}|u|\,dx  \right)+C\,B(K_1),
\end{equation*}
for all $0<\sigma_0<\tau_0\leq 1$, so that Lemma \ref{lemma:uss} entails
\begin{equation*}
    B\big(\|u\|_{L^\infty(B_{1/2})} \big)\leq C\,B\left(\int_{B_{1}}|u|\,dx  \right)+C\,B(K_1)\leq 2\,C\,B\left(\int_{B_{1}}|u|\,dx +K_1 \right),
\end{equation*}
with $C=C(n,\l,\L,i_a,s_a)\geq 1$, the last inequality due to the monotonicity of $B$. Taking $B^{-1}$ to both sides of the above inequality, and using \eqref{B2t}, we get the desired estimate \eqref{sol:Linf} in the case $R=1$. The local boundedness of $u$ in $B_R$ is then obtained by a standard covering argument.
\end{proof}

Next, we state and prove local boundedness at the boundary for solutions to Dirichlet or Neumann boundary value problems in the upper half ball.

\begin{theorem}[Local boundedness, Dirichlet problems]\label{thm:bdddir}
    Suppose $\A$ fulfills \eqref{weaker:A}, and let $u\in W^{1,B}(B_R^+)$ be a weak solution to
    \begin{equation}\label{eq:bddir}
        \begin{cases}
            -\mathrm{div}\big( \A(x,Du)\big)=0\quad &\text{in $B^+_R$}
            \\
            u=0\quad & \text{on $B^0_R$.}
        \end{cases}
    \end{equation}
    Then $ u\in L^\infty(B^+_{r})$ for all $0<r<R$, and there exists $C=C(n,\l,\L,i_a,s_a)$ such that \begin{equation}\label{caccioppoli:dir}
    \mint_{B^+_{R/2}}B\big(|Du| \big)\,dx\leq C\,\mint_{B^+_{R}} B\left( \frac{|u|}{R}\right)\,dx+C\,B(K_1),
\end{equation}
and 
\begin{equation}\label{sol:Linfdir}
    \sup_{B^+_{R/2}}|u|\leq C\,\mint_{B^+_{R}}|u|\,dx+C\,K_1\,R.
\end{equation}
\end{theorem}

\begin{proof}
We recall that, being $u$ a weak solution to \eqref{eq:bddir}, we have 
\begin{equation*}
    \int_{B_R^+}\A(x,Du)\cdot D\phi\,dx=0
\end{equation*}
    for all test functions $\phi\in W^{1,B}(B_R^+)$ such that $\phi=0$ on $\partial B_R^+$. In particular, we may take $\phi=u\,\eta^{s_B}$ or $\phi=(u-\kappa)_+\,\eta^{s_B}$ for $\eta\in C^\infty_c(B_R)$, and for all $\kappa>0$.

    Therefore, the proof is completely identical to that of Theorem \ref{thm:bdd}, save that all the integrals have to be evaluated in upper half balls $B^+_1$ in place of $B_1$, and one has to use Sobolev inequalities in half balls \eqref{half:sobol}. We leave the details to the reader.
\end{proof}

\begin{theorem}[Local boundedness, Neumann problems]\label{thm:bddneu}
    Suppose $\A$ fulfills \eqref{weaker:A}, and $h:B_{R}^0\to \R$ satisfies
    \begin{equation}\label{new:condneu}
         |h(x')|\leq H\quad\text{for all $x'\in B^0_{R_0}$.}
    \end{equation}
Let $u\in W^{1,B}(B_{R}^+)$ be a weak solution to
    \begin{equation}\label{eq:bddneu}
        \begin{cases}
            -\mathrm{div}\big( \A(x,Du)\big)=0\quad &\text{in $B^+_{R}$}
            \\
            \A(x,Du)\cdot e_n+h=0\quad & \text{on $B^0_{R}$.}
        \end{cases}
    \end{equation}
Then $ u\in L^\infty(B^+_{r})$ for all $0<r<R$, and there exists $C=C(n,\l,\L,i_a,s_a)>0$ such that
\begin{equation}\label{caccioppoli:neu}
    \mint_{B^+_{R/2}}B\big(|Du| \big)\,dx\leq C\,\mint_{B^+_{R}} B\left( \frac{|u|}{R}\right)\,dx+C\,B\big(K_1+b^{-1}(H)\big),
\end{equation}
and 
\begin{equation}\label{sol:Linfneu}
    \sup_{B^+_{R/2}}|u|\leq C\,\mint_{B^+_{R}}|u|\,dx+C\,\big(K_1+b^{-1}(H)\big)\,R.
\end{equation}
    \end{theorem}

\begin{proof}
We briefly sketch the proof, noting that the only difference from the proof of Theorem~\ref{thm:bdd} concerns the boundary term.  By the scaling procedure of Remark \ref{rem:scaling}, we may assume $R=1$.

 Let $\eta\in C^\infty_c(B_1)$, be a cut-off function, with $0\leq \eta \leq 1$ and $\max_{B_1} |D\eta|\geq 1$. Let
 \begin{equation*}
     \tilde u =\text{either $u$ or $ (u-\kappa)_+$,}
 \end{equation*}
and, in both cases, testing \eqref{eq:bddneu} with $\tilde{u}\,\eta^{s_B}$ we get
 \begin{equation*}
     \int_{B^+_1}\A(x,D\tilde u)\cdot D\tilde u\,\eta^{s_B}\,dx+s_B\,\int_{B_1^+} \A(x,D\tilde u)\cdot D\eta\,\tilde u\,\eta^{s_B-1}\,dx=\int_{B_1^0} h(x')\,\tilde u(x')\,\eta^{s_B}(x')\,dx'.
 \end{equation*}
 The terms on the left-hand side are estimated as in \eqref{esempio:cac0}-\eqref{esempio:cac1}, so we find $C_0=C_0(n,\l,\L,i_a,s_a)>0$ such that 
 \begin{equation}\label{absorb:neu}
 \begin{split}
    \int_{B_1^+} B\big(|D\tilde u| \big)\,\eta^{s_B}\,dx\leq  &\,C_0\,\big(\max |D\eta| \big)^{s_B}\,\int_{B^+_1\cap \mathrm{spt}\,\eta} B(|\tilde u|)\,dx
    \\
    &+C_0\,B(K_1)\,|\mathrm{spt}(\eta)\cap\{\tilde u\neq 0\}|
    \\
    &+\int_{B_1^0} h(x')\,\tilde u(x')\,\eta^{s_B}(x')\,dx' .
\end{split}
 \end{equation}
By \eqref{new:condneu}, the trace inequality \eqref{in:trace0}, Young's inequality \eqref{Young} and \eqref{B2t}, we estimate the boundary integral as follows:
\begin{equation}
\begin{split}
    \bigg|\int_{B_1^0} h(x')\,\tilde u(x') \,\eta^{s_B}(x') &\,dx'\bigg| \leq \int_{B_1^0} H\,|\tilde u(x')|\,\eta^{s_B}(x')\,dx'\leq C(n)\,\int_{B_1^+} H\,\big|D\big[|\tilde u|\,\eta^{s_B} \big]\big|\,dx
    \\
    \leq &\, C(n)\,\Big\{\int_{B_1^+} H\,|D\tilde u|\,\eta^{s_B}+H\,s_B|\tilde u|\,\eta^{s_B-1}|D\eta|\,dx \Big\}
\\
    \leq & \,C_\delta\,\wB(H)\,|\mathrm{spt}\eta\cap \{\tilde u\neq 0\}|+\delta\,\int_{B_1^+} B(|D\tilde u|)\,\eta^{s_B}\,dx 
    \\
    &+C\,\big(\max |D\eta|\big)^{s_B}\,\int_{B_1^+\cap \mathrm{spt}\,\eta} B(|\tilde u|)\,dx,
    \end{split}
\end{equation}
 for all $\delta\in (0,1)$, with $C$ depending on $n,i_a,s_a$, and $C_\delta$ also depending on $\delta$. 
 
 Choosing $\delta=\delta(n,\l,\L,i_a,s_a)\in (0,1)$ sufficiently small, we may re-absorb terms in \eqref{absorb:neu}, and get
 \begin{equation}\label{neu:newstart}
 \begin{split}
     \int_{B_1^+} B\big(|D\tilde u| \big)\,\eta^{s_B}\,dx\leq &\,  \,C\,\big(\max |D\eta| \big)^{s_B}\,\int_{B^+_1\cap \mathrm{spt}\,\eta} B(|\tilde u|)\,dx
 \\
 &+C\,B\big(K_1+b^{-1}(H)\big)\,|\mathrm{spt}(\eta)\cap\{\tilde u\neq 0\}|\,,
 \end{split}
 \end{equation}
 where we used that $\wB(H)\leq C(i_a,s_a) B\big(b^{-1}(H)\big)$ by \eqref{come:BwB}, and the trivial inequality $B(K_1)+B(b^{-1}(H))\leq 2\,B\big(K_1+b^{-1}(H) \big)$.
 
 Starting from \eqref{neu:newstart}, the proof proceeds exactly as in  Theorem \ref{thm:bdd}, replacing $K_1$ with $K_1+b^{-1}(H)$, evaluating the integrals in the half ball $B_1^+$, and using Sobolev inequality \eqref{half:sobol} on half balls. We omit the details.
\end{proof}

\begin{remark*}
\rm{
Our assumptions on the vector field $\A$ in~\eqref{weaker:A}, on the Young function $B$ in~\eqref{iasa}, and on the boundary datum ensuring the boundedness of solutions are by no means optimal. Nevertheless, estimates of the form~\eqref{caccioppoli:dir}-\eqref{sol:Linfdir}, which will be important in later sections, are difficult to find explicitly in the literature. For this reason, we have chosen to state and prove them here. For results establishing boundedness under weaker assumptions, we refer, for instance, to \cite{S64, T67, T79, Ta91, K90, L91, MP96,C90, C97, C00, BCM23, CMM23, G26, CM26}.}
\end{remark*}

\section{Interior gradient regularity: homogeneous problems}\label{sec:int0}

In the following section, we establish gradient H\"older regularity for the local solution $v\in W^{1,B}_{loc}(\Omega)$ of the homogeneous equation
\begin{equation}\label{eq:homint}
    -\mathrm{div}\big(\A(Dv) \big)=0\quad\text{in $\Omega$.}
\end{equation}

Since additive constants on $\A$ do not alter the equation, we may assume that $\A(0)=0$.

The main theorem of this section is the following.

\begin{theorem}\label{thm:inthom}
    Let $v\in W^{1,B}_{loc}(\Omega)$ be a local weak solution to \eqref{eq:homint}, under the assumption that the stress field $\A\in C^0(\rn) \cap C^1(\rn\setminus\{0\})$ satisfies \eqref{ass:Aaut}.
    
    Then there exists $\alpha_\mathrm h=\alpha_{\mathrm{h}}(n,\l,\L,i_a,s_a)\in (0,1)$ such that
    \begin{equation}
        v\in C^{1,\alpha_{\mathrm{h}}}_{loc}(\Omega). 
    \end{equation}
    Moreover, for every ball  $B_{2R}=B_{2R}(x_0)\Subset \Omega$, the  $L^\infty$-$L^1$ estimate 
    \begin{equation}\label{inf:hom}
        \sup_{B_{R/2}}|Dv|\leq c_{\mathrm{h}}\,\mint_{B_R} |Dv|\,dx\,,
    \end{equation}
    holds; we have the excess decay estimate
    \begin{equation}\label{exc:hom}
        \mint_{B_r}|Dv-(Dv)_{B_r}|\,dx\leq c_{\mathrm{h}}\,\left( \frac{r}{R}\right)^{\alpha_\mathrm{h}} \mint_{B_R}|Dv-(Dv)_{B_R}|\,dx,
    \end{equation}
    and the oscillation estimates
\begin{equation}\label{oscdes}
    \operatorname*{osc}_{B_r} Dv\leq C_{\mathrm{h}}\left(\frac{r}{R} \right)^{\alpha_h}\,\mint_{B_R}|Dv-(Dv)_{B_R}|\,dx,\quad \text{for all }0<r\leq R/2,
\end{equation}
   and
\begin{equation}\label{fin:oscdec}
     \operatorname*{osc}_{B_r} Dv\leq C_{\mathrm{h}}\left(\frac{r}{R} \right)^{\alpha_h} \operatorname*{osc}_{B_R} Dv\leq C'_{\mathrm{h}}\,\left(\frac{r}{R} \right)^{\alpha_h}\,\mint_{B_{2R}} |Dv|\,dx\quad \text{for all }0<r\leq R,
\end{equation} 
    where $c_\mathrm{h},C_{\mathrm{h}},C_{\mathrm{h}}'>0$ depend on $n,\l,\L,i_a,s_a$, and $B_r\subset B_R\subset B_{2R}$ are concentric balls.
\end{theorem}

\noindent \textbf{The regularized problem.} Due to the lack of  a priori regularity of $v$, we first establish the estimates in Theorem \ref{thm:inthom} for the function $v_\e\in W^{1,2}_{loc}(\Omega)$ solving

\begin{equation}\label{eq:homeint}
     -\mathrm{div}\big(\Ae(Dv_\e) \big)=0\quad \text{in $\Omega$,}
\end{equation}
where $\Ae$ is the vector field provided by Lemma \ref{lemma:Ae}. \footnote{The choice of approximating stress field $\mathcal A_\varepsilon$ is not unique; see Remark~\ref{rem:approx} below.}

Most importantly, taking into account Remark \ref{remark:importante}, these estimates will hold uniformly in $\varepsilon>0$, so that the conclusion of Theorem~\ref{thm:inthom} will follow
via a limiting argument.
We recall, by \eqref{bBquadratic}, the equivalence of Sobolev spaces $W^{1,B_\e}=W^{1,2}$. Here and in what follows, we will denote by
\begin{equation*}
    \nabla_\xi \Ae(\xi)=\{\nabla_\xi \Ae(\xi)\}_{\substack{i=1,\dots,n\\j=1,\dots n}}=\frac{\partial \Ae^i}{\partial \xi_j}(\xi),
\end{equation*}
and $\nabla_\xi \A(D\ve)=\nabla_\xi \A(\xi)|_{\xi=D\ve}$.
\vspace{0.2cm}

We start showing that $v_\e$ are, in fact, smooth. Indeed, we have the following

\begin{proposition}\label{prop:veregular}
    Let $\e\in (0,1)$ fixed,  let $\Omega\subset \rn$ be open, and suppose that $v_\e\in W^{1,2}_{loc}(\Omega)$ is a local weak solution to \eqref{eq:homeint}. Then
    \begin{equation}\label{ve:regular}
    v_\e\in C^{\infty}(\Omega)\,,
\end{equation}
and for all $k=1,\dots,n$, $D_k\ve$ is a classical solution to
\begin{equation}\label{eq:classic}
    -\mathrm{div}\big(\nabla_\xi \Ae(D\ve)\,D(D_k\ve) \big)=0\quad\text{in $\Omega$.}
\end{equation}
\end{proposition}

\begin{proof}
  As a matter of fact, \eqref{ve:regular} can be viewed (and proven) as the analogue of 
Hilbert XIX problem in the context of elliptic equations. For the sake of completeness, we provide the details of the proof. First, one shows that
    \begin{equation}\label{ve:W22}
        \ve\in W^{2,2}_{loc}(B_{2R})\,.
    \end{equation}
    via the difference quotients method (see also \cite[Theorems 8.1-8.2]{BF02}). Specifically, let
    \begin{equation}\label{diff:quot}
        \Delta_h^i v(x)=\frac{1}{h}[v(x+h\,e_i)-v(x)]\,,
    \end{equation}
    for $i=1,\dots,n$, and $0<|h|<\mathrm{dist}(x,\partial  \Omega)$. We test the weak formulation of \eqref{eq:homeint} with 
    \begin{equation*}
        \varphi=\Delta^i_{-h}\big(\eta^2\Delta_h^i \ve\big)\,,
    \end{equation*}
where $\eta\in C^\infty_c(\Omega)$ is a cut-off function. By the discrete divergence theorem, we get
\begin{equation*}
\begin{split}
    0&=\int_\Omega \Ae(D\ve)\cdot D\left(\Delta^i_{-h}\big(\eta^2\Delta_h^i u\big) \right)\,dx=\int_\Omega \Delta_h^i\Ae(D\ve)\cdot D\big(\eta^2\Delta_h^i \ve\big) \,dx
    \\
    &=\int_\Omega  \Delta_h^i\Ae(D\ve)\cdot D(\Delta_h^i \ve)\,\eta^2\,dx+2\int_\Omega \Delta_h^i\Ae(D\ve)\cdot D\eta\,\eta\,\Delta_h^i \ve\,dx.
    \end{split}
\end{equation*}
    By the fundamental theorem of calculus
    \begin{equation*}
            \Delta_h^i \Ae(D\ve)=\bigg[\int_0^1\nabla_\xi\Ae\Big(tD\ve(x+he_i)+(1-t)D\ve(x) \Big)\,dt\bigg]\,
            D\big(\Delta_h^i (D\ve)\big)\,,
    \end{equation*}
and owing to \eqref{nonzero:Ae}, we have
\begin{equation*}
    \int_\Omega  \Delta_h^i\Ae(D\ve)\cdot D(\Delta_h^i \ve)\,\eta^2\,dx\geq c\,\e\,\int_\Omega |D(\Delta_h^i \ve)|^2\,\eta^2\,dx\,,
\end{equation*}
    and by \eqref{nonzero:Ae} and Young's inequality 
    \begin{equation*}
    \begin{split}
       \bigg|  & 2\int_\Omega \Delta_h^i\Ae(D\ve)\cdot D\eta\,\eta\,\Delta_h^i \ve\,dx \bigg|\leq 2\,C\,\e^{-1}\int_\Omega |D(\Delta_h^i \ve)|\,|D\eta|\,\eta|\,|\Delta_h^i \ve|\,dx
       \\
       &\leq \frac{c\,\e}{2}\int_\Omega |D(\Delta^i_h \ve)|^2\,\eta^2\,dx+C'\,\e^{-3}\int_\Omega  |\Delta_h^i \ve|^2\,|D\eta|^2\,dx.
       \end{split}
    \end{equation*}
    Connecting the four inequalities above, we get
    \begin{equation}\label{df0}
        \int_\Omega |D(\Delta_h^i\ve)|^2\,\eta^2\,dx\leq C\,\e^{-4}\int_\Omega  |\Delta_h^i \ve|^2\,|D\eta|^2\,dx,
    \end{equation}
 for all $i=1,\dots,n$, for all $0<|h|<\mathrm{dist}(\mathrm{supp}\,\eta,\,\partial \Omega)$. Starting from \eqref{df0}, the  $W^{2,2}_{loc}(\Omega)$-regularity of $\ve$ follows in a standard way using the properties of the difference quotients \cite[Lemmas 7.23-7.24]{GT}-- see also \cite[Proof of Theorem 8.8, pag. 185]{GT}.

Thanks to \eqref{ve:W22}, we may differentiate Equation \eqref{eq:homeint} with respect to $k=1,\dots,n$, and find that $D_k\ve\in W^{1,2}_{loc}(\Omega)$ is a weak solution to
\begin{equation*}
    -\mathrm{div}\big(\nabla_\xi\Ae(D\ve)D(D_k\ve) \big)=0\quad \text{in $\Omega$,}
\end{equation*}
for all $k=1,\dots,n$.

 Then by \eqref{nonzero:Ae}, we may appeal to De Giorgi-Nash-Moser theory \cite[Corollary 4.18]{HL11}, \cite[Theorem 8.13]{GM12} so we find $D_k \ve\in C^{0,\alpha_\e}_{loc}$, for some $\alpha_\e\in (0,1)$, and for all $k=1,\dots,n$. Thus $\nabla_\xi \Ae(D\ve)\in C^{0,\alpha_\e}_{loc}(\Omega)$, so Shauder's theory for uniformly elliptic equations in divergence form \cite[Theorem 5.19]{GM12} gives $D_k\ve\in C^{1,\alpha_\e}_{loc}(\Omega)$ for all $k=1,\dots,n$, and then a bootstrap argument using \cite[Theorem 5.20]{GM12} finally yields \eqref{ve:regular}.
\end{proof}

 \begin{remark}\label{rem:approx}\rm{ The regularization procedure \eqref{eq:homeint} is only introduced to justify the forthcoming computations. However, the specific choice of the approximating functions $a_\varepsilon$ and $\mathcal A_\varepsilon$ is not essential. As  will become apparent along the proofs, the arguments leading to \eqref{inf:hom}-\eqref{fin:oscdec} remain valid for any  $a_\varepsilon(\cdot)$ satisfying 
 \[
 -1 < \mathfrak{i} \le i_{a_\varepsilon} \le s_{a_\varepsilon} \le \mathfrak{s} < \infty\quad \text{for all  $\e>0$}
 \]
 for some constants $\mathfrak{i} \le \mathfrak{s}$,  for any stress field $\mathcal A_\varepsilon$ fulfilling \eqref{coer:Ae}, and for any $\ve\in W^{1,B_\e}_{loc}(\Omega)$ solution to \eqref{eq:homeint} that is regular enough to justify the computations, and such that $v_\e\xrightarrow{\e\to0} v$ in a suitable sense.
 For instance, in the case of  $p$-Laplace problems \eqref{eq:plapl}, a standard choice is $a_\varepsilon(t) = (\varepsilon^2 + t^2)^{\frac{p-2}{2}}$ and $\mathcal A_\varepsilon(\xi) = (\varepsilon^2 + |\xi|^2)^{\frac{p-2}{2}} \xi$.}
\end{remark}

We now turn to the proof of \eqref{inf:hom}, which is based on the Bernstein method. 
That is, we show that the function $B_\varepsilon(|D v_\varepsilon|)$ 
is a subsolution of a uniformly elliptic linear equation\footnote{Differently from what stated in the Introduction-see \eqref{def:unifell}-in the case of linear equations in divergence $\mathrm{div}(M(x)Du)$ or nondivergence form $\mathrm{tr}\big(M(x)D^2u\big)$, uniform ellipticity will mean that the coefficient matrix $M(x)$ satisfies \eqref{M:elliptic}. This condition is also often referred to as \emph{strict ellipticity}.
}, from which \eqref{inf:hom} 
follows via the weak Harnack inequality.

\begin{proposition}\label{prop:bernstein}
    Suppose $\ve\in W^{1,B_\e}_{loc}(\Omega)$ is a local weak solution to \eqref{eq:homeint}, and set
\begin{equation}\label{Ae:uf}
    \mathbb{A}_\e(x)=\frac{\nabla_\xi \Ae(D\ve)}{a_\e(|D\ve|)}\,.
\end{equation}
Then the function $V_\e=B_\varepsilon(|D v_\varepsilon|)$ is a (local) weak subsolution of
\begin{equation}\label{subsol}
    -\mathrm{div}\big(\mathbb{A}_\e(x)\,DV_\e \big)\leq 0\quad\text{in $\Omega$.}
\end{equation}
Moreover, the $L^\infty$-$L^1$ estimate 
\begin{equation}\label{Degiorgi:ve}
    \sup_{B_{R/2}}|D\ve|\leq C_{\mathrm{b}}\,\mint_{B_R}|D\ve|\,dx
\end{equation}
holds for every ball $B_{R}\Subset \Omega$, $R\leq 1$, with $C_{\mathrm{b}}>0$ depending on $n,\l,\L,i_a,s_a$.
\end{proposition}

\begin{proof}
    First observe that by \eqref{ve:regular}, \eqref{ae:nonzero} and \eqref{coer:Ae}, the matrix $\mathbb{A}_\e(x)=(\mathbb{A}_\e(x))_{ij}$ is well defined for every $x\in \Omega$, and it satisfies
    \begin{equation}\label{matrix:Aecgr}
        \mathbb{A}_\e(x)\,\eta\cdot \eta\geq c\,|\eta|^2\quad\text{and}\quad \sum_{i,j=1}^n|(\mathbb{A}_\e(x))_{ij}|\leq C\quad \text{for all $x\in \Omega$,}
    \end{equation}
 and  for all $\eta\in \rn$, with $c,C>0$ depending on $n,\l,\L,i_a,s_a$. Now observe that, by the chain rule and \eqref{def:bBe}, we have
    \begin{equation}\label{der:Beve}
        DV_\e=b_\e\big(|D\ve| \big)\,\frac{D^2\ve\,D\ve}{|D\ve|}=a_\e\big(|D\ve| \big)\,D\left(\frac{|D\ve|^2}{2}\right)\,.
    \end{equation}
Now let $\varphi\in C^\infty_c(\Omega)$ be a nonnegative function; we test \eqref{eq:classic} with $(D_k\ve)\,\varphi$ and summing over $k=1,\dots,n$ we find
\begin{equation}\label{ozza}
\begin{split}
    0&=\sum_{k=1}^n\int_\Omega \nabla_\xi\Ae(D\ve)\,D(D_k\ve)\,D_k\ve\,\cdot D\varphi\,dx+\sum_{k=1}^n\int_\Omega\nabla_\xi\Ae(D\ve)\,D(D_k\ve)\cdot D(D_k\ve)\,\varphi\,dx
    \\
    &=\int_\Omega\mathbb{A}_\e(x)\,a_\e\big(|D\ve|)\,D\left(\frac{|D\ve|^2}{2}\right)\,\cdot D\varphi\,dx+\sum_{k=1}^n\int_\Omega\nabla_\xi\Ae(D\ve)\,D(D_k\ve)\cdot D(D_k\ve)\,\varphi\,dx
    \\
    &\geq \int_\Omega \mathbb{A}_\e(x)\,DV_\e\cdot D\varphi\,dx\,,
    \end{split}
\end{equation}
where in the last inequality we used \eqref{der:Beve} and the fact that $\nabla_\xi\A_\e(D\ve)\,D(D_k\ve)\cdot D(D_k\ve)\geq 0$ thanks to \eqref{coer:Ae}. Inequality \eqref{subsol} is thus proven. 
\vspace{0.1cm}

Let us now show \eqref{Degiorgi:ve}. We prove it in the case $R=1$, as the general argument will then follow via the scaling argument of Remark \ref{rem:scaling}.

By \eqref{subsol} and \eqref{matrix:Aecgr}, we may use the weak Harnack inequality \cite[Theorem 4.1]{HL11} and deduce that for all $0<r_1<r_2<1$, we have
\begin{equation}\label{Linf11}
    \sup_{B_{r_1}} B_\e(|D\ve|)\leq \frac{C}{(r_2-r_1)^n}\int_{B_{r_2}}B_\e\big(|D\ve| \big)\,dx.
\end{equation}
 Then by  \eqref{Bcomeb}, \eqref{young1} and \eqref{wBbt}, we deduce
\begin{equation}\label{Linf111}
    \begin{split}
        \int_{B_{r_2}}B_\e\big(|D\ve| \big)\,dx&\leq \frac{C}{(r_2-r_1)^n}\,b_\e\big( \|D\ve\|_{L^\infty(B_{r_2})}\big)\int_{B_{r_2}}|D\ve|dx
        \\
        &\leq \frac{1}{2}\, B_\e\big( \|D\ve\|_{L^\infty(B_{r_2})}\big)+B_\e\bigg(\frac{C}{(r_2-r_1)^n}\int_{B_{r_2}}|D\ve|\,dx \bigg)
        \\
        &\leq \frac{1}{2}\,B_\e\big( \|D\ve\|_{L^\infty(B_{r_2})}\big)+\frac{C^{\max\{s_B,2\}}}{(r_2-r_1)^{n\max\{s_B,2\}}} B_\e\left(\int_{B_1}|D\ve|\,dx \right),
    \end{split}
\end{equation}
 for all $0<r_1<r_2\leq 1$, with $C>0$ depending on $n,\l,\L,i_a,s_a$. Observe that, in the last inequality, we exploited \eqref{B2t} (with $B_\e$ in place of $B$) coupled  with  Remark \ref{remark:importante}. 
 Thereby using  Lemma \ref{lemma:uss} with $Z(t)=B_\e(\|D\ve\|_{L^\infty(B_t)})$, we get
\begin{equation*}
    B_\e(\|D\ve\|_{L^\infty(B_{1/2})})\leq C\,B_\e\left(\int_{B_1}|D\ve|\,dx \right)\,,
\end{equation*}
for some constant $C=C(n,\l,\L,i_a,s_a)\geq 1$. By applying $B_\e^{-1}$ to both sides of the above inequality, we get
\begin{equation}\label{inf:B1}
    \|D\ve\|_{L^\infty(B_{1/2})}\leq B_\e^{-1}\left(C\,B_\e\left(\int_{B_1}|D\ve|\,dx \right)\right)\leq C'\,\int_{B_1}|D\ve|\,dx\,,
\end{equation}
with $C,C'=C,C'(n,\l,\L,i_a,s_a)\geq 1$, where in the last inequality we used \eqref{B2t} (with $B_\e$ in place of $B$) while taking into account Remark \ref{remark:importante}.
Estimate \eqref{Degiorgi:ve} in a generic ball $B_R$, $R\leq 1$, then  follows from \eqref{inf:B1} via the scaling argument of Remark \ref{rem:scaling}.
\end{proof}

\noindent We now move onto the proof of the H\"older continuity of $D v$. This is based on the so-called \textit{fundamental alternative}. The key idea, originating in the work
of De Giorgi \cite{DG56}, can be summarized as follows. Fix a ball
$\mathcal B\Subset\Omega$. If the set where \eqref{eq:classic} is degenerate (namely where
$|D v|$ is small) occupies only a small portion of $\mathcal B$, then this
degeneracy can be controlled, i.e., $|Dv|$ is bounded away from zero in a smaller
concentric ball, and thus the equation \eqref{eq:classic} is nondegenerate. If, on the other hand, $|Dv|$ is small in a large portion of $\mathcal B$, then it can be  compared with its radius. 
\vspace{0.2cm}

Let us now state and prove the result. Given $B_r(x_0)\equiv B_r\Subset \Omega$, we set
\begin{equation}\label{def:Mr}
    M(r)=\max_{k=1,\dots,n}\sup_{B_r}|D_k\ve|\,.
\end{equation}

\begin{theorem}[The fundamental alternative]\label{thm:alternative}
    Let $\ve\in W^{1,B_\e}_{loc}(\Omega)$ be a local weak solution to \eqref{eq:homeint}, and let $B_{2R}\Subset \Omega$. 
    There exist \textit{universal numbers} 
    \begin{equation}\label{universal}
        \mu_0,\,\eta_0\in (0,1) \quad\text{depending only on $n,\l,\L,i_a,s_a$}
    \end{equation}
    such that the following alternatives hold.
    \\
    (i) If for some $k=1,\dots,n$, we have either
    \begin{equation}\label{alt:0}
    \begin{split}
       \big |\{D_k\ve<M(2R)/2\}&\cap B_{2R}\big|\leq \mu_0\,|B_{2R}|
       \\
       &\text{or}
       \\
       \big |\{D_k\ve>-M(2R)/2&\}\cap B_{2R}\big|\leq \mu_0\,|B_{2R}|
       \end{split}
    \end{equation}  
        then
\begin{equation}\label{thalt1:modDu}
    |D\ve|\geq M(2R)/4\quad\text{in $B_R$.}
\end{equation}
    \\
    (ii) In the complementary case, that is, if  \begin{equation}\label{alt:3}
        \begin{cases}
            \big |\{D_k\ve<M(2R)/2\}\cap B_{2R}\big|>\mu_0\,|B_{2R}|\quad\text{and}
            \\
            \\
            \big |\{D_k\ve>-M(2R)/2\}\cap B_{2R}\big|> \mu_0\,|B_{2R}|,\quad\text{for all $k=1,\dots,n$,}
        \end{cases}
    \end{equation}
    then
    \begin{equation}\label{thesis:alt3}
        M(R)\leq \eta_0\,M(2R)\,.
    \end{equation}
\end{theorem}

Clearly, it suffices to prove Theorem \ref{thm:alternative} only in the case $M(2R)>0$, for otherwise $D\ve \equiv 0$ on $B_{2R}$. To prove it, we start with an auxiliary lemma, which is a simple consequence of the equation \eqref{eq:classic} and standard computations using the properties of $\nabla_\xi\Ae$.

\begin{lemma}\label{lemma:parti}
    Let $\ve\in W^{1,B_\e}_{loc}(\Omega)$ be a local weak solution to \eqref{eq:homeint}. Let $g:\R\to \R$ be a Lipschitz function  such that $g'(t)\geq 0$  for a.e. $ t\in \R$. Then 
    \begin{equation}\label{temp:parti}
    \begin{split}
        \int_{\Omega} a_\e\big( |D\ve|\big)\,|D(D_k\ve)|^2\,g'(D_k\ve)\,\phi^2\,dx\leq C\, &\,\int_{\Omega} a_\e\big(|D\ve|\big)\,|D\ve|^2\,g'(D_k\ve)\,|D\phi|^2\,dx
        \\
        &+C\,\int_\Omega b_\e\big(|D\ve|\big)\,\big|g\big(D_k\ve\big)\big|\,|D^2\phi^2|\,dx,
        \end{split}
    \end{equation}     
    for all $k=1,\dots,n$, and for every $\phi\in C^2_c(\Omega)$, where $C$ is a positive constant  determined only by $n,\l,\L,i_a,s_a$.

    If in addition $|\{g\neq 0\}\cap\{g'= 0\}|=0 $, there holds
\begin{equation}\label{ooooo}
    \int_{\Omega} a_\e\big( |D\ve|\big)\,|D(D_k\ve)|^2\,g'(D_k\ve)\,\phi^2\,dx\leq C\,\int_{\{D(D_k\ve)\neq 0\}\cap\{g(D_k\ve)\neq 0\}} a_\e(|D\ve|)\,\frac{g^2(D_k\ve)}{g'(D_k\ve)}\,|D\phi|^2\,dx.
\end{equation}
\end{lemma}

\begin{proof}
    We test Equation \eqref{eq:classic} with $g(D_k\ve)\,\phi^2$, thus getting
    \begin{equation*}
        \begin{split}
            \int_\Omega \nabla_\xi\A_\e(D\ve)&\,D(D_k\ve)\cdot D(D_k\ve)\,g'(D_k\ve)\,\phi^2dx
            \\
            &=-\int_\Omega \nabla_\xi\A_\e(D\ve)\,D(D_k\ve)\cdot D\phi^2\, g(D_k\ve)\,dx\,.
        \end{split}
    \end{equation*}
By \eqref{coer:Ae}, Remark \ref{remark:importante} and the hypothesis on the sign of $g'$, we have
\begin{equation*}
    \int_\Omega \nabla_\xi\A_\e(D\ve)\,D(D_k\ve)\cdot D(D_k\ve)\, g'(D_k\ve)\,\phi^2dx\geq c\,\int_\Omega a_\e(|D\ve|)\,|D(D_k\ve)|^2\,g'(D_k\ve)\,\phi^2\,dx,
\end{equation*}
with $c=c(n,\l,\L,i_a,s_a)$. The two identities above give
\begin{equation}\label{tempo:baseineq}
    c\,\int_\Omega a_\e(|D\ve|)\,|D(D_k\ve)|^2\,g'(D_k\ve)\,\phi^2\,dx\leq -\int_\Omega \nabla_\xi\A_\e(D\ve)\,D(D_k\ve)\cdot D\phi^2\, g(D_k\ve)\,dx.
\end{equation}
On the other hand, via integration by parts, we find
\begin{equation*}
    \begin{split}
        -\int_\Omega &\,\nabla_\xi\A_\e(D\ve)\,D(D_k\ve)\cdot D\phi^2\, g(D_k\ve)\,dx=-\int_\Omega D_k\big(\Ae(D\ve)\big)\,\cdot D\phi^2\,g(D_k\ve)\,dx
        \\
        &=\int_\Omega \Ae(D\ve)\cdot D\phi^2\,g'(D_k\ve)\,D_{kk}\ve\,dx+\int_\Omega \Ae(D\ve)\cdot D(D_k\phi^2)\,g(D_k\ve)\,dx\,.
    \end{split}
\end{equation*}
Moreover, by \eqref{co:gr}, \eqref{def:bBe} and Young's inequality, we find
\begin{equation*}
    \begin{split}
       \bigg| \int_\Omega &\Ae(D\ve)\cdot D\phi^2\,g'(D_k\ve)\,D_{kk}\ve\,dx\bigg|\leq 2\,\int_\Omega b_\e(|D\ve|)\,|D\phi|\,\phi\,g'(D_k\ve)\,|D(D_k\ve)|\,dx
       \\
       &\leq \delta\,\int_\Omega a_\e(|D\ve|)\,|D(D_k\ve)|^2\,g'(D\ve)\,\phi^2dx+C_\delta\,\int_{\Omega}a_\e(|D\ve|)\,|D\ve|^2\,g'(D_k\ve)\,|D\phi|^2\,dx,
    \end{split}
\end{equation*}
for all $\delta \in (0,1)$, and
\begin{equation*}
    \bigg|\int_\Omega \Ae(D\ve)\cdot D(D_k\phi^2)\,g(D_k\ve)\,dx\bigg|\leq C\,\int_\Omega b_\e(|D\ve|)\,|g(D_k\ve)|\,|D^2\phi^2|\,dx\,.
\end{equation*}
Coupling the three identities above with \eqref{tempo:baseineq}, and by choosing $\delta$, determined by $n,\l,\L,i_a,s_a$, small enough to reabsorb terms, we get \eqref{temp:parti}.

Next, by \eqref{coer:Ae} and Young's inequality\footnote{The computations in \eqref{tojustify} are justified since $\big|\{D(D_k\ve)\neq 0\}\cap \{g(D_k\ve)\neq 0\}\cap \{g'(D_k\ve)=0\}\big|=0$. This is an immediate consequence of the coarea formula and the assumption $|\{g'=0\}\cap\{g\neq 0\}|=0$.
}
\begin{equation}\label{tojustify}
\begin{split}
    \bigg|  \int_\Omega \nabla_\xi\A_\e(D\ve) &\,D(D_k\ve)\cdot D\phi^2\, g(D_k\ve)\,dx\bigg|
    \\
    \leq &\,C\,\int_{\Omega }a_\e(|D\ve|)\,|D(D_k\ve)|\,|g(D_k\ve)|\,|\phi|\,|D\phi|\,dx
    \\
    \leq &\, \delta\int_\Omega a_\e(|D\ve|)\,|D(D_k\ve)|^2\,g'(D_k\ve)\,\phi^2\,dx
    \\
    &+C_\delta\int_{\{D(D_k\ve)\neq 0\}\cap\{g(D_k\ve)\neq 0\}}a_\e(|D_k\ve|)\,\frac{g^2(D_k\ve)}{g'(D_k\ve)}\,|D\phi|^2\,dx\,.
    \end{split}
\end{equation}
Combining the above estimate with \eqref{tempo:baseineq}, and choosing $\delta=\delta(n,\l,\L, i_a, s_a)$ small enough, we may reabsorb terms and finally obtain \eqref{ooooo}.
\end{proof}




In the next, important lemma we establish an estimate for the integral of $|D(D_k\ve)|^2$ restricted to the level sets of $D_k\ve$ lying above $\gamma\, M(2R)$, $\gamma\in (0,1)$; specifically, we show that such integral is bounded in terms of the measure of the corresponding superlevel sets. This is very close in spirit to the proof of H\"older continuity for functions in the De Giorgi's classes-- see~\cite[Chapter 2, Section 6]{LU68} or~\cite[Chapter 7]{G03}. 
We remark that these inequalities are of quadratic type, reflecting the fact that we are working with the linearized equation \eqref{eq:classic}.

\begin{lemma}\label{lemma:levquadratic}
    Let $\ve$ be a solution to \eqref{eq:homeint}, let  $B_{2R}\Subset \Omega$, and let $\gamma\in (0,1)$ be fixed. Then for every $0<r_1<r_2\leq 2R$, and for every $\kappa<\ell $ such that
    \begin{equation}\label{cond:ellkap}
        \gamma\,M(2R)\leq \kappa<\ell\leq M(2R)\,,
    \end{equation}
  with  $M(2R)$ as in \eqref{def:Mr}, the following De Giorgi's type inequalities are valid for all $k=1,\dots,n$: 
\begin{equation}\label{levels:eq}
        \int\limits_{\{\kappa\leq D_k\ve<\ell \}\cap B_{r_1}}|D(D_k\ve)|^2\,dx\leq C_\gamma\,\frac{\big(M(2R)\big)^2}{(r_2-r_1)^2}\big|\{D_k\ve<\ell\}\cap B_{r_2}\big|\,,
\end{equation}
and
\begin{equation}\label{oiii}
    \int\limits_{\{D_k\ve>\kappa\}\cap B_{r_1}  } |D(D_k\ve)|^2\,dx\leq C_\gamma\,\frac{(M(2R)-\kappa)^2}{(r_2-r_1)^2}\,|\{D_k\ve>\kappa\}\cap B_{r_2}|
\end{equation}
   where $C_\gamma>0$ depends on $n,\l,\L,i_a,s_a$ and $\gamma$.
\end{lemma}

\begin{proof}
First observe that $|D\ve|\leq \sqrt{n}\,M(2R)$ in $B_{2R}$; so, by the monotonicity of $b_\e(t)=a_\e(t)\,t$, we have
\begin{equation}\label{brecall}
   a_\e(|D\ve|)\,|D\ve|\leq b_\e(\sqrt{n}\,M(2R))\leq C(n,i_a,s_a)\,a_\e\big(M(2R)\big)\,\mathrm{M(2R)},
\end{equation}
where in the second inequality we used \eqref{b2t} combined with Remark \ref{remark:importante}.

    We then use  Equation \eqref{temp:parti} with function $g(t)=-(t-\ell)_{-}=(t-\ell)\chi_{\{t<\ell\}}$, and a cut-off function $\phi\in C^\infty_c(B_{r_2})$ such that $0\leq \phi\leq 1$, and
    \begin{equation}\label{cutoff}
       \phi\equiv 1\quad \text{in $B_{r_1}$,}\quad  |D\phi|\leq \frac{c(n)}{(r_2-r_1)},\quad\text{and  }\,|D^2\phi|\leq \frac{c(n)}{(r_2-r_1)^2}.
    \end{equation}    
So, we obtain
\begin{equation}\label{temporaneo:lev}
        \begin{split}
            \int\limits_{\{D_k\ve<\ell\}\cap B_{r_1}} &\,a_\e(|D\ve|)\,|D(D_k\ve)|^2\,dx
            \\
            &\leq \frac{C}{(r_2-r_1)^2}\int\limits_{\{D_k\ve<\ell\}\cap B_{r_2}}a_\e(|D\ve|)\,|D\ve|\Big\{ |D\ve|+(\ell-D\ve)\Big\}\,dx
            \\
            &\leq 3\,C\,a_\e\big(M(2R)\big)\,\frac{(M(2R))^2}{(r_2-r_1)^2}\,\big|\{D_k\ve<\ell\}\cap B_{r_2}\big|\,,
        \end{split}
    \end{equation}
    the last inequality due to \eqref{brecall}.
    
    Now observe that when $D_k\ve\geq \gamma\,M(2R)$, then $\gamma\,M(2R)\leq |D\ve|\leq \sqrt{n}\,M(2R)$ in $B_{2R}$. Therefore, by \eqref{ultra:utile}, Remark \ref{remark:importante} and \eqref{cond:ellkap}, we have
    \begin{equation}\label{control:ae}
     c_\gamma\leq \frac{a_\e(|D\ve|)}{a_\e(M(2R))} \leq C_\gamma\quad \text{in $\{D_k\ve\geq\kappa\}\cap B_{2R}$},
    \end{equation}
    with $c_\gamma,C_\gamma=c_\gamma,C_\gamma(n,\l,\L,i_a,s_a,\gamma)>0$.
  Coupling this piece of information with  \eqref{temporaneo:lev} yields
  \begin{equation*}
  \begin{split}
      \int_{\{\kappa\leq D_k\ve<\ell \}\cap B_{r_1} }|D(D_k\ve)|^2\,dx &\leq c_\gamma^{-1}\int_{\{\kappa\leq D_k\ve<\ell \}\cap B_{r_1} } \frac{a_\e(|D\ve|)}{a_\e(M(2R))}\,|D(D_k\ve)|^2\,dx
      \\
      &\leq c_\gamma^{-1}C\,\frac{(M(2R))^2}{(r_2-r_1)^2}\,\big|\{D_k\ve<\ell\}\cap B_{r_2}\big|\,,
      \end{split}
  \end{equation*}
and \eqref{levels:eq} is proven.

Finally, to obtain \eqref{oiii}, we use \eqref{ooooo} with $g(t)=(t-\kappa)_+$, and recalling \eqref{control:ae} and \eqref{cutoff}, we get
\begin{equation*}
\begin{split}
    a_\e\big( M(2R)\big)\, &\int_{\{D_k\ve>\kappa\}\cap B_{r_1}}\,|D(D_k\ve)|^2\,dx  
    \\
    &\leq c_\gamma^{-1}
        \int_{\{D_k\ve>\kappa\}\cap B_{r_1}} a_\e\big( |D\ve|\big)\,|D(D_k\ve)|^2\,\phi^2\,dx
        \\
        &\leq \frac{C_\gamma}{(r_2-r_1)^2}\,\int_{\{D_k\ve>\kappa\}\cap B_{r_2}}a_\e(|D\ve|)\,\big(D_k\ve-\kappa \big)^2\,dx
        \\
        &\leq \frac{C_\gamma'\,a_\e\big(M(2R)\big)}{(r_2-r_1)^2}\,\big(M(2R)-\kappa\big)^2\,|\{D_k\ve>\kappa\}\cap B_{r_2}|\,.
        \end{split}
\end{equation*}
 with $C_\gamma,C_\gamma'=C_\gamma,C_\gamma'(n,\l,\L,i_a,s_a,\gamma)>0$, where in the last inequality we used \eqref{control:ae} and we trivially majorized $(D_k\ve-\kappa)\leq (M(2R)-\kappa) $ in $\{D_k\ve>\kappa\}\cap B_{2R}$. Dividing both sides of this inequality by $a_\e\big(M(2R)\big)$ yields \eqref{oiii}.
\end{proof}


We are now in the position to prove the first part of Theorem \ref{thm:alternative}. This is essentially the content of the next lemma. From now on, for notational simplicity, we set
\begin{equation*}
     \mathrm M\equiv M(2R)=\max_{k=1,\dots,n}\sup_{B_{2R}}|D_k\ve|\,.
\end{equation*}

\begin{lemma}[The first alternative]\label{lem:1alt}
    Let $\ve$ be a solution to \eqref{eq:homeint}, and let $k\in \{1,\dots,n\}$. There exists $\mu_0=\mu_0(n,\l,\L,i_a,s_a)\in (0,2^{-n-1})$ such that if
    \begin{equation}\label{alt:1}
         \big |\{D_k\ve<\mathrm M/2\}\cap B_{2R}\big|\leq \mu_0\,|B_{2R}|,
    \end{equation}
then
\begin{equation}\label{thesis:alt1}
    D_k\ve\geq \mathrm M/4\quad \text{in $B_R$.}
\end{equation}
    Analogously, if
    \begin{equation}\label{alt:2}
        \big |\{D_k\ve>-\mathrm M/2\}\cap B_{2R}\big|\leq \mu_0\,|B_{2R}|,\quad\text{ then $D_k\ve<-\mathrm M/4$  in $B_R$.}
    \end{equation}
\end{lemma}

\begin{proof}
   Suppose that \eqref{alt:1} is in force, with $\mu_0$ to be determined later on. Define the sequences
   \begin{equation}\label{start:iter}
       R_m=R+\frac{R}{2^{m}},\quad \text{and}\quad\kappa_m=\frac{\mathrm{M}}{4}+\frac{\mathrm{M}}{2^{m+2}}\,,\quad m=0,1,2,\dots,
   \end{equation}
   and we also set
   \begin{equation*}
       A^-(\kappa,\vrho)\coloneqq\{D_k\ve<\kappa\}\cap B_{\vrho}.
   \end{equation*}
We apply Lemma \ref{lemma:levels} to the function $u=-D_k\ve$, with levels $\ell=-\kappa_{m+1}$, $\kappa=-\kappa_{m}$, and get
\begin{equation}\label{temp:newlevels}
\begin{split}
    \frac{\mathrm{M}}{2^{m+3}} |A^-(\kappa_{m+1}&,R_{m+1})|^{\frac{n-1}{n}}
    \\
    &\leq c(n)\,\frac{|B_{R_{m+1}}|}{|B_{R_{m+1}}\setminus A^-(\kappa_m,R_{m+1})|}\int_{\{\kappa_{m+1}\leq D_k\ve<\kappa_m  \}\cap B_{R_{m+1}}}|D(D_k\ve)|\,dx.
    \end{split}
\end{equation}
Moreover, by \eqref{alt:1} and since $R\leq R_{m+1}\leq 2R$ and $\kappa_m\leq \mathrm{M}/2$, we have 
\begin{equation}\label{azz}
\begin{split}
|B_{R_{m+1}}\setminus A^-(\kappa_m,R_{m+1})|&=|B_{R_{m+1}}|-|A^-(\kappa_m,R_{m+1})|
\\
&\geq |B_{R_{m+1}}|-\left|A^-\left(\frac{\mathrm{M}}{2},2R\right)\right|
\\
&\geq (1-\mu_0\,2^n)\,|B_{R_{m+1}}|\geq \frac{1}{2}|B_{R_{m+1}}|\,, 
\end{split}
\end{equation}
 provided we take $0<\mu_0\leq 2^{-n-1}$. 
 
 We combine this piece of information with \eqref{temp:newlevels}, H\"older's inequality and \eqref{levels:eq} with $\gamma=1/8$, $r_2=R_m$ and $r_1=R_{m+1}$, thus obtaining
\begin{equation*}
    \begin{split}
        \frac{\mathrm{M}}{2^{m+3}}|A^-(\kappa_{m+1}&,R_{m+1})|^{\frac{n-1}{n}}\leq C(n)\,\int_{\{\kappa_{m+1}\leq D_k\ve<\kappa_m  \}\cap B_{R_{m+1}}}|D(D_k\ve)|\,dx
        \\
        &\leq C(n)\,|A^-(\kappa_m,R_m)|^{\frac{1}{2}}\,\left(\int_{\{\kappa_{m+1}\leq D_k\ve<\kappa_m  \}\cap B_{R_{m+1}}}|D(D_k\ve)|^2\,dx \right)^{1/2}
        \\
        &\leq C'\,2^{m+2}\,\frac{\mathrm{M}}{R}\,|A^-(\kappa_m,R_m)|
    \end{split}
\end{equation*}
with $C'=C'(n,\l,\L,i_a,s_a)$.
We divide both sides by $\mathrm{M}$, we use that $R\leq R_m\leq 2R$, and we write the result in dimensionless form, i.e., by setting
 \begin{equation*}
     Z_m=\frac{|A^-(\kappa_m,R_m)|}{|B_{R_m}|}\,,
 \end{equation*}
from the above inequality we obtain 
\begin{equation}\label{iterazione:1}
    Z_{m+1}\leq C\,(4^{\frac{n}{n-1}})^{m+2}\,Z_m^{\frac{n}{n-1}},\quad\text{and}\quad Z_0\leq \mu_0\,,
\end{equation}
  with $C=C(n,\l,\L,i_a,s_a)>0$, where the condition on $Z_0$ is a consequence of \eqref{alt:1}. 
  
  It then follow by Lemma \ref{lem:hyp} that 
  \begin{equation*}
      \lim_{m\to \infty} Z_m=0, \quad\text{for $\mu_0=\mu_0(n,\l,\L,i_a,s_a)\in (0,2^{-n-1})$ small enough.}
  \end{equation*}
As $R_m\to R$ and $\kappa_m\geq \mathrm{M}/4$, we have 
\begin{equation}\label{fine:iter}
    0=\lim_{m\to \infty} Z_m=\lim_{m\to\infty} A^-(\kappa_m,R_m)/|B_{R_m}|\geq|\{D_k\ve<\mathrm{M}/4\}\cap B_{R}|/|B_R|,
\end{equation}
which implies \eqref{thesis:alt1}. 
\vspace{0.2cm}

Finally, in order to show \eqref{alt:2}, it suffices to reproduce the above proof with the function $\bar{v}_\e=-\ve$, since it solves the equation $-\mathrm{div}\big(\bar{\A}_{\e}(D\bar{v}_\e)\big)=0$ in $\Omega$, where $\bar{\A}_{\e}(\xi)=-\Ae(-\xi)$ satisfies the same  properties of $\Ae$ in Lemma \ref{lemma:Ae}.
\end{proof}

\begin{proof}[Second proof of Lemma \ref{lem:1alt}]
Here we provide a second proof of the first alternative via a Moser type iteration. By the scaling argument of Remark \ref{rem:scaling} (see \eqref{uR:scale}), we may assume that $R=1$. In particular, Equation \eqref{alt:1} takes the form
\begin{equation}\label{alt:1scale}
    |\{D_k\ve<\mathrm{M}/2\}\cap B_2|\leq c(n)\,\mu_0\,.
\end{equation}
Let $0<r_1<r_2\leq 1$, and let $\phi$ be as in \eqref{cutoff}; for $q\geq 0$ large enough, we consider 
\begin{equation*}
 g(t)=-\min\big\{\mathrm{M}/4,\max\{\mathrm{M}/2-t,0\}\big\}^{2q+1}
\end{equation*}
and we also set 
$$w=\min\big\{\mathrm{M}/4,\max\{\mathrm{M}/2-D_k\ve,\,0\}\big\},\quad\text{so that}\quad g(D_k\ve)=-w^{2q+1}.$$ 
 We remark that in the set $\{g'(D_k\ve)\neq 0\}$ we have $\mathrm{M}/4\leq D_k\ve$, and thus $\mathrm{M}/4\leq |D\ve|\leq \sqrt{n}\,\mathrm{M}$; so  by \eqref{ultra:utile} and Remark \ref{remark:importante}, there holds $c\,a_\e(\mathrm{M}) \leq a_\e(|D\ve|)\leq C\,a_\e(\mathrm{M})$, with $c,C=c,C(n,\l,\L,i_a,s_a)>0$.
 
Taking advantage of this information and \eqref{brecall}, we make use of Equation \eqref{temp:parti} with $g(t)$ and $\phi$ as above, and dividing both sides of the resulting equation by $a_\e(\mathrm{M})$, we obtain
\begin{equation*}
    \begin{split}
        \int_{\{g'(D_k\ve)\neq 0 \} } |D(w^{q+1})|^2\,\phi^2\,dx\leq  &\,C\,q^2\,\int_{\{g'(D_k\ve)\neq 0 \}} |D\ve|^2\,\,w^{2q}\,|D\phi|^2\,dx
        \\
        &+C\,q\,\mathrm{M}\int_\Omega \,w^{2q+1}\,|D^2\phi^2|\,dx,
    \end{split} 
\end{equation*}
where in the left hand side we used that $|D(D_k\ve)|=|Dw|$ on $\{g'(D_k\ve)\neq 0 \}$.
We then use Sobolev inequality, we estimate $|D\ve|^2\leq \mathrm M^2$, $w\leq \mathrm M$ in $B_{2R}$, and we use the properties of $\phi$ in \eqref{cutoff} to infer
\begin{equation}\label{mos:0}
    \begin{split}
        \Big(\int_{B_{r_1}} w^{2\kappa (q+1)}\,dx \Big)^{1/\kappa}\leq C\,q^2\frac{\mathrm{M}^2}{(r_2-r_1)^2}\,\int_{B_{r_2}}w^{2q}\,dx.
    \end{split}
\end{equation}
where we set
\begin{equation}\label{sob:kappa}
    \kappa=\begin{cases}
        \frac{n}{n-2}\quad &n>2
        \\
        \text{any number $>1$}\quad & n=2.
    \end{cases}
\end{equation}
Now, if we let $\vartheta=\vartheta(n)=\frac{2}{\kappa-1}$, so that $\vartheta\,\kappa=\vartheta+2$, and setting
\begin{equation*}
    d\mu=\Big(\frac{\mathrm{M}}{w}\Big)^{\vartheta+2}\,\chi_{\{w\neq 0\}}\,dx,
\end{equation*}
then \eqref{mos:0} can be rewritten as
\begin{equation}\label{mosa:1}
   \bigg( \int_{B_{r_1}} w^{\kappa(2q+\vartheta+2)}\,d\mu\bigg)^{1/\kappa}\leq C\,\frac{q^2}{(r_2-r_1)^2}\,\int_{B_{r_2}}w^{2q+\vartheta+2}\,d\mu\,.
\end{equation}
for some $C=C(n,\l,\L,i_a,s_a)>0$.
We now consider radii $r_m=1+1/2^m$ for $m=0,1,2,\dots$, we take $r_1=r_m, r_2=r_{m+1}$, and set
\[
\gamma_0=\vartheta+2,\quad \gamma_{m+1}=\kappa\,\gamma_m=\dots=\kappa^{m+1}\,\gamma_0
\]
so iterating \eqref{mosa:1} yields
\begin{equation}\label{iter:chiamo}
\begin{split}
    \|w\|_{L^{\gamma_{m+1}}(B_{r_{m+1}};d\mu)}&\leq (C\,\gamma_0)^{1/\gamma_m}\,(4\kappa)^{m/\gamma_m}\,\|w\|_{L^{\gamma_{m}}(B_{r_{m}};d\mu)}
    \\
    &\leq \dots\leq  (C\,\gamma_0)^{\frac{1}{\gamma_0}\sum_{m=0}^\infty\frac{1}{\kappa^m}}\,(4\kappa)^{\frac{1}{\gamma_0}\sum_{m=0}^\infty\frac{m}{\kappa^m}}\|w\|_{L^{\gamma_0}(B_2;d\mu)}.
\end{split}
\end{equation}
Letting $m\to\infty$, we deduce that
\begin{equation*}
\begin{split}
    \sup_{B_1} w\leq C\,\|w\|_{L^{\gamma_0}(B_2;d\mu)}&=C\,\bigg(\int_{B_2\cap \{w\neq 0\}}\mathrm{M^{\vartheta+2}}dx\bigg)^{1/(\vartheta+2)}
    \\
    &=C\,\mathrm{M}\,|\{w\neq 0\}\cap B_2|^{1/(\vartheta+2)}\stackrel{\eqref{alt:1scale}}{\leq} C'\,\mu_0^{1/(\vartheta+2)}\,\mathrm{M}\,.
    \end{split}
\end{equation*}
where $C,C'>0$ depend only on $n,\l,\L,i_a,s_a$. Choosing $\mu_0=1/(4C')^{(\vartheta+2)}$, we deduce $w\leq \mathrm{M}/4$ in $B_1$, that is $D_k\ve \geq \mathrm{M}/4$ in $B_1$ by definition of $w$. This concludes the proof. 
\end{proof}

For the second alternative of Theorem \ref{thm:alternative}, we need some preliminary results. The first one is a simple remark.

\begin{remark}\label{remark:nu0}
   \rm{ Let $\mu_0\in (0,2^{n-1})$ be given by Lemma \ref{lem:1alt}, and suppose that, for some $k=1,\dots,n$, there holds
        \begin{equation}\label{temp:alt3}
        \big|\{D_k\ve<\mathrm{M}/2\}\cap B_{2R} \big|>\mu_0\,|B_{2R}|\,.
    \end{equation}
Then
\begin{equation}\label{peso:misura}
    \big|\{D_k\ve>\mathrm{M}/2\}\cap B_{2\nu_0 R} \big|\leq \left(1-\frac{\mu_0}{2}\right)|B_{2\nu_0R}|\,,
\end{equation}
where  $\nu_0=\nu_0(n,\l,\L,i_a,s_a)\in (1/2,1)$ is defined as
\begin{equation}\label{def:nu0}
    \nu_0=\bigg(\frac{1-\mu_0}{1-\mu_0/2}\bigg)^{1/n}=\bigg(2-\frac{2}{2-\mu_0} \bigg)^{1/n}.
\end{equation}
Indeed, since $\nu_0\in(0,1)$, by \eqref{temp:alt3} and \eqref{def:nu0},
\begin{equation*}
\begin{split}
     \big| &\{D_k\ve> \mathrm{M}/2\}\cap B_{2\nu_0 R} \big|\leq \big|\{D_k\ve> \mathrm{M}/2\}\cap B_{2R} \big|
     \\
     &\leq (1-\mu_0)|B_{2R}|=\frac{(1-\mu_0)}{\nu_0^n}|B_{2\nu_0R}|=\left(1-\frac{\mu_0}{2}\right)\,|B_{2\nu_0R}|.
     \end{split}
\end{equation*}
Observe that $\nu_0>1/2$ as we chose $\mu_0\leq 2^{-n-1}$.

Analogously, if
        \begin{equation}\label{temp:alt4}
        \big|\{D_k\ve>-\mathrm{M}/2\}\cap B_{2R} \big|>\mu_0\,|B_{2R}|\,,
    \end{equation}
    then
\begin{equation}\label{peso:misura2}
    \big|\{D_k\ve<-\mathrm{M}/2\}\cap B_{2\nu_0 R} \big|\leq \left(1-\frac{\mu_0}{2}\right)|B_{2\nu_0R}|\,.
\end{equation}
    
   }
\end{remark}

We now show that when \eqref{temp:alt3} is in force, then we can make the set $\{D_k\ve>\kappa\}$ arbitrarily small in measure, provided  we take $\kappa$ sufficiently close to $\mathrm{M}$. This is the content of the following lemma.

\begin{lemma}\label{lemma:alt3small}
    Suppose that \eqref{temp:alt3} is in force, and let $\nu_0$ be given by \eqref{def:nu0}. Then for every $\theta_0\in (0,1)$, there exists $s_0=s_0(n,\l,\L,i_a,s_a,\theta_0)\in \N$ large enough such that
    \begin{equation}\label{arbitr:small}
        \bigg|\Big\{D_k\ve>\Big(1-\frac{1}{2^{s_0}} \Big)\,\mathrm{M} \Big\}\cap B_{2\nu_0R} \bigg|\leq \theta_0\, |B_{2\nu_0R}|\,.
    \end{equation}
\end{lemma}

\begin{proof}
    For $s=1,2,\dots$, we define
    \begin{equation}
        \kappa_s=\left(1-\frac{1}{2^s} \right)\,\mathrm{M},\quad \text{and}\quad A^+_s\coloneqq\{D_k\ve>\kappa_s\}\cap B_{2\nu_0R}\,.
    \end{equation}
Since $\kappa_s\geq \mathrm{M}/2$, from \eqref{peso:misura} we deduce
\begin{equation}\label{peso:misura3}
    |B_{2\nu_0R}\setminus A^+_s|\geq \frac{\mu_0}{2}|B_{2\nu_0R}|\quad\text{for all $s=1,2,\dots$}
\end{equation}
    Applying Lemma \ref{lemma:levels} to the function $u=D_k\ve$, and levels $\ell=\kappa_{s+1}$, $\kappa=\kappa_{s}$, and using H\"older's inequality and \eqref{peso:misura3},  we get
    \begin{equation*}
    \begin{split}
        \frac{\mathrm{M}}{2^{s+1}}\,|A^+_{s+1}|& \leq c(n)\,|A^+_{s+1}|^{\frac{1}{n}}\frac{|B_{2\nu_0R}|}{|B_{2\nu_0R}\setminus A^+_s|}\int_{A^+_{s}\setminus A^+_{s+1}}|D(D_k\ve)|\,dx
        \\
        &\leq C\,R\,\left( \int_{A^+_{s}\setminus A^+_{s+1}}|D(D_k\ve)|^2\,dx\right)^{1/2}\,|A^+_{s}\setminus A^+_{s+1}|^{1/2}\,,
        \end{split}
    \end{equation*}
   with $C=C(n,\l,\L,i_a,s_a)>0$, where in the last inequality we also used $|A^+_{s+1}|\leq |B_{2\nu_0R}|\leq C(n)\,R^n$, the dependency on the data of $\mu_0$ given by \eqref{universal}, and that $\nu_0\in (1/2,1)$.
   
   Then we exploit \eqref{oiii} with $\gamma=1/8$, $r_1=2\nu_0R$ and $r_2=2R$, and recalling the dependency on the data of $\mu_0$, $\nu_0\in (1/2,1)$, and the definition of $\kappa_s$, we deduce
   \begin{equation*}
       \left( \int_{A^+_{s}\setminus A^+_{s+1}}|D(D_k\ve)|^2\,dx\right)^{1/2}\leq C\,\frac{(\mathrm{M}-\kappa_{s})}{R}\,|B_{2R}|^{1/2}\leq C'\,\frac{\mathrm{M}}{2^{s}R}\,|B_{2\nu_0R}|^{1/2},
   \end{equation*}
   for $C,C'=C,C'(n,\l,\L,i_a,s_a)>0$. Connecting the two inequalities above, and squaring both sides of the resulting equation yields
   \begin{equation*}
       |A^+_{s+1}|^2\leq C\,|B_{2\nu_0 R}|\,|A^+_{s}\setminus A^+_{s+1}|\,.
   \end{equation*}
 with $C=C(n,\l,\L,i_a,s_a)>0$.  We sum this inequality over $s=1,2,\dots,s_0-1$, and telescoping the right-hand side, while using $|A^+_{s+1}|\geq |A^+_{s_0}|$ on the left-hand side, we find
   \begin{equation}
       (s_0-2)|A^+_{s_0}|^2\leq \sum_{s=1}^{s_0-1}|A^+_{s+1}|^2\leq C_0\, |B_{2\nu_0 R}|\big(|A^+_1|-|A^+_{s_0}| \big)\leq C_0\, |B_{2\nu_0 R}|^2,
   \end{equation}
   where $C_0=C_0(n,\l,\L,i_a,s_a)>0$. Choosing $s_0=2+C_0/\theta_0^2$ yields $|A^+_{s_0}|\leq \theta_0\,|B_{2\nu_0R}|$, that is \eqref{arbitr:small}, our thesis.
    \end{proof}

As the counterpart of Lemma \ref{lemma:alt3small}, in the case when \eqref{temp:alt4} holds, we have the following lemma, whose proof is completely identical.

\begin{lemma}\label{lemma:alt4small}
     Suppose that \eqref{temp:alt4} is in force for some $k=1,\dots,n$. Then for every $\theta_0\in (0,1)$, there exists $s_0=s_0(n,\l,\L,i_a,s_a,\theta_0)\in \N$ large enough such that
    \begin{equation}\label{arbitr:small1}
        \left|\left\{D_k\ve<-\left(1-\frac{1}{2^{s_0}} \right)\,\mathrm{M} \right\}\cap B_{2\nu_0R} \right|\leq \theta_0\, |B_{2\nu_0R}|\,.
    \end{equation}
\end{lemma}

The next lemma establishes the second alternative in Theorem~\ref{thm:alternative}.

\begin{lemma}[The second alternative]\label{lemma:altfin}
There exists $\eta_0\in (0,1)$ depending on $n,\l,\L,i_a,s_a$ such that, if \eqref{temp:alt3} holds for some $k=1,\dots,n$, then 
\begin{equation}\label{eta0:max}
    D_k\ve\leq \eta_0\,\mathrm{M}\quad \text{in $B_{R}$.}
\end{equation}
Analogously, if \eqref{temp:alt4} holds for some $k=1,\dots,n$, then
\begin{equation}\label{eta0:min}
    D_k\ve\geq-\eta_0\,\mathrm{M}\quad\text{in $B_R$.} 
\end{equation}
\end{lemma}

\begin{proof}
    Let us fix $\theta_0=\theta_0(n,\l,\L,i_a,s_a)\in (0,1)$ to be determined later and, correspondingly, from Lemma \ref{lemma:alt3small} we can find $s_0=s_0(n,\l,\L,i_a,s_a)\in \N$ such that \eqref{arbitr:small} holds.

   From this point on, the proof of \eqref{eta0:max} is very similar to that of \eqref{thesis:alt1}-- see Equations \eqref{start:iter}-\eqref{iterazione:1}. Specifically, for $m=0,1,2,3,\dots$, we set
   \begin{equation}\label{azz:in}
       \kappa_m=\left(1-\frac{1}{2^{s_0}} \right)\mathrm{M}+\left( 1-\frac{1}{2^m}\right)\,\frac{\mathrm{M}}{2^{s_0+1}},\quad R_m=R+(2\nu_0R-R)\left(\frac{1}{2^m} \right),
   \end{equation}
where $\nu_0\in \big(\tfrac{1}{2},1\big)$ is given by \eqref{def:nu0}, and we also set
\begin{equation}
    A^+(\kappa_m,R_m)\coloneqq\{D_k\ve>\kappa_m\}\cap B_{R_m}.
\end{equation}
Then, by \eqref{arbitr:small}, and since $\kappa_m\geq \kappa_0= (1-\tfrac{1}{2^{s_0}})\,\mathrm{M}$ and $R\leq R_m\leq 2\nu_0R\leq 2R$, we have
\begin{equation}\label{azz:1}
\begin{split}
    |B_{R_{m+1}} &\setminus A^+(\kappa_m,R_{m+1})|\geq |B_{R_{m+1}}|-|A^+(\kappa_0,2\nu_0R)|
    \\
    &\geq (1-\theta_0\,(2\nu_0)^n)|B_{R_{m+1}}|\geq  \frac{1}{2}\,|B_{R_{m+1}}|\,,
\end{split}
\end{equation}
provided we choose $0<\theta_0\leq 2^{-n-1}\nu_0^{-n}$.

   We then use Lemma \ref{lemma:levels} with function $u=D_k\ve$, and levels $\ell=\kappa_{m+1}$, $\kappa=\kappa_{m}$, and find
   \begin{equation*}
   \begin{split}
       \frac{\mathrm{M}}{2^{s_0+m+2}}&\, |A^+(\kappa_{m+1},R_{m+1})|^{\frac{n-1}{n}}
       \\
       &\leq C\,\frac{|B_{R_{m+1}}|}{|B_{R_{m+1}} \setminus A^+(\kappa_m,R_{m+1})|}\int_{A^+(\kappa_m,R_{m+1})\setminus A^+(\kappa_{m+1},R_{m+1})}|D(D_k\ve)|\,dx
       \\
       &\leq C'\,\left(\int_{A^+(\kappa_m,R_{m+1})\setminus A^+(\kappa_{m+1},R_{m+1})}|D(D_k\ve)|^2\,dx \right)^{1/2}\,|A^+(\kappa_m,R_{m})|^{1/2}\,,
       \end{split}
   \end{equation*}
with $C,C'=C,C'(n,\l,\L,i_a,s_a)>0$, where in the last estimate we used H\"older's inequality and \eqref{azz:1}.

We now exploit \eqref{oiii} with $\gamma=1/8$, $r_2=R_m$ and $r_1=R_{m+1}$, and using that
\[(\mathrm{M}-\kappa_m)\leq \frac{\mathrm{M}}{2^{s_0}},\quad\text{and}\quad
R_m-R_{m+1}=\frac{(2\nu_0-1)\,R}{2^{m+1}}= c(n,\l,\L,i_a,s_a)\,\frac{R}{2^{m+1}}
\]
for all $m$ (as $\nu_0=\nu_0(n,\l,\L,i_a,s_a)\in(1/2,1)$), we find
\begin{equation*}
    \left(\int_{A^+(\kappa_m,R_{m+1})\setminus A^+(\kappa_{m+1},R_{m+1})}|D(D_k\ve)|^2\,dx \right)^{1/2}\leq C\,\frac{2^{m+2}}{R}\,\frac{\mathrm{M}}{2^{s_0} }|A^+(\kappa_m,R_m)|^{1/2}.
\end{equation*}
Merging the content of the two inequalities above, and dividing both sides of the resulting equation by $\mathrm{M}/2^{s_0}$, we get
\begin{equation}\label{iterazione2}
    |A^+(\kappa_{m+1},R_{m+1})|^{\frac{n-1}{n}}\leq C\,\frac{4^{m}}{R}\,|A^+(\kappa_m,R_m)|\,.
\end{equation}
for some $C=C(n,\l,\L,i_a,s_a)>0$. Hence, by setting
\begin{equation*}
    Z_m=\frac{|A^+(\kappa_m,R_m)|}{|B_{R_m}|},
\end{equation*}
and by exploiting that $R\leq R_m\leq 2\nu_0 R\leq 2R$ as $\nu_0\in (1/2,1)$, from\eqref{iterazione2} we find
\begin{equation*}
    Z_{m+1}\leq C\, (4^{\frac{n}{n-1}})^{m}\,Z_m^{\frac{n}{n-1}},\quad\text{and}\quad Z_0\leq \theta_0\,,
\end{equation*}
with $C=C(n,\l,\L,i_a,s_a)$, where the initial condition on $Z_0$ stems from \eqref{arbitr:small}. Hence, by Lemma \ref{lem:hyp}, if we take $\theta_0=\theta_0(n,\l,\L,i_a,s_a)$ small enough, we get $\lim_{m\to \infty }Z_m=0$, which implies
\begin{equation}\label{azz:fin}
    D_k\ve\leq \left(1-\frac{1}{2^{s_0}} \right)\mathrm{M}+\frac{\mathrm{M}}{2^{s_0+1}}\equiv \eta_0\, \mathrm{M}\,,\quad \text{a.e. in $B_R$,}
\end{equation}
where we set $\eta_0=\eta_0(n,\l, \L,i_a,s_a)=(1-\frac{1}{2^{s_0+1}})$, with $s_0$ provided by Lemma \ref{lemma:alt3small} and the corresponding $\theta_0$ we just fixed.  
Equation \eqref{eta0:max} is thus proven.

Finally, the proof of \eqref{eta0:min} is completely specular, using Lemma \ref{lemma:alt4small} in place of Lemma \ref{lemma:alt3small}. We leave the details to the reader.
\end{proof}

\begin{proof}[Second proof of Lemma \ref{lemma:altfin}]
    Here we provide an alternative proof of the second alternative via a Moser type iteration. Owing to Remark \ref{rem:scaling}, we may assume that $R=1$ (see in particular \eqref{uR:scale}). We fix $\theta_0=\theta_0(n,\l,\L,i_a,s_a)$, and then Lemma \ref{lemma:alt3small} gives $s_0(n,\l,\L,i_a,s_a)\in \N$ such that the scaled version of \eqref{arbitr:small} holds, i.e.,
    \begin{equation}\label{arbsm:scale}
         \bigg|\Big\{D_k\ve>\Big(1-\frac{1}{2^{s_0}} \Big)\,\mathrm{M} \Big\}\cap B_{2\nu_0} \bigg|\leq c(n)\,\theta_0.
    \end{equation}
Now let 
\begin{equation*}
    g(t)=\max\{t-(1-1/2^{s_0})\,\mathrm{M},0\}^{2q+1},
\end{equation*}
    for $q\geq 0$ large enough, and we also define
    \[
    w=\max\big\{D_k\ve-(1-1/2^{s_0})\mathrm{M},0\big\},\quad\text{so that}\quad g(D_k\ve)=w^{2q+1}.
    \]
We notice that
\[
(1-1/2^{s_0})\,\mathrm{M}\leq  D_k\ve \leq |D\ve|\leq \sqrt{n}\, \mathrm{M}\quad \text{in $\{g(D_k\ve)\neq 0\}\cap B_2$,}
\]
 hence from \eqref{ultra:utile} and Remark \ref{remark:importante}, we deduce
 \[
 c\,a_\e(\mathrm{M})\leq a_\e(|D\ve|)\leq C\,a_\e(\mathrm{M}),\quad\text{in $\{g(D_k\ve)\neq 0\}\cap B_{2}$,}
 \]
 with $c,C=c,C(n,\l,\L,i_a,s_a)>0$. Let us then consider \eqref{ooooo}, and divide the resulting equation by $a_\e(\mathrm{M})$, thus getting
\begin{equation*}
    \begin{split}
        \int_{B_{r_2}} |D(w
        ^{q+1}\phi)|^2\,dx\leq C\,\int_{B_{r_2}} w^{2(q+1)}|D\phi|^2\,dx\leq C\,\Big(\frac{\mathrm{M}}{2^{s_0}} \Big)^2\int_{B_{r_2}}w^{2q}|D\phi|^2\,dx
    \end{split}
\end{equation*}
    where in the last inequality we used that $w\leq \big(\frac{\mathrm{M}}{2^{s_0}}\big) $. Next we esploit Sobolev inequality, and by setting $\vartheta=\vartheta(n)=\frac{2}{\kappa-1}$, with $\kappa$ given by \eqref{sob:kappa}, and using the properties of $\phi$ \eqref{cutoff}, we infer
\begin{equation}\label{mos:2}
    \Big(\int_{B_{r_1}}w^{\kappa(2q+\vartheta+2)}\,d\mu\Big)^{1/\kappa}\leq \frac{C}{(r_2-r_1)^2}\int_{B_{r_2}} w^{2q+\vartheta+2}\,d\mu,
\end{equation}
    where we defined the measure
    \begin{equation*}
        d\mu=\Big(\frac{\mathrm{M}}{2^{s_0}}\frac{1}{w} \Big)^{\vartheta+2}\,\chi_{\{w\neq 0\}}\,dx.
    \end{equation*}
    Starting from \eqref{mos:2}, we may iterate exactly as in \eqref{iter:chiamo}, the only difference being the radii $r_m=1+(2\nu_0-1)\,2^{-m}$, $m=0,1,2,\dots$. We thus obtain
\begin{equation*}
    \sup_{B_1} w\leq C\,\|w\|_{L^{\vartheta+2}(B_{2\nu_0};d\mu)}=C\,|\{w\neq 0\}\cap B_{2\nu_0}|^{1/(\vartheta+2)}\Big( \frac{\mathrm{M}}{2^{s_0}}\Big)\stackrel{\eqref{arbsm:scale}}{\leq} C'\,\theta_0^{1/(\vartheta+2)}\Big( \frac{\mathrm{M}}{2^{s_0}}\Big).
\end{equation*}
   Thereby choosing $\theta_0=\theta_0(n,\l,\L,i_a,s_a)=1/(2C')^{\vartheta+2}$, we get $w\leq (\mathrm{M}/2^{s_0+1})$ in $B_1$, that is 
   \[
   D_k\ve \leq \Big(1-\frac{1}{2^{s_0}}\Big)\,\mathrm{M}+\frac{\mathrm{M}}{2^{s_0+1}}\equiv \eta_0\,\mathrm{M}\quad\text{in $B_1$,}
   \]
with $s_0=s_0(n,\l,\L,i_a,s_a)\in \N$ provided by Lemma \ref{lemma:alt3small}. This concludes the proof.   
\end{proof}

We can now prove the fundamental alternative Theorem \ref{thm:alternative}, which is an immediate consequence of Lemmas \ref{lem:1alt}- \ref{lemma:altfin}.

\begin{proof}[Proof of Theorem \ref{thm:alternative}.]
If \eqref{alt:0} holds for some $k\in \{1,\dots,n\}$, then by Lemma \ref{lem:1alt} we have $M(R)\geq |D_k\ve|\geq M(2R)/4$. On the other hand, if  \eqref{alt:3} holds, then  Lemma \ref{lemma:altfin} gives
\begin{equation*}
    -\eta_0\,M(2R)\leq D_k\ve\leq \eta_0\,M(2R)\quad\text{in $B_R$}
\end{equation*}
 for all $k=1,\dots,n$,  hence $M(R)\leq\eta_0 M(2R)$, that is our thesis.  
\end{proof}

Having Theorem \ref{thm:alternative} at our disposal, a standard iteration implies the (quantitative)  H\"older continuity of $D\ve$.
The idea is the following: we consider $B_{2R_0}\Subset \Omega$, and a sequence of dyadic radii $R_m=R_0/2^{m}$ for $m=0,1,2,\dots$.

If for some $m=m_0>0$ the first alternative is valid, i.e., either \eqref{alt:1} or \eqref{alt:2} hold for $R=R_{m_0}$, then $\sqrt{n}\,M(2R_{m_0})\geq|D\ve|\geq M(2R_{m_0})/4$  in $B_{R_{m_0+1}}$. It follows from \eqref{coer:Ae}, \eqref{ultra:utile} and Remark \ref{remark:importante}, that the matrix
\begin{equation*}
    \nabla_\xi \Ae(D\ve)\approx a_\e(|M(2R_{m_0})|)\,\mathrm{Id}
\end{equation*}
whence \eqref{eq:classic} is a uniformly elliptic linear equation, and the H\"older continuity of $D_k\ve$ in $B_{R_{m_0}+1}$ follows from the De Giorgi-Nash-Moser theory \cite[Corollary 4.18]{HL11}. In the complementary case, that is if \eqref{alt:3} holds for all $R=R_m$, $m=0,1,2,3,\dots$, iterating one gets
\[
M(R_{m})\leq \eta_0\,M(R_{m-1})\leq\dots\leq \eta_0^{m}\,M(R_0).
\]
Then for any $\vrho>0$, there exists $m$ such that $R_{m+1}\leq \vrho\leq R_m$, so that 
\begin{equation*}
    \operatorname*{osc}_{B_\vrho} (D\ve)\leq C(n)  \,M(\varrho)\leq C(n)  \,M(R_m) \lesssim \left(\frac{\varrho}{R_0}\right)^{\alpha_1} M(R_0),\quad  \alpha_1=-\log_2 \eta_0.
\end{equation*}
that is the desired result.
We refer, for instance, to \cite[pp. 27-29]{M86} for the details.
With a little more effort, but with a similar reasoning, below we prove the excess decay estimate \eqref{exc:hom}. First let us recall the decay estimate for uniformly elliptic linear equations.

\begin{lemma}\label{lemma:degiorginash}
  Let $w\in W^{1,2}_{loc}(\Omega)$ be a local weak solution to
  \begin{equation}\label{eq:DNM}
      -\mathrm{div}\big( \tilde A(x)\,Dw\big)=0\quad\text{in $\Omega$,}
  \end{equation}
  where the matrix $\tilde A(x)=\{\tilde A_{ij}(x)\}_{i,j=1,\dots,n}$ has measurable entries and satisfies  ellipticity
 and growth bounds
 \begin{equation}\label{A:elliptic}
     c_*\,|\eta|^2\leq \tilde A(x)\,\eta\cdot\eta,\quad \sum_{i,j=1}^n|\tilde A_{ij}(x)|\leq C_*\quad\text{for a.e. $x\in \Omega$,}
 \end{equation}
 and for every $\eta\in \R^n$, for some positive constants $c_*,C_*$. Then there exist constants $C_g\geq 1$ and $\beta_g\in (0,1)$ depending on $n$ and on the ratio $C_*/c_*$ such that
 \begin{equation}\label{DNM:exc}
     \mint_{B_r}|w-(w)_{B_r}|\,dx\leq C_g\left( \frac{r}{R}\right)^{\beta_g}\, \mint_{B_R}|w-(w)_{B_R}|\,dx\,,
 \end{equation}
for every concentric balls $B_r\subset B_{R}\Subset \Omega$.
\end{lemma}

\begin{proof}
     The result is a standard consequence of De Giorgi-Nash-Moser theory for linear elliptic equations. Specifically, by  \cite[Theorems 8.22 and 8.17]{GT} applied to $u=w-(w)_{B_R}$ (which is still a solution to \eqref{eq:DNM}), we get 
     \begin{equation*}
         \operatorname*{osc}_{B_r} w\leq C\,\left( \frac{r}{R}\right)^{\beta_g}\,\mint_{B_R}|w-(w)_{B_R}|\,dx,\quad\text{for all  $r<R/2$,}
     \end{equation*}
     with $\beta_g\in (0,1)$ and $C>0$ depending on $n,C_*/c_*$.
Since, trivially,
\[
\mint_{B_r}|w-(w)_{B_r}|\,dx\leq  \operatorname*{osc}_{B_r} w,
\]
Equation \eqref{DNM:exc} is thus proven in the case $0<r\leq R/2$.
Let us show its validity when  $R/2<r\leq R$. In this case, using \eqref{medie}, we have

\begin{equation*}
\begin{split}
     \mint_{B_r}|w-(w)_{B_r}|\,dx&\leq 2\,\left(\frac{R}{r}\right)^n\mint_{B_R}|w-(w)_{B_R}|\,dx
     \\
     &\leq 2^{n+\beta_g+1}\,\left(\frac{r}{R}\right)^{\beta_g}\mint_{B_R}|w-(w)_{B_R}|\,dx\,,
     \end{split}
\end{equation*}
hence \eqref{DNM:exc} is proven for all $0<r\leq R$.
\end{proof}

We are now ready to prove the excess decay estimate for $D\ve$.
\begin{proposition}\label{exs:Dve}
    Let $\ve\in W^{1,2}_{loc}(\Omega)$ be a local weak solution to \eqref{eq:homeint}. Then
\begin{equation}\label{exc:homee}
        \mint_{B_r}|D\ve-(D\ve)_{B_r}|\,dx\leq c_{\mathrm{h}}\,\left( \frac{r}{R}\right)^{\alpha_\mathrm{h}} \mint_{B_R}|D\ve-(D\ve)_{B_R}|\,dx,
    \end{equation}
    for every concentric ball $B_r\subset B_R\subset B_{2R}\Subset \Omega$,  where $\alpha_\mathrm{h}\in (0,1)$ and $c_\mathrm{h}>0$ depend only on $n,\l,\L,i_a,s_a$.
\end{proposition}

\begin{proof}
The proof follows closely the argument of \cite[Theorem~3.1]{DM10}. Let $M(r)$ be given by \eqref{def:Mr}, and define the excess functional
   \begin{equation}
       E(r)\coloneqq \mint_{B_r}|D\ve-(D\ve)_{B_r}|\,dx.
   \end{equation}
   By Theorem \ref{thm:alternative}, only the two alternatives are possible:  either
    \begin{equation}\label{new:alt1}
        |D\ve|\geq M(2r)/4\quad\text{in $B_r$  }
    \end{equation}
    or
\begin{equation}\label{new:alt2}
    M(r)\leq \eta_0\,M(2r).
\end{equation}
happen for all $0<r\leq R$, with $\eta_0\in (0,1)$ given by \eqref{universal}. 
\\

\noindent \textit{Step 1. The nondegenerate case.} In the case \eqref{new:alt1} holds for some radius $r\leq R$,  then by \eqref{ultra:utile} and Remark \ref{remark:importante}, we have
 $$c_{**}\,a_\e(M(2r))\leq a_\e(|D\ve|)\leq C_{**}\,a_\e(M(2r))\quad \text{in } B_{r},$$
 where $c_{**},C_{**}=c_{**},C_{**}(n,\l,\L,i_a,s_a)$, and hence by \eqref{coer:Ae}, we have
\begin{equation*}
    \nabla_\xi \Ae(D\ve)\,\eta\cdot \eta\geq c'_{**}\,a_\e(M(2r))\,|\eta|^2\quad\text{and }\quad |\nabla_\xi \Ae(D\ve)|\leq C'_{**}\,a_\e(M(2r))\quad \text{in $B_{r}$,}
\end{equation*}
for all $\eta\in \rn$, where $c'_{**},C'_{**}$ depend on the same data as $c_{**},C_{**}$.

Therefore, for all $k=1,\dots,n$, owing to \eqref{eq:classic}, $D_k\ve$ solves a uniformly elliptic linear equation as in \eqref{eq:DNM}. Lemma \ref{lemma:degiorginash}  thus entails
 \begin{equation}\label{temp:DNMexc}
     \mint_{B_\vrho}|D\ve-(D\ve)_{B_\vrho}|\,dx\leq C_{\mathrm{d}}\,\left( \frac{\vrho}{r}\right)^{\beta_{\mathrm{d}}}\mint_{B_t}|D\ve-(D\ve)_{B_t}|\,dx,\quad \text{for all $0<\vrho\leq r$},
 \end{equation}
 with $C_{\mathrm{d}}\geq 1$ and $\beta_{\mathrm{d}}\in (0,1)$ determined only by $n,\l,\L,i_a,s_a$. The main point here is that these constants do not depend on $M(r)$. 
\\

\noindent\textit{Step 2. Choice of constants.} We now choose two constants, whose utility will become apparent later. Let us fix $H_1=H_1(n,\l,\L,i_a,s_a)\geq 1$ such that
\begin{equation}\label{H1}
    8\sqrt{n}\,C_{\mathrm{b}}\eta_0^{H_1-1}\leq 1\,,
\end{equation}
where $C_{\mathrm{b}}>0$ is the constant appearing in \eqref{Degiorgi:ve}. In turn, we fix another parameter $K_1=K_1(n,\l,\L,i_a,s_a)\geq 1$ satisfying
\begin{equation}\label{K1}
    2^{nH_1+2}\eta_0^{K_1}\leq 1.
\end{equation}
Finally, we set
\begin{equation}\label{tbetad} 
\tilde\beta\coloneqq\frac{1}{(H_1+K_1)}\in (0,1]\quad\text{and}\quad \alpha_{\mathrm{h}}\coloneqq\min\big\{\beta_{\mathrm{d}},\, \tilde\beta\big\}\,.
\end{equation}
where $\beta_{\mathrm{d}}$ appears in \eqref{temp:DNMexc}, so that $\tilde\beta,\alpha_{\mathrm{h}}$ are determined only by $n,\l,\L,i_a,s_a$.
\\

\noindent\textit{Step 3. The degenerate case I.} We consider the following situation: there exists a radius $t\leq R$ such that \eqref{new:alt2} happens to hold whenever $r=t/2^i$ for all $1\leq i\leq H_1\in \N$, and we also assume that
\begin{equation}\label{new:alt3}
    |(D\ve)_{B_t}|\leq 2\sqrt{n}\,M\big(2^{-H_1}t \big)\,.
\end{equation}
By iterating \eqref{new:alt2} we find $M(2^{-H_1}t)\leq \eta_0^{H_1-1}\,M(t/2)$. Therefore
\begin{equation*}
    \begin{split}
        M(2^{-H_1}t) &\leq \eta_0^{H_1-1}\,M(t/2)\stackrel{\eqref{Degiorgi:ve}}{\leq} C_{\mathrm{b}}\,\eta_0^{H_1-1}\mint_{B_t}|D\ve|\,dx
        \\
        &\leq C_{\mathrm{b}}\,\eta_0^{H_1-1}\mint_{B_t}|D\ve-(D\ve)_{B_t}|\,dx+C_{\mathrm{b}}\,\eta_0^{H_1-1}|(D\ve)_{B_t}|
        \\
        &\stackrel{\eqref{new:alt3}}{\leq} C_{\mathrm{b}}\,\eta_0^{H_1-1}\,E(t)+2\sqrt{n}C_{\mathrm{b}}\,\eta_0^{H_1-1} M(2^{-H_1}t)
        \\
        &\stackrel{\eqref{H1}}{\leq} C_{\mathrm{b}}\,\eta_0^{H_1-1}\,E(t)+\frac{1}{2}M(2^{-H_1}t).
    \end{split}
\end{equation*}
hence, reabsorbing term, we get $M(2^{-H_1}t)\leq 2\,C_{\mathrm{b}}\,\eta_0^{H_1-1}\,E(t)$. Noticing that, trivially, 
\[E(2^{-H_1}t)\leq 2\sqrt{n} M(2^{H_1}t),\]
making use of this information and  using \eqref{H1} again, we infer
\begin{equation}\label{deg:I}
    E(2^{-H_1}t)\leq \frac{1}{2} E(t)\,.
\end{equation}
\vspace{0.2cm}

\noindent\textit{Step 4. The degenerate case II.} Continuing the reasoning of the previous step, we now assume that there exists a radius $t\leq R$ such that \eqref{new:alt2} holds whenever $r=t/2^i$ and $1\leq i\leq H_1+K_1\in \N$, and we also assume, in alternative to \eqref{new:alt3}, that
\begin{equation}\label{new:alt4}
    |(D\ve)_{B_t}|>2\sqrt{n}\,M(2^{-H_1}t)\,.
\end{equation}
In particular, this implies
\begin{equation}\label{abz}
    |D\ve-(D\ve)_{B_t}|>\sqrt{n}\,M(2^{-H_1}t)\quad\text{in $B_{2^{-H_1}t}$.}
\end{equation}
Iterating \eqref{new:alt2}, we get
\begin{equation*}
    \begin{split}
        E(2^{-(H_1+K_1)}t) &\leq 2\sqrt{n}\,M(2^{-(H_1+K_1)}t)\leq 2\sqrt{n}\,\eta_0^{K_1}M(2^{-H_1}t)
        \\
        &\stackrel{\eqref{abz}}{\leq} 2\,\eta_0^{K_1}\mint_{B_{2^{-H_1}t}}|D\ve-(D\ve)_{B_t}|\,dx
        \\
        &\leq 2^{nH_1+1}\,\eta_0^{K_1}\mint_{B_{t}}|D\ve-(D\ve)_{B_t}|\,dx,
    \end{split}
\end{equation*}
and using \eqref{K1}, we deduce
\begin{equation}\label{deg:II}
    E(2^{-(H_1+K_1)}t)\leq \frac{1}{2}E(t)\,.
\end{equation}
\vspace{0.2cm}

\noindent\textit{Conclusion.} We conclude by a two-speed iteration combined with a
 certain alternative. With $H_1$ and $K_1$ defined in Step 2, we set

\begin{equation}\label{stilde}
    \tilde\sigma\coloneqq2^{-H_1}\quad \text{and}\quad\tilde\tau\coloneqq 2^{-(H_1+K_1)}\,.
\end{equation}
For $0<r\leq R$ as in the statement of the theorem, we consider the set
\begin{equation}\label{S:iter}
    \mathcal{S}=\{i\in \N:\,\text{ \eqref{new:alt2} holds for } r=R/2^i,\,i\geq 1 \},
\end{equation}
and consider the following alternative.
\vspace{0.2cm}

\textit{Case 1:} $\mathcal{S}=\N\setminus\{0\}$. We start by setting $\tau_0=1$, and $t=R$. As $\mathcal{S}=\N\setminus\{0\}$, we have that either Step 3 or Step 4 is in force, that is either
\begin{equation*}
    E(\tilde\sigma\, t)\leq\frac{1}{2} E(R)\quad\text{or}\quad  E(\tilde\tau\, t)\leq\frac{1}{2} E(R)
\end{equation*}
holds. Now, we check whether Step 4 works, and if this is the case we set $\tau_1=\tilde\tau$. If not, then Step 3 applies, and we set $\tau_1=\tilde \sigma$.  In both cases, we have
\begin{equation}\label{Etau1:R}
    E(\tau_1 R)\leq \frac{1}{2} E(R).
\end{equation}

Next, we set $t=\tau_1 R$ and re-examine the alternative between Steps 3 and 4. Taking \eqref{Etau1:R} into account, we thus have that one of the two
 inequalities applies :
 \begin{equation*}
     E(\tilde\tau\,\tau_1R)\leq (1/2)^2 E(R)\quad\text{or}\quad E(\tilde\sigma\,\tau_1R)\leq (1/2)^2 E(R)
 \end{equation*}
 Again, if Step 4 applies we set $\tau_2=\tilde\tau\,\tau_1$ and, if not, then Step 3 holds and we set $\tau_2=\tilde\sigma\,\tau_1$. In any case, we have
 \begin{equation}
     E(\tau_2 R)\leq \left(\frac{1}{2}\right)^2 E(R)\,,
 \end{equation}
 and then we restart by $t=\tau_2 R$ to re-study whether Step 3 or 4 holds.

 Proceeding inductively, we find sequences $\{\tau_i\}\subset (0,1)$, $\{s_i\}\subset \N$, $\{h_i\}\subset \N$ such that
 \begin{equation}\label{tauuu:1}
     \tau_i=(\tilde\sigma)^{s_i}\,(\tilde\tau)^{h_i},\quad s_i+h_i=i
 \end{equation}
for every $i\in \N$. The sequences $\{s_i\}$ and $\{h_i\}$ are such that either $s_{i+1} = s_i +1$
 and $h_{i+1} = h_i$ or $s_{i+1} = s_i$ and $h_{i+1} = h_i+1$; moreover, the inductive procedure gives
\begin{equation}\label{indazzz}
    E(\tau_i\,R)\leq \left(\frac{1}{2}\right)^i\,E(R),\quad \text{for all $i\geq0$.}
\end{equation}
More in detail, to define $\tau_{i+1}$ starting from $\tau_i$, we let $t = \tau_i\,R$; then, if Step 4 holds we
 let $h_{i+1} = h_i + 1$, otherwise Step 3 holds and in this case we let $s_{i+1} = s_i+ 1$; we finally define $\tau_{i+1}$ according to \eqref{tauuu:1}. We notice that $\tau_i$ is a strictly decreasing
 sequence, while $\{s_i\}$, $\{h_i\}$ are nondecreasing.

Now, for $0<r\leq R$, we may find $i\in \N$ be such that $\tau_{i+1}R\leq r\leq \tau_iR$, and notice that by \eqref{tauuu:1} and \eqref{stilde}, we have 
\begin{equation}\label{blockrho:tauR}
    \frac{\tau_i R}{r}+\frac{r}{\tau_{i+1}R}\leq 2\frac{\tau_i}{\tau_{i+1}}\leq C(n,\l,\L,i_a,s_a),
\end{equation}
where the last inequality follows from the dependence of $H_1$ and $K_1$ on the data. Moreover, by definition of $\tau_i$ in \eqref{tauuu:1} and that of $\tilde\beta,\,\tilde\tau$ in \eqref{tbetad} and \eqref{stilde}, respectively, we get 
\begin{equation}\label{domin:taui}
    (\tau_i)^{\tilde\beta}\geq \big(\tilde\tau \big)^{i\,\tilde\beta}=(1/2)^i.
\end{equation}
So, by also using \eqref{medie} we get 
\begin{equation}\label{argue:utile}
    \begin{split}
        E(r) \leq C\,\mint_{B_r}|D\ve&-(D\ve)_{B_{\tau_iR}}|\,dx\leq C\left(\frac{\tau_i R}{r} \right)^{n}E(\tau_iR)
        \\
        &\stackrel{\eqref{blockrho:tauR}}{\leq} C'\,E(\tau_iR)\stackrel{\eqref{indazzz}}{\leq} C'\left(\frac 1 2 \right)^i\, E(R)
        \\
        &=C'\left(\frac 1 2 \right)^i\left(\frac{R}{r} \right)^{\tilde\beta}\left(\frac{r}{R} \right)^{\tilde\beta}\,E(R)
        \\
        &\stackrel{\eqref{blockrho:tauR}}{\leq}C''\left(\frac 1 2 \right)^i\tau_{i}^{-\tilde\beta}\left(\frac{r}{R} \right)^{\tilde\beta}\,E(R)
        \\
        &\stackrel{\eqref{domin:taui}}{\leq} C''\left(\frac{\vrho}{R} \right)^{\tilde\beta}\,E(R)\leq C''\left(\frac{\vrho}{R} \right)^{\alpha_{\mathrm{h}}}\,E(R),
    \end{split}
\end{equation}
where $C,C',C''>0$ depend on $n,\l,\L,i_a,s_a$, and where in the last inequality we used $\alpha_\mathrm{h}\leq \tilde \beta$ by \eqref{tbetad}. Equation  \eqref{exc:homee} is thus proven in this case.
\\

\textit{Case 2.} $\mathcal{S}\neq \N\setminus\{0\}.$ We set $\mathrm{m}\coloneqq \min\big\{ (\N\setminus \{0\})\setminus \mathcal{S}\big\}$, so that \eqref{new:alt1} implies 
\[
|D\ve|\geq M(R/2^{\mathrm{m}-1})/4\quad\text{in $B_{R/2^{\mathrm{m}}}$.}
\]
Therefore, we are in the setting of Step 1 with $r=R/2^{\mathrm{m}}$, and by \eqref{temp:DNMexc} we have
\begin{equation}\label{temp11:DNMexc}
    E(r)\leq C_{\mathrm{d}}\,\left(\frac{r}{R/2^{\mathrm{m}}} \right)^{\alpha_\mathrm{h}}E(R/2^{\mathrm{m}}),\quad \text{for every $0<r\leq R/2^{\mathrm{m}},$}
\end{equation}
where we also used $\alpha_\mathrm{h}\leq \beta_{\mathrm{d}}$ by \eqref{tbetad}. 

In order to pass from the previous inequality to the full form \eqref{exc:homee}, we iterate exactly as in Case 1, and we stop as soon as we find a certain number $\tau_iR\leq R/2^{\mathrm{m}}$. 
 
 More precisely, we proceed as follows: we start checking if $\tilde\tau R\geq R/2^{\mathrm{m}}$;  in this case
 we perform the alternative between Step 3 and Step 4, and we define $\tau_1$ as in Case 1. Otherwise, if $\tilde\tau R\leq R/2^{\mathrm{m}}$, we stop  and set $\gamma_0=0$, and $\tau_\gamma=\tau_0=1$. 
 
Then we restart, and check whether $\tilde\tau\,\tau_1R\leq R/2^{\mathrm{m}} $; if this is the case, we perform the alternative and define $\tau_2$, otherwise we stop and set $\gamma=1$. Proceeding in this way, after a finite number of times we find  numbers $\{\tau_0,\tau_1,\dots,\tau_\gamma\}$ with the property
\begin{equation}\label{bho:tauig}
    \tau_i\geq (\tilde\tau)^i\quad \text{for all $i=1,\dots,\gamma$,}
\end{equation}
\begin{equation}\label{again:indazzz}
      E(\tau_i\,R)\leq \left(\frac{1}{2}\right)^i\,E(R),\quad \text{for all $0\leq i\leq \gamma\in \N$,}
\end{equation}
and
\begin{equation}\label{tau:gamma}
    \tau_\gamma R\geq R/2^{\mathrm{m}}\geq \tau_\gamma\,\tilde\tau\,R\,.
\end{equation}
Using this last inequality, the definition of $\tilde\tau$ in \eqref{stilde}, and \eqref{medie}, we get 
\begin{equation*}
\begin{split}
    E(R/2^{\mathrm{m}})&\leq 2\,\mint_{B_{R/2^{\mathrm{m}}}} |D\ve-(D\ve)_{B_{\tau_\gamma R}}|\,dx
    \\
    &\leq (\tilde\tau)^n\,E(\tau_\gamma R)\stackrel{\eqref{again:indazzz}}{\leq} C\,(1/2)^\gamma E(R)\,,
    \end{split}
\end{equation*}
with $C=C(n,\l,\L,i_a,s_a)$. Then we estimate
\begin{equation*}
    \begin{split}2^{\mathrm{m}\alpha_\mathrm{h}} &\stackrel{\eqref{tbetad}}{\leq} 2^{\frac{\mathrm{m}}{(H_1+K_1)}}\stackrel{\eqref{tau:gamma}}{\leq} (\tau_\gamma\tilde\tau)^{-\frac{1}{(H_1+K_1)}}\stackrel{\eqref{bho:tauig}} {\leq} (\tilde\tau)^{-\frac{\gamma+1}{(H_1+K_1)}}\stackrel{\eqref{stilde}}{=} 2^{\gamma+1}\,.
    \end{split}
\end{equation*}
Combining the two inequalities above with \eqref{temp11:DNMexc},  we infer the validity of \eqref{exc:homee} for $0<r\leq R/2^{\mathrm{m}}$.
\\

We are left to study the case $r>R/2^{\mathrm{m}}$, and we distinguish two cases. If $r\geq \tau_\gamma R$, then there exists $i\in \{0,1,\dots,\gamma-1\}$ such that $\tau_{i+1}R\leq r< \tau_i R $, and arguing similarly to \eqref{argue:utile}, we find 
\begin{equation}\label{qquasi:ex}
        E(r)\leq C\,\left(\frac{r}{R} \right)^{\alpha_\mathrm{h}} E(R),\quad \text{if }\,\tau_\gamma R\leq r.
\end{equation}
Finally, if $R/2^{\mathrm{m}}\leq r\leq \tau_\gamma R$, we have
\begin{equation*}
\begin{split}
    E(r)& \leq C\,\left(\frac{\tau_\gamma R}{r} \right)^{n}\,E(\tau_\gamma R)\leq C\,(\tau_\gamma\,2^{\mathrm{m}})^{n}\,E(\tau_\gamma R)
    \\
    &\stackrel{\eqref{tau:gamma}}{\leq } C'\,E(\tau_\gamma R) \stackrel{\eqref{qquasi:ex}}{\leq} C''\,(\tau_\gamma)^{\alpha_{\mathrm{h}}} E(R)
    \\
    &= C''(\tau_\gamma R/r)^{\alpha_{\mathrm{h}}} (r/R)^{\alpha_{\mathrm{h}}} E(R)
    \\
    &\stackrel{\eqref{tau:gamma}}{\leq }C''(\tilde\tau)^{-\alpha_{\mathrm{h}}}(r/R)^{\alpha_{\mathrm{h}}}\,E(R)=C'''\,(r/R)^{\alpha_{\mathrm{h}}}\,E(R).
    \end{split}
\end{equation*}
where $C,C',C'',C'''$ depend on $n,\l,\L,i_a,s_a$. The proof is complete.
\end{proof}

As an immediate consequence of the excess decay estimate \eqref{exc:homee}, we obtain the following oscillation estimate.

\begin{corollary}\label{osc:decay}
    Let $\ve\in W^{1,2}_{loc}(\Omega)$ be solution to \eqref{eq:homint}. Then for every $B_{2R}\Subset \Omega$, we have
    \begin{equation}\label{campanato0}
        \operatorname*{osc}_{B_r} Dv_\e\leq C_{\mathrm{h}}\,\left( \frac{r}{R}\right)^{\alpha_{\mathrm h}}\,\mint_{B_R}|D\ve-(D\ve)_{B_R}|\,dx\quad\text{ for all $0<r\leq R/2$,}
    \end{equation}
    and for every $0<r\leq R$, we have \begin{equation}\label{campanato}
        \operatorname*{osc}_{B_r} D\ve\leq C_{\mathrm{h}}\left(\frac{r}{R} \right)^{\alpha_h} \operatorname*{osc}_{B_R} D\ve\leq C'_\mathrm{h}\,\left(\frac{r}{R} \right)^{\alpha_h}\,\mint_{B_{2R}}|D\ve|\,dx
    \end{equation}
    where $\alpha_{\mathrm{h}}\in (0,1)$ is given by     
    \eqref{exc:homee}, and $C_{\mathrm{h}},C'_{\mathrm{h}}=C_{\mathrm{h}},C'_{\mathrm{h}}(n,\l,\L,i_a,s_a)>0$.
\end{corollary}

\begin{proof}
    For the moment let us take $0<r\leq R/2$. Let $0<\vrho\leq R/2$, and $x_0\in B_{r}$. By \eqref{exc:homee}, we have
\begin{equation*}
\begin{split}
    \frac{1}{\vrho^{\alpha_{\mathrm{h}}}}\mint_{B_\vrho(x_0)}|D\ve-(D\ve)_{B_{\vrho}(x_0)}|dx&\leq \frac{C}{R^{\alpha_{\mathrm{h}}}}\,\mint_{B_{R/2}(x_0)}|D\ve-(D\ve)_{B_{R/2}(x_0)}|dx
\\            
&\leq \frac{C}{R^{\alpha_{\mathrm{h}}}}\,\mint_{B_{R}}|D\ve-(D\ve)_{B_{R}}|dx
    \end{split}
\end{equation*}
    where in the last inequality we used \eqref{medie} and that $B_{\vrho}(x_0)\subset B_{r}(x_0)\subset B_R$. If instead $\vrho\geq R/2$, then by \eqref{medie} we immediately get 
\begin{equation*}
    \frac{1}{\vrho^{\alpha_{\mathrm{h}}+n}}\int_{B_r\cap B_{\vrho}(x_0)}|D\ve-(D\ve)_{B_r\cap B_{\vrho}(x_0)}|\,dx\leq \frac{C}{R^{\alpha_{\mathrm h}}}\,\mint_{B_R}|D\ve-(D\ve)_{B_R}|\,dx.
\end{equation*}
The two inequalities above are valid for all $x_0\in B_r$, so by Campanato characterization of H\"older continuity \cite[Theorem 5.5]{GM12} we get
\begin{equation*}
    \operatorname*{osc}_{B_r} D\ve\leq C\,\left( \frac{r}{R}\right)^{\alpha_{\mathrm h}}\,\mint_{B_R}|D\ve-(D\ve)_{B_R}|\,dx\leq C\left( \frac{r}{R}\right)^{\alpha_{\mathrm h}}\operatorname*{osc}_{B_R} D\ve
\end{equation*}
in the case $0<r\leq R/2$, so \eqref{campanato0} is proven.
On the other hand, when $R/2\leq r\leq R$ we trivially have
\begin{equation*}
     \operatorname*{osc}_{B_r} D\ve\leq  \operatorname*{osc}_{B_R} D\ve\leq 2^{\alpha_{\mathrm{h}}}\left( \frac{r}{R}\right)^{\alpha_{\mathrm h}}\operatorname*{osc}_{B_R} D\ve.
\end{equation*}
Finally, the second inequality in \eqref{campanato} follows from the elementary inequality $\operatorname{osc}_{B_R} Dv_\e\leq C(n)\,\sup_{B_R} |D\ve|$, and \eqref{Degiorgi:ve}.
\end{proof}

Having Thereom \ref{thm:alternative} at disposal, we can now prove Theorem \ref{thm:inthom} via an approximation procedure. 
Since $v\in W^{1,B}$ does not necessarily belong to $W^{1,2}$, while the vector field $\A_\e$ has quadratic growth  by \eqref{bBquadratic} and \eqref{coAe:gr} (see also \eqref{nonzero:Ae}), an additional approximation step is required.

\begin{proof}[Proof of Theorem \ref{thm:inthom}]
Let $B_{2R}\Subset \Omega$, and for $k\in \N$ large enough so that $B_{2R+1/k}\subset \Omega$, define the regularized function
\begin{equation*}
    v_k^{bd}(x)=v\ast \rho_{1/k}(x).
\end{equation*}
Here $bd$ stands for boundary, to  emphasize that $v_k^{bd}$ will be used as boundary data, while approximating the original solution $v$-see \eqref{eq:veku}.

By the properties of convolution \cite[Theorem 4.4.7
]{HHbook}, $v_k^{bd}\in C^{\infty}(\overline{B}_{2R})$ and we have $v_k^{bd}\xrightarrow{k\to\infty} v\quad\text{in $W^{1,B}(B_{2R})$}$,
    so that  by \eqref{equiv:modular} and \eqref{triangle},
    \begin{equation}\label{bd:kconv}
        \lim_{k\to \infty}\int_{B_{2R}}B\big( |v_k^{bd}|)\,dx=\int_{B_{2R}}B\big( |v|)\,dx,\quad\text{and}\quad  
        \lim_{k\to \infty}\int_{B_{2R}}B\big( |Dv_k^{bd}|)\,dx=\int_{B_{2R}}B\big( |Dv|)\,dx.
    \end{equation}
Moreover, by \eqref{unif:refl}, we have that $v\in W^{1,i_B}_{loc}(\Omega)$, so by the properties of convolution and the continuity of the trace operator
\begin{equation}\label{vktov_trace}
    v_k^{bd}\xrightarrow{k\to\infty}v\quad\text{in $L^1(\partial B_{2R})$ and in the sense of traces on $\partial B_{2R}$.}
\end{equation}
    
For $\e>0$ and $k$ as above, let $v_{\e,k}\in W^{1,2}(B_{2R})$ be the unique solution to
\begin{equation}\label{eq:veku}
\begin{cases}
    -\mathrm{div}\big(\Ae(Dv_{\e,k})\big)=0\quad &\text{in $B_{2R}$}
    \\
    v_{\e,k}=v_k^{bd}\quad &\text{on $\partial B_{2R}$.}
    \end{cases}
\end{equation}
    Thanks to \eqref{strong:coer}, the properties of $\Ae$ in Lemma \ref{lemma:Ae} and \eqref{bBquadratic}, existence and uniqueness of $v_{\e,k}$ readily follows from the theory of monotone operators \cite[Theorem 26.A]{Z90}-- see also the proof of Proposition \ref{prop:ex} below. 
    Next, let us prove a uniform energy bound for $v_{\e,k}$. We test the weak formulation of \eqref{eq:veku} with $v_{\e,k}-v_k^{bd}$, and get
    \begin{equation*}
        \int_{B_{2R}}\Ae(Dv_{\e,k})\cdot Dv_{\e,k}\,dx=\int_{B_{2R}} \Ae(Dv_{\e,k})\cdot Dv_k^{bd}\,dx\,.
    \end{equation*}
    By \eqref{coAe:gr} and  Young's inequality \eqref{young1}, while keeping in mind Remark \ref{remark:importante}, we get
    \begin{equation*}
        \begin{split}
            \int_{B_{2R}}\Ae(Dv_{\e,k})\cdot Dv_{\e,k}\,dx&\geq c_0\,\int_{B_{2R}} B_\e\big(|Dv_{\e,k}| \big)\,dx
            \\
           \bigg| \int_{B_{2R}} \Ae(Dv_{\e,k})\cdot Dv_k^{bd}\,dx\bigg|&\leq C_0\int_{B_{2R}}b_\e\big(|Dv_{\e,k}|\big)\,|Dv_k^{bd}|\,dx
           \\
           &\leq \frac{c_0}{2}\int_{B_{2R}}B_\e\big(|Dv_{\e,k}| \big)\,dx+C\,\int_{B_{2R}}B_\e\big(|Dv_{k}^{bd}| \big)\,dx\,,
        \end{split}
    \end{equation*}
    where all the involved constants depend on $n,\l,\L,i_a,s_a$. Therefore
    \begin{equation*}
        \int_{B_{2R}}B_\e\big(|Dv_{\e,k}| \big)\,dx\leq C\,\int_{B_{2R}}B_\e\big(|Dv_{k}^{bd}| \big)\,dx\,.
    \end{equation*}
    and coupling this information with \eqref{BE:unif} and \eqref{bd:kconv}, we deduce    \begin{equation}\label{energy:bound}
\limsup_{\e\to 0} \int_{B_{2R}}B_\e\big(|Dv_{\e,k}| \big)\,dx\leq C\,\int_{B_{2R}}B(|Dv_k^{bd}|)\,dx\leq 2\,C\,\int_{B_{2R}}B(|Dv|)\,dx\,,
    \end{equation}
    for $k\in \N$ large enough, with $C=C(n,\l,\L,i_a,s_a)$. Then, from \eqref{int:simple} and \eqref{energy:bound}, we get
    \begin{equation}\label{L1:unif}
        \int_{B_{2R}}|Dv_{\e,k}|\,dx\leq C(i_a,s_a)\,\int_{B_{2R}} B_\e\big( |Dv_{\e,k}|\big)\,dx+|B_{2R}|\leq C\,\int_{B_{2R}} B(|Dv|)\,dx+|B_{2R}|
    \end{equation}
    for all $0<\e<\e_0$ small enough, with $C>0$ independent of $\e,k$.
    

By the triangle and Poincar\'e inequalities
\begin{equation}\label{to:limit4}
\begin{split}
   \int_{B_{2R}}|v_{\e,k}|\,dx &\leq\int_{B_{2R}}|v_{\e,k}-v_k^{bd}|\,dx+\int_{B_{2R}}|v_k^{bd}|\,dx 
   \\
   &\leq C\,R\,\int_{B_{2R}}|Dv_{\e,k}-Dv_k^{bd}|\,dx+\int_{B_{2R}}|v_k^{bd}|\,dx 
   \\
   &\leq C\int_{B_{2R}}|Dv_{\e,k}|\,dx+C\int_{B_{2R}}|Dv_k^{bd}|\,dx +C\int_{B_{2R}}|v_k^{bd}|\,dx\leq C',
    \end{split}
\end{equation}
with $C,C'$ independent of $\e,k$, where in the last estimate we exploited \eqref{L1:unif}, \eqref{int:simple0} and \eqref{bd:kconv}.


Estimate \eqref{to:limit4} coupled with Theorem \ref{thm:bdd} (taking into account Remark \ref{remark:importante}), Proposition \ref{prop:bernstein}, Corollary \ref{osc:decay} and a standard covering argument (applied to $v_{\e,k}$ in place of $v_\e$) yields
\begin{equation*}
    \|v_{\e,k}\|_{C^{1,\alpha_{\mathrm{h}}}(\mathcal{B})}\leq \overline{C}(\mathcal{B})\,,
\end{equation*}
for every $\mathcal{B}\Subset B_{2R}$, with $\overline{C}(\mathcal{B})$ independent of $\e,k$. By Ascoli-Arzel\'a theorem, we infer that, up to subsequences, we have
\begin{equation}\label{ennersimo:unif}
 \lim_{k\to \infty}\lim_{\e\to 0}v_{\e,k}=w\quad\text{in $C^1_{loc}(B_{2R})$,}
\end{equation}
for some function $w\in C^1(B_{2R})$. Let us show that $w\in W^{1,B}(B_{2R})$ and $w=v$ on $\partial B_{2R}$. To this end, take $\mathcal{B}\Subset B_{2R}$, and thanks to \eqref{BE:unif} and \eqref{ennersimo:unif}, we have
\[
\lim_{k\to\infty}\lim_{\e\to 0}B_\e\big(|Dv_{\e,k}| \big)=B\big(|Dw| \big)\quad\text{uniformly in $\mathcal{B}$}
\]
which together with \eqref{energy:bound} gives
\begin{equation}\label{calBmon}
    \int_{\mathcal{B}} B(|Dw|)\,dx=\lim_{k\to\infty}\lim_{\e\to 0}\int_{\mathcal{B}} B_\e(|Dv_{\e,k}|)\,dx\leq C\,\int_{B_{2R}}B(|Dv|)\,dx,
\end{equation}
hence letting $\mathcal{B}\nearrow B_{2R}$, by monotone convergence theorem we deduce that $w\in W^{1,B}(B_{2R})$.
Then, by \eqref{unif:refl} and \eqref{energy:bound}, the sequence $\{v_{\e,k}\}$ is uniformly bounded in $W^{1,\min\{i_B,2\}}(B_{2R})$, with $i_B=i_a+2>1$; by the reflexivity of such space, the boundary condition \eqref{eq:veku}, the convergence \eqref{vktov_trace}, and the continuity of the trace operator w.r.t. weak convergence, we deduce that $w=v$ on $\partial B_{2R}$.

To conclude, we need to show that $w=v$ on $B_{2R}$. Testing the weak formulation of \eqref{eq:veku} with a function $\varphi\in C^\infty_c(B_{2R})$, and using  \eqref{Aeunif} and \eqref{ennersimo:unif}, by letting $\e\to 0$ and then $k\to \infty$ we find that $w\in W^{1,B}(B_{2R})\cap C^1(B_{2R})$ satisfies
\begin{equation*}
    \int_{B_{2R}} \A(Dw)\cdot D\varphi\,dx=0;
\end{equation*}
via a density argument, the above equation holds for all test functions $\varphi\in W^{1,B}_0(B_{2R})$. Thus
 $w\in W^{1,B}(B_{2R})\cap C^1(B_{2R})$ is solution to $-\mathrm{div}\big( \A(Dw)\big)=0$ in $B_{2R}$, and $w=v$ on $\partial B_{2R} $, and therefore $w\equiv v$  by uniqueness. Finally, using \eqref{ennersimo:unif}, we may pass to the limit in Equations  \eqref{Degiorgi:ve}, \eqref{exc:homee}, \eqref{campanato0}-\eqref{campanato} (with $v_\e$ replaced by $v_{\e,k}$) and finally obtain \eqref{inf:hom}-\eqref{fin:oscdec}, thus completing the proof. 
\end{proof}

\section{Boundary regularity for equations in trace form}\label{sec:trace}

 The following section is devoted to the proof of boundary regularity for uniformly elliptic equations in trace form. These auxiliary results will play an important role in establishing boundary regularity for solutions to the Dirichlet problem \eqref{eq:dir1}.
 
To start with, let $ A(x)=(A_{ij}(x))_{i,j=1}^n$ be a matrix satisfying growth and coercivity conditions
    \begin{equation}\label{Ax:again}
        A(x)\,\eta\cdot \eta\geq \l_0\,|\eta|^2,\quad \sum_{i,j=1}^n |A_{ij}(x)|\leq \L_0,\quad \text{for all $x\in B_R^+$},
    \end{equation}
   and for all $\eta\in \R^n$, with given constants $0<\l_0<\L_0$. We also define the operator \begin{equation}\label{operator:L}
      L\equiv\sum_{i,j=1}^nA_{ij}(x)\,D_{ij}\,,\quad\text{so that}\quad  Lu(x)=\mathrm{tr}\big( A(x)D^2u(x)\big)
  \end{equation}

Our goal is to provide a control on the supremum and the oscillation of $u/x_n$ via delicate barrier arguments, based on the work of Krylov \cite{K84}.
For analogous results, we also refer to \cite{L86,L88}, \cite[Theorem 9.31]{GT}, \cite[Lemmas 11.13-11.14]{L13}, and \cite[Lemmas 7.46-7.47]{L96} for the parabolic setting.

  \begin{lemma}\label{lemma:barrier1}
    Let $0<r\leq R\leq 1$, and let $ A, L$ be given by \eqref{Ax:again} and \eqref{operator:L}, respectively. Suppose that $u\in C^2(B_r^+)\cap C^0(\overline
    B_r^+)$ is solution to
    \begin{equation}\label{eq:trace1}
           \begin{cases} \big|Lu(x)\big|\leq K\,x_n^{\alpha-1}\quad &\text{for $x=(x',x_n) \in B_r^+$,}
           \\[1ex]
           u(x',0)=0\quad &\text{for $x=(x',0)\in B_r^0$.}
           \end{cases}
    \end{equation}
    where $K>0$ and $\alpha \in (0,1)$ are given constants. Then there exists $C_0=
    C_0(n,\l_0,\L_0,\alpha)>0$ such that
    \begin{equation}\label{uxn:bdd}
        \sup_{B^+_{r/4}} \left(\dfrac{|u|}{x_n}\right)\leq C_0\,\sup_{B_{r/2}^+} \left(\dfrac{|u|}{r}\right)+C_0\,K\,r^\alpha.
    \end{equation}
\end{lemma}

\begin{proof}
    Fix $x_0\in B^+_{r/4}$ and, for $x=(x',x_n)\in B^+_{r/2}$, we define the barriers
\begin{equation}\label{def:barrier}
\begin{split}
   \mathfrak{b}^{\pm}(x)\equiv \mathfrak{b}_{x_0}^{\pm}(x',x_n)=\pm 16\,\Big(\sup_{B^+_{r/2}}|u|\Big) &\,\bigg\{\, \frac{\L_0}{\l_0}\Big(\frac{x_n}{r}-\frac{x_n^2}{r^2}\Big)+\frac{|x-(x_0',0)|^2}{r^2} \bigg\}
    \\
    &\pm \frac{K}{\l_0\,(1+\alpha)\,\alpha}\big\{r^\alpha x_n-x_n^{1+\alpha}\big\}.
    \end{split}
\end{equation}
    We claim that 
    \begin{equation}\label{barrier:subsol}
        L\mathfrak{b}^+_{x_0}\leq -K\,x_n^{\alpha-1},\quad L\mathfrak{b}^-_{x_0}\geq K\, x_n^{\alpha-1}\quad\text{in $B_{r/2}^+$}
    \end{equation}
and
\begin{equation}\label{barrier:bdr}
    \mathfrak{b}^-_{x_0}\leq u\leq \mathfrak{b}^+_{x_0}\quad\text{on $\partial B_{r/2}^+$.}
\end{equation}
Suppose the claims are true; then by \eqref{eq:trace1} and the maximum principle \cite[Theorem 3.1]{GT} we have $ \mathfrak{b}^-_{x_0}(x)\leq u(x)\leq \mathfrak{b}^+_{x_0}(x)$ for all $x\in B^+_{r/2}$; in particular by evaluating at $x=x_0$  we get 
\begin{equation*}
    \begin{split}-16\,\Big(\sup_{B^+_{r/2}}|u|\Big) \,\frac{\L_0}{\l_0}\,\frac{x_{0n}}{r}- \frac{K\,r^\alpha}{\l_0\,(1+\alpha)\,\alpha}x_{0n}&\leq \mathfrak{b}^-(x_0)
    \\
    &\leq u(x_0)
    \\&\leq \mathfrak{b}^+(x_0)\leq 16\,\Big(\sup_{B^+_{r/2}}|u|\Big) \,\frac{\L_0}{\l_0}\,\frac{x_{0n}}{r}+ \frac{K\,r^\alpha}{\l_0\,(1+\alpha)\,\alpha}x_{0n};
    \end{split}
\end{equation*}
 dividing by $x_{0n}$, and by the arbitrariness of $x_0\in B^+_{r/4}$, this proves \eqref{uxn:bdd}.

We are left to prove \eqref{barrier:subsol}-\eqref{barrier:bdr}. Clearly it suffices to show them only for $\mathfrak{b}^+$. We have 
\begin{equation*}
\begin{split}
    D^2\mathfrak{b}^+(x)=32\,\big(\sup_{B^+_{r/2}}|u|\big)r^{-2}\,\bigg\{&-\Big(\frac{\L_0} {\l_0}\Big)\, e_n\otimes e_n+\,\mathrm{Id} \bigg\}-\frac{K}{\l_0}x_n^{\alpha-1}\,e_n\otimes e_n,
    \end{split}
\end{equation*}
where $\mathrm{Id}$ is the identity matrix. Hence \eqref{barrier:subsol} immediately follows from \eqref{Ax:again} and a simple computation. Then, $\mathfrak{b}^+\geq 0=u$ on $B^0_{r}$, whereas if $x\in \partial B^+_{r/2}\setminus B^0_r$, by Young's inequality we have
\begin{equation*}
    \begin{split}
        \frac{|x-(x'_0,0)|^2}{r^2}&=\frac{|x'|^2-2\,x'\cdot x'_0+|x_0'|^2+x_n^2}{r^2}\geq \frac{(1/2)\,|x|^2-|x'_0|^2}{r^2}\geq \frac{1}{16}\,,
    \end{split}
\end{equation*}
where in the last inequality we used  $|x|=r/2$ and $|x'_0|\leq r/4$. From the inequality above, and since $x_n/r\geq (x_n/r)^2$ and $r^\alpha x_n\geq x_n^{1+\alpha}$, it follows that $\mathfrak{b}^+\geq \sup_{B^+_{r/2}} |u|$ on $\partial B^+_{r/2}\setminus B^0_{r/2} $, and \eqref{barrier:bdr} is proven.
\end{proof}

Next, in order to control the oscillation of the normal derivative, we need two additional lemmas. These are better stated in terms of suitable rectangles, so we introduce the sets
\begin{equation}\label{rectangles}
\begin{split}
    &\mathcal{R}(\sigma,R)=\{x=(x',x_n):\,|x'|<R,\,0<x_n<\sigma\,R \}=Q'_R\times (0,\sigma R)
    \\
    &\widetilde{\mathcal{R}}(\sigma,R)=\{x=(x',x_n):\,|x'|< R,\,\sigma R<x_n<2\sigma\,R \}=Q'_R\times (\sigma R,2\sigma R),
    \end{split}
\end{equation}
where $Q'_{\vrho}$ is the $(n-1)$-dimensional cube of wedge $2\vrho$.

\begin{lemma}\label{lemma:barrier2}
    Let $A,L$ be as in Lemma \ref{lemma:barrier1}. Suppose $u\in C^2(B_R^+)\cap C^0(\overline{B}_R^+)$ satisfies
    \begin{equation}\label{traceeq:2}
        Lu(x)\leq K\,x_n^{\alpha-1},\quad u(x)\geq 0,\quad x=(x',x_n)\in U\subset B^+_R.
    \end{equation}
    with $U$ open in $\R^n$, and for some constants $K>0$ and $\alpha\in(0,1)$. Then, setting
\begin{equation}\label{sigma:0}
    \sigma_0=\frac{\l_0}{2+(2n+4)\L_0},
\end{equation}
 for all $0<r<R$ such that $\mathcal{R}(\sigma_0,4r)\subset U$, we have
    \begin{equation}\label{barries:temp}
        \inf_{\widetilde{\mathcal{R}}(\sigma_0,2r)} \Big(\frac{u}{x_n}\Big)\leq 4\inf_{\mathcal{R}(\sigma_0,r)}\Big(\frac{u}{x_n}\Big)+C_0\,K\,r^{\alpha},
    \end{equation}
    for some constant $C_0=C_0(n,\l_0,\L_0,\alpha)$.
\end{lemma}

\begin{proof}
    We set
\begin{equation*}
    \widetilde{\mathit{m}}=\inf_{\widetilde{\mathcal{R}}(\sigma_0,2r)}\Big(\frac{u}{x_n}\Big),
\end{equation*}
    and introduce the  functions
    \begin{equation*}
        \begin{split}
           &\mathfrak{b}_1(x)=\bigg(1-\frac{x_n}{2\sigma_0r}+\frac{|x'|^2}{r^2} \bigg)\,x_n
           \\
           &\mathfrak{b}_2(x)=\frac{1}{\l_0\,\alpha(\alpha+1)}\big[(2\sigma_0 r)^\alpha-x_n^{\alpha} \big]\,x_n
           \\
           &w(x)=u(x)-\widetilde{\mathit{m}}\,x_n+\frac{\widetilde{\mathit{m}}}{4}\,\mathfrak{b}_1(x)+K\,\mathfrak{b}_2(x)
        \end{split}
    \end{equation*}
Taking into account \eqref{traceeq:2}, \eqref{Ax:again} and \eqref{sigma:0}, a simple computation shows that 
\begin{equation*}
    Lw\leq 0\quad\text{in $\mathcal{R}(\sigma_0,2r)$.}
\end{equation*}
Moreover, using that $u\geq 0$, $\mathfrak{b}_2\geq 0$, $\mathfrak{b}_2=0$ if $\{x_n=2\sigma_0r\}$, and $\mathfrak{b_1}\geq 4x_n$ if $|x'|=2r$,  we deduce that, on $\mathcal{R}(\sigma_0,2r)$, we have
\begin{equation*}
    \begin{split}
   & w=u\geq 0\quad\text{if $x_n=0$ }
   \\
        &w\geq u-\widetilde{m}\,x_n+(\widetilde{m}/4)\,\mathfrak{b}_1\geq 0\quad\text{ if $ |x'|=2r$}
        \\
       & w\geq u-\widetilde{m}x_n\geq 0\quad \text{if $x_n=2\sigma_0 r$}
    \end{split}
\end{equation*}
the last inequality due to the definition of $\widetilde{m}$. The maximum principle \cite[Theorem 3.1]{GT} implies $w\geq 0$ on $\mathcal{R}(\sigma_0,2r)$, hence in particular on $\mathcal{R}(\sigma_0,r)$. Since $\mathfrak{b}_1\leq 3x_n$ on $\mathcal{R}(\sigma_0,r)$, it follows that  
\begin{equation*}
    u(x)\geq \Big[\Big(\frac{\widetilde{m}}{4}\Big)-K\frac{(2\sigma_0\,r)^\alpha}{(\lambda_0)}\Big]\,x_n,\quad x\in \mathcal{R}(\sigma_0,r),
\end{equation*}
and dividing by $x_n$, this proves \eqref{barries:temp}.
\end{proof} 

We now combine \eqref{barries:temp} with the weak Harnack inequality to obtain an oscillation estimate for $u/x_n$.

\begin{lemma}\label{lemma:barrier3}
    Let $A,L$ be as in Lemma \ref{lemma:barrier2}, and suppose that $u\in C^2(B_R^+)\cap C^0(\overline{B}^+_R)$, is such that $u/x_n\in L^\infty(B_R^+)$, and it satisfies
    \begin{equation}
        |Lu(x)|\leq K\,x_n^{\alpha-1}\quad x\in B^+_R.
    \end{equation}
  for some $\alpha\in (0,1)$ and $K>0$. Then there exist constants $C_0>0$ and $\theta_0\in (0,1)$ determined only by $n,\l_0,\L_0,\alpha$ such that
\begin{equation}\label{oscillation}
     \operatorname*{osc}_{B_r^+} \Big(\frac{u}{x_n}\Big)\leq C_0\,\left( \frac{r}{R}\right)^{\theta_0}\bigg\{ \operatorname*{osc}_{B_R^+}\Big(\frac{u}{x_n}\Big)+K\,R^\alpha\bigg\}.
\end{equation}

\begin{proof}
    Let us first assume that $100\,n\,r\leq R$, and for $i\in \N$ we set
    \begin{equation*}
        \mathrm{m}_i=\inf_{\mathcal{R}(\sigma_0,ir)}\Big(\frac{u}{x_n}\Big),\quad \mathrm{M}_i=\sup_{\mathcal{R}(\sigma_0,ir)}\Big(\frac{u}{x_n}\Big),
    \end{equation*}
where $\mathcal{R}(\sigma,r)$ and $\sigma_0$ are defined by \eqref{rectangles} and \eqref{sigma:0}, respectively.

We now apply the weak Harnack inequality \cite[Theorem 9.22]{GT} to the  function $u-\mathrm{m}_4 x_n$ in $\widetilde{\mathcal{R}}(\sigma_0,2r)$, and find $p_0=p_0(n,\l_0,\L_0)>0$ such that \footnote{\cite[Theorem 9.22]{GT} is stated for balls instead of rectangles. Nonetheless, the result 
 extends immediately to rectangles of the form \eqref{rectangles} via the bi-Lipschitz transformation $\widetilde{\mathcal{R}}(\sigma_0,2r)\iff B_{2r}$, whose bi-Lipschitz constants depend only on $n$ and $\sigma_0$.}

\begin{equation*}
    \begin{split}\left(\mint_{\widetilde{\mathcal{R}}(\sigma_0,2r)} (u-\mathrm{m}_4x_n)^{p_0}\,dx\right)^{1/p_0} &\leq C\inf_{\widetilde{\mathcal{R}}(\sigma_0,2r)}(u-\mathrm{m}_4x_n)+C\,K\,r\bigg(\int_{\widetilde{\mathcal{R}}(\sigma_0,2r)}x_n^{n(\alpha-1)}\,dx\bigg)^{1/n}
    \\
    &\leq C'\,r\,\inf_{\widetilde{\mathcal{R}}(\sigma_0,2r)}\Big(\frac{u}{x_n}-\mathrm{m}_4\Big)+C'\,K\,r^{1+\alpha},
    \end{split}
\end{equation*}
with $C,C'=C,C'(n,\l_0,\L_0)$, where in the last inequality we used that $2\sigma_0r<x_n\leq 4\sigma_0r$ in $\widetilde{\mathcal{R}}(\sigma_0,2r)$. As $v=u-\mathrm{m}_4 x_n$ is a nonnegative solution to $|Lv(x)|\leq K\,x_n^{\alpha-1}$ in $\mathcal{R}(\sigma_0,4r)$, we can make use of Lemma \ref{lemma:barrier2} and get
\begin{equation}\label{bla}
    \left(\mint_{\widetilde{\mathcal{R}}(\sigma_0,2r)} (u-\mathrm{m}_4x_n)^{p_0}\,dx\right)^{1/p_0}\leq C\,r\,\big[\mathrm{m}_1-\mathrm{m}_4+K\,r^\alpha \big].
\end{equation}
The same argument applied to $\mathrm{M}_4x_n -u$ then gives
\begin{equation}\label{blaa}
     \left(\mint_{\widetilde{\mathcal{R}}(\sigma_0,2r)} (\mathrm{M}_4x_n-u)^{p_0}\,dx\right)^{1/p_0}\leq C\,r\,\big[\mathrm{M}_4-\mathrm{M}_1+K\,r^\alpha \big].
\end{equation}
By using the elementary inequality 
    \begin{equation*}
        \Big(\int(v+w)^{p_0}dx\Big)^{1/p_0}\leq C(n,p_0)\,\bigg[\Big(\int v^{p_0}dx\Big)^{1/p_0}+\int w^{p_0}dx\Big)^{1/p_0} \bigg]
    \end{equation*}
for all nonnegative functions $v,w$, adding \eqref{bla}-\eqref{blaa} and using once more that $2\sigma_0r<x_n\leq 4\sigma_0r$ in $\widetilde{\mathcal{R}}(\sigma_0,2r)$, we obtain
\begin{equation*}
   \big[\mathrm{M}_4-\mathrm{m}_4\big]\,r\leq C\,r\,\big[\mathrm{m}_1-\mathrm{m}_4+\mathrm{M}_4-\mathrm{M}_1+K\,r^\alpha \big]
\end{equation*}
for some constant $C=C(n,\l_0,\L_0)\geq 1$, from which we deduce
\begin{equation*}
    \operatorname*{osc}_{\mathcal{R}(\sigma_0,r)} \Big( \frac{u}{x_n}\Big)\leq \Big(\frac{C}{1+C}\Big)\, \operatorname*{osc}_{\mathcal{R}(\sigma_0,4r)} \Big( \frac{u}{x_n}\Big)+C\,K\,r^\alpha.
\end{equation*}
An application of Lemma \ref{lemma:osc} thus yields
\begin{equation*}
    \operatorname*{osc}_{\mathcal{R}(\sigma_0,r)} \Big(\frac{u}{x_n}\Big)\leq C\,\left( \frac{r}{R}\right)^{\theta_0}\bigg\{ \operatorname*{osc}_{\mathcal{R}\big(\sigma_0,\,R/(100n)\big)}\Big(\frac{u}{x_n}\Big)+K\,R^\alpha\bigg\},
\end{equation*}
and since $B^+_{\sigma_0\vrho}\subset \mathcal{R}(\sigma_0,\vrho)\subset B^+_{\sqrt{n}\vrho} $ for every $\vrho>0$, this implies  \eqref{oscillation} in the case $0<r\leq R/(100 n) $. However, \eqref{oscillation} is valid also for $R/(100n)\leq r\leq R$, as 
\begin{equation*}
     \operatorname*{osc}_{B^+_r}\Big( \frac{u}{x_n}\Big)\leq  \operatorname*{osc}_{B^+_R}\Big( \frac{u}{x_n}\Big)\leq (100n)^{\theta_0}\,\left(\frac{r}{R} \right)^{\theta_0}\operatorname*{osc}_{B^+_R}\Big( \frac{u}{x_n}\Big).
\end{equation*}
The proof is thus complete.
\end{proof}

\begin{remark}\label{rem:dernormale}
    \rm{Let $u$ be as in Lemma \ref{lemma:barrier3}. Then thanks to \eqref{oscillation}, the normal derivative $D_n u(x_0')$ exists at all points $x'_0\in B^0_R$. Indeed,  let $\vrho>0$ be such that $B_\vrho^+(x'_0)\subset B_R
    ^+$.
Take a sequence of points $\{x^{(k)}=(x_0',x_n^{(k)})\}_{k\in \N}$ such that $x^{(k)}_n\xrightarrow{k\to\infty} 0$; by \eqref{oscillation} we have
\begin{equation}\label{oscillation1}
    \Bigg|\frac{u(x^{(k)})}{x^{(k)}_n}-\frac{u(x^{(j)})}{x_n^{(j)}} \Bigg|\leq C\,\left( \frac{r}{\vrho}\right)^{\theta_0}\bigg\{ \operatorname*{osc}_{B_\vrho^+(x_0')}\Big(\frac{u}{x_n}\Big)+K\,\vrho^\alpha\bigg\}
\end{equation}
for every $0<r\leq \vrho$ such that $|x_n^{(k)}-x_n^{(j)}|<r$.
This implies that the sequence $\{u(x^{(k)})/x_n^{(k)}\}_{k\in\N}$ is Cauchy so, up to a subsequence, it will converge to a number $L_0$. Since \eqref{oscillation1} is valid for every sequence $x_n^{(k)}\to 0$, this implies that $L_0=\lim_{k\to \infty} u(x_0',x_n^{(k)})/x_n^{(k)}$ is uniquely determined, and thus $L_0=D_n u(x'_0)$. Moreover,  by \eqref{oscillation}, we have
    \begin{equation*}
        \Bigg|\frac{u(x^{(k)})}{x^{(k)}_n}-\frac{u(y)}{y_n} \Bigg|\leq C\,\left( \frac{r}{\vrho}\right)^{\theta_0}\bigg\{ \operatorname*{osc}_{B_\vrho^+(x_0')}\Big(\frac{u}{x_n}\Big)+K\,\vrho^\alpha\bigg\}
    \end{equation*}
    and by taking the limit as $k\to \infty$ above, we obtain 
    \begin{equation}\label{control:Dnu}
        \Big|D_nu(x'_0)-\frac{u(y)}{y_n} \Big|\leq C\,\left( \frac{r}{\vrho}\right)^{\theta_0}\bigg\{ \operatorname*{osc}_{B_\vrho^+(x_0')}\Big(\frac{u}{x_n}\Big)+K\,\vrho^\alpha\bigg\},
    \end{equation}
    for all $x'_0\in B_R^0$, $B_{\vrho}^+(x'_0)\subset B_R^+$, and for all $y\in B_r^+(x'_0)$, $0<r\leq \vrho$.
    }
\end{remark}
 
\end{lemma}

\section{Boundary gradient regularity: homogeneous Dirichlet problems}\label{sec:dirhomog}

This section is devoted to proving boundary regularity for the gradient of solutions to the  homogeneous Dirichlet problem, with estimates valid up to the flat boundary portion of the upper half ball. Precisely, we consider $v\in W^{1,B}(B_{3R}^+)$ a local weak solution to
\begin{equation}\label{eq:dirhomogen}
    \begin{cases}
        -\mathrm{div}\big(\A(Dv) \big)=0\quad & \text{in $B^+_{2R}$}
        \\
        v=g\quad &\text{on $B^0_{2R}$.}
    \end{cases}
\end{equation}
with $g\in C^{1,\alpha}(\R^{n-1})$ compactly supported.
Observe that, by the interior regularity established in Section~\ref{sec:int0} and the control of tangential derivatives on the flat boundary portion $B^0_{2R}$ given by the Dirichlet datum, it remains to establish regularity for the normal derivative $D_n v$ up to $B^0_{2R}$. But this is exactly provided by the results of the preceding Section \ref{sec:trace}, via the control of the supremum and oscillation of $(v-g)/x_n$.

The main result of this section is the following.
\begin{theorem}\label{thm:gradddir}
    Suppose that the stress field $\A\in C^0(\R^n)\cap C^1(\rn\setminus\{0\})$ fulfills \eqref{ass:Aaut}, and let $g\in C^{1,\alpha}(\R^{n-1})$ be compactly supported.
    
    Let $v\in W^{1,B}(B_{3R}^+)$ be a weak solution to the Dirichlet problem \eqref{eq:dirhomogen}. Then there exist  constant $C_{\mathrm{h}}=C_{\mathrm{h}}(n,\l,\L,i_a,s_a,\alpha)>0$ and $\beta_{\mathrm{h}}=\beta_{\mathrm{h}}(n,\l,\L,i_a,s_a,\alpha)\in (0,1)$ such that 
    \[v\in C^{1,\beta_{\mathrm{h}}}(\overline{B}^+_{R/2}),
    \] and we have
\begin{equation}\label{dir:gradbdd}
        \sup_{B^+_{R/2}} |Dv|\leq C_{\mathrm{h}}\,\mint_{B^+_R} |Dv|\,dx+C_{\mathrm{h}}\,\|g\|_{C^{1,\alpha}}\,,
    \end{equation}
    and, for all $0<r<R/2$, there holds
    \begin{equation}\label{dir:gradholder}
        \operatorname*{osc}_{B_r^+}\, Dv\leq C_{\mathrm{h}}\,\left(\frac{r}{R} \right)^{\beta_{\mathrm{h}}}\,\left(\mint_{B_R^+} |Dv|\,dx+\|g\|_{C^{1,\alpha}} \right).
    \end{equation}
\end{theorem}

As problem \eqref{eq:dirhomogen} is left unchanged by additive constants, we will also assume that $\A(0)=0$. As in the previous sections, we proceed via approximation, and consider $\ve\in W^{1,2}(B_{2R}^+)$ solution to
\begin{equation}\label{eqve:dir}
        \begin{cases}
        -\mathrm{div}\big(\Ae(D\ve) \big)=0\quad & \text{in $B^+_{2R}$}
        \\
        \ve=g\quad &\text{on $B^0_{2R}$,}
    \end{cases}
\end{equation}
where $\Ae$ is given by Lemma \ref{lemma:Ae}. 

\begin{proposition}
    Suppose $\ve \in W^{1,2}(B_{2R}^+)$ is solution to \eqref{eqve:dir}. Then
    \begin{equation}\label{ve:reg1}
        \ve\in C^\infty(B_{2R}^+)\cap C^{0}\big(\overline{B}_r^+\big)\quad\text{for all $0<r<2R$.}
    \end{equation}
\end{proposition}

\begin{proof}
    The interior regularity of $\ve$ is the content of Proposition \ref{prop:veregular}. Regarding the boundary regularity, 
    thanks to \eqref{coAe:gr} and \eqref{bBquadratic}, we can use the boundary Harnack inequality \cite[Theorems 8.25-8.29]{GT} and deduce that $\ve\in C^{0,\alpha_\e}(\overline{B}_r^+)$ for every $0<r<2R$, for some $\alpha_\e \in (0,1)$. 
\end{proof}

We shall also need to extend $g$ to the upper half space in a smooth way. To this end, we exploit the following extension lemma.

\begin{lemma}\label{lemma:extension}
    Let $g\in C^{1,\alpha}(\R^{n-1})$ be compactly supported. Then there exists an extension function $G\in C^{\infty}(\R^n_+)$ such that $  G(x',0)=g(x')$,
    \begin{equation}\label{G:ext}
    \begin{split}
        \|G\|_{C^{1,\alpha}(\R^n_+)} \leq C(n,\alpha)\,\|g\|_{C^{1,\alpha}(\R^{n-1})}
 \quad\text{and}
        \quad
        |D^2 G(x',x_n)|\leq C(n,\alpha)\,\|g\|_{C^{1,\alpha}(\R^{n-1})}\,x_n^{\alpha-1}\,,
        \end{split}
    \end{equation}
    for all $x'\in \R^{n-1}$, and all $x_n>0$.
\end{lemma}

The proof of Lemma~\ref{lemma:extension} can be found in \cite[Lemma~2.3]{GH80} or \cite[Lemma~5.1]{AC25}. We simply remark that both proofs rely on techniques arising in the study of fractional Sobolev spaces. In \cite{GH80}, the extension function is defined via convolution, namely $
G(x',x_n) = g \ast \varrho_{x_n}(x')$,
a method employed to establish properties of fractional Sobolev spaces-see \cite[Section 1.4]{grisvard}. In contrast, the proof in \cite{AC25} is based on the Caffarelli-Silvestre extension \cite{CS07}.

\begin{remark*}
\rm{The assumption that $g$ is compactly defined on the whole space $\mathbb{R}^{n-1}$ is not restrictive. Indeed, if $g$ is defined only on $B'_{4R_0}$, for some $R_0>0$, it suffices to multiply it with a cut-off function $\eta\in C^\infty_c(B'_{4R_{0}})$ such that $\eta\equiv 1$ on $B'_{2R_0}$, $|D\eta|\leq C(n)/R_0$ and $|D^2\eta|\leq C(n)/R_0^2$. Thus we obtain a new function $\tilde g\in C^{1,\alpha}(\R^{n-1})$ compactly supported, which is equal to $g$ in $B'_{2R_0}$, and such that $\|\tilde g\|_{C^{1,\alpha}(\R^{n-1})}\leq C(n,\alpha)/R_0^2\,\|g\|_{C^{1,\alpha}(B'_{4R_0})}$. }
\end{remark*}

It is convenient to introduce the auxiliary function 
\begin{equation}\label{we}
\we=\ve-G,    
\end{equation}
 where $G$ is the extension function given by Lemma \ref{lemma:extension}.
In particular, by \eqref{ve:reg1} and the regularity of $G$ , we have
\[
\we\in C^\infty(B_{2R}^+)\cap C^0(\overline{B}_r^+)\quad\text{for all $0<r<2R$,}
\]
 and, owing to \eqref{eqve:dir}, it is a weak solution to
\begin{equation}\label{we:sol}
    \begin{cases}
         -\mathrm{div}\big(\tilde\Ae(x,D\we) \big)=0\quad & \text{in $B^+_{2R}$}
        \\
        \we=0\quad &\text{on $B^0_{2R}$,}
    \end{cases}
\end{equation}
where $\tilde{\Ae}(x,\xi)=\Ae(\xi+DG(x))$. We claim that
\begin{equation}\label{again:coer}
    \begin{split}
        &\tilde{\Ae}(x,\xi)\cdot \xi \geq c\,B_\e(|\xi|)-C\,B_\e(\|g\|_{C^{1,\alpha}})
        \\
        &\big|\tilde{\Ae}(x,\xi)\big|\leq C\, b_\e\big(|\xi| \big)+C\,b_\e\big(\|g\|_{C^{1,\alpha}} \big).
    \end{split}
\end{equation}
for some constants $c,C=c,C(n,\l,\L,i_a,s_a,\alpha)$. Indeed, by \eqref{coAe:gr}, Young's inequality \eqref{young1} and \eqref{triangle} (and recalling Remark \ref{remark:importante}), we deduce
\begin{equation*}
\begin{split}
    \tilde{\Ae}(x,\xi)\cdot \xi &=\A_\e(\xi+DG)\cdot \xi\geq c\,B_\e\big(|\xi+DG| \big)- \Ae(\xi+DG)\cdot DG
    \\
    &\geq c\,B_\e\big(|\xi+DG| \big)-C\,b_\e\big(|\xi+DG| \big)\,|DG|
    \\
    &\geq  c'\,B_\e\big(|\xi+DG| \big)-C'B_\e\big(|DG|\big)
    \\
    &\geq c'\,B_\e\big(|\xi|\big)-C''B_\e\big(|DG|\big),
    \end{split}
\end{equation*}
where the constants $c,c',C,C',C''>0$ depend only on $n,\l,\L,i_a,s_a$ thanks to Remark \ref{remark:importante}. Taking advantage of \eqref{G:ext} the monotonicity of $B_\e$ and \eqref{B2t}, the first inequality of \eqref{again:coer} follows. The second inequality is simpler, as by \eqref{coAe:gr} and \eqref{triangleb} (coupled with Remark \ref{remark:importante}), we get
\begin{equation*}
    \begin{split}
        |\tilde{\Ae}(x,\xi)|=|\Ae(\xi+DG)|\leq C\,b_\e(|\xi+DG|)\leq C'\,b_\e(|\xi|)+C'b_\e(|DG|)\,,
    \end{split}
\end{equation*}
where $C,C'=C,C'(n,\l,\L,i_a,s_a)>0$. The monotonicity of $b_\e$, \eqref{G:ext} and \eqref{b2t} finally prove \eqref{again:coer}.

Now observe that, by using the chain rule in \eqref{we:sol}, on $B_{2R}^+$ we find
\begin{equation*}
    0=\mathrm{div}\big(\Ae(D\we+DG) \big)=\mathrm{tr}\big(\nabla_\xi \Ae(D\we+DG)\,D^2(\we+G) \big)=\mathrm{tr}\big(\nabla_\xi \Ae(D\ve)\,D^2(\we+G) \big).
\end{equation*}
Dividing the above identity by $a_\e(|D\ve|)$ (which never vanishes by \eqref{ae:nonzero}), setting 
\begin{equation*}
    A_\e(x)=\frac{\nabla_\xi\Ae(D\ve)}{a_\e\big(|D\ve| \big)}\,,
\end{equation*}
we thus have that $w_\e$ solves 
\begin{equation}\label{temp:trace}
    \mathrm{tr}\big(A_\e(x)D^2\we \big)=-\mathrm{tr}\big(A_\e(x)D^2G\big)\quad \text{in $B_{2R}^+.$}
\end{equation}
Now observe that from  \eqref{coer:Ae}, we have
\begin{equation}\label{matrix:Ae}
    A_\e(x)\eta\cdot \eta\geq c\,|\eta|^2\quad \sum_{i,j=1}^n |(A_\e)_{ij}(x)|\leq C\quad \text{for all $\eta\in \rn$, $x\in B^+_{2R}$.}
\end{equation}
with $c,C=c,C(n,\l,\L,i_a,s_a)>0$, and making use of this information and \eqref{G:ext} in \eqref{temp:trace}, we find that $\we=\ve-G$ solves
\begin{equation}\label{we:trace}
    \begin{cases}\big|\mathrm{tr}\big(A_\e(x)D^2\we \big) \big|\leq C\,\|g\|_{C^{1,\alpha}}\,x_n^{\alpha-1}\quad &\text{in $B_{2R}^+$,}
    \\
    \we=0\quad &\text{on $B_{2R}^0$,}
    \end{cases}
\end{equation}
with $C=C(n,\l,\L,i_a,s_a,\alpha)$.

We are now ready to prove the boundedness of $D\ve$.

\begin{proposition}\label{prop:dveDir}
    Let $\ve\in W^{1,2}(B_{2R}^+)$ be a local weak solution to \eqref{eqve:dir}. Then there exist constants $C,C'=C,C'(n,\l,\L,i_a,s_a,\alpha)>0$ such that
    \begin{equation}\label{Dve:dirichlet}
    \begin{split}
        \sup_{B^+_{R/2}} |D\ve|&\leq C\,\mint_{B_R^+}   \bigg|\frac{\ve-G}{R} \bigg|\,dx+C\,\|g\|_{C^{1,\alpha}} 
        \\
       &\leq C'\,\mint_{B_R^+}|D\ve   |\,dx+C'\,\|g\|_{C^{1,\alpha}}, 
        \end{split}
        \end{equation}
and
\begin{equation}\label{vdivisoxn}
\begin{split}
    \sup_{B_{R/2}^+} \left(\frac{|\ve-G|}{x_n} \right) &\leq C\mint_{B^+_{R}}\left(\frac{|\ve-G|}{R}\right)\,dx+C\,\|g\|_{C^{1,\alpha}}
    \\
    &\leq C'\mint_{B_R^+}|D\ve|\,dx+C'\|g\|_{C^{1,\alpha}}.
    \end{split}
\end{equation}
\end{proposition}

\begin{proof}
    Let $\we$ be given by \eqref{we}, and let $0<r\leq 2R$. Owing to \eqref{we:sol} and \eqref{again:coer} (while taking into account \eqref{iaesae} and Remark \ref{remark:importante}), we may use Theorem \ref{thm:bdddir} and find
\begin{equation}\label{wetemp:control}
\begin{split}
    &\sup_{B^+_{r/2}(x_0')}\left(\frac{|\we|}{r}\right)\leq C\,\mint_{B^+_{r}(x_0')}\left(\frac{|\we|}{r}\right)\,dx+C\,\|g\|_{C^{1,\alpha}},
    \end{split}
\end{equation}
    for all $x_0'\in B^0_{2R}$ such that $B_{r}^+(x_0')\subset B_{2R}^+$, with $C=C(n,\l,\L,i_a,s_a)$. 
    Next, since $\we$ is solution to \eqref{we:trace}, we may use Lemma \ref{lemma:barrier1} and then \eqref{wetemp:control}, to obtain
    \begin{equation}\label{wetemp:contr2}
    \begin{split}
        \sup_{B^+_{r/4}(x_0')} \left(\dfrac{|\we|}{x_n}\right) &\leq C\,\sup_{B_{r/2}^+(x_0')} \left(\dfrac{|\we|}{r}\right)+C\,\|g\|_{C^{1,\alpha}}\,r^\alpha
        \\
        &\leq C'\mint_{B^+_{r}(x_0')}\left(\frac{|\we|}{r}\right)\,dx+C'\,\|g\|_{C^{1,\alpha}}.
        \end{split}
    \end{equation} 
    By the arbitrariness of $x_0'\in B^0_{2R}$ such that $B_{r}(x_0')\subset B_{2R}^+$, and recalling the definition of $\we$ in \eqref{we}, from \eqref{wetemp:contr2} and a standard covering argument we obtain the first inequality of \eqref{vdivisoxn}. 
Then, the second inequality in \eqref{vdivisoxn} follows via Poincar\'e inequality \eqref{half:poincare} and \eqref{G:ext}. 



Now let $0<r\leq 2R$, and let $x_0\in B_{r/16}^+$, $x_0=(x_0',d)$, $x_0'\in B^0_{r/16}$, $x_{0n}\eqqcolon d\leq r/16$.
From the $L^\infty$-bound \eqref{inf:hom}, Jensen inequality, the definition of $w_\e$ in \eqref{we}, \eqref{triangle}, \eqref{G:ext}, \eqref{B2t}, the interior Caccioppoli inequality \eqref{caccioppoli} applied to $u=\we$ (which is valid thanks to \eqref{we:sol}-\eqref{again:coer}), and by the monotonicity of $B_\e$, we obtain
\begin{equation*}
    \begin{split}
        B_\e\big(|D\ve(x_0)|\big)&\leq C\,\mint_{B_{d/4}(x_0)}B_\e\big(|D\ve|\big)\,dx\leq C'\,\mint_{B_{d/4}(x_0)}B_\e\big(|D\we|\big)\,dx+C'\,B_\e(\|DG\|_{L^\infty})
        \\
        &\leq C''\mint_{B_{d/2}(x_0)}B_\e\left(\frac{|\we|}{d} \right)+C''\,B_\e(\|g\|_{C^{1,\alpha}})
        \\
        &\leq C''\,\sup_{B_{d/2}(x_0)}B_\e\left(\frac{|\we|}{d} \right)+C''\,B_\e(\|g\|_{C^{1,\alpha}})\leq 2\,C''\,B_\e\bigg(\sup_{B_{d/2}(x_0)}\frac{|\we|}{d}+\|g\|_{C^{1,\alpha}} \bigg),
    \end{split}
\end{equation*}
with $C,C',C''=C,C',C''(n,\l,\L,i_a,s_a,\alpha)$ independent of $\e$ thanks to Remark \ref{remark:importante}. Taking $B_\e^{-1}$ to both sides of the above identity, and using \eqref{B2t}(having Remark \ref{remark:importante} in mind), we get
\begin{equation}\label{uff}
    |D\ve(x_0)|\leq C\,\sup_{B_{d/2}(x_0)}\left(\frac{|\we|}{d} \right)+C\,\|g\|_{C^{1,\alpha}}
\end{equation}
with $C=C(n,\l,\L,i_a,s_a,\alpha)$.

Now observe that, $d=x_{0n}\geq x_n/2$ for all $x
=(x',x_n)\in B_{d/2}(x_0)$, and that 
\[
B_d(x_0)\subset B_{2d}(x_0')\subset B_{r/8}^+(x_0')\subset B_{r/4}^+.
\]
 Using these pieces of information with \eqref{uff} and \eqref{wetemp:contr2}, we infer
\begin{equation*}
    \begin{split}
        |D\ve(x_0)| &\leq C\sup_{B_{d/2}(x_0)}\left(\frac{|\we|}{x_n}\right)+C\,\|g\|_{C^{1,\alpha}}\leq C\, \sup_{B^+_{r/4}}\left(\frac{|\we|}{x_n}\right)+C\,\|g\|_{C^{1,\alpha}}
        \\
        &\leq C'\,\mint_{B_{r}^+}\left(\frac{|\we|}{r}\right)\,dx+C'\,\|g\|_{C^{1,\alpha}}
    \end{split}
\end{equation*}
for all $x_0\in B_{r/16^+}$, with $C,C'=C,C'(n,\l,\L,i_a,s_a,\alpha)$. Therefore, recalling \eqref{we}, we have found
\begin{equation*}
    \sup_{B_{r/16}^+}|D\ve|\leq  C\,\mint_{B_r^+}\left(\frac{|\ve-G|}{r}\right)\,dx+C\,\|g\|_{C^{1,\alpha}}
\end{equation*}
for all $0<r<2R$.  The first inequality in \eqref{Dve:dirichlet} then follows  via a standard covering argument, while the second one is finally obtained using Poincar\'e inequality \eqref{half:poincare} and \eqref{G:ext}.
\end{proof}

We remark that a similar estimate to \eqref{Dve:dirichlet} was obtained in \cite[Theorem 2.2]{CM16} in the case $B(t)=t^p+a_0\,t^q$, $1<p\leq q$, for a fixed constant $a_0\geq 0$, and with zero boundary datum $g\equiv 0$.
\vspace{0.1cm}

Next, we provide a quantitative control on the oscillation of $D\ve$.

\begin{proposition}\label{prop:oscve}
    Let $\ve\in W^{1,2}(B_{2R}^+)$ be a weak solution to \eqref{eqve:dir}. Then there exists  $\beta_{\mathrm{h}}=\beta_{\mathrm{h}}\in (0,1)$ depending on $n,\l,\L,i_a,s_a,\alpha$ such that $\ve\in C^{1,\beta_{\mathrm{h}}}(\overline{B}^+_{R/2})$, and
    \begin{equation}\label{dir:oscillation}
        \operatorname*{osc}_{B^+_r} D\ve\leq C\,\left( \frac{r}{R}\right)^{\beta_{\mathrm{h}}}\bigg\{\mint_{B_R^+} |D\ve|\,dx+\|g\|_{C^{1,\alpha}(\R^{n-1})}\bigg\}
    \end{equation}
    for all $0<r<R/2$, with $C=C(n,\l,\L,i_a,s_a,\alpha)>0$.
\end{proposition}

Before proving Proposition \ref{prop:oscve}, we need some preliminary results. 

\begin{remark}\label{remark}
\rm{Let $\we$ be given by \eqref{we}. Owing to equations \eqref{we:trace} and \eqref{matrix:Ae}, we can use Lemma~\ref{lemma:barrier3} with $K=\|g\|_{C^{1,\alpha}}$, and in particular Remark~\ref{rem:dernormale}, so that the normal derivative $D_n \we(x_0')$ exists at every point 
$x_0'\in B_{2R}^0$, and from \eqref{control:Dnu} we have the bound 
\begin{equation}\label{bound:normal}
\begin{split}
    \bigg|D_n\we(x'_0)-\frac{\we(x)}{x_n} \bigg|&\leq C\,\left(\frac{r}{R} \right)^{\theta_0}\bigg\{\sup_{B_{R/2}^+}\left(\frac{|\we|}{x_n}\right)+\|g\|_{C^{1,\alpha}}R^\alpha \bigg\},
    \\
    &\leq C'\,\left(\frac{r}{R} \right)^{\theta_0}\bigg\{\mint_{B_R^+}|D\ve|\,dx+\|g\|_{C^{1,\alpha}}\bigg\}
    \end{split}
\end{equation}
for all $x_0'\in B_r^0$, for all $x=(x',x_n)\in B^+_{2r}(x_0') $, whenever $0<r\leq R/8$, where in the second inequality we used \eqref{vdivisoxn} and \eqref{G:ext}. Here $C,C'$ depend on $n,\l,\L,i_a,s_a,\alpha$, and $\theta_0\in (0,1)$ depends on $n,\l,\L,i_a,s_a,\alpha$. As $\we\equiv 0$ on $B^0_{2R}$, and it is smooth in $B^+_{2R}$, we thus have that $D\we$ exists at all points $\overline{B}_R^+$, and so does $D\ve=D\we+DG$.
\vspace{0.2cm}

We just remark that, owing to Equation \eqref{eqve:dir} and the property \eqref{nonzero:Ae}, we could also have appealed to \cite{GG84} and deduce that $\ve\in C^{1,\alpha_\e}(\overline{B}_R^+)$ for some $\alpha_\e\in (0,1)$ depending on $\e>0$ as well. 
}
\end{remark}

We shall also need the following interpolation lemma, which will be used to connect the interior gradient regularity \eqref{fin:oscdec}, and the pointwise oscillation estimate \eqref{bound:normal}.

\begin{lemma}\label{lemma:interpol}
    Let $w\in C^0(\overline{\mathcal{B}})\cap C^1(\mathcal{B})$ for some ball $\mathcal{B}=B_{R_0}(x_0)$, and assume that
    \begin{equation}\label{hp:interpol}
        \operatorname*{osc}_{B_r(x_1)} Dw\leq c_1\,\left( \frac{r}{\vrho}\right)^{\alpha}\bigg\{  \operatorname*{osc}_{B_\vrho(x_1)} Dw+K\,\vrho^\alpha\bigg\},
    \end{equation}
   whenever $0<r\leq \vrho$, and $B_\vrho(x_1)\subset \mathcal{B}$, for some constants $c_1,K>0$ and $\alpha\in (0,1)$. Then for any $L\in \rn$ and $U\in \R$, we have
   \begin{equation}\label{thesis:interpol}
       \sup_{B_{R_0/2}(x_0)}|Dw-L|\leq C(c_1,n,\alpha)\bigg\{ R_0^{-1}\sup_{\mathcal{B}}|w-L\cdot x-U|+K\,{R}_0^\alpha\bigg\}.
   \end{equation}
\end{lemma}

We postpone the proof of Lemma \ref{lemma:interpol} to Appendix A, and we also refer to \cite[Lemma 12.4]{L96} for the same lemma in the parabolic setting. We now have all the ingredients  to prove the oscillation estimate \eqref{dir:oscillation}.

\begin{proof}[Proof of Proposition \ref{prop:oscve}]
We first observe that, by \eqref{we} and \eqref{G:ext}, we have 
\[
\begin{split}
 &\operatorname*{osc}_{B_r(x_0)}D\we\leq  \operatorname*{osc}_{B_r(x_0)}D\ve+ \operatorname*{osc}_{B_r(x_0)}DG \leq  \operatorname*{osc}_{B_r(x_0)}D\ve+ C(n,\alpha)\,\|g\|_{C^{1,\alpha}}\,r^\alpha
 \\
& \operatorname*{osc}_{B_r(x_0)}D\ve\leq  \operatorname*{osc}_{B_r(x_0)}D\we+ \operatorname*{osc}_{B_r(x_0)}DG \leq  \operatorname*{osc}_{B_r(x_0)}D\we+ C(n,\alpha)\,\|g\|_{C^{1,\alpha}}\,r^\alpha,
 \end{split}
\]
which together with the interior oscillation estimate \eqref{fin:oscdec}, yields
\begin{equation}\label{osc:Dwe}
    \operatorname*{osc}_{B_r(x_0)}D\we\leq C\left(\frac{r}{\vrho} \right)^{\tilde{\mathrm{\alpha}}_h}\bigg\{\operatorname*{osc}_{B_\vrho(x_0)}D\we +\|g\|_{C^{1,\alpha}}\,\vrho^{\tilde{\mathrm{\alpha}}_h}\bigg\},
\end{equation}
    for all $x_0$ such that $B_{2\vrho}(x_0)\subset B^+_{2R}$, and all $0<r\leq \vrho$. Here we set ${\tilde{\mathrm{\alpha}}_h}=\min\{\alpha_{\mathrm h},\alpha\}$ and we used that $0<\vrho\leq 1$.

Now let $0<r\leq R/8$, and we fix a point $x_0=(x'_0,d)\in B^+_r$.
  Also observe that 
  \begin{equation}\label{DweDnwe}
      D\we(x_0')=D_n\we(x_0')e_n\quad\text{ on $B_{2R}^0$}
  \end{equation}
 as $\we\equiv 0$ on $B_{2R}^0$. Therefore, since  $B_{d/4}(x_0)\subset B_{2r}^+(x_0')$, owing to \eqref{bound:normal}, and using that $x_n\leq 2d$, for all $x\in B_{d/4}(x_0)$ we get
\begin{equation}\label{uffi}
\begin{split}
    |\we(x)-D\we(x_0')\cdot x|&=|\we(x)-D_n\we(x_0')\cdot x_n|=x_n\bigg|\frac{\we(x)}{x_n}-D_n\we(x_0')\bigg|
    \\
    &\leq C\,d\,\left(\frac{r}{R} \right)^{\theta_0}\bigg\{\mint_{B_R^+}|D\ve|\,dx+\|g\|_{C^{1,\alpha}}\bigg\},
    \end{split}
\end{equation}
Next, thanks to \eqref{osc:Dwe}, we may use Lemma \ref{lemma:interpol} with $L=D\we(x_0')$, $U=0$, $R_0=d/4$, and coupling the resulting equation with \eqref{uffi}, we infer
\begin{equation}
    |D\we(x_0)-D\we(x_0')|\leq C\left( \frac{r}{R}\right)^{\beta_{\mathrm{h}}}\bigg\{\mint_{B_R^+}|D\ve|\,dx+\|g\|_{C^{1,\alpha}}\bigg\},
\end{equation}
 where $C>0$, and we set $\beta_{\mathrm{h}}=\min\{\tilde{\alpha}_h,\theta_0\}\in (0,1)$, both depending on $n,\l,\L,i_a,s_a,\alpha$.
Moreover, using \eqref{DweDnwe},  taking $x=(0',x_n)$ in \eqref{bound:normal} and letting $x_n\to 0$ , we find
\begin{equation*}
    |D\we(x_0')-D\we(0)|=|D_n\we(x_0')-D_n\we(0)|\leq C\,\left(\frac{r}{R} \right)^{\theta_0}\bigg\{\mint_{B_R^+}|D\ve|\,dx+\|g\|_{C^{1,\alpha}}\bigg\},
\end{equation*}
The two identities above are valid for all $x_0\in B^+_r$, $0<r\leq R/8$, so we get
\begin{equation*}
    \operatorname*{osc}_{B_r^+} D\we\leq C(n)\,\sup_{x_0\in B_r^+}|D\we(x_0)-D\we(0)|\leq C\left( \frac{r}{R}\right)^{\beta_{\mathrm{h}}}\bigg\{\mint_{B_R^+}|D\ve|\,dx+\|g\|_{C^{1,\alpha}}\bigg\}, 
\end{equation*}
with $C=C(n,\l,\L,i_a,s_a,\alpha)$; coupling the above inequality with 
\[
\operatorname*{osc}_{B^+_r}D\ve\leq  \operatorname*{osc}_{B^+_r}D\we+ \operatorname*{osc}_{B^+_r}DG \leq  \operatorname*{osc}_{B_r^+}D\we+ C(n,\alpha)\,\|g\|_{C^{1,\alpha}}\,r^\alpha
\]
which stem from \eqref{we} and \eqref{G:ext}, we finally obtain \eqref{dir:oscillation} in the case $0<r\leq R/8$. 
On the other hand, if $R/8\leq r\leq R/2$, \eqref{dir:oscillation} is still valid, since by \eqref{Dve:dirichlet} we have
\begin{equation*}
    \begin{split}
        \operatorname*{osc}_{B_r^+} D\ve\leq \operatorname*{osc}_{B_{R/2}^+} D\ve &\leq C(n)\,\sup_{B_{R/2}^+} |D\ve|\leq C\,\mint_{B_R^+}|D\ve|\,dx+C\,\|g\|_{C^{1,\alpha}}
        \\
        &\leq C\,8^{\beta_{\mathrm{h}}}\left(\frac{r}{R} \right)^{\beta_{\mathrm{h}}}\bigg\{\mint_{B_R^+}|D\ve|\,dx+C\,\|g\|_{C^{1,\alpha}}\bigg\}.
    \end{split}
\end{equation*}
This completes the proof.
\end{proof}

By means of an approximation argument, similar to that of Theorem \ref{thm:inthom}, we now prove the main result of this section.

\begin{proof}[Proof of Theorem \ref{thm:gradddir}]
    Let $v\in W^{1,B}(B^+_{3R})$ be as in the statement. Since $v=g=G$ on $B^0_{3R}$, we may consider the odd reflection
    \begin{equation*}
        (v-G)^o(x',x_n)=\begin{cases}
            (v-G)(x',x_n),\quad &x_n>0
            \\
             -(v-G)(x',x_n),\quad &x_n\leq 0
        \end{cases}
    \end{equation*}
    which clearly belongs to $W^{1,B}(B_{3R})$. Now we claim that its regularization $(v-G)^o\ast \rho_{1/k}$ vanishes on $B^0_{3R-2/k}$. Indeed, for $(x'_0,0)=x'_0\in B^0_{3R-2/k}$, and denoting by  $B^-_r(x'_0)$ the lower half ball centered at $x'_0$ of radius $r>0$, we have
    \begin{equation*}
        \begin{split}
            (v-G)^o\ast \rho_{1/k}(x'_0)=&\int_{B_{1/k}(x'_0,0)}(v-G)^o(y)\,\rho_{1/k}(y-(x'_0,0))\,dy
            \\
            =&\int_{B^+_{1/k}(x'_0)}(v-G)^o(y)\,\rho_{1/k}(y-(x'_0))\,dy
            \\
            &+\int_{B^-_{1/k}(x'_0)}(v-G)^o(y)\,\rho_{1/k}(y-(x'_0,0))\,dy=0,
        \end{split}
    \end{equation*}
  the last equality stemming from the symmetries of $(v-G)^o$ and of the mollifier $\rho$. Now let
\begin{equation}
    v_k^{bd}\coloneqq(v-G)^o\ast \rho_{1/k}+G,
\end{equation}
    so that $v_k^{bd}\in C^{1,\alpha}(\overline{B^+_{2R}})$, $v_k^{bd}=G=g$ on $B_{2R}^0$ and by the properties of convolution \cite[Theorem 4.4.7]{HHbook}, \eqref{unif:refl} and the properties of the trace operator
\begin{equation}\label{uffffii}
    v_k^{bd}\xrightarrow{k\to\infty} v\quad\text{in $W^{1,B}(B^+_{2R})$, in the sense of traces and in $L^1(\partial B^+_{2R})$.}
\end{equation}
    Let $v_{\e,k}\in W^{1,2}(B_{2R}^+)$ be the unique solution to
    \begin{equation}\label{temo:weakdir}
        \begin{cases}
            -\mathrm{div}\big( \A_\e(D v_{\e,k})\big)=0\quad&\text{in $B_{2R}^+$}
            \\
            v_{\e,k}=v_k^{bd}&\text{on $\partial B_{2R}^+$}
        \end{cases}
    \end{equation}
   Existence and uniqueness follow exactly as in the proof of Theorem \ref{thm:inthom}--see the discussion after Equation \eqref{eq:veku}, or the proof of Proposition \ref{prop:ex}. Also, testing the weak formulation of \eqref{temo:weakdir}  with $v_{\e,k}-v_k^{bd}$,  and arguing exactly as in the proof of \eqref{energy:bound}, while taking into account \eqref{uffffii},  we obtain
    \begin{equation}\label{dir:enbound}
       \limsup_{k\to\infty} \limsup_{\e\to 0}\int_{B_{2R}^+}B_\e(|Dv_{\e,k}|)\,dx\leq C\,\int_{B^+_{2R}} B(|Dv|)\,dx,
    \end{equation}
    for $C>0$ independent of $\e,k$. Then, by the interior and boundary estimates Theorems \ref{thm:bdd}-\ref{thm:bdddir} (while taking into account Remark \ref{remark:importante}), Theorem \ref{thm:inthom}, estimates \eqref{Dve:dirichlet} and \eqref{dir:oscillation}, and a simple covering argument, we infer
\begin{equation*}
    \|v_{\e,k}\|_{C^{1,\beta_{\mathrm h}}({\overline{B}_r^+})}\leq C\,\bigg(1+\int_{B_{2R}^+}|Dv_{\e,k}|\,dx \bigg)\,dx\leq C'
\end{equation*}
    for every $0<r<2R$, with $C,C'>0$ independent of $\e,k$, where in the last inequality we used \eqref{dir:enbound} and \eqref{int:simple}. Hence, up to a diagonalization argument,  by Ascoli-Arzel\'a theorem we find 
    \begin{equation}\label{C1ekdir}
        \lim_{k\to \infty}\lim_{\e\to 0}v_{\e,k}=w\quad\text{in $C^1(\overline{B}_r^+)$,}
    \end{equation}
for some function $w\in C^1(\overline{B}_r^+)$, and for every $0<r<2R$. Then, by \eqref{unif:refl} and \eqref{dir:enbound}, the sequence $\{v_{\e,k}\}_{\e,k}$ is uniformly bounded in $W^{1,\min\{i_B,2\}}(B^+_{2R})$ with $i_B>1$; hence by reflexivity, \eqref{uffffii} and the continuity of the trace operator, we deduce that $w=v$ on $\partial B^+_{2R}$. Then, by using \eqref{dir:enbound}, \eqref{BE:unif}, \eqref{C1ekdir}, and arguing exactly as in the proof of \eqref{calBmon}, we find that $w\in W^{1,B}(B^+_{2R})$.

 Exploiting \eqref{Aeunif} and \eqref{C1ekdir}, by passing to the limit in the weak formulation of \eqref{temo:weakdir},  we find that $w\in W^{1,B}(B_{2R}^+)$ solves the Dirichlet problem $-\mathrm{div}\big(\A(Dw) \big)=0$ in $B^+_{2R}$, $w=u$ on $\partial B^+_{2R} $, hence $w=v$ by uniqueness. Finally, inequalities \eqref{dir:gradbdd}-\eqref{dir:gradholder} are obtained, via \eqref{C1ekdir}, by passing to the limit in the estimates \eqref{Dve:dirichlet} and \eqref{dir:oscillation}. This completes the proof.
\end{proof}

\section{Boundary gradient regularity: homogeneous Neumann problems}\label{sec:neuhom}

In this section we study boundary regularity for the gradient of solution to the Neumann problem
\begin{equation}\label{homog:neum}
    \begin{cases}
        -\mathrm{div}\big( \A(Dv)\big)=0\quad&\text{in $B_{3R_0}^+$}
        \\
        \A(Dv)\cdot e_n+h_0=0\quad &\text{on $B_{3R_0}^0$}.
    \end{cases}
\end{equation}
for a given constant $h_0\in \R$. We recall that $v\in W^{1,B}(B_{3R_0}^+)$ is a weak solution to \eqref{homog:neum} if
\begin{equation}
    \int_{B_{3R_0}^+}\A(Dv)\cdot D\varphi\,dx=h_0\,\int_{B^0_{3R_0}}\varphi(x')\,dx'
\end{equation}
for all test functions $\varphi\in W^{1,B}_c(B_{3R_0})$. We may also assume that $\A(0)=0$. Indeed, if this is not the case, we can replace $\A(\xi)$ with $\A(\xi)-\A(0)$, so the function $v$ then satisfies the same equation, with boundary datum shifted to $h_{0}+\A(0)\cdot e_n$, which is also controlled since, by \eqref{A:hold}$_2$, we have $|\A(0)|\leq \L$.
\vspace{0.1cm}

The main result of this section is  the following.
\begin{theorem}\label{thm:homneu}
    Suppose the stress field $\A\in C^0(\R^n)\cap  C^1(\rn\setminus\{0\})$ satisfies \eqref{ass:Aaut} and $\A(0)=0$.  Let $v\in W^{1,B}(B_{3R_0}^+)$ be a weak solution to \eqref{homog:neum}.
    
    Then there exists $\beta_{N}\in (0,1)$ depending on $n,\l,\L,i_a,s_a$ such that 
    \[
    v\in C^{1,\beta_N}(\overline{B}^+_{R_0/2}),
    \]
    and there exist constants $C,C'=C,C'(n,\l,\L,i_a,s_a)>0$ such that
\begin{equation}\label{neu:EST}
    \sup_{B_{R_0}^+} |Dv|\leq C\,\mint_{B_{2R_0}^+}|Dv|\,dx+C\,b^{-1}\big(|h_0|\big)\,.
\end{equation}
Moreover, for all $0<r\leq R\leq R_0/2$, there holds the excess decay estimate
\begin{equation}\label{decay:NEU}
    \mint_{B^+_r}|Dv-(Dv)_{B^+_r}|\,dx\leq C\,\left(\frac{r}{R} \right)^{\beta_N}\,\mint_{B^+_{R}} |Dv-(Dv)_{B^+_{R}}|\,dx\,,
\end{equation}
and the oscillation estimates
\begin{equation}\label{neu:EST1}
    \operatorname*{osc}_{B_r^+} Dv\leq C\,\left(\frac{r}{R} \right)^{\beta_N}\operatorname*{osc}_{B_R^+} Dv\leq  C'\,\left(\frac{r}{R} \right)^{\beta_N}\bigg\{\mint_{B_{2R}^+}|Dv|\,dx+C\,b^{-1}\big(|h_0|\big) \bigg\},
\end{equation}
\begin{equation}\label{NEU:est2}
      \operatorname*{osc}_{B_r^+} Dv\leq C\,\left(\frac{r}{R} \right)^{\beta_N}\,\mint_{B^+_R}|Dv-(Dv)_{B^+_R}|\,dx,\quad \text{for}\,\,0<r\leq R/2.
\end{equation}
\end{theorem}

As per usual, we proceed  via approximation by considering $\ve\in W^{1,2}(B_{2R_0}^+)$ solution to
\begin{equation}\label{homve:neu}
    \begin{cases}
         -\mathrm{div}\big( \Ae(D\ve)\big)=0\quad\text{in $B_{2R_0}^+$}
        \\
        \Ae(D\ve)\cdot e_n+h_0=0\quad \text{on $B_{2R_0}^0$}.
    \end{cases}
\end{equation}
where $\Ae$ is given by Lemma \ref{lemma:Ae}.
Let us first recollect some regularity properties of $\ve$.

\begin{proposition}
    Let $\ve\in W^{1,2}(B_{2R_0}^+)$ be a weak solution to \eqref{homve:neu}. Then
    \begin{equation}\label{reg:neuve}
        \ve\in C^\infty(B_{2R_0}^+)\cap W^{2,2}(B_r^+)\quad \text{for all $0<r<2R_0$.}
    \end{equation}
\end{proposition}

\begin{proof}
    The $C^\infty$-smoothness of $\ve$ in $B_{2R_0}^+$ is guaranteed by Proposition \ref{prop:veregular}. For what concerns the $W^{2,2}$- regularity, let us consider the difference quotients  $\Delta_h^i \ve$ for $i=1,\dots,n-1$, which are defined by \eqref{diff:quot}. We test the weak formulation of \eqref{homve:neu} with
    \begin{equation*}
        \Delta_{-h}^i(\eta^2\Delta_h^i\ve),\quad \eta\in C^{\infty}_c(B_{2R_0-2h}),
    \end{equation*}
    which is admissible since $\N\ni i<n$, and repeating verbatim the computations done for \eqref{df0}, we arrive at
\begin{equation*}
    \int_{B_{2R_0}^+}|D(\Delta_h^i\ve)|^2\,\eta^2\,dx\leq C_\e\,\int_{B_{2R_0}^+}|\Delta_h^i\ve|^2\,|D\eta|^2\,dx, 
\end{equation*}
    and from this inequality and the properties of the difference quotients \cite[Section 7.11]{GT}, we deduce
    \begin{equation}\label{W22:neu}
        \int_{B_r^+}|D(D_i\ve)|^2\,dx\leq C(\e,r)\, \int_{B_{2R_0}^+}|D\ve|^2\,dx\,,
    \end{equation}
    for all $i=1,\dots,n-1$. We now need to estimate the $L^2$-norm of $D_{nn}\ve$. By the chain rule, we may write \eqref{eq:homeint} in the trace form
    \begin{equation}\label{trace:utile}
        \sum_{i,j=1}^n\frac{\partial \Ae^i}{\partial \xi_j}(D\ve(x))\,D_{ij}\ve(x)=0,\quad \text{for every $x\in B^+_{2R_0}$.}
    \end{equation}
In particular, 
    \begin{equation*}
        D_{nn}\ve=-\left( \frac{\partial \Ae^n}{\partial \xi_n}(D\ve)\right)^{-1}\sum_{\substack{ i,j=1  \\ i\lor j \neq n}}^n
\frac{\partial \Ae^i}{\partial \xi_j}(D\ve)\,D_{ij}\ve(x),\quad \text{in $ B^+_{2R_0}$,}
    \end{equation*}
and using that $|\frac{\partial \Ae}{\partial \xi_j}|\leq C\e^{-1}$ and $\frac{\partial \Ae^n}{\partial \xi_n}(D\ve)\geq c\,\e$ by \eqref{nonzero:Ae}, we deduce
\begin{equation*}
    |D_{nn}\ve|\leq C(\e)\,\sum_{i=1}^{n-1}|D(D_i\ve)|\quad \text{in $B_{2R_0}^+$,}
\end{equation*}
We square the above expression, and integrate it over $\mathcal{B}^+\Subset B_{r}^+$ so that, by also using \eqref{W22:neu}, we get
\begin{equation*}
    \int_{\mathcal{B}^+}|D_{nn}\ve|^2\,dx\leq C(\e,r)\,\int_{B_{2R_0}^+}|D\ve|^2\,dx.
\end{equation*}
Taking a sequence of sets $\mathcal{B}^+\nearrow B_{r}^+$ and using mononotone convergence theorem finally yields $D_{nn}\ve\in L^2(B_r^+)$, that is our thesis.
\end{proof}

\subsection{Some properties of \texorpdfstring{$D\ve$}{Dv}}
Here we recollect some useful identities and properties of $\ve$ and its derivatives. Let us consider a given radius $R\leq R_0$. First, the weak formulation of \eqref{homve:neu} tells that
\begin{equation}\label{test:neuve}
    \int_{B^+_{2R}}\Ae(D\ve)\cdot D\varphi\,dx=h_0\,\int_{B_{2R}^0}\varphi\,dx\,,
\end{equation}
for all $\varphi\in W^{1,2}_c(B_{2R})$.
Testing \eqref{test:neuve} with $\varphi=D_k\eta$, $\eta\in C^\infty_c(B_{2R})$, integrating by parts and using the chain rule for $W^{2,2}$-functions, we find
\begin{equation*}
    \int_{B^+_{2R}}\nabla_\xi \Ae(D\ve)\,D(D_k\ve)\cdot D\eta\,dx=-\int_{B_{2R}^0}\Ae(D\ve)\cdot D\eta\,\delta_{kn}\,dx'-h_0\,\int_{B_{2R}^0} D_k\eta\,dx',
\end{equation*}
 where $\delta_{kn}$ is the Kronecker delta. In particular, if $1\leq k<n$,  by the divergence theorem on $B_{2R}^0$ and a density argument on $\eta$, we deduce
\begin{equation}\label{neu:Dkve}
    \int_{B_{2R}^+}\nabla_\xi \Ae(D\ve)\,D(D_k\ve)\cdot D\eta\,dx=0,\quad k=1,\dots,n-1
\end{equation}
for all $\eta\in W^{1,2}_c(B_{2R})$. On the other hand, for $k=n$ we have
\begin{equation}\label{lauso:Dnv}
    \int_{B_{2R}^+}\nabla_\xi \Ae(D\ve)\,D(D_n\ve)\cdot D\eta\,dx=-\int_{B_{2R}^0}\Ae(D\ve)\cdot D\eta\,dx'-h_0\,\int_{B_{2R}^0} D_n\eta\,dx',
\end{equation}
for all $\eta\in C^\infty_c(B_{2R})$. Observe that, by the regularity property \eqref{reg:neuve}, the boundary condition $\Ae^n(D\ve)+h_0=0$ holds in the sense of traces, that is for a.e. $x'\in B_{2R}^0$. 
Therefore
\[
\Ae(D\ve)\cdot D\eta=\Ae'(D\ve)\cdot D'\eta+\Ae^n(D\ve)\,D_n\eta=\Ae'(D\ve)\cdot D'\eta-h_0\,D_n\eta\quad\text{ a.e. in $B_{2R}^0$,}
\]
  so equation \eqref{lauso:Dnv} simplifies to
\begin{equation}\label{neu:Dnve}
    \int_{B_{2R}^+}\nabla_\xi \Ae(D\ve)\,D(D_n\ve)\cdot D\eta\,dx=-\int_{B_{2R}^0}\Ae'(D\ve)\cdot D'\eta\,dx',
\end{equation}
for all $\eta\in C^\infty_c(B_{2R})$ and, via a density argument, for all $\eta\in W^{1,2}_c(B_{2R})$ whose trace fulfills $\eta\in W^{1,2}(B^0_{2R})$. 
\vspace{0.2cm}

Next, analogously to Proposition \ref{prop:bernstein}, we show that $V_\e=B_\e\big( |D\ve|\big)$ is a weak sub-solution (in an integral sense) of a certain equation. This is the content of the following
\begin{proposition}\label{prop:berNeu}
{\rm
    Let $\ve\in W^{1,B_\e}(B^+_{2R_0})$ be a weak solution to \eqref{homve:neu}, let $0<R\leq R_0$, and let $V_\e:= B_\e\big( |D\ve|\big)$.
    Then $V_\e\in W^{1,1}(B^+_r)$ for all $0<r<2R_0$, and the following integral inequality holds:
\begin{equation}\label{new:Bernstein}
    \int_{B_{2R}^+}\mathbb{A}_\e(x)\,DV_\e\cdot D\varphi\,dx\leq 0,\quad\text{for all nonnegative   $\varphi\in W^{1,\infty}_c(B_{2R})$ s.t.  $\varphi\equiv 0$ on $B^0_{2R}$,}
\end{equation}  
 where $\mathbb{A}_\e$ is the matrix-valued function defined by \eqref{Ae:uf}.
}
\end{proposition}

\begin{proof}
  By the chain rule, \eqref{ae:nonzero} and \eqref{def:bBe}, we have 
   \begin{equation}\label{new:chain}|DV_\e|=b_\e(|D\ve|)\,\big|D|D\ve|\big|\leq \e^{-1}|D\ve| |D^2\ve|.
   \end{equation}
   which belongs to $L^{1}(B^+_r)$ for every $0<r< 2R_0$ owing to \eqref{reg:neuve}.
   
   To prove \eqref{new:Bernstein}, we pick a test function $\varphi\in C^\infty_c(B_{2R})$  such that $\varphi\geq 0$ in $B_{2R}$, and $\varphi\equiv 0$ on $B_{2R}^0$, and we follow the same computations as in \eqref{der:Beve}-\eqref{ozza}. Specifically, we test Equations \eqref{neu:Dkve} and \eqref{neu:Dnve} with $(D_k\ve)\,\varphi$ and $(D_n\ve)\,\varphi$, respectively. Observing that the boundary integral in \eqref{neu:Dnve} dies off thanks to our choice of $\varphi$, by summing over $k=1,\dots,n$, and using \eqref{der:Beve} and \eqref{coer:Ae}, we get
   \begin{equation*}
       \begin{split}
0&=\sum_{k=1}^n\int_{B^+_{2R}}\nabla_\xi \A_\e(D\ve)\,D(D_k\ve)\,D_k\ve\,\cdot D\varphi\,dx+\sum_{k=1}^n\int_{B^+_{2R}}\nabla_\xi\Ae(D\ve)\,D(D_k\ve)\cdot D(D_k\ve)\,\varphi\,dx
\\
&=\int_{B^+_{2R}}\mathbb{A}_\e(x)\,a_\e\big(|D\ve|)\,D\left(\frac{|D\ve|^2}{2}\right)\,\cdot D\varphi\,dx+\sum_{k=1}^n\int_{B^+_{2R}}\nabla_\xi\Ae(D\ve)\,D(D_k\ve)\cdot D(D_k\ve)\,\varphi\,dx
    \\
    &\geq \int_{B^+_{2R}} \mathbb{A}_\e(x)\,DV_\e\cdot D\varphi\,dx\,,
       \end{split}
   \end{equation*}
   which, after performing a standard density argument on $\varphi$, it is our thesis.
\end{proof}




Next, observe that by \eqref{trace:utile} and \eqref{coer:Ae}, in $B_{2R}^+$ we have
{\small
\begin{equation*}
\begin{split}
    c\,a_\e\big(|D\ve| \big)\,|D_{nn}\ve|&\leq \frac{\partial \Ae^n}{\partial \xi_n}|D_{nn}\ve| \leq \Bigg|\sum_{\substack{ i,j=1  \\ i\lor j \neq n}}^n
\frac{\partial \Ae^i}{\partial \xi_j}(D\ve)\,D_{ij}\ve(x)\Bigg|\leq C\,a_\e\big(|D\ve| \big)\,|D(D'\ve)|,
    \end{split}
\end{equation*}
}
so dividing by $a_\e\big(|D\ve| \big)$ (which never vanishes by \eqref{ae:nonzero}), we deduce the useful hessian estimate
\begin{equation}\label{control:D2ve}
    |D^2\ve|\leq C\,|D(D'\ve)|\quad\text{in $B^+_{2R}$,}
\end{equation}
for $C=C(n,\l,\L,i_a,s_a)>0$.
\vspace{0.2cm}

\subsection{Bounds for \texorpdfstring{$D_n\ve$}{Dnv} on the flat boundary.} Exploiting the boundary condition and the coercivity properties of $\A$, here we  derive supremum and oscillation bounds for $D_n\ve$, in terms of the tangential derivatives, on the flat part $B^0_{2R}$. First, for all $0<r\leq 2R$, we claim 
\begin{equation}\label{trbd:neu}
    |D\ve|\leq C\,|D'\ve|+C\,b^{-1}\big(|h_0| \big)\quad\text{a.e. on $B^0_{r}$,}
\end{equation}
for some constant $C=C(n,\l,\L,i_a,s_a)>0$, for $0<\e<\e_0$ small enough ($\e_0$ depending only on $b(\cdot)$ and an upper bound on  $|h_0|$, as we will see below using \eqref{be:unif}).

To prove it, we use
\eqref{coAe:gr}, the boundary condition $\Ae^n(D\ve)\equiv-h_0$ a.e. on $B_{r}^0$, and Young's inequality \eqref{young1}, thus getting
\begin{equation*}
    \begin{split}c_0\,B_\e\big(|D\ve|\big) &\leq \Ae(D\ve)\cdot D\ve=\Ae'(D\ve)\cdot D'\ve+\A^n(D\ve)\,D_n\ve
\\
&=\Ae'(D\ve)\cdot D'\ve-h_0\,D_n\ve\leq C\,b_\e\big(|D\ve|\big)\,|D'\ve|+|h_0|\,|D\ve|
\\
&\leq \frac{c_0}{2}\,B_\e\big(|D\ve|\big)+C'\,B_\e\big(|D'\ve|\big)+C'\,\widetilde{B}_\e
\big(|h_0| \big),    \end{split}
\end{equation*}
 where the constants $c_0,C,C'>0$ depend on $n,\l,\L,i_a,s_a$, but are independent of $\e$ thanks to Remark \ref{remark:importante}. Reabsorbing terms, we get
\begin{equation}\label{unaltro}
    B_\e\big( |D\ve|\big)\leq C\,B_\e\big(|D'\ve|\big)+C\,\widetilde{B}_\e
\big(|h_0| \big)\quad\text{a.e. on $B^0_{r}$,}
\end{equation}
 so taking $B_\e^{-1}$ to both sides, and using \eqref{B2t} and \eqref{come:BwB}, we find 
\[
|D\ve|\leq C\,|D'\ve|+C\,b_\e^{-1}(|h_0|)\leq C\,|D'\ve|+2C\,b^{-1}(|h_0|),
\] 
 the last inequality due to the uniform convergence \eqref{be:unif}, whenever $0<\e<\e_0$ is small enough. Again, here $C=C(n,\l,\L,i_a,s_a)>0$ thanks to Remark \ref{remark:importante}, so \eqref{trbd:neu} is proven.
\vspace{0.2cm}

\noindent Next, for all $0<r<2R$, we have the oscillation bound
 \begin{equation}\label{ttt:osc}
     |D_n\ve(x')-D_n\ve(y')|\leq C_{\mathrm{a}}\,\max_{k=1,\dots,n-1}|D_k\ve(x')-D_k\ve(y')|\quad\text{for a.e. $x',y'\in B^0_{r}$,}
 \end{equation}
 for some constant $C_{\mathrm{a}}=C_{\mathrm{a}}(n,\l,\L,i_a,s_a)>0$.
 
 To prove it, we first observe that, since $\Ae^n(D\ve)\equiv-h_0$ a.e. on $B_{r}^0$, we have
 {\small
\begin{equation*}
\begin{split}
    \Big(\Ae\big(D\ve(x')\big)-\Ae&\big(D\ve(y')\big)\Big)\cdot\big(D\ve(x')-D\ve(y')\big)
    \\
    &= \Big(\Ae'\big(D\ve(x')\big)-\Ae'\big(D\ve(y')\big)\Big)\cdot\big(D'\ve(x')-D'\ve(y')\Big)\quad \text{for a.e. $x',y'\in B^0_{r}$.}
\end{split}
\end{equation*}
}
 On the other hand, by the fundamental theorem of calculus and \eqref{coer:Ae} with Remark \ref{remark:importante}, we have
 {\small
\begin{equation*}
    \begin{split}
        \Big(&\Ae\big(D\ve(x')\big)-\Ae\big(D\ve(y')\big)\Big)\cdot\big(D\ve(x')-D\ve(y')\big)\\
        &=\bigg(\int_0^1
\nabla_\xi \Ae\Big(tD\ve(x')+(1-t)D\ve(y') \Big)dt \bigg)\,\Big(D\ve(x')-D\ve(y') \Big)\cdot\Big(D\ve(x')-D\ve(y') \Big)
\\
&\geq c\,\bigg(\int_0^1 a_\e\big(tD\ve(x')+(1-t)D\ve(y') \big)\,dt\bigg)\,|D\ve(x')-D\ve(y')|^2,
\end{split}
\end{equation*}
}
and
{\small
\begin{equation*}
    \begin{split}
    \bigg|\Big(\Ae'&\big(D\ve(x')\big)-\Ae'\big(D\ve(y')\big)\Big)\cdot\big(D'\ve(x')-D'\ve(y')\Big)\bigg|
    \\
    =&\,\Bigg|\sum_{i=1}^{n-1}\sum_{j=1}^{n}\bigg(\int_0^1\frac{\partial \Ae^i}{\partial \xi_j}\big(tD\ve(x')+(1-t)D\ve(y')\big)dt\bigg)\,\big(D_i\ve(x')-D_i\ve(y')\big)\,\big(D_j\ve(x')-D_j\ve(y')\big)\Bigg|
    \\
    \leq &\, C\,\bigg(\int_0^1 a_\e\big( tD\ve(x')+(1-t)D\ve(y')\big)dt\bigg)\,|D\ve(x')-D\ve(y')|\,|D'\ve(x')-D\ve(y')|
    \\
    \leq &\,\frac{c}{2}\bigg(\int_0^1 a_\e\big( tD\ve(x')+(1-t)D\ve(y')\big)dt\bigg)\,|D\ve(x')-D\ve(y')|^2
    \\
    &+C'\,\bigg(\int_0^1 a_\e\big( tD\ve(x')+(1-t)D\ve(y')\big)dt\bigg)\,|D'\ve(x')-D'\ve(y')|^2
    \end{split}
\end{equation*}
}
with $c,C,C'$ depending on $n,\l,\L,i_a,s_a$, where we used weighted Young's inequality in the  last estimate. Combining the three identities above, and simplifying the resulting expression (which is admissible since $a_\varepsilon(t)$ never vanishes by \eqref{ae:nonzero}) yields \eqref{ttt:osc}.

\subsection{Gradient boundedness}
In a first step, we establish that
\begin{equation}\label{temp:Dvebddd}
D v_\varepsilon \in L^\infty(B_r^+), \quad \text{for all } 0 < r < 2R,
\end{equation}
without providing a quantitative bound. At this stage, we do not aim to obtain an $\e$-independent estimate, since \eqref{temp:Dvebddd} will only be used to justify the computations later on.
\vspace{0.2cm}

We start with the boundedness of the tangential gradient $D'\ve$. This immediately follows from \eqref{neu:Dkve}, \eqref{nonzero:Ae} and the theory of quadratic equations in $B_r^+$. Specifically, for $\kappa>0$, we test \eqref{neu:Dkve} with $(D_k\ve-\kappa)_+\,\phi^2$, for $\phi\in C^\infty_c(B_{r_2})$, $0\leq \phi\leq 1$, $\phi\equiv 1$ on $B_{r_1}$, $|D\phi|\leq C(n)/(r_2-r_1)$. Using \eqref{nonzero:Ae}, weighted Young's inequality, and that $D(D_k\ve)=D(D_k\ve-\kappa)_+$ in  $\{(D_k\ve-\kappa)_+\neq 0\}$, from \eqref{neu:Dkve} we get
\begin{equation*}
    \begin{split}c&\,\e\int_{B_{2R}^+}|D(D_k\ve-\kappa)_+|^2\phi^2\,dx\leq \int_{B_{2R}^+}\nabla_\xi \Ae(D\ve)D(D_k\ve)\cdot D(D_k\ve-\kappa)_+\phi^2dx
    \\
    &= -2\int_{B_{2R}^+} \nabla_\xi \Ae(D\ve)D(D_k\ve)\cdot D\phi\,\phi\,(D_k\ve-\kappa)_+\,dx
    \\
    &\leq C\,\e^{-1}\,\int_{B_{2R}^+}|D(D_k\ve)|\,|D\phi|\,\phi\,|(D_n\ve-\kappa)_+|\,dx
    \\
    &\leq \frac{c\,\e}{2}\int_{B_{2R}^+}|D(D_k\ve-\kappa)_+|^2\phi^2\,dx+C'\,\e^{-3}\int_{B_{2R}^+} |(D_k\ve-\kappa)_+|^2\,|D\phi|^2\,dx.
    \end{split}
\end{equation*}
Reabsorbing terms, and using the properties of $\phi$, we find
\begin{equation*}
    \int_{B_{r_1}^+}|D(D_k\ve-\kappa)_+|^2\,dx\leq \frac{C_\e}{(r_2-r_1)^2}\,\int_{B_{r_2}^+} |(D_k\ve-\kappa)_+|^2\,dx
\end{equation*}
for all $k=1,\dots,n-1$. Therefore, $D_k\ve$ belongs to the De Giorgi's class \cite[Section 7.2]{G03} with $B_r^+$ in place of $Q_r$, so the proofs of \cite[Section 7.2]{G03} can be repeated verbatim replacing $Q_r$ with $B_r^+$ (see also Remark \ref{remark:evenext} below), thus obtaining
\begin{equation}\label{tmp:neubdd}
    \sup_{B_r^+}|D'\ve|\leq C(\e,r)\,\mint_{B_{2R}^+}|D\ve|\,dx
\end{equation}
for all $0<r<2R$.
\vspace{0.2cm}

Next, owing to the bounds \eqref{tmp:neubdd}, \eqref{trbd:neu}, and the properties of the trace operator, we find
\begin{equation*}
    |D\ve|\leq C_{\e,r}\,\mint_{B_{2R}^+}|D\ve|\,dx+ C_{\e,r}\,b^{-1}\big(|h_0|\big)\eqqcolon \widehat C_{\e,r}\quad\text{a.e. on $B_{r}^0$.}
\end{equation*}

Therefore, for every $\kappa>\widehat{C}_{\e,r}$, the function $(D_n\ve-\kappa)_+$ vanishes on $B^0_r$; thus, by testing \eqref{neu:Dnve} with $\eta=(D_n\ve-\kappa)_+\,\phi^2$, $\phi\in C^\infty_c(B_r)$, the boundary integral in \eqref{neu:Dnve} dies off, and using \eqref{nonzero:Ae} and Young's inequality, we obtain once more
\begin{equation*}
    \begin{split}c&\,\e\int_{B_{2R}^+}|D(D_n\ve-\kappa)_+|^2\phi^2\,dx\leq \int_{B_{2R}^+}\nabla_\xi \Ae(D\ve)D(D_n\ve)\cdot D(D_n\ve-\kappa)_+\phi^2dx
    \\
    &= -2\int_{B_{2R}^+} \nabla_\xi \Ae(D\ve)D(D_n\ve)\cdot D\phi\,\phi\,(D_n\ve-\kappa)_+\,dx
    \\
    &\leq C\,\e^{-1}\,\int_{B_{2R}^+}|D(D_n\ve)|\,|D\phi|\,\phi\,|(D_n\ve-\kappa)_+|\,dx
    \\
    &\leq \frac{c\,\e}{2}\int_{B_{2R}^+}|D(D_n\ve-\kappa)_+|^2\phi^2\,dx+C'\,\e^{-3}\int_{B_{2R}^+} |(D_n\ve-\kappa)_+|^2\,|D\phi|^2\,dx.
    \end{split}
\end{equation*}
Reabsorbing terms, and taking $\phi\in C^\infty_c(B_{r_2})$, $\phi\equiv 1$ in $B_{r_1}$, $|D\phi|\leq C(n)/(r_2-r_1)$ yields
\begin{equation*}
    \int_{B_{r_1}^+}|D(D_n\ve-\kappa)_+|^2\phi^2\,dx\leq \frac{C_\e}{(r_2-r_1)^2}\int_{B_{r_2}^+} |(D_n\ve-\kappa)_+|^2\,dx, 
\end{equation*}
for all $\kappa>\widehat{C}_{\e,r}$, and all $0<r_1<r_2<r$. Again, it follows by \cite[Theorem 7.2-7.5]{G03} with $B_r^+$ in place of $Q_r$ (see also Remark \ref{remark:evenext} below) that $D_n\ve\in L^\infty(B_{r}^+)$, and thus \eqref{temp:Dvebddd} is proven.
\vspace{0.2cm}

\begin{remark}\label{remark:evenext}
\rm{
Suppose that \(w \in W^{1,2}(B_R^+)\) satisfies one of the following De~Giorgi-type inequalities:
\begin{equation}\label{DGrem}
    \int_{B_{r_1}^+} |D(w-\kappa)_+|^2\,dx 
    \le \frac{C}{(r_2 - r_1)^2} 
    \int_{B_{r_2}^+} |(w-\kappa)_+|^2\,dx, 
    \qquad \kappa > \kappa_0,
\end{equation}
or
\begin{equation}\label{DGrem1}
    \int_{B_{r_1}^+} |D(w-\ell)_-|^2\,dx 
    \le \frac{C}{(r_2 - r_1)^2} 
    \int_{B_{r_2}^+} |(w-\ell)_-|^2\,dx, 
    \qquad \ell < \ell_0,
\end{equation}
for every \(0 < r_1 < r_2 < R\). Then its even extension $w^e$ defined by \eqref{def:even}
satisfies the same De~Giorgi-type inequalities in the full ball \(B_R\). 
In particular, one may apply the results of \cite[Section~7]{G03} to obtain boundedness and H\"older continuity of \(w^{e}\), and hence of \(w\), since $
\sup_{B_r} |w^{e}| = \sup_{B_r^+} |w|$ and $\operatorname{osc}_{B_r} w^{e} = \operatorname{osc}_{B_r^+} w$. 

We also remark that if $w\in W^{1,2}(B_R^+)$ satisfies 
\begin{equation}\label{satisfies}
\begin{split}
    \int_{B_R^+}&A(x)Dw\cdot D\phi\,dx=0\quad \text{for all $\phi\in C^\infty_c(B_R)$}
    \\
     \sum_{i,j=1}^n|&A_{ij}(x)|\leq \L\quad A(x)\eta\cdot \eta\geq \l|\eta|^2\quad \text{for a.e. $x$, for all $\eta\in \R^n$}
    \end{split}
\end{equation}
 then the standard computations (i.e., testing with $(w-\kappa)_+\phi^2$ or $(w-\ell)_-\phi^2$ and using weighted Young's inequality) show that $w$ fulfills \eqref{DGrem}-\eqref{DGrem1}.

Further details concerning boundary De Giorgi's classes can be found in \cite[Chapter 2, Section 7]{LU68}, \cite[Sections 10.7-10.8]{DB10}, and \cite[Section 5.10]{L13}.
}
\end{remark}

\begin{remark}
  \rm{ It is also possible to prove that 
    \begin{equation*}
        D\ve\in C^{0,\alpha_\e}(\overline{B}^+_r)\quad 0<r<2R.
    \end{equation*}
for some $\alpha_\e\in (0,1)$. Since we do not need this fact, we just sketch the idea of the proof: for $\ell>0$, we test \eqref{neu:Dkve} with $(D_k\ve-\ell)_-\phi^2$, and via the usual computations (\eqref{nonzero:Ae} and weighted Young's inequality), 
one obtains
\begin{equation*}
    \int_{B_{r_1}^+}|D(D_k\ve-\ell)_-|^2\,dx\leq \frac{C_\e}{(r_2-r_1)^2}\int_{B_{r_2}^+}|(D_k\ve-\ell)_-|^2\,dx,
\end{equation*}
for all $k=1,\dots,n-1$; so, by Remark \ref{remark:evenext} and \cite[Theorem 7.6]{G03}, we have $D_k\ve\in C^{0,\alpha_\e}(\overline{B}_r^+)$. Next, by testing \eqref{neu:Dnve} with $\eta=(D_n\ve-\ell)_-\,\phi^2$, with 
$\phi\in C^\infty_c(B_r)$, and let
\begin{equation*}
  \ell<\ell_0:=-C\sup_{B^0_{2R}}|D'\ve|-C\,b^{-1}(|h_0|),  
\end{equation*}
with $C>0$ being the constant in \eqref{trbd:neu}. It follows by said expression that $\eta\equiv 0$ on $B^0_{2R}$, so the boundary integral in \eqref{neu:Dnve} vanishes, and we may perform the routine computations  (\eqref{nonzero:Ae} and Young's inequality), and get
\begin{equation*}
    \int_{B_{r_1}^+}|D(D_n\ve-\ell)_-|^2\,dx\leq \frac{C_\e}{(r_2-r_1)^2}\int_{B_{r_2}^+}|(D_n\ve-\ell)_-|^2\,dx
\end{equation*}
and, in view of \eqref{ttt:osc} and the H\"older continuity of $D'\ve$ just proven, we also have an oscillation controll of $D_n\ve$ on the flat boundary: 
\begin{equation*}
    \operatorname*{osc}_{B_r^0}D_n\ve\leq C_{\mathrm{a}} \max_{k=1,\dots,n-1} \operatorname*{osc}_{B^0_r}D_k\ve\leq C'r^{\alpha_\e}.
\end{equation*} 
It then follows by Remark \ref{remark:evenext} and a simple modification of \cite[proof of Theorem 7.8]{G03} (see also \cite[Theorem 7.1]{DB10} or \cite[Chapter 2, Theorem 7.1]{LU68}), that $D_n\ve\in C^{0,\alpha_\e}(\overline{B}_r^+)$, and the claim is proven.
}
\end{remark}

We now prove the quantitative boundedness of $D\ve$. First, inspired by the proof of \cite[Theorem 10.2.1]{LU68} and \cite[pp. 1214]{L88}, we use a weighted Moser iteration to show that the tangential gradient $D'\ve$ is bounded, uniformly in $\e>0$. Then by using the fact that $V_\e=B_\e(|D\ve|)$ is a  sub-solution to a uniform elliptic equation (see Proposition \ref{prop:berNeu}), and it is bounded on the flat part of the boundary by \eqref{trbd:neu} and the newfound estimate for $|D'\ve|$, the  (uniform in $\e$) bound on $|D\ve|$ will follow from the weak Harnack inequality.

\begin{proposition}\label{prop:gradneuLinf}
    Let $\ve\in W^{1,2}(B_{2R_0}^+)$ be a weak solution to \eqref{homve:neu}. Then  $D\ve\in L^\infty(B^+_{R_0})$, and there exists a constant $C=C(n,\l,\L,i_a,s_a)>0$ such that
    \begin{equation}\label{bbb:neu}
        \sup_{B^+_{R_0}}|D\ve|\leq C\,\mint_{B_{2R_0}^+}|D\ve|\,dx+C\,b^{-1}\big(|h_0|\big).
    \end{equation}
\end{proposition}

\begin{proof}
By scaling (see Remark \ref{rem:scaling}), we may assume that $R_0=1$.    Let $\phi\in C^\infty_c(B_{2})$, $0\leq \phi\leq 1$ be a cut-off function, and let $q>q_0$ for $q_0\geq 2$ large enough to be determined later. Let $\psi_\e$ be the function defined by \eqref{def:psi} with $a_\e$ in place of $a$;  we test equation \eqref{neu:Dkve} with \begin{equation*}
        \eta=\psi_\e\big( |D'\ve|\big)^q\,D_k\ve\bigg(\frac{|D\ve|^2}{|D'\ve|^2}\bigg)\,\phi^{2},
    \end{equation*}
     which is admissible  thanks to \eqref{temp:Dvebddd} and \eqref{reg:neuve},  and we sum the resulting equations for $k=1,\dots,n-1$, thus obtaining
    \begin{equation}\label{lungo.test}
        \begin{split}
            0= &\,q\sum_{i,j=1}^n\sum_{k=1}^{n-1}\int_{B_{2}^+}\frac{\partial \Ae^i}{\partial \xi_j}(D\ve)\,D_j(D_k\ve)\,D_k\ve\,D_i\psi_\e\big(|D'\ve| \big)\,\psi_\e\big(|D'\ve| \big)^{q-1}\frac{|D\ve|^2}{|D'\ve|^2}\,\phi^2\,dx
            \\
            &-2\sum_{i,j=1}^n\sum_{k=1}^{n-1}\int_{B_{2}^+}\frac{\partial \Ae^i}{\partial \xi_j}(D\ve)\,D_j(D_k\ve)\,D_k\ve\,\psi_\e\big(|D'\ve| \big)^q\,\frac{|D\ve|^2}{|D'\ve|^3}\,D_i|D'\ve|\,\phi^2\,dx
            \\
            &+\sum_{i,j=1}^n\sum_{k=1}^{n-1}\int_{B_{2}^+}\frac{\partial \Ae^i}{\partial \xi_j}(D\ve)\,D_j(D_k\ve)\,D_i(D_k\ve)\,\psi_\e\big(|D'\ve| \big)^q\frac{|D\ve|^2}{|D'\ve|^2}\,\phi^2\,dx
            \\
            &+\sum_{i,j=1}^n\sum_{k=1}^{n-1}\int_{B_{2}^+}\frac{\partial \Ae^i}{\partial \xi_j}(D\ve)\,D_j(D_k\ve)\,D_k\ve\,\psi_\e\big(|D'\ve| \big)^q\frac{D_i|D\ve|^2}{|D'\ve|^2}\,\phi^2\,dx
            \\
            &+2\,\sum_{i,j=1}^n\sum_{k=1}^{n-1}\int_{B_{2}^+}\frac{\partial \Ae^i}{\partial \xi_j}(D\ve)\,D_j(D_k\ve)\,D_k\ve\,\psi_\e\big(|D'\ve| \big)^q\frac{|D\ve|^2}{|D'\ve|^2}\,\phi\,\,D_i\phi\,dx
            \\
            \eqqcolon & (I)+(II)+(III)+(IV)+(V).
        \end{split}
    \end{equation}
Now observe that
\begin{equation}\label{rivative}
    \sum_{k=1}^{n-1} D_j(D_k\ve)\,D_k\ve=D_j|D'\ve|\,|D'\ve|,\quad D_j\psi_\e\big(|D'\ve| \big)=\psi_\e'\big(|D'\ve| \big)\,D_j|D'\ve|,
\end{equation}
and by \eqref{ultima} and Remark \ref{remark:importante}, we have
\begin{equation}\label{ultima1}
    c(s_a)\,\psi_\e(t)\leq\psi'_\e(t)\,t=a_\e(t)^{1/2}t\leq C(i_a)\,\psi_\e(t).
\end{equation}
 Therefore, by using \eqref{rivative}, \eqref{ultima1} and  \eqref{coer:Ae}, we obtain  
\begin{equation*}
    \begin{split}
        (I)=q\,\sum_{i,j=1}^n\int_{B_{2}^+}\frac{\partial \Ae^i}{\partial \xi_j}&(D\ve)\,D_j|D'\ve|\,\psi_\e'\big(|D\ve| \big)|D'\ve|\,D_i\psi_\e\big( |D'\ve|\big)
        \\
        &\times\psi_\e\big( |D'\ve|\big)^{q-1}\frac{|D\ve|^2}{|D'\ve|^2\,\psi_\e'(D\ve)}\,\phi^2\,dx
        \\
       =q\,\sum_{i,j=1}^n\int_{B_{2}^+}\frac{\partial \Ae^i}{\partial \xi_j}&(D\ve)\,D_j\psi_\e\big(|D'\ve|\big)\,D_i\psi_\e\big( |D'\ve|\big)
       \\
       &\times\psi_\e\big( |D'\ve|\big)^{q-1}\frac{|D\ve|^2}{|D'\ve|\,\psi_\e'(|D'\ve|)}\,\phi^2\,dx
       \\
       \geq  c_0\,q\,\int_{B^+_{2}} a_\e\big(|D&\ve| \big)\,|D\ve|^2\,\Big|D\big[\psi_\e\big(|D'\ve|\big)\big]\Big|^2\,\psi_\e\big(|D'\ve|\big)^{q-2}\,\phi^2\,dx,
    \end{split}
\end{equation*}
with $c_0=c_0(n,\l,\L,i_a,s_a)\in (0,1)$. Then, by \eqref{rivative}, \eqref{coer:Ae} and \eqref{ultima1}, we find
\begin{equation*}
    \begin{split}
        |(II)|=&\,2\,\bigg|\sum_{i,j=1}^n\int_{B^+_{2}} \frac{\partial \Ae^i}{\partial \xi_j}(D\ve)\,D_j|D'\ve|\,D_i|D'\ve|\,\psi_\e'\big(|D'\ve| \big)^2
        \\
         &\hspace{5cm}\times\psi_\e\big(|D'\ve| \big)^q\frac{|D\ve|^2}{|D'\ve|^2\psi_\e'\big(|D'\ve| \big)^2}\,\phi^2\,dx\bigg|
        \\
        \leq &\, C_0\int_{B^+_{2}}a_\e\big(|D\ve| \big)\,|D\ve|^2\,\big|D\big[\psi_\e\big(|D'\ve| \big)\big]\big|^2\frac{\psi_\e\big(|D'\ve| \big)^q}{|D'\ve|^2\,\psi_\e'\big(|D'\ve| \big)^2 }\,\phi^2\,dx
        \\
        \leq &\, C_1\int_{B^+_{2}}a_\e\big(|D\ve| \big)\,|D\ve|^2\,\big|D\big[\psi_\e\big(|D'\ve| \big)\big]\big|^2\psi_\e\big(|D'\ve| \big)^{q-2}\,\phi^2\,dx\,,
    \end{split}
\end{equation*}
with $C_0,C_1=C_0,C_1(n,\l,\L,i_a,s_a)>0$. Next, by using \eqref{coer:Ae} and \eqref{control:D2ve}, we get
\begin{equation*}
    \begin{split}
        (III)\geq &\,c_2\,\int_{B_{2}^+}a_\e\big(|D\ve| \big)\,|D(D'\ve)|^2\,\psi_\e\big(|D'\ve|\big)^q\,\frac{|D\ve|^2}{|D'\ve|^2}\,\phi^2\,dx
        \\
        \geq &\,c_3\int_{B_{2}^+}a_\e\big(|D\ve| \big)\,|D^2\ve|^2\,\psi_\e\big(|D'\ve| \big)^q\,\frac{|D\ve|^2}{|D'\ve|^2}\,\phi^2\,dx
    \end{split}
\end{equation*}
with $c_2,c_3>0$ depending on $n,\l,\L,i_a,s_a$. By \eqref{coer:Ae}, \eqref{rivative}-\eqref{ultima1}, and weighted Young's inequality, we find
\begin{equation*}
    \begin{split}
        |(IV)|=&\bigg| \sum_{i,j=1}^n\frac{\partial \A^i_\e}{\partial \xi_j}(D\ve)\,D_j|D'\ve|\,\psi_\e\big(|D'\ve| \big)^q\bigg(\frac{2\,|D\ve|\,D_i |D\ve|}{|D'\ve|}\bigg)\,\phi^2\,dx\bigg|
        \\
        \leq &\,C_3\,\int_{B_{2}^+} a_\e\big(|D\ve| \big)\,|D\ve|\,|D\psi_\e(|D'\ve|)|\,\frac{\psi_\e\big(|D'\ve| \big)^q }{|D'\ve|\,\psi'_\e\big( |D'\ve|\big)}|D^2\ve|\,\phi^2\,dx
        \\
        \leq &\,C_4\,\int_{B_{2}^+} a_\e\big( |D\ve|\big)\,|D\ve|\,|D\psi_\e(|D'\ve|)|\,\psi_\e\big(|D'\ve| \big)^{q-1}|D^2\ve|\,\phi^2\,dx
        \\
        \leq & \,C_5\int_{B_{2}^+}a_\e\big(|D\ve| \big)\,|D\ve|^2\,|D\psi_\e(|D'\ve|)|^2\,\psi_\e\big(|D'\ve| \big)^{q-2}\,\phi^2\,dx
        \\
        &+\frac{c_3}{4}\int_{B_{2}^+}a_\e\big(|D\ve| \big)\,|D^2\ve|^2\,\psi_\e\big(|D'\ve| \big)^q\,\phi^2\,dx,
    \end{split}
\end{equation*}
where $C_3,C_4,C_5\geq 2$ depend on $n,\l,\L,i_a,s_a$. Finally, by \eqref{coer:Ae} and weighted Young's inequality
\begin{equation*}
    \begin{split}
        |(V)|\leq &\,C_6\int_{B_{2}^+} a_\e\big(|D\ve| \big)\,|D^2\ve|\,\psi_\e\big(|D'\ve| \big)^q\,\frac{|D\ve|^2}{|D'\ve|}\,\phi\,|D\phi|\,dx
        \\
        \leq &\,\frac{c_3}{4}\,\int_{B_{2}^+}a_\e\big(|D\ve| \big)\,|D^2\ve|^2\, \psi_\e\big(|D'\ve| \big)^q\frac{|D\ve|^2}{|D'\ve|^2}\,\phi^2\,dx
        \\
        &+C_7\,\int_{B_{2}^+}a_\e\big(|D\ve| \big)\,|D\ve|^2\psi_\e\big(|D'\ve| \big)^q\,|D\phi|^2\,dx.
    \end{split}
\end{equation*}
with $C_6,C_7=C_6,C_7(n,\l,\L,i_a,s_a)>0$. Inserting the five inequalities above into \eqref{lungo.test}, we obtain
\begin{equation*}
\begin{split}
    \Big(c_0\,q-C_1-C_5\Big)&\,\int_{B^+_{2}} a_\e\big(|D\ve| \big)\,|D\ve|^2\,\Big|D\big[\psi_\e\big(|D'\ve|\big)\big]\Big|^2\,\psi_\e\big(|D'\ve|\big)^{q-2}\,\phi^2\,dx
    \\
    &+\frac{c_3}{2}\int_{B_{2}^+}a_\e\big(|D\ve| \big)\,|D^2\ve|^2\,\psi_\e\big(|D'\ve| \big)^q\,\frac{|D\ve|^2}{|D'\ve|^2}\,\phi^2\,dx
    \\
    \leq &\,C_7\,\int_{B_{2}^+}a_\e\big(|D\ve| \big)\,|D\ve|^2\psi_\e\big(|D'\ve| \big)^q\,|D\phi|^2\,dx.
    \end{split}
\end{equation*}
Thereby choosing $q\geq q_0:=2\frac{C_1+C_5}{c_0}>2$, and using \eqref{ultima1} and $\psi_\e(|D\ve|)\geq \psi_\e(|D'\ve|)$ by the monotonicity of $\psi_\e$, the above expression implies
\begin{equation}\label{zazan}
    \int_{B_{2}^+}\bigg| D\Big[\psi_\e\big(|D'\ve| \big)^{q/2}\psi_\e\big(|D\ve| \big)\,\phi  \Big]\bigg|^2\,dx\leq C\,q\, \int_{B_{2}^+}\psi_\e\big(|D'\ve| \big)^{q}\psi_\e\big(|D\ve| \big)^2\,|D\phi|^2\,dx
\end{equation}
with $C=C(n,\l,\L,i_a,s_a)>0$. 
By applying  Sobolev inequality \eqref{half:sobol} to \eqref{zazan} yields
\begin{equation}\label{moser1}
    \begin{split}
        \bigg(\int_{B_2^+} \psi_\e&\big(|D'\ve| \big)^{\kappa(q+2)}\left(\frac{\psi_\e\big(|D\ve| \big)}{\psi_\e\big(|D'\ve| \big)}\right)^2\,\phi^{2\kappa}\,dx\bigg)^{1/\kappa}
        \\
        &\leq \bigg(\int_{B_2^+} \psi_\e\big(|D'\ve| \big)^{\kappa(q+2)}\left(\frac{\psi_\e\big(|D\ve| \big)}{\psi_\e\big(|D'\ve| \big)}\right)^{2\kappa}\,\phi^{2\kappa}\,dx \bigg)^{1/\kappa}
        \\
        &\leq C\,q\, \int_{B_{2}^+}\psi_\e\big(|D'\ve| \big)^{q+2}\left(\frac{\psi_\e\big(|D\ve| \big)}{\psi_\e\big(|D'\ve| \big)}\right)^2\,|D\phi|^2\,dx.
    \end{split}
\end{equation}
where $\kappa>1$ is defined by \eqref{sob:kappa}, and in the first inequality we used the monotonicity of $\psi_\e(t)$ and that $\kappa\geq 1$. We are in the setting of Moser's iteration: fix two radii $0<r_1<r_2<2$, and for $m=0,1,2,\dots$ let
\begin{equation*}
    r_m=r_1+\frac{r_2-r_1}{2^m},\quad q_{m+1}+2=\kappa\,(q_m+2),\quad \gamma_m=q_m+2.
\end{equation*}
In Equation \eqref{moser1}, we take 
\[
q=q_m,\quad \phi=\phi_m\,\text{ such that $\phi_m\in C^\infty_c(B_{r_m})$, $\phi_m\equiv 1$ in $B_{r_{m+1}}$, and $|D\phi_m|\leq \frac{C(n)}{(r_m-r_{m+1})}$,}
\]
and if we define the measure 
\begin{equation*}
    d\mu=\left(\frac{\psi_\e\big(|D\ve| \big)}{\psi_\e\big(|D'\ve| \big)}\right)^2\,\chi_{\{D'\ve\neq 0\}}\,dx,
\end{equation*}
we arrive at
\begin{equation*}
    \big\|\psi_\e\big(|D'\ve|\big)\big\|_{L^{\gamma_{m+1}}(B^+_{r_{m+1}};\,d\mu)}\leq \left(\frac{C\,\gamma_m 4^m}{(r_2-r_1)^2} \right)^{1/\gamma_m}\,\big\|\psi_\e\big(|D'\ve|\big)\big\|_{L^{\gamma_{m}}(B^+_{r_{m}};\,d\mu)}.
\end{equation*}
with  $C=C(n,\l,\L,i_a,s_a)$. 

As $\gamma_m=\kappa^m\,\gamma_0$, iterating the above expression yields
\begin{equation}\label{finito:iter}
\begin{split}
    \big\|\psi_\e\big(|D'\ve|\big)&\big\|_{L^{\gamma_{m+1}}(B^+_{r_{m+1}};\,d\mu)}
    \\
    &\leq \left(\frac{C\,\gamma_0}{(r_2-r_1)^2} \right)^{\textstyle\frac{1}{\gamma_0}\sum_{m=0}^\infty \frac{1}{\kappa^m}}\,(4\kappa)^{\textstyle\frac{1}{\gamma_0}\sum_{m=0}^\infty \frac{m}{\kappa^m}}\big\|\psi_\e(|D'\ve|)\big\|_{L^{\gamma_0}(B^+_{r_2};\,d\mu)}
    \\
    &\leq \frac{C_1}{(r_2-r_1)^{n_0/\gamma_0}}\big\|\psi_\e(|D'\ve|)\big\|_{L^{\gamma_0}(B^+_{r_2};\,d\mu)},
    \end{split}
\end{equation}
with $C_1=C_1(n,\l,\L,i_a,s_a)>0$, where we set 
\begin{equation*}
    n_0=\begin{cases}
        n\quad &\text{if $n\geq 3$}
        \\
        \text{any number $>2$}\quad &\text{if $n= 2$}
    \end{cases}
\end{equation*}
 and in the last inequality we used that $\sum_{m=0}^\infty 1/\kappa^m=n_0/2 $, $\sum_{m=0}^\infty m/\kappa^m=c(n)<\infty$, and that $\gamma_0=q_0+2=2\frac{C_1+C_5}{c_0}+2$ depends only on $n,\l,\L,i_a,s_a$.

Letting $m\to\infty$ in \eqref{finito:iter}, and expanding the resulting expression yields
\begin{equation*}
    \begin{split}
        \sup_{B_{r_1}^+}\psi_\e\big(|D'\ve| \big)&\leq C\,\left(\frac{1}{(r_2-r_1)^{n_0}} \int_{B_{r_2}^+}\psi_\e\big(|D'\ve| \big)^{\gamma_0-2}\,\psi_\e\big(|D\ve| \big)^2dx\right)^{1/\gamma_0}
        \\
        &\leq C\,\sup_{B^+_{r_2}}\psi_\e\big(|D'\ve| \big)^{\frac{\gamma_0-2}{\gamma_0}}\left(\frac{1}{(r_2-r_1)^{n_0}} \int_{B_{r_2}^+}\psi_\e\big(|D\ve| \big)^2dx\right)^{1/\gamma_0}
        \\
         &\leq \frac{1}{2}\,\sup_{B^+_{r_2}}\psi_\e\big(|D'\ve| \big)+C'\left(\frac{1}{(r_2-r_1)^{n_0}} \int_{B_{r_2}^+}\psi_\e\big(|D\ve| \big)^2dx\right)^{1/2}.
    \end{split}
\end{equation*}
for all $0<r_1<r_2<2$, where in the last estimate we used weighted Young's inequality with exponents $(\frac{\gamma_0}{\gamma_0-2},\frac{\gamma_0}{2})$. Therefore,  using the interpolation Lemma \ref{lemma:uss}, squaring the resulting expression and recalling \eqref{ultima1} and \eqref{Bcomeb} (with Remark \ref{remark:importante} in mind), we obtain
\begin{equation}\label{oo:Dve}
    \sup_{B^+_{r_1}}B_\e\big(|D'\ve| \big)\leq \frac{C}{(r_2-r_1)^{n_0}}\,\int_{B_{r_2}^+} B_\e\big(|D\ve| \big)\,dx\,.
\end{equation}
for every $0<r_1<r_2<2$, where $C=C(n,\l,\L,i_a,s_a)>0$.
\vspace{0.2cm}

Next we observe that, owing to \eqref{oo:Dve} and \eqref{trbd:neu} (see also \eqref{unaltro}), the following boundary condition holds:
\begin{equation}\label{kkk}
B_\e\big(|D\ve| \big)\leq \frac{C}{(r_2-r_1)^{n_0}}\,\int_{B_{r_2}^+} B_\e\big(|D\ve| \big)\,dx+C\,\widetilde{B}_\e\big(|h_0| \big)\quad\text{a.e. on $B^0_{\frac{r_1+r_2}{2}}$,}
\end{equation}
for every $0<r_1<r_2<2$.
\vspace{0.1cm}

Now let $s,t$ be such that $r_1<s<t<\frac{r_1+r_2}{2}$, and let us take a test function $\phi\in C^\infty_c(B_t)$, $\phi\equiv 1$ on $B_s$, $|D\phi|\leq C(n)/(t-s)$. Let $V_\e=B_\e\big(|D\ve|\big)$, and observe that,  thanks to \eqref{reg:neuve}, \eqref{temp:Dvebddd} and the chain rule \eqref{new:chain}, we have $V_\e\in W^{1,2}(B_r^+)$ for every $0<r<2$.

Thus, via a density argument, the Bernstein-type inequality \eqref{new:Bernstein} remains valid for all test functions $\varphi\in W^{1,2}_c(B^+_{2})$ such that $\varphi\equiv0$ on $B^0_{2}$.

Consequently, thanks to \eqref{kkk}, we may take $\varphi=(V_\e-\ell)_+\,\phi^2$ as a test function in \eqref{new:Bernstein} for all $\phi$ as above, and for all $\ell>0$ such that
\begin{equation}\label{neu:ell}
    \ell>\frac{C}{(r_2-r_1)^{n_0}}\,\int_{B_{r_2}^+} B_\e\big(|D\ve| \big)\,dx+C\,\widetilde{B}_\e\big(|h_0| \big)
\end{equation} 

Using \eqref{matrix:Aecgr}, $DV_\e=D(V_{\e}-\ell)_+$ on $\{(V_{\e}-\ell)_+\neq 0\}$ and weighted Young's inequality, we get
\begin{equation*}
\begin{split}
    c&\,\int_{B_{2}^+} |D(V_{\e}-\ell)_+|^2\,\phi^2\,dx \leq \int_{B_{2}^+} \mathbb{A}_\e(x)\,DV_\e\cdot D(V_{\e}-\ell)_+\phi^2\,dx
    \\
    &= -\int_{B_{2}^+} \mathbb{A}_\e(x)\,D(V_\e-\ell)_+(V_\e-\ell)_+ \,\phi\,D\phi\,dx
    \\
    &\leq \frac{c}{2}\int_{B_{2}^+} |D(V_\e-\ell)_+|^2\,\phi^2\,dx+C\,\big(\sup|D\phi|^2\big)\,\int_{B_{2}^+} |(V_\e-\ell)_+|^2\,dx
    \end{split}
\end{equation*}
with $c,C=c,C(n,\l,\L,i_a,s_a)$. Thus, reabsorbing terms and using the properties of $\phi$, we infer 
\begin{equation}\label{la:riuso}
    \int_{B_{s}^+} |D(V_\e-\ell)_+|^2\,dx \leq \frac{C}{(t-s)^2}\int_{B_{t}^+} |(V_\e-\ell)_+|^2\,dx,
\end{equation}
for all $r_1<s<t<\frac{r_1+r_2}{2}$, and $\ell$ as in \eqref{neu:ell}
By \eqref{la:riuso} and Remark \ref{remark:evenext}, we may apply \cite[Theorem 7.3]{G03} (with $B_R^+$ in place of $Q_R$), and deduce
\begin{equation*}
    \begin{split}\sup_{B^+_{r_1}}B_\e\big(|D\ve|\big) & \leq \frac{C}{(r_2-r_1)^{n_0}}\int_{B^+_{r_2}}B_\e\big(|D\ve|\big)\,dx+C\,\widetilde{B}_\e\big(|h_0| \big),
    \end{split}
\end{equation*}
with $C=C(n,\l,\L,i_a,s_a)>0$.
We now apply the same interpolation argument of\eqref{Linf111}: from the above expression,  Young's inequality \eqref{young1} and \eqref{B2t} (and recalling Remark \ref{remark:importante}), we get
\begin{equation*}
    \begin{split}
        \sup_{B^+_{r_1}}B_\e\big(|D\ve|\big) & \leq \frac{C}{(r_2-r_1)^{n_0}}\sup_{B_{r_2}^+}b_\e\big(|D\ve| \big)\,\int_{B_{r_2}^+}|D\ve|\,dx+C\,\widetilde{B}_\e\big(|h_0| \big)
        \\
        &\leq \frac{1}{2}\sup_{B_{r_2}^+}B_\e\big(|D\ve| \big)+\frac{C'}{(r_2-r_1)^{n_0\max\{s_B,2 \} }}B_\e\bigg(\int_{B_{2}^+}|D\ve|\,dx \bigg)+C\,\widetilde{B}_\e\big(|h_0| \big).
    \end{split}
\end{equation*}
for all $0<r_1<r_2<2$. Thereby using the interpolation  Lemma \ref{lemma:uss}, taking $B_\e^{-1}$ to both sides of the resulting equation and recalling \eqref{triangle}, \eqref{come:BwB}, \eqref{be:unif} and Remark \ref{remark:importante}, we finally obtain
\begin{equation*}
    \sup_{B^+_{1}}|D\ve|\leq C\,\int_{B^+_2}|D\ve|\,dx+C\,b^{-1}\big(|h_0| \big),
\end{equation*}
for $0<\e<\e_0$ small enough ($\e_0$ depending only on $b(\cdot)$ and an upper bound on $|h_0|$), and with $C=C(n,\l,\L,i_a,s_a)>0$, that is our thesis.
\end{proof}

\subsection{H\"older continuity of \texorpdfstring{$D\ve$}{Dv}}
We now move onto the proof of H\"older continuity of $D\ve$. To this end, we introduce some  notation. For $0<r<2R$, we set
\begin{equation}\label{maxplus}
    M^+(r)=\max_{k=1,\dots,n}\sup_{B_r^+} |D_k\ve|,\quad  T^+(r)=\max_{k=1,\dots,n-1}\sup_{B_r^+} |D_k\ve|,
\end{equation}
that is the supremum of the gradient and of the tangential gradient of $\ve$, respectively.
\vspace{0.2cm}

We start by showing that a slightly modified version of the fundamental alternative holds true for the tangential derivatives. More precisely, we have the following lemma.

\begin{lemma}[The fundamental alternative for $D'\ve$]\label{lem:alttang}
    Let $\ve \in W^{1,B_\e}(B_{2R_0}^+)$ be a weak solution to \eqref{homve:neu}, and let $0<R\leq R_0/2$. Then for every $\tau\in (0,1)$, there exists $\eta_{\tau}=\eta_{\tau}(n,\l,\L,i_a,s_a,\tau)\in (0,1)$ such that
    \begin{equation}\label{alt:tang}
        \text{either }\, T^+(\tau R)\geq M^+(2R)/4\quad \text{or}\quad T^+(\tau R)\leq \eta_\tau\,M^+(2R).
    \end{equation}
\end{lemma}

\begin{proof}
    Let $\gamma \in (0,1)$ be a fixed constant. First we observe that, starting from \eqref{neu:Dkve} (in place of the weak formulation of \eqref{eq:classic}), we may repeat verbatim the proof of Lemmas \ref{lemma:parti}- \ref{lemma:levquadratic}, (replacing $B_r$ with $B_r^+$, and $M(2R)$ with $M^+(2R)$) and find that for every $\kappa,\ell$ such that
\begin{equation*}
    \gamma/8\, M^+(2R)\leq \kappa<\ell\leq M^+(2R),
\end{equation*}
the following integral inequalities hold:
\begin{equation}\label{oaa:neu}
        \int\limits_{\{\kappa\leq D_k\ve<\ell \}\cap B_{r_1}^+}|D(D_k\ve)|^2\,dx\leq C_\gamma\,\frac{\big(M^+(2R)\big)^2}{(r_2-r_1)^2}\big|\{D_k\ve<\ell\}\cap B_{r_2}^+\big|\,,
\end{equation}
and
\begin{equation}\label{oaa:neu1}
    \int\limits_{\{D_k\ve>\kappa\}\cap B_{r_1}^+  } |D(D_k\ve)|^2\,dx\leq C_\gamma\,\frac{(M^+(2R)-\kappa)^2}{(r_2-r_1)^2}\,|\{D_k\ve>\kappa\}\cap B_{r_2}^+|
\end{equation}
    for every $0<r_1<r_2\leq 2R$, and for all $k=1,\dots,n-1$. 
    
    Starting from these inequalities, the proof of \eqref{alt:tang} is almost identical to that of Lemma \ref{lem:1alt} and Lemmas \ref{lemma:alt3small}-\ref{lemma:altfin}, save that we have to suitably modify the radii in the iteration, and use \eqref{oaa:neu}-\eqref{oaa:neu1} in place of \eqref{levels:eq}-\eqref{oiii}, respectively.
Specifically, regarding the first alternative, we follow the iterative scheme of Lemma \ref{lem:1alt}, i.e., Equations \eqref{start:iter}-\eqref{fine:iter}, with the following modifications: we work with $B^+_R, M^+(2R)$ in place of $B_R, M(2R)$, respectively, we use \eqref{oaa:neu} in place of \eqref{levels:eq},  and we consider the sequence of radii \begin{equation*}
    R_m=\tau\,R+(1-\tau)\,\frac{R}{2^m},\quad m=0,1,2,\dots.
\end{equation*}
With this choice, since along the estimates we also use the inequalities $\tau R \leq R_{m} \leq 2R$, the constants arising in the iteration will depend on $\tau$ as well.
In the end, we obtain a parameter $\mu_\tau=\mu_\tau(n,\l,\L,i_a,s_a,\tau)\in (0,2^{-n-1})$ small enough such that, if
\begin{equation}\label{fails}
\begin{split}
    |\{D_k\ve<\mathrm{M}^+(2R)/2 \}&\cap B^+_{2R}|\leq \mu_\tau\,|B^+_{2R}|,
    \\
    &\text{or}
    \\
    |\{D_k\ve>-\mathrm{M}^+(2R)/2 &\}\cap B^+_{2R}|\leq \mu_\tau\,|B^+_{2R}|
\end{split}
\end{equation}
    for some $k=1,\dots,n-1$, then 
    \begin{equation*}
        |D_k\ve|\geq \mathrm{M}^+(2R)/4 \quad\text{in $B^+_{\tau R}$,}
    \end{equation*}
    so the first alternative in \eqref{alt:tang} is true.
    
    If, on the other hand, \eqref{fails} fails for all $k=1,\dots,n-1$, then we follow the arguments of Lemmas \ref{lemma:alt3small}-\ref{lemma:altfin}: specifically, as in Remark \ref{remark:nu0}, we find 
    $\nu_\tau=\nu_\tau(n,\l,\L,i_a,s_a,\tau)\in (1/2,1)$ of the form
\[
\nu_\tau=\Big(\frac{1-\mu_\tau}{1-\mu_\tau/2}\Big)^{1/n}
\]
 such that
\begin{equation*}
\begin{split}
    |\{D_k\ve>M^+(2R)/2\}\cap &\,B^+_{2\nu_\tau R}|\leq \Big(1-\frac{\mu_\tau}{2} \Big)\,|B^+_{2\nu_\tau R}|
\\
&\text{and}
\\
|\{D_k\ve<-M^+(2R)/2\}&\cap B^+_{2\nu_\tau R}|\leq \Big(1-\frac{\mu_\tau}{2} \Big)\,|B^+_{2\nu_\tau R}|.
    \end{split}
\end{equation*}
Then we repeat the proof of Lemmas \ref{lemma:alt3small}-\ref{lemma:alt4small}, with $A^+_s=\{D_k\ve>\kappa_s\}\cap B^+_{2\nu_\tau R}$ and using \eqref{oaa:neu1} in place of \eqref{oiii}; by so doing, we find that for every $\theta_0\in (0,1)$, there exists $s_\tau=s_\tau(n,\l,\L,i_a,s_a,\tau,\theta_0)\in \N$ large enough such that
\begin{equation}\label{purtroppo}
\begin{split}
    \Big|\big\{D_k\ve>\Big(1-\frac{1}{2^{s_\tau}} \Big)&\,M^+(2R)\big\}\cap B^+_{2\nu_\tau R} \Big|\leq \theta_0\,|B^+_{2\nu_\tau R}|
    \\
&\text{and}
\\
\Big|\big\{D_k\ve<-\Big(1-\frac{1}{2^{s_\tau}}& \Big)\,M^+(2R)\big\}\cap B^+_{2\nu_\tau R} \Big|\leq \theta_0\,|B^+_{2\nu_\tau R}|.
    \end{split}
\end{equation}
Finally, we fix $\theta_\tau=\theta_\tau(n,\l,\L,i_a,s_a,\tau)$, and we repeat the argument of Lemma \ref{lemma:altfin} with radii
\begin{equation*}
    R_m=\tau\,R+(2\nu_\tau R-\tau R)\Big(\frac{1}{2^m} \Big)\quad m=0,1,2,\dots,
\end{equation*}
 so that, by using \eqref{oaa:neu1} and \eqref{purtroppo}, we determine a parameter $\eta_\tau=\eta_\tau(n,\l,\L,i_a,s_a,\tau)\in (0,1)$ such that 
  \begin{equation*}
      |D_k\ve|<\eta_\tau\,\mathrm{M}^+(2R)\quad\text{in $B^+_{\tau R}$}
  \end{equation*}
for all $k=1,\dots,n-1$; this shows the second alternative \eqref{alt:tang} is valid, which ends the proof.
\end{proof}

We now control the oscillation of the tangential gradient $D'\ve$, showing that it can be made arbitrarily small by reducing the size of $B_r^+$. To this end, let us define 
\begin{equation}\label{def:osctang}
    \omega'(r)=\max_{k=1,\dots,n-1}\operatorname*{osc}_{ B_r^+} D_k\ve,
\end{equation}
the oscillation of the tangential gradient $D'\ve=(D_1\ve,\dots, D_{n-1}\ve)$.

\begin{lemma}\label{lemma:oscDive}
    Let $\ve\in W^{1,B_\e}(B^+_{2R_0})$ be a solution to \eqref{homve:neu}, and let $0<R\leq R_0/2$. Then for any $\gamma\in (0,1)$, there exists $\tau_\gamma\in (0,1)$ depending on $n,\l,\L,i_a,s_a$ and $\gamma$ such that
    \begin{equation}\label{osc:Dineu}
       \omega'\big(\tau_\gamma R \big)\leq \frac{\gamma}{2} M^+(2R).
    \end{equation}
\end{lemma}

\begin{proof} For clarity of exposition, we divide the proof into a few steps.
\vspace{0.2cm}

\noindent \textit{Step 1.} Assume $\vrho\in (0,2R]$ is such that
\begin{equation}\label{hip:1}
    \omega'(\vrho)> \frac{\gamma}{2} M^+(2R).
\end{equation}
    In particular, by definition of $T^+(\vrho)$ in \eqref{maxplus}, this implies 
    \begin{equation}\label{hip:2}
        T^+(\vrho)>\gamma\, M^+(2R).
    \end{equation}
   Since $T^+(\vrho)\leq  M^+(2R)$, \eqref{hip:2}  together with \eqref{oaa:neu}-\eqref{oaa:neu1} entails
\begin{equation}\label{oaa:neu2}
        \int\limits_{\{\kappa\leq D_k\ve<\ell \}\cap B_{r_1}^+}|D(D_k\ve)|^2\,dx\leq C'_\gamma\,\frac{\big(T^+(\vrho)\big)^2}{(r_2-r_1)^2}\big|\{D_k\ve<\ell\}\cap B_{r_2}^+\big|\,,
\end{equation}
and
\begin{equation}\label{oaa:neu3}
    \int\limits_{\{D_k\ve>\kappa\}\cap B_{r_1}^+  } |D(D_k\ve)|^2\,dx\leq C'_\gamma\,\frac{\big(T^+(\vrho)-\kappa\big)^2}{(r_2-r_1)^2}\,|\{D_k\ve>\kappa\}\cap B_{r_2}^+|,
\end{equation}
   for every $T^+(\vrho)/8\leq \kappa<\ell\leq T^+(\vrho)$, and for every $0<r_1<r_2<2R$, where $C'_\gamma$ depends on $n,\l,\L,i_a,s_a,\gamma$. We now divide the remainder of Step 1 into a number of cases.
\vspace{0.2cm}

\noindent\textit{Case 1. The first alternative.}
Starting from \eqref{oaa:neu2}, we repeat the proof of Lemma \ref{lem:1alt}.
In particular, we reproduce the iterative scheme given by 
\eqref{start:iter}--\eqref{fine:iter}, with the following modifications:
we keep track of the constants that, in view of \eqref{oaa:neu2}, now depend also on \(\gamma\),
we apply Lemma~\ref{lemma:levels} on \(B_\varrho^+\)
and, in place of \eqref{start:iter}, we consider the sequences
   \begin{equation*}
    \kappa_m=\frac{T^+(\vrho)}{4}+\frac{T^+(\vrho)}{2^{m+2}},\quad R_m=\frac{\vrho}{2}+\frac{\vrho}{2^{m+1}},\quad m=0,1,2,\dots.
\end{equation*}

\noindent By so doing, we determine a parameter $\mu_\gamma\in (0,2^{-n-1})$ depending on $n,\l,\L,i_a,s_a,\gamma$ such that, if
\begin{equation*}
\begin{split}
    \big|\big\{D_k\ve<T^+(\vrho)/2\big\}&\cap B^+_\vrho\big|\leq \mu_\gamma\,|B^+_\vrho|
    \\
   & \text{or}
    \\
    \big|\big\{D_k\ve>-T^+(\vrho)/2&\big\}\cap B^+_\vrho\big|\leq \mu_\gamma\,|B^+_\vrho|
    \end{split}
\end{equation*}
hold for some $k=1,\dots,n-1$, then 
\begin{equation*}
   |D_k\ve|\geq \frac{T^+(\vrho)}{4}\quad\text{in $B^+_{\vrho/2}$}.
\end{equation*}

\noindent\textit{Case 2. The second alternative.} In the complementary case, that is if
\begin{equation*}
    \begin{split}
    \big|\big\{D_k\ve<T^+(\vrho)/2\big\}&\cap B^+_\vrho\big|> \mu_\gamma\,|B^+_\vrho|
    \\
   & \text{and}
    \\
    \big|\big\{D_k\ve>-T^+(\vrho)/2&\big\}\cap B^+_\vrho\big|> \mu_\gamma\,|B^+_\vrho|,\quad\text{hold for all $k=1,\dots,n-1$},
    \end{split}
\end{equation*}
we follow the proof of Lemmas \ref{lemma:alt3small}-\ref{lemma:altfin}. First, as in Remark \ref{remark:nu0}, we find a parameter $\nu_\gamma=\nu_\gamma(n,\l,\L,i_a,s_a,\gamma)\in (1/2,1)$ of the form
\[
\nu_\gamma=\bigg(\frac{1-\mu_\gamma}{1-\mu_\gamma/2}\bigg)^{1/n}
\]
such that
\begin{equation*}
\begin{split}
    \big|\{D_k\ve>T^+(\vrho)/2\}&\cap B^+_{\nu_\gamma \vrho} \big|\leq \left(1-\frac{\mu_\gamma}{2}\right)|B^+_{\nu_\gamma\vrho}|
    \\
     & \text{and}
    \\
    \big|\{D_k\ve<-T^+(\vrho)/&2\}\cap B^+_{\nu_\gamma \vrho} \big|\leq \left(1-\frac{\mu_\gamma}{2}\right)|B^+_{\nu_\gamma\vrho}|,
    \end{split}
\end{equation*}
for all $k=1,\dots,n-1$. Then, taking advantage of \eqref{oaa:neu3}, we repeat the proof of Lemma \ref{lemma:alt3small} with sequences 
\begin{equation*}
 \kappa_s=\Big(1-\frac{1}{2^s}\Big)\,T^+(\vrho),   \quad A^+_s=\big\{D_k\ve>\kappa_s\big\}\cap B^+_{\nu_\gamma\,\vrho},
\end{equation*}
and find that for every $\theta_0\in (0,1)$, there exists  $s_\gamma=s_\gamma(n,\l,\L,i_a,s_a,\gamma,\theta_0)\in \N$ large enough such that
\begin{equation*}
\begin{split}
            \bigg|\bigg\{D_k\ve>\left(1-\frac{1}{2^{s_\gamma}} \right)\,T^+(\vrho) \bigg\}&\cap B^+_{\nu_\gamma\vrho} \bigg|\leq \theta_0\, |B^+_{\nu_\gamma\vrho}|\,
            \\
                 & \text{and}
    \\
        \bigg|\bigg\{D_k\ve<-\left(1-\frac{1}{2^{s_\gamma}} \right)\,T^+(\vrho)& \bigg\}\cap B^+_{\nu_\gamma \vrho} \bigg|\leq \theta_0\, |B^+_{\nu_\gamma\vrho}|
            \end{split}
\end{equation*}
for all $k=1,\dots,n-1$. Then, by using \eqref{oaa:neu3}, we follow the iterative scheme of \eqref{azz:in}-\eqref{azz:fin}, with sequences 
\begin{equation*}
      \kappa_m=\left(1-\frac{1}{2^{s_\gamma}} \right)\,T^+(\vrho)+\left( 1-\frac{1}{2^m}\right)\,\frac{\mathrm{T^+(\vrho)}}{2^{s_\gamma+1}},\quad R_m=\frac{\vrho}{2}+(\nu_0\vrho-\vrho/2)\left(\frac{1}{2^m} \right),\quad m=0,1,\dots.
\end{equation*}
We thus find a parameter $\eta_\gamma=(1-\frac{1}{2^{s_\gamma+1}})\in (0,1)$ depending on $n,\l,\L,i_a,s_a,\gamma$ such that
\begin{equation*}
    T^+(\vrho/2)\leq \eta_\gamma\,T^+(\vrho).
\end{equation*}

\noindent\textit{Step 2.} We now combine the two alternatives. When Case 1 occurs we have
\begin{equation*}
    M^+(2R)\geq |D\ve|\geq |D_{k_0}\ve|\geq T^+(\vrho)\stackrel{\eqref{hip:2}}{\geq} \gamma\,M^+(2R)\quad\text{in $B_{\vrho/2}^+$,}
\end{equation*}
for some $k_0=1,\dots,n-1$. It then follows by \eqref{ultra:utile} and Remark \ref{remark:importante}, that
\begin{equation}
    c^\ast_\gamma\leq\frac{a_\e\big(|D\ve| \big)}{a_\e\big(M^+(2R) \big)}\leq C^\ast_\gamma\quad \text{in $B^+_{\vrho/2}$,}
\end{equation}
with $c^\ast_\gamma,C_\gamma^\ast$ depending on $n,\l,\L,i_a,s_a$ and $\gamma$. Therefore, by \eqref{coer:Ae}, equation \eqref{neu:Dkve} is linear, uniformly elliptic in $B_{\vrho/2}^+$, for all $k=1,\dots,n-1$, with ellipticity constants depending on $c_\gamma^*,C_\gamma^*$.  Recalling Remark \ref{remark:evenext}, we may appeal to the De Giorgi-Nash-Moser theory \cite[Equation (7.44)]{G03}, and get
\begin{equation*}
   \omega'(\vrho/8)\leq  \eta^*_{\gamma}\,\omega'(\vrho/2)\leq \eta^*_{\gamma}\,\omega'(\vrho)
\end{equation*}
for some $\eta^\ast_\gamma\in (0,1)$ depending on $n,\l,\L,i_a,s_a,\gamma$.
In the complementary case, i.e.,  whenever Case 2 occurs, we have
\begin{equation*}
    T^+(\vrho/8)\leq \eta_\gamma\,T^+(\vrho)\,.
\end{equation*}
All in all, we have shown that if \eqref{hip:1} occurs, then either
\begin{align}
    \omega'(\vrho/8)&\leq \delta_\gamma\, \omega'(\vrho) \label{again:alt}
    \\
    &\text{or}\nonumber
    \\
        T^+(\vrho/8)&\leq \delta_{\gamma}\,T^+(\vrho),\label{again:alt1}
\end{align}
happen, where we set $\delta_{\gamma}=\max\{\eta_\gamma,\,\eta^\ast_\gamma\}\in (0,1)$.
\vspace{0.2cm}

\noindent\textit{Step 3. Iteration and conclusion.}
For $\gamma$ as in the statement of the lemma, and $\delta_\gamma$ defined in \eqref{again:alt}-\eqref{again:alt1}, we fix $m_\gamma=m_\gamma(n,\l,\L,i_a,s_a,\gamma)>0$ such that
\begin{equation}\label{m:gamma}
    (\delta_\gamma)^{m_\gamma}\leq \gamma/4,
\end{equation}
and we claim that
\begin{equation*}
    \tau_\gamma=\tau_\gamma(n,\l,\L,i_a,s_a,\gamma)\coloneqq 8^{-(2m_\gamma+1)}
\end{equation*}
satisfies \eqref{osc:Dineu}. Let us consider the sequence of radii
\begin{equation*}
    \vrho_m=\frac{R}{8^m},\quad m=1,2,\dots,
\end{equation*}
and we check whether \eqref{hip:1} holds for $\vrho=\vrho_m$, $m=1,2,\dots$. If \eqref{hip:1} fails for some \(\varrho_m\), we stop the iteration; otherwise, we test the next radius \(\varrho_{m+1}\). 

We perform this procedure up to \(2m_\gamma\) times, so only two alternatives are possible:
either 
\begin{equation*}
    \omega'(\vrho_{m_\ast})\leq \frac{\gamma}{2} M^+(2R),\quad\text{for some $m_\ast\leq 2m_\gamma$,}
\end{equation*}
and thus \eqref{osc:Dineu} is verified, since
\[
\omega'(\tau_\gamma R)
= \omega'(\vrho_{2m_\gamma+1})
\le \omega'(\vrho_{m_\ast}).
\]
Otherwise, one has
\begin{equation*}
\omega'(\vrho_m)>\frac{\gamma}{2}\,M^+(2R),
\qquad \text{for all } m=1,2,\dots,2m_\gamma,
\end{equation*}
and we are therefore left to analyze this situation. In this case, by Step~2, for every $m=1,\dots,2m_\gamma$, either
\eqref{again:alt} or \eqref{again:alt1} occurs.  We thus determine two sets of indices
\begin{equation*}
\begin{split}
    &\mathcal{I}_{\omega'}=\big\{m\in \{1,\dots,2m_\gamma\}: \eqref{again:alt}\, \text{occurs for  } \vrho=\vrho_m\big\}
    \\
    &\mathcal{I}_{T^+}=\big\{m\in \{1,\dots,2m_\gamma\}: \eqref{again:alt1}\, \text{occurs for  } \vrho=\vrho_m\big\}.
    \end{split}
\end{equation*}
whose cardinalities satisfy $| \mathcal{I}_{\omega'}|+| \mathcal{I}_{T^+}|=2m_\gamma$. 

Necessarily, either $| \mathcal{I}_{\omega'}|\geq m_\gamma$ or $| \mathcal{I}_{T^+}|\geq m_\gamma$ must happen. Suppose the first case holds, so we may find  indices $j_1>j_2>\dots >j_{m_\gamma}\geq 2m_\gamma$, $j_m\in I_{\omega'}$ for all $m=1,\dots,m_\gamma$. By definition of $I_{\omega'}$ and $\vrho_m$, we have
\begin{equation*}
    \omega'(\vrho_{j_{m+1}})\leq \omega'(\vrho_{j_{m}}/8)\leq \delta_\gamma\,\omega'(\vrho_{j_{m}}),\quad \text{for all }\,m=1,\dots,m_\gamma.
\end{equation*}
Using that $\tau_\gamma R=\vrho_{2m_\gamma+1}\leq \vrho_{j_{m_\gamma+1}}$, iterating the above inequality yields 
\begin{equation*}
\begin{split}
    \omega'(\tau_\gamma R) &\leq \omega'(\vrho_{j_{m_\gamma+1}}) \leq \delta_\gamma\,\omega'(\vrho_{j_{m_\gamma}})\leq \dots\leq (\delta_\gamma)^{m_\gamma}\,\omega'(\vrho_{j_1})
    \\
    &\leq 2\,(\delta_\gamma)^{m_\gamma}\,M^+(2R)\stackrel{\eqref{m:gamma}}{\leq }\frac{\gamma}{2}\,M^+(2R),
    \end{split}
\end{equation*}
that is \eqref{osc:Dineu}. In the other case, that is if $|\mathcal{I}_{T^+}|\geq m_\gamma$, we may find a sequence $i_1>i_2>\dots i_{m_\gamma}\geq 2m_\gamma$ such that $i_m\in \mathcal{I}_{T^+}$ for all $m=1,\dots,m_\gamma$. So, by definition of $\mathcal{I}_{T^+}$ and $\vrho_m$, we have
\begin{equation*}
    T^+(\vrho_{i_{m+1}})\leq T^+(\vrho_{i_{m}}/8)\leq \delta_\gamma\,T^+(\vrho_{i_{m}}),\quad \text{for all }\,m=1,\dots,m_\gamma.
\end{equation*}
Thereby iterating and using that $\tau_\gamma R\leq \vrho_{i_{m_\gamma+1}}$, we obtain
\begin{equation*}
\begin{split}
    \omega'(\tau_\gamma R)&\leq \omega'(\vrho_{i_{m_\gamma+1}})\leq 2\, T^+(\vrho_{i_{m_\gamma+1}})\leq 2\,\delta_\gamma\, T^+(\vrho_{i_{m_\gamma}})
    \\
    &\leq \dots\leq 2\,(\delta_\gamma)^{m_\gamma}T^+(\vrho_{i_1})\stackrel{\eqref{m:gamma}}{\leq }\frac{\gamma}{2}\,M^+(2R),
    \end{split}
\end{equation*}
and \eqref{osc:Dineu} is proven in this case as well, thus concluding the proof. 
\end{proof}

We now turn onto the H\"older continuity of $D_n\ve$. To this end, we first derive some integral inequalities as in  Lemma \ref{lemma:parti}.

\begin{lemma}\label{lemmaNEU:DG}
    Let $\ve\in W^{1,2}(B^+_{2R_0})$ be solution \eqref{homve:neu}, and let $0<R\leq R_0$. Then, for every  $\phi\in C^\infty_c(B_{2R})$ and  for every Lipschitz function $g:\R\to \R$ such that $g'(t)\geq 0$ for a.e. $t\in \R$ and
    \begin{equation}\label{g:const}
        g'(D_n\ve)\equiv 0\quad \text{on $B_{2R}^0\cap \mathrm{spt}\,\phi$,}
    \end{equation}
the integral inequalities
    \begin{equation}\label{qqq}
    \begin{split}
        \int_{B^+_{2R}} a_\e&\big(|D\ve| \big)\,|D(D_n\ve)|^2\,g'(D_n\ve)\,\phi^2\,dx\leq  C\,\int_{B^+_{2R}}a_\e\big(|D\ve| \big)|D\ve|^2\,g'(D_n\ve)\,|D\phi|^2\,dx
        \\
        &+C\,\int_{B^+_{2R}}b_\e\big(|D\ve| \big)\,|g(D_n\ve)|\,|D^2\phi^2|\,dx+C\,|h_0|\,\bigg|\int_{B^0_{2R}} D_{n}\phi^2\,g(D_n\ve)\,dx'\bigg|\,,
        \end{split}
    \end{equation}
    and
    \begin{equation}\label{quai}
    \begin{split}
        \int_{B^+_{2R}} a_\e \big(|D\ve| \big)\,|D(D_n\ve)|^2\,g'(D_n\ve)\,\phi^2\,dx\leq  C\,\int_{B^+_{2R}}&a_\e\big(|D\ve| \big)|D\ve|^2\,g'(D_n\ve)\,|D\phi|^2\,dx
        \\
        +C\,\int_{B^+_{2R}}b_\e\big(|D\ve| \big)\,|g(D_n\ve)|\,|D^2\phi^2|\,dx&+ C\,|h_0|\,\int_{B^+_{2R}}|D^2\phi^2|\,|g(D_n\ve)|\,dx
        \\&+C|h_0|^2\int_{B^+_{2R}}\frac{g'(D_n\ve)}{a_\e(|D\ve|)}|D\phi|^2\,dx,
    \end{split}
    \end{equation}
  hold true, with $C=C(n,\l,\L,i_a,s_a)>0$. 
    
    If in addition $g(D_n\ve)\equiv 0$ on $B^0_{2R}\cap \mathrm{spt}\,\phi$, and~$|\{g\neq 0\}\cap \{g'= 0\}|=0$, then
    \begin{equation}\label{quai1}
    \begin{split}
        \int_{B^+_{2R}} &a_\e\big( |D\ve|\big)\,|D(D_n\ve)|^2\,g'(D_n\ve)\,\phi^2\,dx
        \\
        &\leq C\,\int_{B^+_{2R}\cap \{g(D_n\ve)\neq 0\}\cap\{D(D_n\ve)\neq 0\}} a_\e(|D\ve|)\,\frac{g^2(D_n\ve)}{g'(D_n\ve)}\,|D\phi|^2\,dx.
        \end{split}
    \end{equation}
\end{lemma}

\begin{proof}
    We test equation \eqref{neu:Dnve} with $\eta=g(D_n\ve)\,\phi^2$, thus getting
    \begin{equation}\label{bdn:st}
        \begin{split}
            \int_{B^+_{2R}}&\nabla_\xi \Ae(D\ve)\,D(D_n\ve)\cdot D(D_n\ve)\,g'(D_n\ve)\,\phi^2\,dx
            \\
            &+\int_{B^+_{2R}}\nabla_\xi \Ae(D\ve)\,D(D_n\ve)\cdot D\phi^2\, g(D_n\ve)\,dx=-\int_{B_{2R}^0}\Ae'(D\ve)\cdot D'(g(D_n\ve)\phi^2\big)\,dx'
        \end{split}
    \end{equation}
By \eqref{g:const}, $g(D_n\ve)$ is constant on $B^0_{2R}\cap \mathrm{spt}\phi$, so we have
\begin{equation*}
    -\int_{B_{2R}^0}\Ae'(D\ve)\cdot D'(g(D_n\ve)\phi^2\big)\,dx'=-\int_{B_{2R}^0}\Ae'(D\ve)\cdot D'\phi^2\,g(D_n\ve)\,dx',
\end{equation*}
    whilst using the chain rule and integrating by parts, we find
    \begin{equation*}
        \begin{split}
            \int_{B^+_{2R}}\nabla_\xi &\Ae(D\ve)\,D(D_n\ve)\cdot D\phi^2\, g(D_n\ve)\,dx=\int_{B^+_{2R}}D_n\big(\Ae(D\ve)\big)\cdot D\phi^2\, g(D_n\ve)\,dx
            \\
            =&\int_{B^+_{2R}}\Ae(D\ve)\cdot D_n\big( D\phi^2\,g(D_n\ve)\big)\,dx-\int_{B_{2R}^0}\Ae(D\ve)\cdot D\phi^2\,g(D_n\ve)\,dx
            \\
            =&\int_{B^+_{2R}}\Ae(D\ve)\cdot D(D_n\phi^2)\,g(D_n\ve)\,dx+\int_{B^+_{2R}}\Ae(D\ve)\cdot D\phi^2\, D_{nn}\ve\,g'(D_n\ve)\,dx
            \\
            &-\int_{B_{2R}^0}\Ae'(D\ve)\cdot D'\phi^2\,g(D_n\ve)\,dx'+h_0\int_{B^0_{2R}} D_{n}\phi^2\,g(D_n\ve)\,dx',
        \end{split}
    \end{equation*}
    where in the last equality we used the boundary condition $\Ae^n(D\ve)=-h_0$ a.e. on $B^0_{2R}$.

Coupling the three identities above yields
\begin{equation}\label{chiamal}
    \begin{split}
        (O)\coloneqq\int_{B^+_{2R}}&\nabla_\xi \Ae(D\ve)\,D(D_n\ve)\cdot D(D_n\ve)\,g'(D_n\ve)\,\phi^2\,dx
        \\
        =&\int_{B^+_{2R}}\Ae(D\ve)\cdot D(D_n\phi^2)\,g(D_n\ve)\,dx+\int_{B^+_{2R}}\Ae(D\ve)\cdot D\phi^2 D_{nn}\ve\,g'(D_n\ve) \,dx
        \\
        & -h_0\,\int_{B^0_{2R}} D_{n}\phi^2\,g(D_n\ve)\,dx'\eqqcolon  (I)+(II)+(III).
    \end{split}
\end{equation}
    We now repeat the computations of Lemma \ref{lemma:parti}. By \eqref{coer:Ae}, we get
    \begin{equation}\label{O1}
        (O)\geq c_1\,\int_{B_{2R}^+} a_\e(|D\ve|)\,|D(D_n\ve)|^2\,g'(D_n\ve)\,\phi^2\,dx;
    \end{equation}
with $c_1=c_1(n,\l,\L,i_a,s_a)$, and owing to \eqref{coAe:gr}, we have
    \begin{equation*}
        |(I)|\leq C\,\int_{B_{2R}^+}b_\e(|D\ve|)\,|g(D_n\ve)|\,|D^2\phi^2|\,dx,
    \end{equation*}
    and by \eqref{coAe:gr}, \eqref{def:bBe} and weighted Young's inequality, we find
    \begin{equation*}
    \begin{split}
        |(II)|\leq &\, C\,\int_{B_{2R}^+}b_\e(|D\ve|)|D(D_n\ve)| g'(D_n\ve)\,|D\phi|\,\phi\,dx
        \\
        \leq &\, \frac{c_1}{4}\int_{B_{2R}^+} a_\e(|D\ve|)\,|D(D_n\ve)|^2\,g'(D_n\ve)\,\phi^2\,dx
        \\
        &+C'\,\int_{B^+_{2R}}a_\e(|D\ve|)\,|D\ve|^2g'(D_n\ve)\,|D\phi|^2\,dx,
        \end{split}
    \end{equation*}
    with $C,C'=C,C'(n,\l,\L,i_a,s_a)>0$. Merging the content of the four inequalities above, and reabsorbing terms we get \eqref{qqq}.  To obtain \eqref{quai}, we use \eqref{g:const} and the divergence theorem:
    \begin{equation}\label{diverg:0}
    \begin{split}
        h_0&\int_{B^0_{2R}} D_{n}\phi^2\,g(D_n\ve)\,dx'=-h_0\int_{B^+_{2R}}D_n\big( D_n\phi^2\,g(D_n\ve)\big)\,dx
        \\
        &=-h_0\int_{B^+_{2R}}D_{nn}\phi^2 g(D_n\ve)\,dx-2h_0\int_{B^+_{2R}} D_n\phi\,\phi\,g'(D_n\ve)\,D_{nn}\ve\,dx
        \end{split}
    \end{equation}
and then estimate $|h_0\int_{B^+_{2R}}D_{nn}\phi^2 g(D_n\ve)\,dx|\leq |h_0|\int_{B^+_{2R}}|D^2\phi^2|\,|g(D_n\ve)|\,dx$, and  via weighted Young's inequality
\begin{equation}\label{O2}
\begin{split}
    \bigg|2h_0\int_{B^+_{2R}} D_n\phi\,\phi\,g'(D_n\ve)\,D_{nn}\ve\,dx \bigg|\leq  &\,\delta\, \int_{B_{2R}^+} a_\e(|D\ve|)\,|D(D_n\ve)|^2\,g'(D_n\ve)\,\phi^2\,dx
    \\
    &+C_\delta\,|h_0|^2\int_{B^+_{2R}}\frac{g'(D_n\ve)}{a_\e(|D\ve|)}|D\phi|^2\,dx,
    \end{split}
\end{equation}
Choosing $\delta\in (0,1)$ small enough to reabsorb terms to the left hand side of \eqref{qqq} yields \eqref{quai}.

Finally, if $g(D_n\ve)\equiv 0$ on $B^0_{2R}\cap \mathrm{spt}\,\phi$, the last integral in \eqref{bdn:st} is zero, so \eqref{quai1} follows via the same computations as in the proof of Equation \eqref{ooooo}.
\end{proof}

\begin{remark}[Another integral inequality for $D_n\ve$]\label{rem:identity}
\rm{  Let $\mathbf{V}\in \R^n$ be a constant vector, and set~$\tilde{\A_\e}(\xi):=\A_\e(\xi)-\mathbf{V}$. Then, by \eqref{homve:neu} we have that $v_\e$ satisfies the Neumann problem 
    \begin{equation*}
            \begin{cases}
         -\mathrm{div}\big( \tilde\Ae(D\ve)\big)=0\quad\text{in $B_{2R}^+$}
        \\
        \tilde\Ae(Dv)\cdot e_n+\tilde{h}_0=0\quad \text{on $B_{2R}^0$}.
    \end{cases}
    \end{equation*}
    where $\tilde h_0=h_0+\mathbf{V}\cdot e_n$. Hence, testing the weak formulation with  $\phi=D_n\eta$, and performing the same computations as in \eqref{lauso:Dnv}-\eqref{neu:Dnve}, we find
    \begin{equation*}
        \int_{B_{2R}^+}\nabla_\xi \Ae(D\ve)\,D(D_n\ve)\cdot D\eta\,dx=-\int_{B_{2R}^0}\tilde{\Ae'}(D\ve)\cdot D'\eta\,dx'.
    \end{equation*}
Now let $g:\R\to \R$ be a Lipschitz function such that $g'(t)\geq 0$ for a.e. $t$, and satisfying \eqref{g:const}. Testing the above identity with $g(D_n\ve)\,\phi^2$, we may repeat the computations of \eqref{chiamal} and of \eqref{diverg:0}, the only difference being that $\Ae$ and $h_0$ are replaced by $\tilde\Ae$ and $\tilde{h}_0$, respectively. By doing so, we arrive at
\begin{equation*}
    \begin{split}
\int_{B^+_{2R}}&\nabla_\xi \Ae(D\ve)\,D(D_n\ve)\cdot D(D_n\ve)\,g'(D_n\ve)\,\phi^2\,dx
        \\
        =&\int_{B^+_{2R}}\tilde\Ae(D\ve)\cdot D(D_n\phi^2)\,g(D_n\ve)\,dx+\int_{B^+_{2R}}\tilde\Ae(D\ve)\cdot D\phi^2 D_{nn}\ve\,g'(D_n\ve) \,dx
        \\
        &+\tilde h_0\int_{B^+_{2R}}D_{nn}\phi^2 g(D_n\ve)\,dx+2\tilde h_0\int_{B^+_{2R}} D_n\phi\,\phi\,g'(D_n\ve)\,D_{nn}\ve\,dx.
    \end{split}
\end{equation*}
    We estimate the left-hand side and the last integral on the right hand side via \eqref{O1} and \eqref{O2}, thus obtaining
\begin{equation}\label{inednv}
    \begin{split}
        \int_{B^+_{2R}}&a_\e\big(|D\ve|\big)\,|D(D_n\ve)|^2\,g'(D_n\ve)\,\phi^2\,dx 
        \\
        \leq &\,C\,\int_{B^+_{2R}}\big|\tilde\Ae(D\ve)\big|\,|D^2\phi^2|\,|g(D_n\ve)|\,dx+C\,\int_{B^+_{2R}}\big|\tilde\Ae(D\ve)\big|\,| D\phi\,\phi|\,|D_{nn}\ve|\,g'(D_n\ve) \,dx
               \\
               &+C\,|\tilde h_0|\,\int_{B^+_{2R}}|D^2\phi^2|\,| g(D_n\ve)|\,dx+C\,|\tilde h_0|\,\int_{B^+_{2R}} |D\phi\,\phi|\,g'(D_n\ve)\,|D_{nn}\ve|\,dx,
    \end{split}
\end{equation}
with constant $C=C(n,\l,\L,i_a,s_a)>0$.
}
\end{remark}

\begin{lemma}[The first alternative for $D_n\ve$]\label{lemma:alt1Dn}
    Let $\ve\in W^{1,B_\e}(B_{2R_0}^+)$ be a solution to \eqref{homve:neu}, and let $0<R\leq R_0/2$. Then there exist $\mu_0\in (0,2^{-n-1})$ and $\gamma_0\in (0,1)$, both depending on  $n,\l,\L,i_a,s_a$ such that if $\gamma\leq \gamma_0$, and $\tau_\gamma\in (0,1)$ is the parameter given by Lemma \ref{lemma:oscDive} the following holds: if
\begin{equation}\label{a1:Dnv}
    \Big|\{D_n\ve<  M^+(2R)/2\}\cap B^+_{\tau_\gamma R} \Big|\leq \mu_0\,|B^+_{\tau_\gamma R}|,
\end{equation}
    then
    \begin{equation}\label{tesi}
        D_n\ve\geq M^+(2R)/4\quad \text{in $B^+_{\tau_\gamma R/2}$.}
    \end{equation}
    Analogously, if 
    \begin{equation}\label{specular0}
     \big|\{D_n\ve> - M^+(2R)/2\}\cap B^+_{\tau_\gamma R} \big|\leq \mu_0\,|B_{\tau_\gamma R}|,\quad\text{then}\quad D_n\ve<-\frac{M^+(2R)}{4}\quad\text{in $B^+_{\tau_\gamma R/2}$.}
    \end{equation}
\end{lemma}

\begin{proof}
For notational simplicity, we set $\mathrm{M}^+=M^+(2R)$.
    We distinguish two  cases.

    \noindent\textit{Case 1:} Suppose that
    \begin{equation}\label{caso1:dnv}
        D_n\ve \geq \frac{3}{8} \mathrm{M}^+\quad \text{a.e. on $B^0_{\tau_\gamma R}$.}
    \end{equation}
We consider a parameter $\ell$ such that
\[
\ell\in \Big(\frac{\mathrm{M}^+}{4},\frac{3\mathrm{M}^+}{8}\Big)
\]
and let $g(t)= -(t-\ell)_-$, so that owing to \eqref{caso1:dnv} we have that  $
g(D_n\ve)\equiv 0$ on $B_{\tau_\gamma R}^0$.

Then let $\kappa$ be such that $\mathrm{M}^+/4<\kappa<\ell$, so that by \eqref{ultra:utile} and Remark \ref{remark:importante}, we have
\begin{equation*}
    c\,a_\e(\mathrm{M}^+)\leq  a_\e(|D\ve|)\leq C\,a_\e(\mathrm{M}^+)\quad\text{in the set $\big\{\kappa\leq D_n\ve<\ell\big\}\cap B^+_{2R}$,}
\end{equation*}
 with $c,C=c,C(n,\l,\L,i_a,s_a)>0$.
Therefore, by taking a cut-off function $\phi\in C^\infty_c(B_{\tau_\gamma R})$, $0\leq \phi\leq 1$ satisfying \eqref{cutoff} with $0<r_1<r_2\leq \tau_\gamma R$, by exploiting Equation \eqref{qqq}, and dividing both sides of the resulting equation by $a_\e(\mathrm{M^+})$, we deduce
\begin{equation}\label{ff}
    \int_{\{\kappa\leq D_n\ve<\ell \}\cap B^+_{r_1}} |D(D_n\ve)|^2\,dx\leq C\,\frac{(\mathrm{M}^+)^2}{(r_2-r_1)^2}\,\big|\{D_n\ve<\ell\}\cap B^+_{r_2}\big|;
\end{equation}
we remark that, in order to obtain the full form of \eqref{ff}, we also used $|g(D_n\ve)|\leq \mathrm{M}^+$ and 
\begin{equation}\label{brecall1}
   a_\e(|D\ve|)\,|D\ve|=b_\e\big( |D\ve|\big)\leq C(n,i_a,s_a)\,b_\e(\mathrm{M}^+)=C(n,i_a,s_a)\,a_\e(\mathrm{M}^+)\,\mathrm{M}^+, 
\end{equation}
 which stems from $|D\ve|\leq \sqrt{n}\,\mathrm{M}^+$ in $B^+_{2R}$, the monotonicity of $b_\e(t)=a_\e(t)\,t$, \eqref{b2t} and Remark \ref{remark:importante}--compare with the proof of \eqref{levels:eq}, i.e., Equations \eqref{brecall},  \eqref{temporaneo:lev} and \eqref{control:ae}. 

Now set $ A^-(\kappa,\vrho)\coloneqq \{D_n\ve <\kappa\}\cap B^+_\vrho$, and define the sequences
\begin{equation*}
    R_m=\frac{\tau_\gamma R}{2}+\frac{\tau_\gamma R}{2^{m+1}},\quad \kappa_m=\frac{\mathrm{M}^+}{4}+\frac{\mathrm{M}^+}{2^{m+3}},\quad m=0,1,2,\dots.
\end{equation*}
Since $\tau_\gamma R/2\leq R_m\leq \tau_\gamma R$ and  $\kappa_m<3/8\mathrm{M}^+<\mathrm{M}
^+/2$, from \eqref{a1:Dnv} we deduce

    \begin{equation}\label{azz0}
\begin{split}
|B^+_{R_{m+1}}\setminus A^-(\kappa_m,R_{m+1})|&=|B^+_{R_{m+1}}|-|A^-(\kappa_m,R_{m+1})|
\\
&\geq |B^+_{R_{m+1}}|-\left|A^-\left(\frac{\mathrm{M}^+}{2},\tau_\gamma R\right)\right|
\\
&\geq (1-\mu_0\,2^n)\,|B^+_{R_{m+1}}|\geq \frac{1}{2}|B^+_{R_{m+1}}|\,, 
\end{split}
\end{equation}
provided that $\mu_0\in (0,2^{-n-1})$. Owing to \eqref{ff} and \eqref{azz0}, we may perform the same iterative scheme of \eqref{temp:newlevels}-\eqref{fine:iter}; that is, defining $Z_m=\frac{A^-(\kappa_m,R_m)}{|B^+_{R_m}|}$, we find $\lim_{m\to \infty} Z_m=0$ for $\mu_0=\mu_0(n,\l,\L,i_a,s_a)\in (0,2^{-n-1})$  small enough, which implies \eqref{tesi}, thus completing the proof in this case.
\vspace{0.2cm}

\noindent\textit{Case 2.} Suppose now that
\begin{equation}\label{caso2:dnv}
    D_n\ve(x'_0)<\frac{3}{8}\mathrm{M}^+\quad \text{for some point $x_0'\in B^0_{\tau_\gamma R}$.}
\end{equation}
 We aim to show that this cannot happen as long as $\gamma\leq \gamma_0$, for $\gamma_0\in (0,1)$ chosen small enough. Owing to the boundary oscillation \eqref{ttt:osc} and \eqref{osc:Dineu}, for a.e. $x'\in B^0_{\tau_\gamma R}$ we find
\begin{equation*}
    \begin{split}
        D_n\ve(x')\leq |D_n\ve(x')-D_n\ve (x'_0)|+D_n\ve (x'_0)\leq C_{\mathrm{a}}\,\omega'(\tau_\gamma R)+ \frac{3}{8}\mathrm{M}^+\leq\Big( C_{\mathrm{a}}\frac{\gamma}{2}+\frac{3}{8}\Big)\,\mathrm{M}^+,
    \end{split}
\end{equation*}
hence we have
\begin{equation}\label{dn:b01}
    D_n\ve \leq \frac{7}{16}\mathrm{M}^+\quad \text{a.e. on $B^0_{\tau_\gamma R}$}
\end{equation}
provided we take
\begin{equation}\label{1cond:g0}
    \gamma\leq \gamma_0(n,\l,\L,i_a,s_a)=\frac{1}{(8C_{\mathrm{a}})}.
\end{equation} 
Then for  
\[
\frac{7}{16}\mathrm{M}^+<\kappa<\ell<\frac{\mathrm{M}^+}{2},
\]
 consider the function
\begin{equation*}
    g(t)=\begin{cases}
        0\quad & t\geq \ell
        \\
        t-\ell\quad & \kappa<t<\ell
        \\
        \kappa-\ell\quad &t\leq \kappa.
    \end{cases}
\end{equation*}
In particular, by \eqref{dn:b01}, we have $g'(D_n\ve)=0$ on $B^0_{\tau_\gamma R}$, so for a given cut-off function $\phi\in C^\infty_c(B_{\tau_\gamma R})$, we may exploit Equation \eqref{quai}; in said expression, we estimate the left-hand side and the first two terms on the right-hand side exactly as in Case 1. For what concerns the other terms, we exploit the inequalities $g(D_n\ve)\leq \mathrm M^+$,  $$|h_0|=|\A^n(D\ve)|\stackrel{\eqref{coAe:gr}}{\leq }Cb_\e\big(|D\ve|\big)\stackrel{\eqref{brecall1}}{\leq} C\,a_\e(\mathrm{M}^+)\,\mathrm{M}^+,$$
coming from the boundary condition $\A^n(D\ve)=-h_0$ on $B^0_{2R}$ and the continuity of the trace operator, and that 
\begin{equation}\label{rihae}
c\,a_\e(\mathrm{M}^+)\leq a_\e(|D\ve|)\leq C\,a_\e(\mathrm{M}^+)\quad\text{in $\{g'(D_n\ve)\neq 0\}\subset \{D_n\ve\geq \kappa\}\cap B^+_{2R}$}
\end{equation}
with $c,C=c,C(n,\l,\L,i_a,s_a)>0$ by our choice of $\kappa$, \eqref{ultra:utile} and Remark \ref{remark:importante}; ultimately, we get
\begin{equation*}
    \begin{split}
        &|h_0|\,\int_{B^+_{2R}}|D^2\phi^2|\,|g(D_n\ve)|\,dx\leq C\,\frac{a_\e(\mathrm{M}^+)\,(\mathrm{M}^+)^2}{(r_2-r_1)^2}\,\big| \{ D_n\ve<\ell\}\cap B^+_{r_2} \big|
        \\
        &|h_0|^2\int_{B^+_{2R}}\frac{g'(D_n\ve)}{a_\e(|D\ve|)}|D\phi|^2\,dx\leq C\,\frac{a_\e(\mathrm{M}^+)\,(\mathrm{M}^+)^2}{(r_2-r_1)^2}\,\big| \{ D_n\ve<\ell\}\cap B^+_{r_2} \big|.
    \end{split}
\end{equation*}
Therefore, inserting these estimates into \eqref{quai}, we once again obtain Equation \eqref{ff}.

Now let $R_{m}$ be as in Case 1, and set 
\[
    \kappa_{m} = \frac{7}{16}\mathrm{M}^{+} + \frac{\mathrm{M}^{+}}{2^{m+4}},   \quad m=0,1,\dots
\]
Since $\kappa_{m} \leq \mathrm{M}^{+}/2$, condition \eqref{azz0} still holds.  
We may therefore carry out the same iteration argument as in the previous case and conclude that
\[
    D_{n}v \geq \kappa_{\infty} = \frac{7}{16}\mathrm{M}^{+}
    \qquad \text{in } B^{+}_{\tau_{\gamma} R/2}.
\]
This contradicts \eqref{dn:b01} and the continuity of the trace operator.  
Henceforth, as long as  $\gamma$ fulfills \eqref{1cond:g0},  only Case~1 can occur, for which we have proved the validity of \eqref{tesi}. This completes the proof
of Lemma~\ref{lemma:alt1Dn} in the situations covered by \eqref{tesi}.

Finally, the proof of \eqref{specular0} is completely specular, so it is left to the reader.
\end{proof}

\begin{lemma}[The second alternative for $D_n\ve$]\label{lemma:alt2Dn}  Let $\ve\in W^{1,B_\e}(B_{2R_0}^+)$ be a solution to \eqref{homve:neu}, and let $0<R\leq R_0/2$. Then there exist $\gamma=\gamma(n,\l,\L,i_a,s_a)\in (0,\gamma_0)$ and $\eta_0=\eta_0(n,\l,\L,i_a,s_a)\in (0,1)$ such that, if 
\begin{equation}\label{a2:Dnv}
    \Big|\{D_n\ve<  M^+(2R)/2\}\cap B^+_{\tau_\gamma R} \Big|> \mu_0\,|B^+_{\tau_\gamma R}|,
\end{equation}
    with $\mu_0, \tau_\gamma, \gamma_0$ as in Lemma \ref{lemma:alt1Dn}, then
    \begin{equation}\label{tteesi}
        D_n\ve \leq \eta_0\,M^+(2R)\quad\text{in $B^+_{\tau_\gamma R/2}$.}
    \end{equation}
Analogously, if
\begin{equation}\label{specular}
\Big|\{D_n\ve>- M^+(2R)/2\}\cap B^+_{\tau_\gamma R} \Big|> \mu_0\,|B^+_{\tau_\gamma R}|,\quad\text{then}\quad D_n\ve \geq -\eta_0\,M^+(2R)\quad\text{in $B^+_{\tau_\gamma R/2}$.}
\end{equation}
\end{lemma}

\begin{proof}
Let us set $\mathrm{M}^+=M^+(2R)$, and from now on we fix
\begin{equation}\label{gfix}
    \gamma=\frac{1}{2^{t_0}}\quad\text{for some $t_0=t_0(n,\l,\L,i_a,s_a)\in \N$ large enough,}
\end{equation}
which will be determined at the end of the proof, and we consider the corresponding $\tau_\gamma=\tau_\gamma(n,\l,\L,i_a,s_a)$ given by Lemma \ref{lemma:oscDive}.

Owing to \eqref{a2:Dnv}, and using the argument of Remark \ref{remark:nu0}, we may find $\nu_0\in (1/2,1)$, depending on $n,\l,\L,i_a,s_a$, such that
\begin{equation}\label{hfd}
    |\{D_n\ve >\mathrm{M}^+/2\}\cap B^+_{\nu_0\tau_\gamma R}|<(1-\mu_0/2)\,|B^+_{\nu_0\tau_\gamma R}|.
\end{equation}
We remark that $\nu_{0}=\left(\frac{1-\mu_{0}}{1-\mu_{0}/2}\right)^{1/n} \in (1/2,1)$ is independent of $\gamma$, since $\mu_{0}=\mu_{0}(n,\lambda,\Lambda,i_{a},s_{a})\in (0,2^{-n-1})$ is.
We now distinguish two cases, which in turn will be divided into sub-steps.
\vspace{0.2cm}

\noindent\textit{Case 1.} Let us suppose that
\begin{equation}\label{2ncase1}
    D_n\ve \leq (1-\gamma)\,\mathrm{M}^+\quad\text{on $B^0_{\tau_\gamma R}$.}
\end{equation}
In this case, we consider $\kappa>(1-\gamma)\,\mathrm{M}^+$, and the function $g(t)=(t-\kappa)_+$, so that $g(D_n\ve)\equiv 0$ on $B^0_{\tau_\gamma R}$. For $\phi\in C^\infty_c(B_{\tau_\gamma R})$ satisfying \eqref{cutoff}, we may use \eqref{quai1}, and deduce
\begin{equation}\label{ijn}
    \int_{B^+_{r_1}\cap\{D_n\ve>\kappa\}} |D(D_n\ve)|^2\,dx\leq \frac{C\,(\mathrm{M^+-\kappa})^2}{(r_2-r_1)^2}\,|\{D_n\ve>\kappa\}\cap B^+_{r_2}|\,,
\end{equation}
where, once again, we used  \eqref{rihae}, as well as the estimate~$g^2(D_n\ve)\leq (\mathrm{M^+-\kappa})^2 $. 
\vspace{0.2cm}

\noindent\textit{Step 1.} We show that for every $\theta_0\in (0,1)$, we may find $s_0=s_0(n,\l,\L,i_a,s_a,\theta_0)\in \N$ large enough, such that
\begin{equation}\label{smdnve}
         \bigg|\Big\{D_n\ve>\Big(1-\frac{1}{2^{s_0+t_0}} \Big)\,\mathrm{M}^+ \Big\}\cap B^+_{\nu_0\tau_\gamma R} \bigg|\leq \theta_0\, |B^+_{\nu_0\tau_\gamma R}|,
\end{equation}
with $t_0\in \N$ given by \eqref{gfix}. 

We mostly reproduce the argument of Lemma \ref{lemma:alt3small}: for $s=t_0+1,t_0+2,\dots,$, we define
\[
\kappa_s=\Big(1-\frac{1}{2^s}\Big)\,\mathrm{M}^+,\quad\text{and }\quad A^+_s=\{D_n\ve>\kappa_s\}\cap B^+_{\nu_0\tau_\gamma R}.
\]
 Since $\kappa_s\geq \mathrm{M}^+/2$, from \eqref{hfd} we deduce
\begin{equation*}
    |B^+_{\nu_0\tau_\gamma R}\setminus A^+_s|\geq \frac{\mu_0}{2}|B^+_{\nu_0\tau_\gamma R}|\quad\text{for all $s=t_0+1,t_0+2,\dots$}
\end{equation*}
    Applying Lemma \ref{lemma:levels} to the function $u=D_n\ve$, and levels $\ell=\kappa_{s+1}$, $\kappa=\kappa_{s}$, using H\"older's and the above inequality, we get
    \begin{equation*}
    \begin{split}
        \frac{\mathrm{M}^+}{2^{s+1}}\,|A^+_{s+1}|& \leq c(n)\,|A^+_{s+1}|^{\frac{1}{n}}\frac{|B^+_{\nu_0 \tau_\gamma R}|}{|B^+_{\nu_0 \tau_\gamma R}\setminus A^+_s|}\int_{A^+_{s}\setminus A^+_{s+1}}|D(D_k\ve)|\,dx
        \\
        &\leq C\,\tau_\gamma\,R\,\left( \int_{A^+_{s}\setminus A^+_{s+1}}|D(D_k\ve)|^2\,dx\right)^{1/2}\,|A^+_{s}\setminus A^+_{s+1}|^{1/2}\,,
        \end{split}
    \end{equation*}
   with $C=C(n,\l,\L,i_a,s_a)>0$, where in the last inequality we also used $|A^+_{s+1}|\leq |B^+_{\nu_0\tau_\gamma R}|\leq C(n)\,R^n\,(\tau_\gamma)^n$ since $\nu_0\in (1/2,1)$.
   
   Then we use \eqref{ijn} with $\kappa=\kappa_s$, $r_1=\nu_0\tau_\gamma R$,  $r_2=\tau_\gamma R$, recalling the dependence on the data of $\mu_0$, $\nu_0$, and the definition of $\kappa_s$, we deduce
   \begin{equation*}
       \left( \int_{A^+_{s}\setminus A^+_{s+1}}|D(D_k\ve)|^2\,dx\right)^{1/2}\leq C\,\frac{(\mathrm{M}^+-\kappa_{s})}{\tau_\gamma R}\,|B^+_{\tau_\gamma R}|^{1/2}\leq C'\,\frac{\mathrm{M}^+}{2^{s}\,\tau_\gamma R}\,|B^+_{\nu_0\tau_\gamma R}|^{1/2},
   \end{equation*}
   for $C,C'=C,C'(n,\l,\L,i_a,s_a)>0$. Connecting the two inequalities above,  and squaring both sides of the resulting equation yields
   \begin{equation*}
       |A^+_{s+1}|^2\leq C\,|B^+_{\nu_0 \tau_\gamma R}|\,|A^+_{s}\setminus A^+_{s+1}|\,.
   \end{equation*}
   We sum this inequality over $s=t_0+1,t_0+2,\dots,t_0+s_0-1$, and telescoping the right-hand side, while using $|A^+_{s+1}|\geq |A^+_{s_0+t_0}|$ on the left-hand side, we find
   \begin{equation*}
       (s_0-2)|A^+_{s_0+t_0}|^2\leq \sum_{s=t_0+1}^{s_0+t_0-1}|A^+_{s+1}|^2\leq C_0\, |B^+_{\nu_0 \tau_\gamma R}|\big(|A^+_{t_0+1}|-|A^+_{s_0+t_0}| \big)\leq C_0\, |B^+_{\nu_0 \tau_\gamma R}|^2,
   \end{equation*}
   where $C_0=C_0(n,\l,\L,i_a,s_a)>0$. Choosing $s_0=s_0(n,\l,\L,i_a,s_a,\theta_0)$ large enough gives \eqref{smdnve}.
   \vspace{0.2cm}

   \noindent\textit{Step 2.} Let us now prove \eqref{tteesi}, following the proof of Lemma \ref{lemma:altfin}. We fix $\theta_0=\theta_0(n,\l,\L,i_a,s_a)\in (0,1)$  and the corresponding $s_0=s_0(n,\l,\L,i_a,s_a)\in \N$ from the previous step. For $m=0,1,2,3,\dots$, we set
   \begin{equation*}
       \kappa_m=\left(1-\frac{1}{2^{s_0+t_0}} \right)\,\mathrm{M}^++\left( 1-\frac{1}{2^m}\right)\,\frac{\mathrm{M}^+}{2^{s_0+t_0+1}},\quad R_m=\frac{\tau_\gamma R}{2}+\Big(\nu_0\tau_\gamma R-\frac{\tau_\gamma R}{2}\Big)\left(\frac{1}{2^m} \right),
   \end{equation*}
 and we also set $
    A^+(\kappa_m,R_m)\coloneqq\{D_n\ve>\kappa_m\}\cap B^+_{R_m}$. Then, by \eqref{smdnve}, and since $\kappa_m\geq \kappa_0= (1-\tfrac{1}{2^{s_0+t_0}})\,\mathrm{M}^+$ and $\tau_\gamma R/2\leq R_m\leq \nu_0 \tau_\gamma R$, we have
\begin{equation*}
\begin{split}
    |B^+_{R_{m+1}} &\setminus A^+(\kappa_m,R_{m+1})|\geq |B^+_{R_{m+1}}|-|A^+(\kappa_0,\nu_0 \tau_\gamma R)|
    \\
    &\geq (1-\theta_0\,(\tau_\gamma\nu_0)^n)|B^+_{R_{m+1}}|\geq  \frac{1}{2}\,|B^+_{R_{m+1}}|\,,
\end{split}
\end{equation*}
provided we choose $0<\theta_0\leq \tau_\gamma^{-n-1}\nu_0^{-n}$. We then use Lemma \ref{lemma:levels} with function $u=D_n\ve$, and levels $\ell=\kappa_{m+1}$, $\kappa=\kappa_{m}$, together with the above inequality, thus finding
   \begin{equation*}
   \begin{split}
       \frac{\mathrm{M}^+}{2^{s_0+t_0+m+2}}&\,|A^+(\kappa_{m+1},R_{m+1})|^{\frac{n-1}{n}} 
       \\
       & \leq C\,\frac{|B^+_{R_{m+1}}|}{|B^+_{R_{m+1}} \setminus A^+(\kappa_m,R_{m+1})|}\int_{A^+(\kappa_m,R_{m+1})\setminus A^+(\kappa_{m+1},R_{m+1})}|D(D_n\ve)|\,dx
       \\
       &\leq C'\,\left(\int_{A^+(\kappa_m,R_{m+1})\setminus A^+(\kappa_{m+1},R_{m+1})}|D(D_n\ve)|^2\,dx \right)^{1/2}\,|A^+(\kappa_m,R_{m})|^{1/2}\,,
       \end{split}
   \end{equation*}
with $C,C'=C,C'(n,\l,\L,i_a,s_a)>0$,  where in the last estimate we used H\"older's inequality. We now use \eqref{ijn} with $r_2=R_m$ and $r_1=R_{m+1}$, and using that $(\mathrm{M}^+-\kappa_m)\leq \mathrm{M}^+/2^{s_0+t_0}$, and that $\nu_0=\nu_0(n,\l,\L,i_a,s_a)\in (1/2,1)$, we find
\begin{equation*}
    \left(\int_{A^+(\kappa_m,R_{m+1})\setminus A^+(\kappa_{m+1},R_{m+1})}|D(D_n\ve)|^2\,dx \right)^{1/2}\leq C\,\frac{2^{m+2}}{\tau_\gamma R}\,\frac{\mathrm{M}^+}{2^{s_0+t_0} }|A^+(\kappa_m,R_m)|^{1/2}.
\end{equation*}
Merging the content of the two inequalities above, and dividing both sides of the resulting equation by $\mathrm{M}^+/2^{s_0+t_0}$, we get
\begin{equation*}
    |A^+(\kappa_{m+1},R_{m+1})|^{\frac{n-1}{n}}\leq C\,\frac{4^{m}}{\tau_\gamma R}\,|A^+(\kappa_m,R_m)|\,.
\end{equation*}
with $C=C(n,\l,\L,i_a,s_a)$. Hence by setting $ Z_m=\frac{|A^+(\kappa_m,R_m)|}{|B^+_{R_m}|}$,
and by exploiting that $\tau_\gamma R/2\leq R_m\leq \nu_0 \tau_\gamma R$ and $\nu_0\in (1/2,1)$,  the above inequality and \eqref{smdnve} imply
\begin{equation*}
    Z_{m+1}\leq C\, (4^{\frac{n}{n-1}})^{m}\,Z_m^{\frac{n}{n-1}},\quad\text{and}\quad Z_0\leq \theta_0\,,
\end{equation*}
with $C=C(n,\l,\L,i_a,s_a)>0$. Thus, by Lemma \ref{lem:hyp}, choosing $\theta_0=\theta_0(n,\l,\L,i_a,s_a)$ small enough, we get $\lim_{m\to \infty }Z_m=0$, which implies
\begin{equation*}
    D_n\ve\leq \Big(1-\frac{1}{2^{s_0+t_0}} \Big)\mathrm{M}^++\frac{\mathrm{M}^+}{2^{s_0+t_0+1}}\equiv \eta_0\, \mathrm{M}^+\,,\quad \text{a.e. in $B^+_{\tau_\gamma R/2}$,}
\end{equation*}
that is \eqref{tteesi}, and this proves the lemma when \eqref{2ncase1} holds.
\vspace{0.2cm}

\noindent\textit{Case 2.} Suppose now that
\begin{equation*}
    D_n\ve(x_0') >(1-\gamma)\,\mathrm{M}^+\quad\text{for some point $x_0'\in B^0_{\tau_\gamma R}$.}
\end{equation*}
We aim to show that this cannot occur, provided $t_0$ in \eqref{gfix} is chosen large enough. By  \eqref{ttt:osc} and \eqref{osc:Dineu}, for a.e. $x'\in B^0_{\tau_\gamma R}$ we find
\begin{equation}\label{csfindn}
\begin{split}
    D_n\ve (x')&\geq -\operatorname*{osc}_{B^0_{\tau_\gamma R}} D_n\ve+D_n\ve(x'_0)\geq -C_{\mathrm{a}}\,\omega'(\tau_\gamma R)+(1-\gamma)\,\mathrm{M}^+
    \\
    &\geq \big(1-\big(C_{\mathrm{a}}/2+1\big)\,\gamma\big)\,\mathrm{M}^+\eqqcolon (1-\hat{C}\,\gamma)\,\mathrm{M}^+=(1-\hat{C}2^{-t_0})\,\mathrm{M}^+.\quad\text{on $B^0_{\tau_\gamma R}$}
    \end{split}
\end{equation}
where we set $\hat{C}=C_{\mathrm{a}}/2+1$.

We take $t_0=t_0(n,\l,\L,i_a,s_a)\in \N$ so large that $\gamma=2^{-t_0}$ satisfies $\hat{C}\gamma<1/4$; then
we consider parameters~$\mathrm{M}^+/2<\kappa<\ell<(1-\hat{C}\gamma)\,\mathrm{M}^+$, and the function
\begin{equation}\label{last:g}
    g(t)=\begin{cases}
        0\quad & t\leq\kappa
        \\
        t-\kappa\quad & \kappa< t< \ell
        \\
        \ell-\kappa\quad &t>\ell.
    \end{cases}
\end{equation}
In particular, by \eqref{csfindn} and our choice of $\ell$, we have $g'(D_n\ve)=0$ on $B^0_{\tau_\gamma R}$. 
\vspace{0.2cm}

\noindent\textit{Step 1.} Our goal is to obtain an integral inequality similar to \eqref{ijn}. To this end we will exploit Remark \ref{rem:identity} with $\mathbf{V}=\Ae\big(D'\ve(0),\mathrm{M}^+\big)$ and $\tilde \Ae(\xi)=\Ae(\xi)-\Ae\big(D'\ve(0),\mathrm{M}^+\big)$.

By the fundamental theorem of calculus and \eqref{coer:Ae}, on the set~$\{g(D_n\ve)\neq 0\}\cap B^+_{\tau_\gamma R}=\{D_n\ve>\kappa\}\cap B^+_{\tau_\gamma R}$, for all $i=1,\dots,n$ we have
\begin{equation}\label{above}
    \begin{split}
        |\tilde{\Ae^i}(D\ve)|=&\big|\Ae^i(D\ve)-\Ae^i\big(D'\ve(0),\mathrm{M}^+\big)\big|
        \\
        \leq &\sum_{j=1}^{n-1}\int_0^1\bigg|\frac{\partial \Ae^i}{\partial \xi_j}\Big(tD\ve+(1-t)\big(D'\ve(0),\mathrm{M}^+\big) \Big)\bigg|\,dt\,\big| D_j\ve-D_j\ve(0)\big|
        \\
        &+\int_0^1 \bigg|\frac{\partial \Ae^i}{\partial \xi_n}\Big(tD\ve+(1-t)\big(D'\ve(0),\mathrm{M}^+\big) \Big)\bigg|\,dt\,(\mathrm{M}^+-D_n\ve)
        \\
        \leq &\, C\,\int_0^1 a_\e\Big(\big|tD\ve+(1-t)\big(D'\ve(0),\mathrm{M}^+\big)\big| \Big)\,dx\,\bigg\{\frac{\gamma\,\mathrm{M^+}}{2}+(\mathrm{M}^+-\kappa)\bigg\},
    \end{split}
\end{equation}
where in the last inequality we also used the oscillation estimate \eqref{osc:Dineu}. Now observe that in $\{D_n\ve>\kappa\}\cap B^+_{2R}$, since $\kappa\geq \mathrm{M}^+/2$, for all $t\in [0,1]$ we have $$C(n)\,\mathrm{M}^+\geq\big|tD\ve+(1-t)(D'\ve(0),\mathrm{M}^+)\big|\geq|tD_n\ve+(1-t)\mathrm{M}^+|\geq \mathrm{M}^+/2. $$
Hence by \eqref{ultra:utile} and Remark \ref{remark:importante}, we have $a_\e\Big(\big|tD\ve+(1-t)\big(D'\ve(0),\mathrm{M}^+\big) \big|\Big)\leq C\,a_\e(\mathrm{M}^+)$, with $C=C(n,\l,\L,i_a,s_a)>0$. Using this information and \eqref{gfix} into \eqref{above}, we infer
\begin{equation}\label{tae:stima}
    |\tilde{\Ae}(D\ve)|\leq C\,a_\e\big(\mathrm{M}^+ \big)\Big\{\frac{\mathrm{M}^+}{2^{t_0+1}}+(\mathrm{M}^+-\kappa) \Big\}\quad\text{in $\{D_n\ve>\kappa\}\cap B^+_{\tau_\gamma R}$.}
\end{equation}
for some $C=C(n,\l,\L,i_a,s_a)>0$. In particular, since $\tilde{h}_0\equiv -\tilde{\Ae^n}(D\ve)$ on $B^0_{\tau_\gamma R}$, by the continuity of the trace operator the above inequality implies
\begin{equation}\label{h0:stima}
    |\tilde{h}_0|\leq C\,a_\e\big(\mathrm{M}^+ \big)\Big\{\frac{\mathrm{M}^+}{2^{t_0+1}}+(\mathrm{M}^+-\kappa) \Big\}
\end{equation}

By the properties of $\phi\in C^\infty_c(B_{\tau_\gamma R})$ in \eqref{cutoff}, and by using that
\[
c\, a_\e(\mathrm{M}^+)\leq a_\e(|D\ve|)\leq C\,a_\e(\mathrm{M}^+)\quad\text{on $\{g(D_n\ve)\neq 0\}\cap B^+_{2R}$}
\]
with $c,C=c,C(n,\l,\L,i_a,s_a)$ thanks to our choice of $g(t)$ in \eqref{last:g}, \eqref{ultra:utile} and Remark \ref{remark:importante}, we infer
\begin{equation*}
    \int_{B^+_{2R}}a_\e\big(|D\ve|\big)\,|D(D_n\ve)|^2\,g'(D_n\ve)\,\phi^2\,dx\geq c\,a_\e(\mathrm{M}^+)\int_{\{\kappa< D_n\ve<\ell\}\cap B^+_{2R}} |D(D_n\ve)|^2\,\phi^2\,dx,
\end{equation*}
and using \eqref{tae:stima}, the estimate $|g(D_n\ve)|\leq (\mathrm{M^+}-\kappa)$ and the properties of $\phi$, we find
\begin{equation*}
    \begin{split}\int_{B^+_{2R}}\big|\tilde\Ae(D\ve)\big|&\,|D^2\phi^2|\,|g(D_n\ve)|\,dx
    \\
    &\leq C\,\frac{a_\e(\mathrm{M}^+)}{(r_2-r_1)^2}\,\Big\{\frac{\mathrm{M}^+}{2^{t_0+1}}+(\mathrm{M}^+-\kappa) \Big\}\,\big(\mathrm{M}^+-\kappa\big)\,\big|\{D_n\ve>\kappa\}\cap B^+_{r_2}\big|,
    \end{split}
\end{equation*}
and also using weighted Young's inequality
\begin{equation*}
    \begin{split}
    \int_{B^+_{2R}}\big|\tilde\Ae(D\ve)\big|&\,| D\phi\,\phi|\,|D_{nn}\ve|\,g'(D_n\ve) \,dx
    \\
    \leq & \,   C\,a_\e(\mathrm{M}^+)\,\Big\{\frac{\mathrm{M}^+}{2^{t_0+1}}+(\mathrm{M}^+-\kappa) \Big\}\,\int_{\{\kappa < D_n\ve<\ell\}\cap B^+_{2R}}|D\phi|\,\phi\,|D(D_n\ve)|\,dx
    \\
    \leq &\,\delta\,a_\e(\mathrm{M}^+)\int_{\{\kappa < D_n\ve<\ell\}\cap B^+_{2R}}|D(D_n\ve)|^2\,\phi^2\,dx
    \\
&+C_\delta\,\frac{a_\e(\mathrm{M}^+)}{(r_2-r_1)^2}\,\Big\{\frac{\mathrm{M}^+}{2^{t_0+1}}+(\mathrm{M}^+-\kappa) \Big\}^2|\{D_n\ve>\kappa\}\cap B^+_{r_2}|,
    \end{split}
\end{equation*}
for all $\delta\in (0,1)$, with $C>0$ depending on $n,\l,\L,i_a,s_a$, and $C_\delta$ depending on $\delta$ as well.

Then, by \eqref{h0:stima}, we obtain
\begin{equation*}
\begin{split}
    |\tilde h_0|\,\int_{B^+_{2R}}&|D^2\phi^2|\,| g(D_n\ve)|\,dx
    \\
    &\leq C\,\frac{a_\e(\mathrm{M}^+)}{(r_2-r_1)^2}\,\Big\{\frac{\mathrm{M}^+}{2^{t_0+1}}+(\mathrm{M}^+-\kappa) \Big\}\,\big(\mathrm{M}^+-\kappa\big)\,|\{D_n\ve>\kappa\}\cap B^+_{r_2}|,
    \end{split}
\end{equation*}
and by also using Young's inequality
\begin{equation*}
    \begin{split}
        |\tilde h_0|\,\int_{B^+_{2R}} &|D\phi\,\phi|\,g'(D_n\ve)\,|D_{nn}\ve|\,dx
        \\
        \leq &\,a_\e(\mathrm{M}^+)\,\Big\{\frac{\mathrm{M}^+}{2^{t_0+1}}+(\mathrm{M}^+-\kappa) \Big\}\int_{\{\kappa< D_n\ve<\ell\}}|D\phi|\,\phi\,|D(D_n\ve)|\,dx
        \\
        \leq &\,\delta\, a_\e(\mathrm{M}^+)\,\int_{\{\kappa < D_n\ve<\ell\}}|D(D_n\ve)|^2\,\phi^2\,dx
        \\
        &+C_\delta\,\frac{a_\e(\mathrm{M}^+)}{(r_2-r_1)^2}\,\Big\{\frac{\mathrm{M}^+}{2^{t_0+1}}+(\mathrm{M}^+-\kappa) \Big\}^2\,|\{D_n\ve>\kappa\}\cap B^+_{r_2}|\,.
    \end{split}
\end{equation*}
Coupling the five estimates above with \eqref{inednv}, choosing $\delta=\delta(n,\l,\L,i_a,s_a)\in (0,1)$ small enough to reabsorb terms, and using that $\phi\equiv 1$ on $B_{r_1}$ finally yields \footnote{By the implicit function theorem, the set
$\{D(D_n\ve) \neq 0\} \cap \{D_n\ve = \kappa\}$
is a smooth hypersurface in $B^+_{2R}$ and therefore has zero Lebesgue measure.
Hence, the integral in \eqref{eccoci} may be taken over
$\{\kappa \le D_n \ve < \ell\}$ or equivalently $\{\kappa < D_n\ve < \ell\}$.}

\begin{equation}\label{eccoci}
    \int_{\{\kappa \leq D_n\ve<\ell\}\cap B^+_{r_1}}|D(D_n\ve)|^2\,dx\leq \frac{C}{(r_2-r_1)^2}\,\Big\{\frac{\mathrm{M}^+}{2^{t_0+1}}+(\mathrm{M}^+-\kappa) \Big\}^2|\{D_n\ve>\kappa\}\cap B^+_{r_2}|,
\end{equation}
with $C=C(n,\l,\L,i_a,s_a)>0$, which is valid for all 
\[
\frac{\mathrm{M}^+}{2}<\kappa<\ell<\bigg(1-\frac{\hat{C}}{2^{t_0}}\bigg)\,\mathrm{M}^+,
\]
where $\hat{C}=\hat{C}(n,\l,\L,i_a,s_a)$ is the constant appearing in \eqref{csfindn}.
\vspace{0.2cm}

\noindent\textit{Step 2.} We show that for every $\theta_0\in (0,1)$, we may find $t_0=t_0(n,\l,\L,i_a,s_a,\theta_0)\in \N$ large enough, such that
\begin{equation}\label{smdnve2}
         \bigg|\Big\{D_n\ve>\Big(1-\frac{\hat{C}}{2^{t_0-2}} \Big)\,\mathrm{M}^+ \Big\}\cap B^+_{\nu_0\tau_\gamma R} \bigg|\leq \theta_0\, |B^+_{\nu_0\tau_\gamma R}|.
\end{equation}
Let us first impose that $t_0\geq \big\lceil\log_2 \hat{C}\big\rceil+6$.
 
 For $s=\big\lceil\log_2 \hat{C}\big\rceil+1,\big\lceil\log_2 \hat{C}\big\rceil+2,\dots,t_0-3$, we define 
 \begin{equation*}
     \kappa_s=\Big(1-\frac{\hat{C}}{2^s}\Big)\,\mathrm{M}^+\quad\text{and}\quad A^+_s=\{D_n\ve>\kappa_s\}\cap B^+_{\nu_0\tau_\gamma R}.
 \end{equation*}
Again, since $\kappa_s\geq \mathrm{M}^+/2$, from \eqref{hfd} we deduce $|B^+_{\nu_0\tau_\gamma R}\setminus A^+_s|\geq \frac{\mu_0}{2}|B^+_{\nu_0\tau_\gamma R}|$.  Thus, by applying Lemma \ref{lemma:levels} to the function $u=D_n\ve$, with levels $\ell=\kappa_{s+1}$, $\kappa=\kappa_{s}$, and by using H\"older's inequality,  we get
    \begin{equation*}
    \begin{split}
        \frac{\mathrm{M}^+}{2^{s+1}}\,|A^+_{s+1}|& \leq c(n)\,|A^+_{s+1}|^{\frac{1}{n}}\frac{|B^+_{\nu_0 \tau_\gamma R}|}{|B^+_{\nu_0 \tau_\gamma R}\setminus A^+_s|}\int_{A^+_{s}\setminus A^+_{s+1}}|D(D_k\ve)|\,dx
        \\
        &\leq C\,\tau_\gamma\,R\,\Big( \int_{A^+_{s}\setminus A^+_{s+1}}|D(D_k\ve)|^2\,dx\Big)^{1/2}\,|A^+_{s}\setminus A^+_{s+1}|^{1/2}\,,
        \end{split}
    \end{equation*}
   with $C=C(n,\l,\L,i_a,s_a)>0$, where in the last inequality we estimated $|A^+_{s+1}|^{\frac{1}{n}}\leq C(n)\,\tau_\gamma\,R$.
   
   Then, by \eqref{eccoci} with $r_1=\nu_0\tau_\gamma R$,  $r_2=\tau_\gamma R$,  and by the estimates 
   \[
   \frac{\mathrm M^+}{2^{t_0+1}}\leq \frac{\mathrm M^+}{2^{s}},\quad (\mathrm{M}^+-\kappa_s)=\frac{\hat{C}\,\mathrm M^+}{2^{s}},
   \]
and $|B^+_{\tau_\gamma R}|\leq C(n)\,|B^+_{\nu_0\tau_\gamma R}|$ as $\nu_0\in (1/2,1)$, we get
   \begin{equation*}
       \Big( \int_{A^+_{s}\setminus A^+_{s+1}}|D(D_k\ve)|^2\,dx\Big)^{1/2}\leq C\,\frac{\mathrm{M}^+}{2^{s}\,\tau_\gamma R}\,|B^+_{\nu_0\tau_\gamma R}|^{1/2},
   \end{equation*}
   Connecting the two inequalities above,  and squaring both sides of the resulting equation yields
   \begin{equation*}
       |A^+_{s+1}|^2\leq C\,|B^+_{\nu_0 \tau_\gamma R}|\,|A^+_{s}\setminus A^+_{s+1}|\,.
   \end{equation*}
   We sum this inequality over $s=\big\lceil\log_2 \hat{C}\big\rceil+1,\big\lceil\log_2 \hat{C}\big\rceil+2,\dots,t_0-3$, and telescoping the right-hand side, while using $|A^+_{s+1}|\geq |A^+_{t_0-2}|$ on the left-hand side, we find
   \begin{equation*}
   \begin{split}
       \big(t_0-\big\lceil\log_2 \hat{C}\big\rceil-5\big)\,|A^+_{t_0-2}|^2&\leq \sum_{s=\big\lceil\log_2 \hat{C}\big\rceil+1}^{t_0-3}|A^+_{s+1}|^2
       \\
       &\leq C_0\, |B^+_{\nu_0 \tau_\gamma R}|\Big(|A^+_{\big\lceil\log_2 \hat{C}\big\rceil+1}|-|A^+_{t_0-2}| \Big)\leq C_0\, |B^+_{\nu_0 \tau_\gamma R}|^2,
       \end{split}
   \end{equation*}
hence choosing $t_0=t_0(n,\l,\L,i_a,s_a,\theta_0)\in \N$ large enough yields \eqref{smdnve2}.
\vspace{0.2cm}

\noindent\textit{Step 3.} 
We fix $\theta_{0}=\theta_{0}(n,\lambda,\Lambda,i_{a},s_{a})\in (0,1)$ sufficiently small, to be specified later.  
From the previous step we then determine the parameter $t_{0}=t_{0}(n,\lambda,\Lambda,i_{a},s_{a})\in \N$ large enough such that \eqref{smdnve2} holds.  
At this point we finally determine the parameter $\gamma$ precisely, namely $\gamma = 2^{-t_{0}}$ according to \eqref{g:const}, which also determines the corresponding value of $\tau_{\gamma}$.

For $m=0,1,2,3,\dots$, we set
   \begin{equation*}
       \kappa_m=\Big(1-\frac{\hat{C}}{2^{t_0-2}} \Big)\,\mathrm{M}^++\hat{C}\,\Big( 1-\frac{1}{2^m}\Big)\,\frac{\mathrm{M}^+}{2^{t_0-1}},\quad R_m=\frac{\tau_\gamma R}{2}+\Big(\nu_0\tau_\gamma R-\frac{\tau_\gamma R}{2}\Big)\left(\frac{1}{2^m} \right),
   \end{equation*}
 and we also set 
 \[
    A^+(\kappa_m,R_m)\coloneqq\{D_n\ve>\kappa_m\}\cap B^+_{R_m}.
 \] 
    Then, by \eqref{smdnve2}, and since $\kappa_m\geq \kappa_0= (1-\tfrac{\hat{C}}{2^{t_0-2}})\,\mathrm{M}^+$ and $\tau_\gamma R/2\leq R_m\leq \nu_0 \tau_\gamma R$, we have
\begin{equation*}
\begin{split}
    |B^+_{R_{m+1}} &\setminus A^+(\kappa_m,R_{m+1})|\geq |B^+_{R_{m+1}}|-|A^+(\kappa_0,\nu_0 \tau_\gamma R)|
    \\
    &\geq (1-\theta_0\,(\tau_\gamma\nu_0)^n)|B^+_{R_{m+1}}|\geq  \frac{1}{2}\,|B^+_{R_{m+1}}|\,,
\end{split}
\end{equation*}
provided we choose $0<\theta_0\leq \tau_\gamma^{-n-1}\nu_0^{-n}$. By  Lemma \ref{lemma:levels} with function $u=D_n\ve$, and levels $\ell=\kappa_{m+1}$, $\kappa=\kappa_{m}$, the above inequality and H\"older's inequality, we get
   \begin{equation*}
   \begin{split}
       \frac{\mathrm{M}^+}{2^{t_0+m+1}}&\,|A^+(\kappa_{m+1},R_{m+1})|^{\frac{n-1}{n}} 
       \\
       & \leq C\,\frac{|B^+_{R_{m+1}}|}{|B^+_{R_{m+1}} \setminus A^+(\kappa_m,R_{m+1})|}\int_{A^+(\kappa_m,R_{m+1})\setminus A^+(\kappa_{m+1},R_{m+1})}|D(D_n\ve)|\,dx
       \\
       &\leq C'\,\left(\int_{A^+(\kappa_m,R_{m+1})\setminus A^+(\kappa_{m+1},R_{m+1})}|D(D_n\ve)|^2\,dx \right)^{1/2}\,|A^+(\kappa_m,R_{m})|^{1/2}\,,
       \end{split}
   \end{equation*}
with $C,C'=C,C'(n,\l,\L,i_a,s_a)>0$. Then we use \eqref{eccoci} with $r_2=R_m$ and $r_1=R_{m+1}$, and  that $(\mathrm{M}^+-\kappa_m)\leq \hat{C}\,\mathrm{M}^+/2^{t_0-2}$, so  we find
\begin{equation*}
    \left(\int_{A^+(\kappa_m,R_{m+1})\setminus A^+(\kappa_{m+1},R_{m+1})}|D(D_n\ve)|^2\,dx \right)^{1/2}\leq C\,\frac{2^{m+2}}{\tau_\gamma R}\,\frac{\mathrm{M}^+}{2^{t_0-2} }|A^+(\kappa_m,R_m)|^{1/2}.
\end{equation*}
Merging the content of the two inequalities above, and dividing both sides of the resulting equation by $\mathrm{M}^+/2^{t_0}$, we get
\begin{equation*}\label{iterazione7}
    |A^+(\kappa_{m+1},R_{m+1})|^{\frac{n-1}{n}}\leq C\,\frac{4^{m}}{\tau_\gamma R}\,|A^+(\kappa_m,R_m)|\,.
\end{equation*}
with $C=C(n,\l,\L,i_a,s_a)>0$. Hence by setting $ Z_m=\frac{|A^+(\kappa_m,R_m)|}{|B^+_{R_m}|}$,
and by exploiting that $\tau_\gamma R/2\leq R_m\leq \nu_0 \tau_\gamma R$ and $\nu_0\in (1/2,1)$,  the above inequality and \eqref{smdnve2} imply
\begin{equation*}
    Z_{m+1}\leq C\, (4^{\frac{n}{n-1}})^{m}\,Z_m^{\frac{n}{n-1}},\quad\text{and}\quad Z_0\leq \theta_0\,,
\end{equation*}
with $C=C(n,\l,\L,i_a,s_a)$. Thus, by Lemma \ref{lem:hyp}, choosing $\theta_0=\theta_0(n,\l,\L,i_a,s_a)$ small enough, we get $\lim_{m\to \infty }Z_m=0$, which implies
\begin{equation*}
    D_n\ve\leq \Big(1-\frac{\hat{C}}{2^{t_0-2}} \Big)\mathrm{M}^++\hat{C}\frac{\mathrm{M}^+}{2^{t_0-1}}=\Big(1-\frac{\hat{C}}{2^{t_0-1}} \Big)\,\mathrm{M}^+\,,\quad \text{a.e. in $B^+_{\tau_\gamma R/2}$,}
\end{equation*}
which is in contradiction with \eqref{csfindn} and the continuity of the trace operator. Hence only Case 1 can occur, for which we proved the validity of the lemma.
This concludes the proof of \eqref{tteesi}. Finally, the proof of \eqref{specular} is completely specular, and is left to the reader.
\end{proof}

We are now in the position to prove the excess decay and oscillation estimates for the approximating functions $v_\e$ solutions to \eqref{homve:neu}.

\begin{proposition}\label{propneu:excess}
    Let $\ve\in W^{1,2}(B^+_{2R_0})$ be a weak solution to \eqref{homve:neu}. Then there exists $\beta_N\in (0,1)$ depending only in $n,\l,\L,i_a,s_a$ such that $\ve\in C^{1,\beta_N}(B^+_{R_0/2})$. Moreover, for every $0<r\leq R\leq R_0/2$, the excess decay estimate
\begin{equation}\label{fpt}
   \mint_{B_r^+} |D\ve-(D \ve)_{B_r^+}|\,dx\leq C\,\Big( \frac{r}{R}\Big)^{\beta_N}\,\mint_{B_R^+}|D\ve-(D \ve)_{B_R^+}|\,dx,
\end{equation}
holds, and
\begin{equation}\label{fqt}
    \operatorname*{osc}_{B_r^+} D\ve\leq C\,\Big( \frac{r}{R}\Big)^{\beta_N}\, \operatorname*{osc}_{B_R^+} D\ve\leq C'\,\Big( \frac{r}{R}\Big)^{\beta_N}\bigg\{ \mint_{B^+_{2R}}|D\ve|\,dx+C\,b^{-1}(|h_0|)\bigg\},
\end{equation}
\begin{equation}\label{fqt1}
      \operatorname*{osc}_{B_r^+} D\ve\leq C\,\left(\frac{r}{R} \right)^{\beta_N}\,\mint_{B^+_R}|D\ve-(D\ve)_{B^+_R}|\,dx,\quad 0<r\leq R/2.
\end{equation}
for constants $C,C'=C,C'(n,\l,\L,i_a,s_a)>0$.
\end{proposition}

\begin{proof}
    By combining Lemmas \ref{lemma:alt1Dn}-\ref{lemma:alt2Dn}, we deduce the existence of $\tau_\gamma=\tau_\gamma(n,\l,\L,i_a,s_a)\in (0,1) $ and $\eta_0=\eta_0(n,\l,\L,i_a,s_a)\in (0,1)$ such that, for every $0<R\leq R_0/2$,
    \begin{equation*}
        \text{either } |D_n\ve|\geq\frac{M^+(2R)}{4}\quad\text{or}\quad |D_n\ve|\leq \eta_0\,M^+(2R)\quad\text{in $B^+_{\tau_\gamma R/2}$.}
    \end{equation*}
Then, by Lemma \ref{lem:alttang} applied with $\tau(n,\l,\L,i_a,s_a)=\tau_\gamma/2$, we find $\eta_1=\eta_1(n,\l,\L,i_a,s_a)\in (0,1)$ such that
\begin{equation*}
    \text{either }\, T^+(\tau_\gamma R/2)\geq M^+(2R)/4\quad \text{or}\quad T^+(\tau_\gamma R/2)\leq \eta_1\,M^+(2R).
\end{equation*}
From the two expressions above, it follows that only the following two alternatives can hold:
\begin{equation}\label{neum:alternative}
     \text{either }\, M^+(\tau_d R)\geq \frac{M^+(2R)}{4}\quad \text{or}\quad M^+(\tau_d R)\leq \eta_2\,M^+(2R)
\end{equation}
for all $0<R\leq R_0/2$, where we set $\tau_d(n,\l,\L,i_a,s_a)=\tau_\gamma/2$ and $\eta_2(n,\l,\L,i_a,s_a)=\max\{\eta_0,\eta_1\}\in (0,1)$. Using \eqref{neum:alternative},  the proof of \eqref{fpt}now follows exactly the same lines as that of Proposition~\ref{exs:Dve}. The only difference is that the excess is defined by
\[
    E^{+}(r)=\mint_{B^{+}_{r}}\bigl|D\ve-(D\ve)_{B^{+}_{r}}\bigr|\,dx,
\]
and that, in the iteration, powers of $\tau_{d}$ replace the negative powers of~$2$.
Then, by using \eqref{fpt}, \eqref{bbb:neu}  and Campanato chacterization of H\"older continuity on half balls \cite[Theorem 5.5]{GM12}, the proof of \eqref{fqt}-\eqref{fqt1} is identical to that of Corollary \ref{osc:decay}.  We omit the details.
\end{proof}

Finally, we are able to prove Theorem \ref{thm:homneu} by means of an approximation procedure, analogous to that of Sections \ref{sec:int0}, \ref{sec:dirhomog}.

\begin{proof}[Proof of Theorem \ref{thm:homneu}]
The argument is very similar to the proof of Theorem \ref{thm:inthom} and Theorem \ref{thm:gradddir}. Let $v\in W^{1,B}(B^+_{3R_0})$, and let $v^{e}$ be its even extension to $B_{3R_0}$ given by \eqref{def:even}. For $k$ large enough, let 
\begin{equation*}
    v_k^{bd}(x)=v^{e}\ast\rho_{1/k}(x),\quad x\in B_{2R_0},
\end{equation*}
and consider $v_{\e,k}\in W^{1,2}(B^+_{2R_0})$ the unique solution to
\begin{equation}\label{neu:vek}
\begin{cases}
    -\mathrm{div}\big(\Ae(Dv_{\e,k}) \big)=0\quad&\text{in $B^+_{2R_0}$}
    \\
    \Ae(Dv_{\e,k})\cdot e_n+h_0=0\quad&\text{on $B^0_{2R_0}$}
    \\
    v_{\e,k}=v_k^{bd}\quad&\text{on $\partial B^+_{2R_0}\setminus B^0_{2R_0}.$}
    \end{cases}
\end{equation}
Existence and uniqueness of \eqref{neu:vek} is guaranteed by the theory of monotone operators \cite[Theorem 26.A]{Z90}-- see also the proof of Proposition \ref{prop:existmixed} below.

Then, by applying Theorem \ref{thm:bdd}, Proposition \ref{prop:gradneuLinf}, Proposition \ref{propneu:excess} and a standard covering argument, we find
\begin{equation}\label{uffa}
    \|v_{\e,k}\|_{C^{1,\beta_N}(\overline B^+_r)}\leq C_r\,\Big(1+\int_{B^+_{2R_0}}|D\ve|\,dx \Big),
\end{equation}
    for every $0<r<2R_0$, with $C_r$ independent of $\e,k$.
Moreover, by testing the weak formulation of \eqref{neu:vek} with $v_{\e,k}-v_k^{bd}$, we get
\begin{equation*}
    \begin{split}
        \int_{B^+_{2R_0}}\Ae(Dv_{\e,k})\cdot D(v_{\e,k}-v^{bd}_k)\,dx=h_0\,\int_{B_{2R_0}^0}(v_{\e,k}-v_k^{bd})dx'\,.
    \end{split}
\end{equation*}
By means of  \eqref{coer:Ae} and Young's inequality \eqref{wBbt}, we infer
\begin{equation*}
    \begin{split}
        &\int_{B^+_{2R_0}}\Ae(Dv_{\e,k})\cdot Dv_{\e,k}\,dx\geq c\,\int_{B^+_{2R_0}}B_\e(|Dv_{\e,k}|)\,dx
        \\
         &\bigg|\int_{B^+_{2R_0}}\Ae(Dv_{\e,k})\cdot Dv_k^{bd}\,dx\bigg|\leq \frac{c}{4}\, \int_{B^+_{2R_0}}B_\e(|Dv_{\e,k}|)\,dx+C\,\int_{B^+_{2R_0}}B_\e(|Dv_k^{bd}|)\,dx,
    \end{split}
\end{equation*}
    while by the trace inequality \eqref{in:trace0},
    Young's inequality \eqref{Young} and Remark \ref{remark:importante}, we get
    \begin{equation*}
        \begin{split}
            \bigg|h_0\,\int_{B_{2R_0}^0}(v_{\e,k}-v_k^{bd})dx'\bigg| &\leq C\,|h_0|\,\int_{B^+_{2R_0}}|Dv_{\e,k}-Dv_k^{bd}|\,dx
            \\
            &\leq \frac{c}{4}\int_{B^+_{2R_0}}B_\e\big(|Dv_{\e,k}| \big)\,dx+C'\,\int_{B^+_{2R_0}}B_\e\big(|Dv_k^{bd}|+1\big)\,dx,
        \end{split}
    \end{equation*}
    with $C,C'$ independent of $\e,k$. Merging the content of the four expressions above, and using \eqref{BE:unif} and that $v_k^{bd}\xrightarrow{k\to\infty}v$ in $W^{1,B}(B^+_{2R_0})$ by the properties of convolution, we deduce 
    \begin{equation*}
        \limsup_{k\to \infty}\limsup_{\e\to 0}\int_{B^+_{2R_0}}B_\e\big(|Dv_{\e,k}| \big)\,dx\leq C\bigg(1+\int_{B^+_{2R_0}}B(|Dv|)\,dx \bigg),
    \end{equation*}
    with $C>0$ independent of $\e,k$. We use this piece of information coupled with \eqref{uffa} and the same argument of  \eqref{energy:bound}-\eqref{calBmon} and discussion below, thus getting, up to subsequences,
    \begin{equation}\label{uffa1}
        \lim_{k\to \infty}\lim_{\e\to 0}v_{\e,k}= w\quad\text{in $C^{1}(\overline B^+_{r})$, for all $0<r<2R_0$,}
    \end{equation}
with $w\in W^{1,B}(B^+_{2R_0})$ satisfying  $w=v$ on $\partial B_{2R_0}^+\setminus B^0_{2R_0}$.
    
Letting $\e\to 0$ and $k\to \infty$ in the weak formulation of \eqref{neu:vek}, and using Lemma \eqref{Aeunif} and \eqref{uffa1}, we deduce that $w\in W^{1,B}(B^+_{2R_0})$ is solution to
\begin{equation*}
\begin{cases}
    -\mathrm{div}\big(\A(Dw) \big)=0\quad & \text{in $B^+_{2R_0}$}
    \\
    \A(Dw)\cdot e_n+h_0=0\quad&\text{on $B^0_{2R_0}$}
    \\
    w=v\quad &\text{on $\partial B^+_{2R_0}\setminus B^0_{2R_0}.$}
    \end{cases}
\end{equation*}
hence $w=v$ by uniqueness. Thereby using \eqref{uffa1}, and  passing to the limit  in \eqref{bbb:neu}, \eqref{fpt}-\eqref{fqt1}, we finally obtain \eqref{neu:EST}-\eqref{NEU:est2}. This concludes the proof.
\end{proof}

\section{Proof Theorem \ref{thm:interior}}\label{sec:thmint}

This section is devoted to the proof of the interior gradient regularity of solutions to \eqref{eq1}.
The argument relies on the so-called perturbation method. 
More precisely, we consider a function $u_{0}$ solving the homogeneous problem with frozen coefficients
\[
    -\operatorname{div}\A(x_{0},Du_{0})=0\quad\text{in $B_R(x_0)$}
\]
and $u_0=u$ on $\partial B_R(x_0)$, for $x_0\in \Omega$, $R>0$ such that $B_R\Subset \Omega$. By the results Section \ref{sec:int0}, $u_{0}$ enjoys fine oscillation estimates, which  
 are then transferred to the original solution $u$ via a comparison argument.  Namely, we test the equations satisfied by $u$ and $u_{0}$ with $u-u_{0}$, and by 
exploiting the coercivity and H\"older continuity of $\A(x,\xi)$, we obtain Campanato-type estimates for $u$, and the H\"older continuity then follows as a consequence of Campanato's theorem.
\vspace{0.1cm}

So let $u\in W^{1,B}_{loc}(\Omega)$ be a weak solution to \eqref{eq1}.
First, we establish existence and some estimates for the solution to the homogeneous frozen problem.

\begin{proposition}\label{prop:ex}
    Let $x_0\in \Omega$ and $B_{R}(x_0)\Subset \Omega$. Then there exists a unique solution $u_0\in W^{1,B}(B_R(x_0))$ of the problem
    \begin{equation}\label{frozen}
        \begin{cases}
            -\mathrm{div}\big(\A(x_0,Du_0) \big)=0\quad& \text{in $B_R(x_0)$}
            \\
            u_0=u\quad & \text{on $\partial B_R(x_0)$.}
        \end{cases}
    \end{equation}
    Moreover, the estimate
\begin{equation}\label{en:u0}
  \int_{B_R(x_0)}B\big(|Du_0| \big) \,dx\leq C\,\int_{B_R(x_0)}\big[ B\big(|Du| +1\big)\big] \,dx,
\end{equation}
    holds true with $C=C(n,\l,\l,i_a,s_a)>0$.
\end{proposition}

\begin{proof}
In the course of the proof, the center of the balls will be implicitly taken to be $x_{0}$,  that is, we write $B_{R}=B_{R}(x_{0})$.
To prove the existence of $u_0$, we want to apply \cite[Theorem 26.A]{Z90} with function space $X=W^{1,B}_0(B_R)$, endowed with norm $\|v\|=\|Dv\|_{L^B(B_R)}$, and operator
    \begin{equation*}
     \langle Av_1,v_2\rangle=\int_{B_R}\A(x_0,Dv_1+Du)\cdot Dv_2\,dx,\quad v_1,v_2\in X.
    \end{equation*}
First observe that $A$ is a monotone operator; indeed, if $v_1\neq v_2$ (and thus $Dv_1\not\equiv Dv_2$), by \eqref{strong:coer} we have
\begin{equation*}
    \langle Av_1-A v_2,v_1-v_2\rangle=\int_{B_R} \big[\A(x_0,Du+Dv_1)-\A(x_0,Du+Dv_2)\big]\cdot D(v_1-v_2)\,dx>0
\end{equation*}
  Let us show the hemicontinuity of $A$, that is the continuity of the map
    \begin{equation*}
        [0,1]\ni t\mapsto \langle A(v_1+tv_2),v_3\rangle=\int_{B_R}\A(x_0, Du+Dv_1+tDv_2)\cdot Dv_3\,dx.
    \end{equation*}
 By\eqref{A:hold}$_2$, \eqref{co:gr}, Young's inequality \eqref{young1}, \eqref{simple} and the monotonicity of $B$, we have 
    \begin{equation*}
    \begin{split}
        \A(x_0, Du+Dv_1+tDv_2)\cdot Dv_3&\leq C\,b\big(|Du|+|D v_1|+|Dv_2| \big)\,|Dv_3|+C\,|Dv_3|
        \\
        &\leq C' B\big(|Du|+|D v_1|+|Dv_2|+|Dv_3| \big)+1
        \end{split}
    \end{equation*}
    hence the hemicontinuity of $A$  follows by the continuity of $\xi\mapsto \A(x,\xi)$
    and dominated convergence theorem. We are left to prove the coercivity of $A$. By \eqref{co:gr}, \eqref{A:hold}$_2$, \eqref{triangle} and \eqref{simple}, we have
    \begin{equation*}
    \begin{split}
        \langle Av,v\rangle=&\int_{B_R}\A(x_0,Du+Dv)\cdot Dv\,dx
        \\
        \geq& \,c\,\int_{B_R}B\big(\big|Dv \big|\big)\,dx-C\,\int_{B_R}B\big(|Du| \big) \,dx-C\,\L\,\int_{B_R}|Dv|\,dx
        \\
        \geq &\,c'\int_{B_R}B(|Dv|)\,dx-C\,\int_{B_R}B\big(|Du| \big) \,dx-C'\,\widetilde{B}(1)\,|B_R|,
        \end{split}
    \end{equation*}
    with $c,c',C,C'>0$ depending on $n,\l,\L,i_a,s_a$. In particular, assuming $\|v\|=\|Dv\|_{L^B(B_R)}\geq 1$, from the above inequality and \eqref{B2t}, we have
    \begin{equation*}
    \begin{split}
        \langle Av,v\rangle&\geq c\,\|Dv\|^{i_B}_{L^B(B_R)}\,\int_{B_R}B\bigg(\frac{|Dv|}{\|Dv\|_{L^B(B_R)}} \bigg)\,dx-C\,\int_{B_R}B\big(|Du| \big) \,dx-C'\,\widetilde{B}(1)\,|B_R|
        \\
        &\stackrel{\eqref{modular1}}{=}c\,\|Dv\|_{L^B(B_R)}^{i_B}-C\,\int_{B_R}B\big(|Du| \big) \,dx-C'\,\widetilde{B}(1)\,|B_R|
\end{split}   
    \end{equation*}
    Since $i_B>1$, it follows that
    \begin{equation*}
        \lim_{\|v\|\to +\infty}\frac{\langle Av,v\rangle}{\|v\|}= +\infty
    \end{equation*}
    that is the desired coercivity. 
    Hence, \cite[Theorem 26.A]{Z90} ensures the existence and uniqueness of $v_0$ solution to $ Av_0=0$ in  $X^*$ (the dual space of $X$), which is equivalent to the solvability of \eqref{frozen} with $u_0=v_0+u$. Next, we test the weak formulation of \eqref{frozen} with $u-u_0$, thus getting
\begin{equation*}
\begin{split}
    \int_{B_R} B\big(|Du_0| \big)\,dx&\stackrel{\eqref{co:gr}}{\leq } C\,\int_{B_R} \big(\A(x_0,Du_0)-\A(x_0,0)\big)\cdot Du_0\,dx
    \\
    &\stackrel{\eqref{frozen}}{=}C\,\int_{B_R} \A(x_0,Du_0)\cdot Du-C   \,\int_{B_R}\A(x_0,0)\cdot Du_0\,dx
    \\
    &\stackrel{\eqref{co:gr},\eqref{A:hold}_2}{\leq} C'\,\int_{B_R} b\big(|Du_0|\big)\,|Du|\,dx+C\,\L\,\int_{B_R}|Du_0|\,dx
    \\
    &\stackrel{\eqref{young1},\eqref{Young},\eqref{tB=1}}{\leq}\frac{1}{2}\,\int_{B_R} B\big(|Du_0|\big)\,dx+C''\,\int_{B_R}\big[ B(|Du|)+1\big]\,dx,
    \end{split}
\end{equation*}
    with $C,C',C''>0$ depending on $n,\l,\L,i_a,s_a$.  Equation \eqref{en:u0} thus follows.
\end{proof}

Next, we recollect some standard estimates on $u_0$, that immediately follow from Theorem \ref{thm:inthom}.
\begin{proposition}\label{pr:otherest}
    Let $u_0\in W^{1,B}(B_R(x_0))$ be the solution to \eqref{frozen}. Then there exists $C=C(n,\l,\L,i_a,s_a)>0$ such that
\begin{equation}\label{du00}
    \mint_{B_r(x_0)} B\big(|Du_0| \big)\,dx\leq C\,\mint_{B_R(x_0)} B\big(|Du_0| \big)\,dx,
\end{equation}
and the following excess decay estimate 
\begin{equation}\label{excBB}
    \mint_{B_r(x_0)}B\Big(\big|Du_0-(Du_0)_{B_r(x_0)} \big|\Big)\,dx\leq C\,\left( \frac{r}{R}\right)^{\alpha_{\mathrm{h}}\,i_B} \mint_{B_R(x_0)}B\Big(\big|Du_0-(Du_0)_{B_R(x_0)} \big|\Big)\,dx,
\end{equation}
   holds for every $0<r\leq R$, with $\alpha_{\mathrm{h}}=\alpha_{\mathrm{h}}(n,\l,\L,i_a,s_a)\in (0,1)$ given by Theorem \ref{thm:inthom}.
\end{proposition}
\begin{proof}
   In the case $R/4\leq r\leq R$, Equation \eqref{du00} is immediate to prove, while \eqref{excBB} easily follows from \eqref{added}. So let us assume that $0<r\leq R/4$. We compute
    \begin{multline*}
        \mint_{B_r(x_0)} B\big( |Du_0|\big)\,dx\leq B\Big(\sup_{B_r(x_0)}|Du_0| \Big)\leq B\Big(\sup_{B_{R/4}(x_0)}|Du_0| \Big)
            \\
            \stackrel{\eqref{inf:hom}}{\leq} B\Big(c_{\mathrm{h}}\mint_{B_{R/2}(x_0)}|Du_0|\,dx \Big)            \stackrel{\eqref{B2t},\text{Jensen}}{\leq} C\,\mint_{B_R(x_0)} B\big(|Du_0| \big)\,dx,
    \end{multline*}
 and \eqref{du00} is proven. Then, by means of \eqref{oscdes}, \eqref{B2t} and Jensen inequality, we obtain
    \begin{equation*}
        \begin{split}
            \mint_{B_r(x_0)}B\Big(\big|Du_0-(Du_0)_{B_r(x_0)} \big|\Big)\,dx &\leq B\Big( \operatorname*{osc}_{B_r(x_0)}Du_0\Big)
            \\
            &\leq B\Big(C\,\Big(\frac{r}{R}\Big)^{\alpha_\mathrm{h}}\,\mint_{B_{R}(x_0)}|Du_0-(Du_0)_{B_R(x_0)}|\,dx\Big)
            \\
            &\leq C'\,\Big(\frac{r}{R}\Big)^{\alpha_\mathrm{h}\,i_B}B\Big(\mint_{B_R(x_0)}|Du_0-(Du_0)_{B_R(x_0)}|\,dx\Big)
            \\
            &\leq C'\,\Big(\frac{r}{R}\Big)^{\alpha_\mathrm{h}\,i_B}\mint_{B_R(x_0)}B\Big(|Du_0-(Du_0)_{B_R(x_0)}|\Big)\,dx,
        \end{split}
    \end{equation*}
which proves \eqref{excBB}.
\end{proof}
Our next, important result is the comparison estimate between $u$ and $u_0$.
\begin{proposition}[Comparison estimate]\label{prop:comparison}
Let $u,u_0, x_0$ and $R$ be as in Proposition \ref{prop:ex}. Then there exist $C=C(n,\l,\L,\L_\mathrm{h},i_a,s_a)>0$ and $\theta=\theta(n,d,\alpha)\in (0,1/2)$ such that
\begin{equation}\label{comp:fin1}
\begin{split}
    \int_{B_R(x_0)} B\big(|Du-Du_0| \big) \,dx\leq C\,\Big( 1+\|f\|_{L^d(B_R(x_0))}\Big)&\,\Big(\int_{B_R(x_0)}\big[B\big(|Du|\big)+1\big]\,dx\Big)\,R^\theta.
    \end{split}
\end{equation}
\end{proposition}
\begin{proof}
    We test the weak formulation of Equations \eqref{eq1} and \eqref{frozen} with $u-u_0$, and using \eqref{strong:coer} we get 
\begin{equation*}
    \begin{split}
    \int_{B_R(x_0)}a&\big(|Du|+|Du_0| \big)\,|Du-Du_0|^2\,dx 
\\
&\leq C\,\int_{B_R(x_0)}\big(\A(x_0,Du)-\A(x_0,Du_0)\big)\cdot \big(Du-Du_0 \big)\,dx
\\
&=C\,\int_{B_R(x_0)}\A(x_0,Du)\cdot \big(Du-Du_0 \big)\,dx
\\
&=C\,\int_{B_R(x_0)}\big(\A(x_0,Du)-\A(x,Du)\big)\cdot \big(Du-Du_0 \big)\,dx+C\,\int_{B_R(x_0)}f\,(u-u_0)\,dx;
    \end{split}
\end{equation*}
then, by using \eqref{A:hold}, H\"older and Sobolev inequalities, \eqref{simple} and \eqref{en:u0}, we obtain
\begin{equation}
        \begin{split}
            \int_{B_R(x_0)}a&\big(|Du|+|Du_0| \big)\,|Du-Du_0|^2\,dx
\\
&\leq C\,R^{\alpha}\,\int_{B_R(x_0)}|Du-Du_0|\,dx+C\,\Big(\int_{B_R(x_0)}|f|^n\,dx \Big)^{1/n}\,\Big( \int_{B_R(x_0)}|u-u_0|^{n'}\,dx\Big)^{1/n'}
\\
&\leq C'\,R^{\alpha}\int_{B_R(x_0)} \big[ B\big(|Du| \big)+1\big]\,dx+C'\Big(\int_{B_R(x_0)}|f|^n\,dx \Big)^{1/n}\, \int_{B_R(x_0)}|Du-Du_0|\,dx
  \\
&\leq  C''\,\Big(R^{\alpha}+R^{1-n/d}\,\|f\|_{L^d(B_R(x_0))}\Big)\,\int_{B_R(x_0)} \big[ B\big(|Du| \big)+1\big]\,dx.
        \end{split}
    \end{equation}
with $C,C',C''>0$ depending on $n,\l,\L,\L_\mathrm{h},i_a,s_a$. Coupling the above inequality with \eqref{el:comparison} and \eqref{en:u0}, we infer
\begin{equation*}
    \int_{B_R(x_0)} B\big(|Du-Du_0| \big)\,dx\leq C''\,\Big\{\delta+\delta^{-1}\Big(R^{\alpha}+R^{1-n/d}\,\|f\|_{L^d(B_R(x_0))}\Big)\Big\}\,\int_{B_R(x_0)} \big[ B\big(|Du| \big)+1\big]\,dx
\end{equation*}
for all $\delta\in (0,1)$. Thereby choosing $\delta=\max\big\{R^{\alpha/2}, R^{(1-n/d)/2}\big\}$, we finally obtain \eqref{comp:fin1} with 
\[\theta=\min\{\alpha/2, (1-n/d)/2\}.\]

\end{proof}

In our next result, we show that $B(|Du|)$ belongs to the Morrey space $\mathcal{L}^{1,\mu}$ for every $\mu\in (0,n)$. For further details on these spaces, we refer to \cite[Section 5.1]{GM12}.

\begin{lemma}[A Morrey-type estimate]\label{lem:morrey}
    Let $\Omega'\Subset \Omega$, and $u\in W^{1,B}_{loc}(\Omega)$ be a weak solution to \eqref{eq1}.  Then for every $\mu\in (0,n)$, there exists a constant  $C_\mu=C_\mu(n,\l,\L,\L_\mathrm{h},i_a,s_a,\mu)>0$ and radius $R_\mu\in (0,1)$ depending on $n,\l,\L,\L_\mathrm{h},i_a,s_a,\alpha,d,\|f\|_{L^d(\Omega')}$ and $\mu$ such that 
\begin{equation}\label{est:morrey}
    \int_{B_R(x_0)} B\big(|Du
    |\big)\,dx\leq C_\mu\,\bigg(1+\|f\|_{L^d(\Omega')}+\frac{1}{R_0^\mu}\int_{B_{R_0}(x_0)}B\big(|Du| \big)\,dx \bigg)\,R^\mu,\quad R\leq R_0
\end{equation}
for every ball $B_{R_0}(x_0)\Subset \Omega'$, with $R_0\leq R_\mu$.
\end{lemma}

\begin{proof}
   Let $B_{R_0}(x_0)\Subset \Omega'$ be as in the statement, and estimate $\|f\|_{L^d(B_R(x_0))}\leq \|f\|_{L^d(\Omega')}$.
Then for all $0<r\leq R\leq R_0$, we compute
{\small
\begin{equation*}
    \begin{split}
        \int_{B_r(x_0)}B\big(|Du| \big)\,dx \stackrel{\eqref{triangle}}{\leq} &C\,\int_{B_r(x_0)}B\big(|Du-Du_0| \big)\,dx+C\,\int_{B_r(x_0)} B\big( |Du_0|\big)\,dx
        \\
        \stackrel{\eqref{du00}}{\leq} &C\,\int_{B_R(x_0)}B\big(|Du-Du_0| \big)\,dx+C'\,\Big(\frac{r}{R} \Big)^n\,\int_{B_R(x_0)} B\big( |Du_0|\big)\,dx
        \\
        \stackrel{\eqref{comp:fin1},\eqref{en:u0}}{\leq} & C''\,\bigg(\Big(\frac{r}{R} \Big)^n+\Big(1+ \|f\|_{L^d(\Omega'))}\Big)\,R^\theta\bigg)\,\int_{B_R(x_0)}B\big(|Du|\big)\,dx +C''\,\Big(1+ \|f\|_{L^d(\Omega')}\Big)\,R^n\,.
    \end{split}
\end{equation*}
}
with $C,C',C''>0$ depending on $n,\l,\L,\L_\mathrm{h},i_a,s_a$. Setting $\phi(r)=\int_{B_r(x_0)}B\big(|Du| \big)\,dx$, and using $R^n\leq R^{\mu}$, the above inequality can be rewritten as
\begin{equation*}
    \phi(r)\leq C''\,\bigg(\Big(\frac{r}{R} \Big)^n+\Big(1+ \|f\|_{L^d(\Omega'))}\Big)\,R^\theta\bigg)\,\phi(R)+C\,\Big(1+ \|f\|_{L^d(\Omega')}\Big)\,R^\mu;
\end{equation*}
hence, by taking 
\begin{equation}\label{Rmu}
    R_0\leq R_\mu:=\bigg(\frac{1}{2C''(1+\|f\|_{L^d(\Omega')})} \bigg)^{\frac{2n}{(n-\mu)\theta}}
\end{equation}
the thesis follows via an application of Lemma \ref{lem:iteration}.
\end{proof}

We now have all the ingredients to prove the interior regularity of the gradient.

\begin{proof}[Proof of Theorem \ref{thm:interior}] 
 Let $\Omega'\Subset \Omega$, and for $\theta=\theta(n,d,\alpha)$ provided by Proposition \ref{prop:comparison}, we fix
 {\small
 \begin{equation}\label{fix:mu}
     \mu=\mu(n,d,\alpha):=n-\frac{\theta}{2}\in (0,n)\xRightarrow{\text{Lemma \ref{lem:morrey}
     }} R_\mu=R_\mu(n,\l,\L,\L_\mathrm{h},i_a,s_a,d,\alpha,\|f\|_{L^d(\Omega')}).
 \end{equation}
 }
Then let $B_{2R_0}\Subset \Omega'$, be such that $R_0\leq R_\mu$, and  consider $x_0\in B_{R_0}$, and  $0< R\leq R_0$. 

We first observe that, by combining \eqref{comp:fin1} with \eqref{est:morrey} and the choice of $\mu$ in \eqref{fix:mu}, we obtain the improved comparison estimate
 \begin{equation}\label{comp:morrey}
     \mint_{B_R(x_0)} B\big(|Du-Du_0| \big)\,dx\leq C_{f,R_0}\,R^{\theta/2}
 \end{equation}
 where, by also using that $B_{R_0}(x_0)\subset B_{2R_0}\subset \Omega'$, we set
 \[
C_{f,R_0}:=C(n,\l,\L,\L_\mathrm{h},i_a,s_a,d,\alpha)\,\Big( 1+\|f\|_{L^d(\Omega')}\Big)\,\bigg(1+\|f\|_{L^d(\Omega')}+\frac{1}{R_0^\mu}\int_{\Omega'} B\big( |Du|\big)\,dx \bigg).
 \]
 Then, for $0<r\leq R\leq R_0$,  we estimate 
{\small
\begin{equation}\label{urcs}
    \begin{split}
        \mint_{B_r(x_0)} &B\big(|Du-(Du)_{B_r(x_0)}| \big)\,dx
        \\
        \stackrel{\eqref{triangle}, \text{Jensen}}{\leq} &\,C\,\mint_{B_r(x_0)}B\big(|Du-Du_0| \big)\,dx+C\,\mint_{B_r(x_0)}B\big(|Du_0-(Du_0)_{B_r(x_0)}| \big)\,dx
        \\
        \stackrel{\eqref{excBB}}{\leq}&\,C\,\Big(\frac{R}{r} \Big)^n\,\mint_{B_R(x_0)}B\big(|Du-Du_0| \big)\,dx
        \\
        &+C'\,\left( \frac{r}{R}\right)^{\alpha_{\mathrm{h}}i_B} \mint_{B_R(x_0)}B\big(|Du_0-(Du_0)_{B_R(x_0)}| \big)\,dx
         \\
        \stackrel{\eqref{triangle}\text{,Jensen}}{\leq} &\,C_1\bigg(\Big(\frac{R}{r} \Big)^n+\Big(\frac{r}{R}\Big)^{\alpha_{\mathrm{h}}i_B} \bigg)\,\mint_{B_R(x_0)}B\big(|Du-Du_0| \big)\,dx
        \\
        &+C''\,\left( \frac{r}{R}\right)^{\alpha_{\mathrm{h}}i_B} \mint_{B_R(x_0)}B\big(|Du-(Du)_{B_R(x_0)}| \big)\,dx
        \\
        \stackrel{\eqref{comp:morrey}}{\leq} &\,C_2\, C_{f,R_0}\,\Big(\frac{R^{n+\theta/2}}{r^n}\Big)+C''\,\left( \frac{r}{R}\right)^{\alpha_{\mathrm{h}}i_B} \mint_{B_R(x_0)}B\big(|Du-(Du)_{B_R(x_0)}| \big)\,dx.
    \end{split}
\end{equation}
}
with $C,C',C'',C_1,C_2>0$ depending on $n,\l,\L,i_a,s_a$.
Let us now set 
\begin{equation}\label{fix:vphi}
    \varphi(t):=\mint_{B_t(x_0)}B\big(|Du-(Du)_{B_t(x_0)}| \big)\,dx,
\end{equation}
and choose a parameter 
\begin{equation}\label{parameter}
    \tau=\tau(n,\l,\L,i_a,s_a)\in (0,1)\quad\text{such that }\, C''\,\tau^{(\alpha_{\mathrm{h}}i_B)/2}\leq 1.
\end{equation}
By doing so, from \eqref{urcs} we deduce
\begin{equation}\label{fin:itur}
    \varphi(\tau R)\leq \tau^{(\alpha_{\mathrm{h}}i_B)/2}\,\varphi(R)+C_2\,\tau^{-n}\,C_{f,R_0} R^{\theta/2}\quad \text{for all $0<R\leq R_0$.}
\end{equation}
Moreover, by \eqref{added} and our choice of $\tau$ in \eqref{parameter}, it is immediate to verify that
\[\varphi(t)\leq C(n,\l,\L,\L_\mathrm{h},i_a,s_a)\,\varphi(R)\quad\text{for all $t\in (\tau^{k+1}R,\tau^kR)$}.
\]
Hence, thanks to \eqref{fin:itur}, we may apply Lemma \ref{lem:ultiter} and deduce
{\small
\begin{equation}\label{Bcampanato}
     \mint_{B_r(x_0)} B\big(|Du-(Du)_{B_r(x_0)}| \big)\,dx\leq C\,\bigg[\frac{1}{R_0^\vartheta}\, \mint_{B_{R_0}(x_0)} B\big(|Du-(Du)_{B_{R_0}(x_0)}| \big)\,dx+C_{f,R_0}\bigg]\,r^\vartheta\leq \hat{C}_{f,R_0}\,r^\vartheta,
\end{equation}
}
for all $0<r\leq R_0$, and $x_0\in B_{R_0}$, where we set  \[\vartheta=\vartheta(n,\l,\L,i_a,s_a,\alpha,d)=\min\big\{(\alpha_{\mathrm{h}}i_B)/2,\theta/2\big\}\in (0,1),
\]
and 
\[
\hat{C}_{f,R_0}=\hat{C}_{f,R_0}(n,\l,\L,\L_\mathrm{h},i_a,s_a,\alpha,d,\|f\|_{L^d(\Omega')},R_0)=C\,\bigg[\frac{1}{R_0^{\vartheta+n}}\, \int_{\Omega'}\Big[ B\big(|Du|\big)+1\Big] \,dx+C_{f,R_0}\bigg].
\]
Finally, by Jensen inequality
\begin{equation*}
    B\Big( \mint_{B_r(x_0)} |Du-(Du)_{B_r(x_0)}|\,dx\Big)\leq  \mint_{B_r(x_0)} B\big(|Du-(Du)_{B_r(x_0)}| \big)\,dx,
\end{equation*}
so coupling this information with \eqref{Bcampanato}, and using  \eqref{B2t} and \eqref{B=1},  we get
\begin{equation}\label{notBcamp}
    \begin{split}
         \mint_{B_r(x_0)} |Du-(Du)_{B_r(x_0)}|\,dx\leq B^{-1}\Big(\hat{C}_{f,R_0}\,r^\vartheta \Big)\leq B^{-1}\Big(\hat{C}_{f,R_0} \Big)\,r^{\vartheta/s_B}\leq C(i_a,s_a)\,(\hat{C}_{f,R_0})^{\frac{1}{i_B}}\,r^{\beta},
    \end{split}
\end{equation}
for all $0<r\leq R_0$, and $x_0\in B_{R_0}$, where we set $\beta=\beta(n,\l,\L,i_a,s_a,\alpha,d):=\vartheta/s_B\in (0,1)$. 

By Campanato characterization of H\"older continuity \cite[Theorem 5.5]{GM12}, estimate \eqref{notBcamp} implies 
\begin{equation}\label{camp:sphere}
    \|Du\|_{C^{0,\beta}(B_{R_0})}\leq C\bigg(n,\l,\L,\L_\mathrm{h},i_a,s_a,\alpha,d,R_0,\int_{\Omega'}B\big(|Du| \big)\,dx,\|f\|_{L^d(\Omega')}\bigg)\,,
\end{equation}
for all $B_{2R_0}\Subset \Omega'$ with $R_0\leq R_\mu$.
\vspace{0.1cm}

Now let $\Omega''\Subset \Omega'\Subset \Omega$; in order to obtain the $C^{1,\beta}$-estimate on $\Omega''$, it suffices to apply a covering argument. More precisely, we cover the set $\overline{\Omega''}$ with balls $\{B_{R_0}(x_i)\}_{i=1}^N$ centered at $x_i\in \overline{\Omega''}$, and of radius 
\begin{equation}\label{R0cov}
    R_0=R_0\big(n,\l,\L,\L_\mathrm{h},i_a,s_a,d,\alpha,\|f\|_{L^d(\Omega')}, \mathrm{dist}(\Omega'',\partial \Omega')\big):=\min\{\mathrm{dist}(\Omega'',\partial \Omega')/4,R_\mu\}
\end{equation}
 so that $B_{2R_0}(x_i)\Subset \Omega'$ for all $i=1,\dots,N$, and the dependence of the constants in~\eqref{R0cov} follows from~\eqref{fix:mu}. Moreover, 
the cardinality of such a cover satisfies
\begin{equation}\label{cover}
    N\leq C(n)\,\bigg(\frac{\mathrm{diam}(\Omega'')}{R_0}\bigg)^n= C\Big(n,\l,\L,\L_\mathrm{h},i_a,s_a,\alpha,d,\|f\|_{L^d(\Omega')},\mathrm{dist}(\Omega'',\partial \Omega'),\mathrm{diam}(\Omega'')\Big).
\end{equation}
Then let $\{\phi^i\}_{i=1}^N$ be a partition of unity associated to $\{B_{R_0}(x_i)\}_{i=1}^N$, and satisfying $\phi^i\in C^\infty_c(B_{R_0}(x_i))$,
\begin{equation*}
  0\leq \phi^i\leq1,\quad   |D\phi^i|\leq C(n)/R_0,
\end{equation*}
We observe that, by \eqref{eq1} and Lemma \ref{lemma:diver}, $u\in W^{1,B}(\Omega')$ is solution to
\begin{equation*}
    -\mathrm{div}\big(\A_f(x,Du)\big)=0\quad\text{in $\Omega'$,}\quad \text{where}\quad \A_f(x,\xi)=\A(x,\xi)+F(x)
\end{equation*}
with $\mathrm{div}F=f$, and by \eqref{co:gr}, \eqref{A:hold}$_2$ and \eqref{divFf} we have
\begin{equation}\label{Afdiv}
\begin{split}
    &\A_f(x,\xi)\cdot \xi\geq c\,B\big(|\xi| \big)-C\,B\Big(\big(1+\|f\|_{L^d(\Omega')}\big)^{\frac{1}{i_b}}\Big) 
    \\
    &|\A_f(x,\xi)|\leq C\,b\big(|\xi| \big)+C\,b\Big(\big(1+\|f\|_{L^d(\Omega')}\big)^{\frac{1}{i_b}}\Big).
    \end{split}
\end{equation}
with $c,C=c,C(n,\l,\L,d,i_a,s_a)>0$, where we also used the elementary inequalities 
\[
\begin{split}
    &(1+t)=B\big(B^{-1}(1+t) \big)\leq C\,B((1+t)^{\frac{1}{i_b+1}})\leq C\,B((1+t)^{\frac{1}{i_b}})
    \\
   & (1+t)=b\big(b^{-1}(1+t) \big)\leq C\,b((1+t)^{\frac{1}{i_b}}),\quad t\geq 0
\end{split}
\]
 which stem from \eqref{b2t}, \eqref{B2t} and the monotonicity of $B$.

Hence, by applying Theorem \ref{thm:bdd} on each ball $B_{R_0}(x_i)$, and recalling \eqref{R0cov}, we deduce
\begin{equation}\label{fin:Linf}
\begin{split}
    \|u\|_{L^\infty(\Omega'')}&\leq \frac{C}{R_0^n}\int_{\Omega'} |u|\,dx+C\,\big(1+\|f\|_{L^d(\Omega')} \big)^{1/i_b}\,R_0
    \\
    &\leq C\bigg(n,\l,\L,i_a,s_a,d,\alpha,\|f\|_{L^d(\Omega')}, \mathrm{dist}(\Omega'',\partial \Omega'),\int_{\Omega'}|u|\,dx\bigg).
    \end{split}
\end{equation}
Then, by \eqref{camp:sphere}-\eqref{cover}, \eqref{fin:Linf} and the properties of $\phi^i$ and $N$, we get
\begin{equation}\label{finalcover}
    \begin{split}
        \|Du\|_{C^{0,\beta}(\Omega'')}&\leq \sum_{i=1}^N  \|(Du)\,\phi^i\|_{C^{0,\beta}(\Omega'')}
        \\
        &\leq C(n)\,\sum_{i=1}^N\bigg\{\|Du\|_{L^\infty(B_{R_0}(x_i))}\,\|\phi^i\|_{C^{0,\beta}(B_{R_0}(x_i))}+\|\phi^i\|_{L^\infty(\Omega'')}\,\|Du\|_{C^{0,\beta}(B_{R_0}(x_i))} \bigg\}
        \\
        &\leq  C(n)\,\Big\{\frac{N}{R_0^\beta}+ N\Big\}\,\|Du\|_{C^{0,\beta}(B_{R_0}(x_i))}
        \\
        &\leq C\bigg(n,\l,\L,\L_\mathrm{h},i_a,s_a,d,\alpha,\|f\|_{L^d(\Omega')}, \mathrm{dist}(\Omega'',\partial \Omega'),\int_{\Omega'}|u|\,dx+\int_{\Omega'}B\big( |Du|\big)\,dx\bigg),
    \end{split}
\end{equation}
which together with \eqref{fin:Linf} yields \eqref{stimaDuint}.
\end{proof}

\section{Boundary \texorpdfstring{$C^{1,\beta}$}{C1b}-regularity, Dirichlet problems}\label{sec:pfdir}
We first study the gradient regularity of solutions $u\in W^{1,B}(B^+_{R_0})$ to the Dirichlet problem
\begin{equation}\label{eq:dirBup}
    \begin{cases}
        -\mathrm{div}\big(\A(x,Du) \big)=f\quad& \text{in $B^+_{R_0}$}
        \\
        u=g\quad &\text{on $B^0_{R_0}$}
    \end{cases}
\end{equation}
with $g\in C^{1,\alpha}(\R^{n-1})$, and $\A(x,\xi)$ satisfying \eqref{new:coer}-\eqref{A:hold1}. In particular, this implies that the quantitative constants will also depend on \(L_\Omega\) or on an upper bound on \(\|\phi\|_{C^{1,\alpha}}\). We keep this dependence explicit throughout the proofs.

We also assume that $R_0\leq 1$,  $B^+_{R_0}\subsetneq \mathcal{V}$ and $f\in L^d(\mathcal{V})$, where $\mathcal V$ is a bounded domain of $\R^{n}_+$.
\vspace{0.1cm}

We then fix $x'_0\in B^0_{R_0}$, and 
 $R\in (0,1)$ such that $B^+_{R}(x'_0)\subset B^+_{R_0}$; let $u_0\in W^{1,B}(B^+_R(x_0'))$ be the weak solution to the homogeneous, frozen Dirichlet problem
\begin{equation}\label{homog:dirx0}
    \begin{cases}
        -\mathrm{div}\big( \A(x'_0,Du_0)=0\quad &\text{in $B^+_R(x_0')$}
        \\
        u_0=u \quad &\text{on $\partial B^+_R(x_0')$}.
    \end{cases}
\end{equation}
As in Section \ref{sec:thmint}, we start with the following
\begin{proposition}\label{propo:dirhomog}
    There exists a unique weak solution $u_0\in W^{1,B}(B^+_R(x_0'))$ to \eqref{homog:dirx0}. Moreover, it satisfies the energy estimate
    \begin{equation}\label{en:u0dir}
        \int_{B^+_{R}(x_0')} B\big(|Du_0| \big)\,dx\leq C\,\int_{B^+_{R}(x_0')}\big[ B\big(|Du| \big)+1\big]\,dx
    \end{equation}
    with $C=C(n,\l,\L,i_a,s_a,L_\Omega)>0$.
\end{proposition}
We omit the proof, as it is entirely analogous to that of Proposition~\ref{prop:ex}, the only difference that \(B_R(x_0)\) is replaced by \(B_R^+(x_0')\), and the presence of $L_\Omega$ in the constants.

The next proposition is the analogue of Proposition \ref{pr:otherest}.
\begin{proposition}\label{prop:direst}
    Let $u_0\in W^{1,B}(B^+_R(x_0'))$ be the unique solution to \eqref{homog:dirx0}. Then there exists $C=C(n,\l,\L,i_a,s_a,\alpha,L_\Omega)>0$ such that, for every $0<r\leq R$, the inequality
    \begin{equation}\label{duu:dir}
        \mint_{B_r^+(x_0')}B\big(|Du_0| \big)\,dx\leq C\,\mint_{B_R^+(x_0')} B\big(|Du_0| \big)\,dx+C\,B\big(\|g\|_{C^{1,\alpha}} \big),
    \end{equation}
    holds true, and we have the decay estimate
    \begin{equation}\label{duurrdir1}
        \mint_{B^+_r(x_0')} B\big(|Du_0-(Du_0)_{B^+_r(x_0')}| \big)\,dx\leq C\,\Big( \frac{r}{R}\Big)^{\beta_{\mathrm{h}}i_B}\,\bigg\{\mint_{B^+_R(x'_0)} B\big(|Du_0| \big)\,dx+B\big(\|g\|_{C^{1,\alpha}} \big)\bigg\}.
    \end{equation}
    with $\beta_{\mathrm{h}}=\beta_{\mathrm{h}}(n,\l,\L,i_a,s_a,\alpha,L_\Omega)\in (0,1)$.
\end{proposition}

\begin{proof}
    In the case $R/8\leq r\leq R$, inequality \eqref{duu:dir} is trivial, while \eqref{duurrdir1} readily follows from \eqref{added}, \eqref{triangle} and Jensen's inequality. Therefore, we may assume that $0<r\leq R/8$. Then, using \eqref{dir:gradbdd}, \eqref{B2t}, \eqref{triangle} and Jensen's inequality, we find
    \begin{equation*}
        \begin{split}
            \mint_{B_r^+(x_0')}B\big(|Du_0| \big)\,dx &\leq \sup_{B^+_r(x_0')} B\big(|Du_0| \big)\leq \sup_{B^+_{R/2}(x_0')} B\big(|Du_0| \big)
            \\
            &\leq C\,B\bigg(\mint_{B^+_R(x_0')}|Du_0|\,dx \bigg)+C\,B\big(\|g\|_{C^{1,\alpha}}\big)
            \\
            &\leq C\,\mint_{B^+_R(x_0')}B\big(|Du_0|\big)\,dx+C\,B\big(\|g\|_{C^{1,\alpha}}\big)
        \end{split}
    \end{equation*}
    with $C=C,(n,\l,\L,i_a,s_a,\alpha,L_\Omega)>0$, so \eqref{duu:dir} is proven. Then by \eqref{dir:gradholder}, \eqref{B2t}-\eqref{triangle} and Jensen's inequality, we find
\begin{equation*}
    \begin{split}
        \mint_{B^+_r(x_0')} B\Big(|Du_0&-(Du_0)_{B^+_r(x_0')}| \Big)\,dx\leq  B\Big(\operatorname*{osc}_{B_r^+(x_0')} Du_0 \Big)
        \\
        &\leq B\bigg(C\,\Big( \frac{r}{R}\Big)^{\beta_{\mathrm{h}}}\Big\{ \mint_{B^+_R(x_0')}|Du_0|\,dx+\|g\|_{C^{1,\alpha}}\Big\} \bigg)
        \\
        &\leq C'\,\Big( \frac{r}{R}\Big)^{\beta_{\mathrm{h}}i_B}\,\bigg\{\mint_{B^+_R} B\big(|Du_0| \big)\,dx+B\big(\|g\|_{C^{1,\alpha}} \big)\bigg\},
    \end{split}
\end{equation*}
    and the proof is complete.
\end{proof}
Next, we state and prove the comparison and Morrey type estimates.

\begin{proposition}[Morrey type estimate]\label{prop:comparisondir}
    Let $u,u_0$  be as above. Then there exist constants 
    \[C=C(n,\l,\L,\L_\mathrm{h},i_a,s_a,L_\Omega,\alpha,\|\phi\|_{C^{1,\alpha}})>0,
    \] 
    and $\theta=\theta(n,d,\alpha)\in (0,1/2)$ such that
\begin{equation}\label{comp:fin2}
\begin{split}
    \int_{B_R^+(x_0')} B\big(|Du-Du_0| \big) \,dx\leq C\,\Big( 1+\|f\|_{L^d(B_R^+(x_0'))}\Big)&\,\Big(\int_{B_R^+(x_0')}\big[B\big(|Du|\big)+1\big]\,dx\Big)\,R^\theta.
    \end{split}
    \end{equation}
Moreover, for every $\mu\in(0,n)$, there exist
\[\begin{split}
&C_\mu=C_\mu(n,\l,\L,\L_\mathrm{h},i_a,s_a,L_\Omega,\alpha,\|\phi\|_{C^{1,\alpha}},\mu)>0
    \\
    &R_\mu=R_\mu(n,\l,\L,\L_\mathrm{h},i_a,s_a,\alpha,d,\|f\|_{L^d(\mathcal{V})},L_\Omega,\|\phi\|_{C^{1,\alpha}},\mu)\in (0,1)
\end{split}\]
 such that, if $R_0\leq R_\mu$, then the Morrey-type estimate
\begin{equation}\label{est:morreydir}
    \int_{B_R^+(x_0')} B\big(|Du
    |\big)\,dx\leq C_\mu\,\bigg(1+B(\|g\|_{C^{1,\alpha}})+\|f\|_{L^d(\mathcal{V})}+\frac{1}{R_0^\mu}\int_{\mathcal{V}}B\big(|Du| \big)\,dx \bigg)\,R^\mu,
\end{equation}
holds true for every $x'_0\in B^0_{R_0/2}$ and  $0<R\leq R_0/2$.
\end{proposition}

\begin{proof}
    We omit the proof of \eqref{comp:fin2}, as it is identical to that of Propositions \ref{prop:comparison}, save that one has to replace $B_R(x_0)$ with $B_R^+(x_0')$, and take into account the presence of $L_\Omega,\|\phi\|_{C^{1,\alpha}}$ in the constants due to \eqref{new:coer}-\eqref{A:hold1}. For what concerns \eqref{est:morreydir}, we argue as in Lemma \ref{lem:morrey}, that is
{\small
\begin{equation*}
    \begin{split}
        \int_{B^+_r(x'_0)}B\big(|Du| \big)\,dx \stackrel{\eqref{triangle}}{\leq} &C\,\int_{B_r^+(x'_0)}B\big(|Du-Du_0| \big)\,dx+C\,\int_{B^+_r(x'_0)} B\big( |Du_0|\big)\,dx
        \\
        \stackrel{\eqref{duu:dir}}{\leq} &\,C\,\int_{B_R^+(x'_0)}B\big(|Du-Du_0| \big)\,dx
        \\
        &+C'\Big(\frac{r}{R} \Big)^n\,\int_{B^+_R(x'_0)} B\big( |Du_0|\big)\,dx+C'\,B\big(\|g\|_{C^{1,\alpha}} \big)\,r^n
        \\
        \stackrel{\eqref{comp:fin2},\eqref{en:u0dir}}{\leq} & C''\,\bigg(\Big(\frac{r}{R} \Big)^n+\Big(1+ \|f\|_{L^d(\mathcal{V}))}\Big)\,R^\theta\bigg)\,\int_{B^+_R(x'_0)}B\big(|Du|\big)\,dx 
        \\
        &+C''\,\Big(1+ B(\|g\|_{C^{1,\alpha}})+\|f\|_{L^d(\mathcal{V})}\Big)\,R^n\,.
    \end{split}
\end{equation*}}
Therefore, by applying  Lemma \ref{lem:iteration}, and using that $B^+_R(x_0')\subset B^+_{R_0}\subset \mathcal{V}$ we obtain \eqref{est:morreydir}.
\end{proof}

We now have all the ingredient to prove the boundary regularity of solutions to Dirichlet problems.

\begin{proof}[Proof of Theorem \ref{thm:dir}]
We divide the proof into a few steps.
\vspace{0.1cm}

\noindent \textit{Step 1.} Let us assume that  $u\in W^{1,B}\big(B^+_{R_0}\big)$ is  solution to \eqref{eq:dirBup}, with $B^+_{R_0}\subsetneq \mathcal{V}$, and $\A(x,\xi)$ satisfying \eqref{new:coer}-\eqref{A:hold1}.

We also fix $\bar\mu=\bar\mu(n,\l,\L,i_a,s_a,d,\alpha)\in (0,n)$ to be determined later in the proof, so that from Proposition \ref{prop:comparisondir} we may find a radius 
\[
R_{\bar\mu}=R_{\bar\mu}(n,d,\l,\L,\L_\mathrm{h},i_a,s_a,\alpha,d,\|f\|_{L^d(\mathcal{V})}, L_\Omega,\|\phi\|_{C^{1,\alpha}})\in (0,1)
\] 
such that \eqref{est:morreydir} holds provided $R_0\leq R_{\bar \mu}$. For the moment, we assume that this condition is satisfied. Let $0<r\leq R\leq R_0/2$, and estimate
{\small
\begin{equation}\label{tocca}
\begin{split}
     \mint_{B^+_r(x'_0)}& B\big(|Du-(Du)_{B_r^+(x'_0)}| \big)\,dx
        \\
        \stackrel{\eqref{triangle}}{\leq}& \,C\,\mint_{B^+_r(x'_0)}B\big(|Du-Du_0| \big)\,dx+C\,\mint_{B^+_r(x'_0)}B\big(|Du_0-(Du_0)_{B_r^+(x_0)}| \big)\,dx
         \\
        \stackrel{\eqref{duurrdir1}}{\leq} & C\Big(\frac{R}{r} \Big)^n\mint_{B^+_R(x'_0)}B\big(|Du-Du_0| \big)\,dx+C'\Big( \frac{r}{R}\Big)^{\beta_{\mathrm{h}}i_B}\,\bigg\{\mint_{B^+_R(x'_0)} B\big(|Du_0| \big)\,dx+B\big(\|g\|_{C^{1,\alpha}} \big)\bigg\}
        \\
        \stackrel{\eqref{comp:fin2},\eqref{en:u0dir}}{\leq} & C''\,\Big(\frac{R}{r} \Big)^n\,\big(1+\|f\|_{L^d(\mathcal{V})} \big)\,\Big(\mint_{B_R^+(x_0')}\big[B\big( |Du|\big)+1\big]\,dx\Big)\,R^\theta
        \\
        &+C''\,\Big( \frac{r}{R}\Big)^{\beta_{\mathrm{h}}i_B}\,\bigg\{\mint_{B_R^+(x_0')}\big[B\big( |Du|\big)+1\big]\,dx+B\big(\|g\|_{C^{1,\alpha}} \big)\bigg\}
        \\
        \stackrel{\eqref{est:morreydir}}{\leq} & C_{\bar \mu}\bigg\{\Big(\frac{R}{r}\Big)^{n}R^{\theta-n+\bar\mu}+\Big(\frac{r}{R} \Big)^{\beta_{\mathrm{h}}i_B}\,R^{\bar\mu-n}\bigg\},
        \end{split}
\end{equation}
}
where we set
\[
C_{\bar \mu}:=C(n,\l,\L,\L_{\mathrm{h}},i_a,s_a,L_\Omega,d,\alpha,\|\phi\|_{C^{1,\alpha}})\,\bigg(1+B\big(\|g\|_{C^{1,\alpha}}\big)+\|f\|_{L^d(\mathcal{V})}+\frac{1}{R_0^{\bar \mu}}\int_{\mathcal{V}}B\big(|Du| \big)\,dx \bigg).
\]
By choosing
\[
\begin{split}
    \bar\mu:=n-\min\Big\{\frac{\theta}{4},\frac{\beta_{\mathrm{h}}i_B\,\theta}{4n\big(1+\tfrac{\theta}{4n}\big)} \Big\},\quad \sigma:=\frac{n}{\theta+\bar \mu}\Big(1+\frac{\theta}{4n}\Big)\in (0,1),\quad\text{and radii}\quad R=r^\sigma,
\end{split}
\]
from \eqref{tocca} we find
\begin{equation}\label{tocca1}
     \mint_{B^+_r(x'_0)} B\big(|Du-(Du)_{B_r^+(x'_0)}| \big)\,dx\leq C_{\bar \mu}\,r^\beta,\quad \beta=\beta(n,\l,\L,i_a,s_a,\alpha,d,L_\Omega)=\frac{\beta_\mathrm{h}i_B\theta}{4}\in (0,1),
\end{equation}
which is valid for all $x_0'\in B_{R_0/2}$, and $0<r\leq R_0/2$. From this inequality, the interior gradient regularity Theorem \ref{thm:interior} and Campanato's theorem \cite[Theorem 5.5]{GM12}, we infer
\begin{equation}\label{temp0:dirC1}
    \|Du \big\|_{C^{0,\beta}(\overline B^+_{R_0/2})}\leq C_0\bigg(n,\l,\L,\L_{\mathrm{h}},i_a,s_a,\alpha,d,L_\Omega,\|\phi\|_{C^{1,\alpha}},\|f\|_{L^d(\mathcal{V})},\|g\|_{C^{1,\alpha}},R_0,\int_{\mathcal{V}}B\big(|Du| \big)\,dx \bigg)
\end{equation}

Moreover, from \eqref{eq:dirBup}, Lemma \ref{divFf} and arguing as in \eqref{we:sol}-\eqref{again:coer} and \eqref{Afdiv}, we deduce that $w=u-G$ is solution to
\begin{equation*}
    \begin{cases}
        -\mathrm{div}\big(\widetilde{\A}_f(x,Dw) \big)=0\quad &\text{in $B^+_{R_0}$}
        \\
        w=0\quad&\text{on $B^0_{R_0}$,}
    \end{cases}
\end{equation*}
with $\widetilde{\A}_f(x,\xi)=\A\big(x,\xi+DG(x)\big)+\mathrm{div}\,F(x)$ satisfying
\begin{equation*}
\begin{split}
   & \widetilde{\A}_f(x,\xi)\cdot \xi\geq c\,B\big( |\xi|\big
)-C\,B\Big(\big(1+\|f\|_{L^d(\mathcal{V})}\big)^{1/i_b} \Big)-C\,B\big(\|g\|_{C^{1,\alpha}} \big)
\\
 &|\widetilde{\A}_f(x,\xi)|\leq C\,b\big(|\xi| \big)+C\,b\Big((1+\|f\|_{L^d(\mathcal{V})})^{1/i_b} \Big)+C\,b\big(\|g\|_{C^{1,\alpha}}\big).
\end{split}
\end{equation*}
with $c,C=c,C(n,\l,\L,\alpha,i_a,s_a,L_\Omega)>0$. Hence we may use Theorem \ref{thm:bdddir} and \eqref{simple}, and deduce that
\begin{equation}\label{temp0:bdddir}
    \|u\|_{L^\infty(B^+_{R_0/2})}\leq \frac{C}{R_0^n}\int_{B_{R_0}^+}|u|\,dx+C\big\{(1+\|f\|_{L^d(\mathcal{V})})^{1/i_b}+\|g\|_{C^{1,\alpha}} \big\},
\end{equation}
for some positive constant $C=C(n,\l,\L,\alpha,i_a,s_a,d,L_\Omega)$.
\vspace{0.1cm}

\noindent\textit{Step 2.} To complete the proof, it suffices to apply the results of the previous steps together with a covering argument. Let $u\in W^{1,B}(\Omega\cap \mathcal{U})$ be solution to \eqref{eq:dir1}, and let $\mathcal{U}'\Subset \mathcal{U} $ be such that $\partial\Omega\cap\mathcal{U}'\neq \varnothing$ (hence it is of class $C^{1,\alpha}$).

By compactness, we cover $\partial\Omega\cap \overline{\mathcal{U}'}$ with a family of cylinders $\{\mathcal{Q}_{\Omega,x_i}\}_{i=1}^{N_0}$, $x_i\in \partial\Omega\cap \overline{\mathcal{U}'}$, defined by \eqref{cilindr}, with corresponding local boundary charts $\phi_i=\phi_{x_i}$ and diffeomorphisms $\Phi_{i}=\Phi_{x_i}$ given by \eqref{diffeom}. In particular, the cardinality of such covering satisfies
\[N_0\leq C(n)\,\bigg(\frac{(1+L_\Omega)\,\mathrm{diam}\,(\mathcal{U})}{R_\Omega}\bigg)^n,
\]
and by the definition in \eqref{normadeom}, we have
\[\|\phi_i\|_{C^{1,\alpha}}\leq\|\partial\Omega\cap \mathcal{U}\|_{C^{1,\alpha}(\mathcal{U}')} \quad\text{for all $i=1,\dots,N_0$.}
\]
 Then we fix the radius 
\begin{equation}\label{rad0:depend}
\begin{split}
&R_0\big(n,\l,\L,\L_\mathrm{h},i_a,s_a,\alpha,d,\mathcal{L}_\Omega,\|\partial\Omega\cap \mathcal{U}\|_{C^{1,\alpha}(\mathcal{U}')},\|f\|_{L^d(\Omega\cap \mathcal{U})},\mathrm{dist}(\mathcal{U}',\partial\mathcal{U})\big)
\\
&R_0:=\min\Big\{\frac{R_{\bar\mu}}{4},\frac{\mathrm{dist}(\mathcal{U}',\partial\mathcal{U})}{C(n)(1+L_\Omega)},\frac{R_\Omega}{C(n)(1+L_\Omega)} \Big\},
\end{split}
\end{equation}
with $C(n)\geq 8$ large, and we cover $\partial\Omega\cap \overline{\mathcal{U}'}$ with sets $\{\Phi_i^{-1}(B^+_{R_0}(y'_j))\}_{j=1}^N$, where we choose points 
\[
y'_j\in B^0_{R_0}\cap \Phi_i(\partial\Omega\cap\,\overline{\mathcal{U}'}).
\] 
Notice that, since $\Phi_i$ is a diffeomorphism, with gradient bounds \eqref{graddifeom}, and thanks to our choice of $R_0$ and $y'_j$, we have \[
B_{\frac{R_0}{C(n)(1+L_\Omega)}}(\Phi^{-1}_i(y'_j))\subset\Phi_i^{-1}(B_{R_0}(y'_j))\subset B_{C(n)(1+L_\Omega)R_0}(\Phi_i^{-1}(y'_j))\Subset \mathcal{U} ,
\]
and the first inclusion implies that the cardinality of the covering satisfies
\[
N\leq C(n)\bigg( \frac{(1+L_\Omega)\,\mathrm{diam}\,\mathcal{U}'}{R_0}\bigg)^n.
\]
Then, we cover \[\overline{\mathcal{U}'}\setminus \bigcup\limits_{i,j} \Phi^{-1}_i\big(B^+_{R_0}(y'_j) \big)\subset \bigcup_{k=1}^{N_{\rm int}} B_{R_0}(x_k),\quad x_k\in \overline{\mathcal{U}'}\setminus \bigcup\limits_{i,j} \Phi^{-1}_i\big(B^+_{R_0}(y'_j)\big),
\]
so we have
\[
N_{\rm int}\leq C(n)\Big(\frac{\mathrm{diam}\,\mathcal{U}}{R_0} \Big)^n
\]
Now let $\{\eta_{ij}\}$, $\{\eta_k\}$ be partitions of unity associated to $\{\Phi_i^{-1}(B^+_{R_0}(y'_j))\}$ and $\{B_{R_0}(x_k)\}$, respectively, and such that  
\[
0\leq \eta_{ij},\eta_k\leq 1,\quad |D\eta_{ij}|+|D\eta_k|\leq \frac{C}{R_0},
\]
with $C=C(n,L_\Omega)$. Then we observe that the function $\hat{u}_i=u\circ\Phi_i^{-1}$ is solution to \eqref{pb:C}$_1$, and in particular they solve the Dirichlet problem \eqref{eq:dirBup} in $B^+_{2R_0}(y'_j)$, with $\A,f,g$ replaced by $\hat{\A},\hat{f},\hat{g}$ defined by \eqref{def:hat}-\eqref{def:hat1}. 

Hence, from the previous step, and in particular from \eqref{temp0:dirC1}-\eqref{temp0:bdddir}, we get
\[
\|u\|_{C^{1,\beta}(\Phi_i^{-1}\big(B^+_{R_0}(y'_j))\big)}\leq C_{\mathrm{data}}
\]
with $C_{\mathrm{data}}$ depending on the same quantities as in \eqref{stimafin:Dir}. Now, the estimate $\|u\|_{L^\infty(\Omega\cap\, \mathcal{U})}$ can be immediately obtained from this inequality and the interior estimate of Theorem \ref{thm:interior}. Then, from the estimate above, Theorem \ref{thm:interior}, the properties of $\eta_{ij},\eta_k$, and the estimates on the cardinalities $N_0,N,N_{\rm int}$ above, and on $R_0$ in \eqref{rad0:depend}, we may argue as in \eqref{finalcover}, and get
\begin{equation*}
        \|Du\|_{C^{1,\beta}(\overline\Omega\cap\,\mathcal{U}')}\leq \sum_{i=1}^{N_0}\sum_{j=1}^{N}\|(Du)\,\eta_{ij}\|_{C^{0,\beta}(\Phi_i^{-1}(B^+_{R_0}(y'_j)))}+\sum_{k=1}^{N_{int}}\|(Du)\,\eta_k\|_{C^{0,\beta}(B_{R_0}(x_k))}
        \leq C'_{\mathrm{data}}
\end{equation*}
with $C'_{\mathrm{data}}$ depending on the same quantities as in \eqref{stimafin:Dir}, hence our thesis.
\end{proof}

\begin{proof}[Proof of Corollary \ref{cor:dirichlet}]
    The result is an immediate consequence of Theorem \ref{thm:dir}, and the following energy estimate on $u$.
    We test \eqref{eq:dir2} with test function $\varphi=u-g\in W^{1,B}_0(\Omega)$, so that by using \eqref{co:gr}, \eqref{A:hold}$_2$ and \eqref{Young}, \eqref{young1}, \eqref{B=1}, we get
    \begin{equation*}
    \begin{split}
        \int_\Omega B\big(|Du| \big)\,dx &\leq C\,\int_\Omega\big[ \A(x,Du)-\A(x,0)\big]\cdot Du\,dx
        \\
        &\leq  C\,\int_\Omega \A(x,Du)\cdot Dg+C'\,\int_\Omega |Du|\,dx
        \\
        &\leq C''\int_\Omega \big[b(|Du|)+1\big]\,|Dg|\,dx+\frac{1}{4}\,\int_\Omega B(|Du|)\,dx+C''\,|\Omega|
        \\
        &\leq \frac{1}{2}\,\int_\Omega B\big(|Du|\big)\,dx+C'''\,\big(|\Omega|+1\big)\,\Big(1+B\big( \|g\|_{C^{1,\alpha}}\big)\Big),
        \end{split}
    \end{equation*}
 with $C,C',C'',C'''>0$ depending on $n,\l,\L,i_a,s_a,\alpha$. Hence, using also \eqref{B2t}, we have found
\begin{equation}\label{stima:enudir}
      \int_\Omega B\big(|Du| \big)\,dx\leq C\,\big(|\Omega|+1 \big)\,\big(1+ \|g\|_{C^{1,\alpha}}\big)^{s_B}.
 \end{equation}
 Next, we fix the ball $
 \mathcal{B}=B_{2\rm{diam}\,\Omega}(x_0)$ for some $x_0\in \Omega$. so that $\Omega \Subset \mathcal{B}$, hence $|\Omega|\leq C(n)\,(\mathrm{diam}\,\Omega)^n$; we also extend $u\equiv g$ in $\mathcal{B}\setminus \Omega$.
Therefore, by Poincar\'e inequality, \eqref{simple} and \eqref{stima:enudir}, we get
\begin{equation}\label{stima:udirff}
\begin{split}
    \int_\Omega |u|\,dx&\leq \int_\mathcal{B} |u-g| \,dx+C\,(\mathrm{diam}\,\Omega)^n\,\|g\|_{L^\infty}
    \\
    &\leq C_1\,\big(\mathrm{diam}\,\Omega\big)\,\int_\Omega |D(u-g)| \,dx+C\,(\mathrm{diam}\,\Omega)^n\,\|g\|_{L^\infty}
    \\
     &\leq C_2\,\big(\mathrm{diam}\,\Omega\big)\,\int_\Omega B\big(|Du|\big)\,dx+C\,(1+\mathrm{diam}\,\Omega)^n\,\|g\|_{C^{1,\alpha}}
    \\
    &\leq C_3\,\big(\mathrm{diam}\,\Omega+1 \big)^n\,\big(1+ \|g\|_{C^{1,\alpha}}\big)^{s_B},
    \end{split}
\end{equation}
 with $C,C_1,C_2,C_3>0$ depending on $n,\l,\L,i_a,s_a,\alpha$. The estimates \eqref{stima:enudir}-\eqref{stima:udirff} and Theorem \ref{thm:dir} finally yield the desired result Equation \eqref{est:globdirichlet}.
\end{proof}

\section{Boundary \texorpdfstring{$C^{1,\beta}$}{C1b}-regularity, Neumann problems}\label{sec:finNeum}
As in  Section \ref{sec:pfdir}, we begin studying regularity of solutions $u\in W^{1,B}(B_{R_0}
^+)$ to the conormal problem
\begin{equation}\label{eq:halfBneu}
    \begin{cases}
        -\mathrm{div}\big(\A(x,Du)\big)=f\quad &\text{in $B^+_{R_0}$}
        \\
        \A(x,Du)\cdot e_n+h=0\quad&\text{on $B^0_{R_0}$}
    \end{cases}
\end{equation}
where $h\in C^{0,\alpha}(\bar B^0_{R_0})$, and $\A(x,\xi)$ fulfilling \eqref{new:coer}-\eqref{A:hold1}, and $B_{R_0}\subsetneq \mathcal{V}$ for some bounded domain $\mathcal{V}\subset \R^n_+$. Again, throughout the proofs we shall keep explicit the dependence on the Lipschitz constant $L_\Omega$ and on the boundary chart norm $\|\phi\|_{C^{1,\alpha}}$.

Let $x'_0\in B^0_{R_0}$ and $R\in (0,1)$ be such that $B_R(x_0')\subset B^+_{R_0}$; we consider $u_0\in W^{1,B}\big(B_R^+(x_0')\big)$ solution to the the homogeneous, frozen mixed problem
\begin{equation}\label{eq:mixed}
    \begin{cases}
        -\mathrm{div}\big(\A(x_0',Du_0) \big)=0\quad &\text{in $B^+_R(x_0')$}
        \\
        \A(x'_0,Du_0)\cdot e_n+h(x_0')=0&\text{on $B^0_R(x_0')$}
        \\
        u_0=u\quad &\text{on $\partial B^+_R(x_0')\setminus B^0_R(x_0')$.}
    \end{cases}
\end{equation}
The scheme of the proofs now follows the lines of Sections \ref{sec:thmint}- \ref{sec:pfdir}, so we start with the following

\begin{proposition}\label{prop:existmixed}
    There exists a unique solution $u_0\in W^{1,B}\big(B^+_{R}(x_0')\big)$ solution to the problem \eqref{eq:mixed}. Furthermore, it satisfies
    \begin{equation}\label{lkj:energia}
        \mint_{B^+_R(x_0')} B\big( |D u_0|\big)\,dx\leq C\,(1+\|h\|_{L^\infty})\mint_{B^+_R(x_0')} \big[B\big( |D u|\big)+1\big]\,dx+C\,\widetilde{B}\big(\|h\|_{L^\infty} \big),
    \end{equation}
    with positive constant $C=C(n,\l,\L,i_a,s_a,L_\Omega)$.
\end{proposition}

\begin{proof}
    The proof is very similar to that of Proposition \ref{prop:ex}, with minor differences due to the co-normal boundary value.     
    We consider the space
\begin{equation*}
    X=\big\{v\in W^{1,B}(B_R^+(x_0')): v=0\quad\text{on $\partial B^+_R(x_0')\setminus B^0_R(x_0')$} \big\},
\end{equation*}
   endowed with norm $\|v\|=\|Dv\|_{L^B(B_R^+(x_0'))}$; by the properties of the trace operator, it can be easily shown that $X$ is reflexive. We then consider the operator 
   \begin{equation*}
       \langle Av_1,v_2\rangle=\int_{B^+_R(x_0')} \A(x_0',Dv_1+Du)\cdot Dv_2\,dx\quad \text{and}\quad\langle\mathscr{H},v\rangle=h(x'_0)\,\int_{B^0_R(x'_0)}v_1\,d\mathcal{H}^{n-1}.
   \end{equation*}
 for $v,v_1,v_2\in X$.  From the computations of Proposition \ref{prop:ex}, the operator $A$ is monotone, hemicontinuous and coercive, while the operator $\mathscr{H}:X\to \R$ is linear and continuous. Indeed, by the trace inequality \eqref{in:trace0} and H\"older's inequality \eqref{holder:orlicz}, we get
 \begin{equation*}
     |\langle\mathscr{H},v\rangle|\leq C\,|h(x_0')|\,\int_{B^+_R(x_0')}|Dv|\,dx\leq C'\,\|Dv\|_{L^B(\Omega)}.
 \end{equation*}
 Therefore, \cite[Theorem 26.A]{Z90} ensures the existence of a unique solution $v_0\in X$ satisfying
 \begin{equation*}
     \langle Av_0,v\rangle=\langle\mathscr{H},v\rangle\quad\text{for all $v\in X$,}
 \end{equation*}
 and this is equivalent to the existence of $u_0=v_0+u$ solution to \eqref{eq:mixed}. Next, we test the weak formulation of \eqref{eq:mixed} with $u-u_0$, so that
 \begin{equation*}
     \begin{split}
         \int_{B^+_R(x_0')}\A(x_0',Du_0)\cdot D(u_0-u)\,dx=h(x_0')\,\int_{B^0_{R}(x_0')}(u_0-u)\,d\mathcal{H}^{n-1},
     \end{split}
 \end{equation*}
 which together with \eqref{co:gr}, \eqref{A:hold1}$_2$, \eqref{in:trace0}, \eqref{Young}, \eqref{young1} and \eqref{simple} yields
 \begin{equation*}
     \begin{split}
         \mint_{B^+_R(x_0')} B\big( |D u_0|\big)\,dx\leq &\, C\,\mint_{B^+_R(x_0')}\A(x_0',Du_0)\cdot D u_0\,dx-C\, \mint_{B^+_R(x_0')}\A(x_0',0)\cdot Du_0\,dx
         \\
         = &\, C\,\mint_{B^+_R(x_0')}\A(x_0',Du_0)\cdot D u\,dx+C\,h(x_0')\,\mint_{B^0_{R}(x_0')}(u_0-u)\,d\mathcal{H}^{n-1}
         \\
         &-C\, \mint_{B^+_R(x_0')}\A(x_0',0)\cdot Du_0\,dx
         \\
         \leq &\,C_1\,\mint_{B^+_R(x_0')}b\big(|Du_0| \big)\,|Du|\,dx+C_1\,\|h\|_{\infty}\,\mint_{B^+_R(x_0')}|D(u_0-u)|\,dx
         \\
         &+C_1\,\mint_{B^+_R(x_0')}\big(|Du_0|+|Du|\big)\,dx
         \\
         \leq &\,\frac{1}{2}\mint_{B^+_R(x_0')} B\big(|Du_0| \big)\,dx+C_2\,\big(1+\|h\|_{L^\infty}\big)\,\mint_{B^+_R(x_0')}[B(|Du|)+1]\,dx
         \\
         &+C_2\,\widetilde{B}\big(\|h\|_{L^\infty} \big),
     \end{split}
 \end{equation*}
 where $C,C_1,C_2>0$ depend on $n,\l,\L,i_a,s_a,L_\Omega$. Equation \eqref{lkj:energia} thus follows.
\end{proof}

Next, analogous to Proposition \ref{pr:otherest}, we have the following

\begin{proposition}\label{prop:neuosc}
    Let $u_0\in W^{1,B}(B^+_R(x'_0))$ be the solution to \eqref{eq:mixed}. Then there exists $C=C(n,\l,\L,i_a,s_a,L_\Omega)>0$ such that
\begin{equation}\label{neu000}
    \mint_{B^+_r(x'_0)} B\big(|Du_0| \big)\,dx\leq C\,\mint_{B^+_R(x'_0)} B\big(|Du_0| \big)\,dx+C\,\widetilde{B}\big(\|h\|_{L^\infty}\big)+C ,
\end{equation}
and the following excess decay estimate 
\begin{equation}\label{exc:BBNEU}
    \mint_{B_r^+(x'_0)}B\Big(\big|Du_0-(Du_0)_{B^+_r(x'_0)} \big|\Big)\,dx\leq C\,\left( \frac{r}{R}\right)^{\beta_Ni_B} \mint_{B^+_R(x'_0)}B\Big(\big|Du_0-(Du_0)_{B_R^+(x'_0)} \big|\Big)\,dx,
\end{equation}
   holds for every $0<r\leq R$, with $\beta_N=\beta_N(n,\l,\L,i_a,s_a,L_\Omega)\in (0,1)$.
\end{proposition}

\begin{proof}
    Estimate \eqref{neu000} is trivial when $R/8\leq r\leq R $, so we may assume that $0<r\leq R/8$. From \eqref{neu:EST}, \eqref{triangle}, \eqref{B2t}, Jensen inequality, \eqref{come:BwB}, \eqref{tB:usef} and \eqref{tB=1}, we get\footnote{Estimate \eqref{neu:EST} is valid under the assumption $\A(0)=0$, but 
in the present setting we may have $\A(x'_0,0)\neq 0$. However,  
introducing $\hat\A(x'_0,\xi)= \A(x'_0,\xi)-\A(x'_0,0)$, we have $\hat\A(x'_0,0)=0$, and that $u_0$ satisfies the same equation with $\hat\A(x'_0,\xi)$, but with boundary datum $h(x'_0)+\A(x'_0,0)\cdot e_n$. This can be controlled  via \eqref{A:hold1}$_2$, i.e., $|h(x'_0)+\A(x'_0,0)\cdot e_n|\leq |h(x'_0)|+C(n,L_\Omega,\L)$.}
    \begin{equation*}
        \begin{split}
            \mint_{B^+_r(x'_0)} B\big(|Du_0| \big)\,dx &\leq B\Big(\sup_{B^+_{R/8}(x'_0)}|Du_0| \Big)\leq B\Big(C\,\mint_{B^+_R(x'_0)}|Du_0|\,dx+C\,b^{-1}\big(|h(x_0')|+C\big) \Big)
            \\
            &\leq C_1\,\mint_{B^+_R(x'_0)}B\big(|Du_0| \big)\,dx+C_1\,\widetilde{B}\big( \|h\|_\infty\big)+C_1\,,
        \end{split}
    \end{equation*}
with $C,C_1=C,C_1(n,\l,\L,i_a,s_a,L_\Omega)>0$, so \eqref{neu000} is proven. Finally, by using \eqref{NEU:est2} in place of \eqref{oscdes}, the proof \eqref{exc:BBNEU} is completely analogous to that of \eqref{excBB}. We omit the details.
\end{proof}

Next we prove the following

\begin{proposition}[Comparison estimate]\label{prop:comparisonNEU}
Let $u_0\in W^{1,B}(B^+_{R}(x'_0))$ be the solution to \eqref{eq:mixed}. Then there exist $C=C(n,\l,\L,\L_\mathrm{h},i_a,s_a,L_\Omega,\alpha,\|\phi\|_{C^{1,\alpha}})>0$
 and $\theta=\theta(n,d,\alpha)\in (0,1/2)$ such that
\begin{equation}\label{comp:NEufin}
\begin{split}
    \int_{B^+_R(x'_0)} B\big(|Du-Du_0| \big) \,dx\leq &\,C\,\Big( 1+\|h\|_{C^{0,\alpha}}+\|f\|_{L^d(B^+_R(x'_0))}\Big)\,\Big(\int_{B^+_R(x'_0)}\big[B\big(|Du|\big)+1\big]\,dx\Big)\,R^\theta
    \\
    &+C\,\Big( 1+\|h\|_{C^{0,\alpha}}+\|f\|_{L^d(B^+_R(x'_0))}\Big)\,\widetilde{B}\big(\|h\|_{L^\infty} \big)\,R^{n+\theta}.
    \end{split}
\end{equation}
\end{proposition}

\begin{proof}
    We argue similarly to Proposition \ref{prop:comparison}. More precisely, by means of \eqref{strong:coer}, \eqref{eq:mixed} and \eqref{eq:halfBneu}, we find
    \begin{equation*}
        \begin{split}
            \int_{B^+_R(x'_0)}a&\big(|Du|+|Du_0| \big)\,|Du-Du_0|^2\,dx 
\\
\leq&\,C\,\int_{B^+_R(x'_0)}\big(\A(x'_0,Du)-\A(x'_0,Du_0)\big)\cdot \big(Du-Du_0 \big)\,dx
\\
= &\,C\,\int_{B^+_R(x'_0)}\A(x'_0,Du)\cdot \big(Du-Du_0 \big)\,dx-\int_{B^0_{R}(x'_0)}h(x'_0)\,(u-u_0)\,dx'
\\
= &\,C\,\int_{B^+_R(x'_0)}\big(\A(x'_0,Du)-\A(x,Du)\big)\cdot \big(Du-Du_0 \big)\,dx+\int_{B^+_R(x'_0)}f\,(u-u_0)\,dx
\\
&+\int_{B^0_R(x'_0)} \big(h(x')-h(x'_0)\big)\,(u-u_0)\,dx',
        \end{split}
    \end{equation*}
   and then, by using  \eqref{A:hold1}, the H\"older, Sobolev and trace inequalities \eqref{half:sobol}, \eqref{in:trace0}, we obtain
\begin{equation}
        \begin{split}
\int_{B^+_R(x'_0)}a&\big(|Du|+|Du_0| \big)\,|Du-Du_0|^2\,dx 
\\
\leq &\, C\,R^{\alpha}\,\int_{B^+_R(x'_0)}|Du-Du_0|\,dx+C\,\|f\|_{L^n(B^+_R(x'_0))}\,\Big( \int_{B_R^+(x'_0)}|u-u_0|^{n'}\,dx\Big)^{1/n'}
\\
&+\|h\|_{C^{0,\alpha}}\,R^\alpha\,\int_{B^0_R(x'_0)}|u-u_0|\,dx'
\\
\leq & \,C'\,(1+\|h\|_{L^\infty})\,R^{\alpha}\int_{B^+_R(x'_0)} \big[ B\big(|Du| \big)+1\big]\,dx+C'\,\widetilde{B}\big(\|h\|_{L^\infty}\big)\,R^{n+\alpha}
 \\
 &+C'\Big(\|f\|_{L^n(B^+_R(x'_0))}+R^\alpha\,\|h\|_{C^{0,\alpha}}\Big)\, \int_{B^+_R(x'_0)}|Du-Du_0|\,dx
  \\
\leq  &\,C''\,\Big(R^{\alpha}(1+\|h\|_{C^{0,\alpha}})+R^{1-n/d}\,\|f\|_{L^d(B^+_R(x'_0))}\Big)\,\int_{B^+_R(x'_0)} \big[ B\big(|Du| \big)+1\big]\,dx
\\
&+C''\,\Big(R^{\alpha}(1+\|h\|_{C^{0,\alpha}})+R^{1-n/d}\,\|f\|_{L^d(B^+_R(x'_0))}\Big)\,\widetilde{B}\big(\|h\|_{L^\infty}\big)\,R^n
        \end{split}
    \end{equation}
    with $C,C',C''>0$ depending on $n,\l,\L,\L_\mathrm{h},i_a,s_a,\alpha,L_\Omega,\|\phi\|_{C^{1,\alpha}}$, where in the last two estimates we used \eqref{simple} and \eqref{lkj:energia}. Coupling this inequality with \eqref{el:comparison} and \eqref{lkj:energia}, we infer
\begin{equation*}
\begin{split}
    \int_{B^+_R(x'_0)} & B\big(|Du-Du_0| \big)\,dx
    \\
    \leq &\, C\,\Big\{\delta+\delta^{-1}\Big(R^{\alpha}\big(1+\|h\|_{C^{0,\alpha}} \big)+R^{1-n/d}\,\|f\|_{L^d(B^+_R(x'_0))}\Big)\Big\}\,\int_{B^+_R(x'_0)} \big[ B\big(|Du| \big)+1\big]\,dx
    \\
    &+C\,\Big\{\delta+\delta^{-1}\Big(R^{\alpha}\big(1+\|h\|_{C^{0,\alpha}} \big)+R^{1-n/d}\,\|f\|_{L^d(B^+_R(x'_0))}\Big)\Big\}\,\widetilde{B}\big(\|h\|_{L^\infty} \big)\,R^n
\end{split} 
\end{equation*}
for all $\delta\in (0,1)$. Choosing $\delta=\max\big\{R^{\alpha/2}, R^{(1-n/d)/2}\big\}$, we finally obtain \eqref{comp:NEufin} with 
\[\theta=\min\{\alpha/2, (1-n/d)/2\}.\]
\end{proof}

Next, we prove the Morrey-type estimate for $B(|Du|)$.

\begin{lemma}[A Morrey-type estimate]\label{lem:morreyneu}
    Let  $u\in W^{1,B}(B^+_{R_0})$ be solution to \eqref{eq:halfBneu}.  Then for every $\mu\in (0,n)$, there exist  \[
    \begin{split}
    &C_\mu=C_\mu(n,\l,\L,\L_\mathrm{h},i_a,s_a,\alpha,L_\Omega,\|\phi\|_{C^{1,\alpha}},\mu)>0
    \\
    &R_\mu=R_\mu(n,\l,\L,\L_\mathrm{h},i_a,s_a,\alpha,d,L_\Omega,\|\phi\|_{C^{1,\alpha}},\|h\|_{C^{0,\alpha}},\|f\|_{L^d(\mathcal{V})},\mu)\in(0,1)
    \end{split}
    \] 
 such that, if $R_0\leq R_\mu$, then 
\begin{equation}\label{est:morreyneu}
    \int_{B^+_R(x'_0)} B\big(|Du
    |\big)\,dx\leq C_\mu\,\bigg(1+\widetilde{B}\big(\|h\|_{C^{0,\alpha}}\big)+\|f\|_{L^d(\mathcal{V})}+\frac{1}{R_0^\mu}\int_{\mathcal{V}}B\big(|Du| \big)\,dx \bigg)\,R^\mu
\end{equation}
is valid for every $x'_0\in B^0_{R_0/2}$ and  $0<R\leq R_0/2$.
\end{lemma}

\begin{proof}
We follow the proof of~\eqref{est:morreydir}, with the necessary modifications due to the different boundary condition. Hence we compute
   {\small
\begin{equation*}
    \begin{split}
        \int_{B^+_r(x'_0)}&B\big(|Du| \big)\,dx \stackrel{\eqref{triangle}}{\leq} C\,\int_{B_r^+(x'_0)}B\big(|Du-Du_0| \big)\,dx+C\,\int_{B^+_r(x'_0)} B\big( |Du_0|\big)\,dx
        \\
        \stackrel{\eqref{neu000}}{\leq} &\,C\,\int_{B_R^+(x'_0)}B\big(|Du-Du_0| \big)\,dx
        \\
        &+C'\Big(\frac{r}{R} \Big)^n\,\int_{B^+_R(x'_0)} B\big( |Du_0|\big)\,dx+C'\,\Big(\widetilde{B}\big(\|h\|_{C^{0,\alpha}}\big)+1\Big)\,r^n
        \\
        \stackrel{\eqref{lkj:energia},\eqref{comp:NEufin}}{\leq}  &C''\,\bigg(\Big(\frac{r}{R} \Big)^n\big(1+\widetilde{B}(\|h\|_{C^{0,\alpha}})\big)+\Big(1+\|h\|_{C^{0,\alpha}}+ \|f\|_{L^d(\mathcal{V}))}\Big)\,R^\theta\bigg)\,\int_{B^+_R(x'_0)}B\big(|Du|\big)\,dx 
        \\
        &+C''\,\Big(1+ \widetilde{B}(\|h\|_{C^{0,\alpha}})+\|f\|_{L^d(\mathcal{V})}\Big)\,R^n\,,
    \end{split}
\end{equation*}}
with $C,C',C''>0$ depending on $n,\l,\L,\L_\mathrm{h},i_a,s_a,\alpha,L_\Omega,\|\phi\|_{C^{1,\alpha}}$, where we also estimated $\|h\|_{C^{0,\alpha}}\leq C(i_a,s_a)\,\widetilde{B}\big
(\|h\|_{C^{0,\alpha}} \big)+1$, stemming from \eqref{simple} with $B$ replaced by $\widetilde{B}$. Equation \eqref{est:morreyneu} thus follows via an application of Lemma \ref{lem:iteration}.
\end{proof}

We now have all the ingredients to prove the boundary regularity for co-normal problems.

\begin{proof}[Proof of Theorem \ref{thm:neu}]
    We first prove the gradient regularity of solutions $u\in W^{1,B}(B^+_{R_0})$ to \eqref{eq:halfBneu}.
    We fix $\bar\mu=\bar\mu(n,d,\alpha)=n-\theta/2$, with $\theta\in (0,1)$ given by Proposition \ref{prop:comparisonNEU}, and from Lemma \ref{lem:morreyneu} we determine
   
    \[\begin{split}
    &C_{\bar \mu}=C_{\bar\mu}(n,\l,\L,\L_\mathrm{h},i_a,s_a,\alpha,L_\Omega,\|\phi\|_{C^{1,\alpha}})>0
    \\
    &R_{\bar\mu}= R_{\bar\mu}(n,\l,\L,\L_\mathrm{h},i_a,s_a,\alpha,d,L_\Omega,\|\phi\|_{C^{1,\alpha}},\|h\|_{C^{0,\alpha}},\|f\|_{L^d(\mathcal{V})})\in (0,1)
    \end{split}
    \]
    such that \eqref{est:morreyneu} holds true with $\mu=\bar \mu$. Therefore, on assuming that $R_0\leq R_{\bar \mu}$, by coupling \eqref{comp:NEufin} and \eqref{est:morreyneu}, and by also estimating $\widetilde{B}(\|h\|_{C^{0,\alpha}})\leq C(i_a,s_a)\,\big(\|h\|_{C^{0,\alpha}}+1\big)^{i_B'}$ by \eqref{B2t} (with $B$ replaced by $\widetilde{B}$ and $s_B$ replaced by $i_B'$), we obtain the improved comparison estimate
    \begin{equation}\label{imprcomp:NEu}
         \mint_{B^+_R(x'_0)} B\big(|Du-Du_0| \big) \,dx\leq \widehat{C}_0\,R^{\theta/2},
    \end{equation}
where
\[
\widehat{C}_0=\widehat{C}_0\Bigg(n,\l,\L,\L_\mathrm{h},i_a,s_a,L_\Omega,\alpha,d,\|h\|_{C^{0,\alpha}},\|\phi\|_{C^{1,\alpha}},\|f\|_{L^d(\mathcal{V})},\frac{1}{R_0^{\bar \mu}}\int_{\mathcal{V}}B\big(|Du| \big)\,dx\Bigg)>0.
\]
Repeating the same computations of \eqref{urcs} with \eqref{exc:BBNEU} and \eqref{imprcomp:NEu} in place of \eqref{excBB} and \eqref{comp:morrey}, respectively, we arrive at
\begin{equation*}
     \mint_{B^+_r(x'_0)} B\big(|Du-(Du)_{B^+_r(x'_0)}| \big) \,dx\leq C\,\widehat{C}_0\,\Big(\frac{R^{n+\theta/2}}{r^n}\Big)+C\,\Big(\frac{r}{R} \Big)^{\beta_N\,i_B} \mint_{B^+_R(x'_0)} B\big(|Du-(Du)_{B^+_R(x'_0)}| \big) \,dx,
\end{equation*}
    with $C=C(n,\l,\L,i_a,s_a,L_\Omega)>0$. Starting from this inequality, and repeating exactly the same argument of \eqref{fix:vphi}-\eqref{camp:sphere}, we deduce
    \[
    \|Du\|_{C^{0,\beta}(B^+_{R_0/2})}\leq C\Bigg(n,\l,\L,\L_\mathrm{h},i_a,s_a,d,\alpha,L_\Omega,\|f\|_{L^d(\mathcal{V})},\|h\|_{C^{0,\alpha}},\|\phi\|_{C^{0,\alpha}},R_0,\int_{\mathcal{V}}B\big(|Du| \big)\,dx \Bigg),
    \]
    with $\beta=\beta(n,\l,\L,i_a,s_a,\alpha,d,L_\Omega)\in (0,1)$. Then, by \eqref{eq:halfBneu} and Lemma \ref{lemma:diver}, we have that $u$ solves
    \begin{equation*}
    \begin{cases}
        -\mathrm{div}\big( \A_f(x,Du)\big)=0\quad &\text{in $B^+_{R_0}$}
        \\
        \A_f(x',Du)\cdot e_n+h(x')=0&\text{on $B^0_{R_0}$}
        \end{cases}
    \end{equation*}
    with $\A_f(x,\xi)=\A(x,\xi)+\mathrm{div}F$ satisfying assumption \eqref{Afdiv} (up to the necessary change of constants due to \eqref{new:coer}, namely the appearance of the factor  $L_\Omega$), so that
    Theorem \ref{thm:bddneu}, yields
    \[ \|u\|_{L^{\infty}(B^+_{R_0/2})}\leq C\bigg(n,\l,\L,i_a,s_a,d,L_\Omega,\|h\|_{L^\infty},\|f\|_{L^d(\mathcal{V})},R_0,\int_{\mathcal{V}} |u|\,dx \bigg).
    \]
Hence, the \(C^{1,\beta}\)-regularity estimates for solutions to \eqref{eq:halfBneu} are established.

Finally, these estimates imply \eqref{stimafin:Neu} by means of the flattening argument of \eqref{change:var}--\eqref{def:hat1}, combined with the same covering argument employed in Step~2 of the proof of Theorem~\ref{thm:dir} (see the discussion following \eqref{temp0:bdddir}). We omit the details.
\end{proof}

We conclude this final section with the following  
\begin{proof}[Proof of Corollary \ref{cor:neu}]
By \eqref{zero:mean}, we may use Poincar\'e inequality
    \begin{equation}\label{poinc:coroll}
        \int_\Omega |u|\,dx\leq C\big(n,\mathrm{diam}\,(\Omega)\big)\,\int_\Omega |Du|\,dx.
    \end{equation}
    We then test Equation \eqref{eq:neu2}, so that
    \begin{equation}\label{wek:teneu}
        \int_\Omega \A(x,Du)\cdot Du\,dx=\int_\Omega f\,u\,dx+\int_{\partial\Omega}h\,u\,d\mathcal{H}^{n-1};
    \end{equation}
    by means of \eqref{co:gr}, \eqref{A:hold}$_2$, \eqref{Young}, \eqref{young1} and \eqref{B=1}, we get
    \begin{equation*}
    \begin{split}
        \int_\Omega \A(x,Du)\cdot Du\,dx&\geq c\,\int_\Omega B\big(|Du| \big)\,dx-C\,\int_\Omega\A(x,0)\cdot Du\,dx
        \\
        &\geq c\,\int_\Omega B\big(|Du| \big)\,dx-C'\int_\Omega |Du|\,dx
        \\
        &\geq c'\,\int_\Omega B\big(|Du| \big)\,dx-C''|\Omega|,
        \end{split}
    \end{equation*}
    with $C,C',C''>0$ depending on $n,\l,\L,i_a,s_a$. Then, by H\"older and Sobolev inequalities (whose quantitative constant depends on $n,\mathcal{L}_\Omega$), \eqref{poinc:coroll} and \eqref{Young}, we get
    \begin{equation*}
        \begin{split}
            \Big|\int_\Omega f&\,u\,dx \Big|\leq \|f\|_{L^n(\Omega)}\,\|u\|_{L^{n'}(\Omega)}\leq C(n,\mathcal{L}_\Omega)\|f\|_{L^n(\Omega)}\bigg\{\int_\Omega |u|\,dx+\int_\Omega |Du|\,dx\bigg\}
            \\
            &\leq  C(n,\mathrm{diam}\,\Omega,\mathcal{L}_\Omega)\|f\|_{L^n(\Omega)}\,\int_\Omega |Du|\,dx\leq \delta\,\int_\Omega B\big(|Du| \big)\,dx+C_\delta,
        \end{split}
    \end{equation*}
    for all $\delta\in (0,1)$, with $C_\delta=C_\delta(n,i_a,s_a,\mathrm{diam}\,\Omega,\mathcal{L}_\Omega,\|f\|_{L^d(\Omega)},\delta)$. Next, by the trace inequality \eqref{in:trace0}, \eqref{poinc:coroll} and \eqref{Young}, we get
    \begin{equation*}
        \begin{split}
            \Big|\int_{\partial \Omega}h&\,u\,d\mathcal{H}^{n-1} \Big|\leq C(n,\mathrm{diam}\,\Omega,\mathcal{L}_\Omega)\,\|h\|_{L^\infty(\partial \Omega)}\,\bigg\{\int_\Omega |u|\,dx+\int_\Omega |Du|\,dx \bigg\}
            \\
            &\leq C'(n,\mathrm{diam}\,\Omega,\mathcal{L}_\Omega)\,\|h\|_{L^\infty(\partial \Omega)}\,\int_\Omega |Du|\,dx \leq \delta\,\int_\Omega B\big(|Du| \big)\,dx+C_\delta,
        \end{split}
    \end{equation*}
    for all $\delta \in (0,1)$, with $C_\delta=C_\delta(n,i_a,s_a,\mathrm{diam}\,\Omega,\mathcal{L}_\Omega,\|h\|_{L^\infty(\partial\Omega)},\delta)$. Inserting the content of the three inequalities above into \eqref{wek:teneu}, and choosing \[\delta=\delta(n,\l,\L,i_a,s_a,\|h\|_{L^\infty(\partial \Omega)},\|f\|_{L^d(\Omega)},\mathrm{diam}\,\Omega,\mathcal{L}_\Omega)\in (0,1)\] 
    small enough, we obtain the energy estimate
\begin{equation}
    \int_\Omega B\big(|Du| \big)\,dx\leq C\Big(n,\l,\L,i_a,s_a,\|h\|_{L^\infty(\partial \Omega)},\|f\|_{L^d(\Omega)},\mathrm{diam}\,\Omega,\mathcal{L}_\Omega \Big).
\end{equation}
   This estimate together with \eqref{poinc:coroll}, \eqref{simple} and Theorem \ref{thm:neu} yields \eqref{est:neucoroll}, that is our thesis. 
\end{proof}

\appendix

\section{Proof of Lemma \ref{lemma:interpol}}\label{appendixA}
Let $\mathcal{B}=B_{R_0}(x_0)$, and without loss of generality let us assume that $x_0=0$. For $x,y\in \mathcal{B}$, we set
\[
d(x):=\mathrm{dist}(x,\partial\mathcal{B})=R_0-|x|,\quad d(x,y):=\min\big\{d(x),d(y) \big\}.
\]
In the same spirit of \cite[Section 6.8]{GT}, for a given vector field $V=\{V^i(x)\}_{i=1}^d$, we define the weighted semi-norms
\[
[V]_\alpha^{(b)}:=\sup_{x\neq y\in \mathcal{B}}d(x,y)^{\alpha+b}\frac{|V(x)-V(y)|}{|x-y|^\alpha},\quad \alpha\in (0,1],\,b\geq 0,
\]
and for $b\geq 0$ we also set
\[
|V|_0=\sup_{\mathcal{B}}|V|,\quad
|V|^{(b)}_0:=\sup_{x\in \mathcal{B}} d^b(x)|V(x)|. \]
We start by proving the interpolation inequality
\begin{equation}\label{0inter:ineq}
    |Dv|_0^{(1)}\leq \frac{2}{\e}|v|_0+2^{1+\alpha}\e^\alpha\,[Dv]_\alpha^{(1)},
\end{equation}
which holds for all $\e\in (0,1/2)$, and for any $C^1$-function $v:\mathcal{B}\to \R$. To prove it, let $x_1\in \mathcal{B}$, and suppose that $Dv(x_1)\neq 0$. Then choose $x_2\in \mathcal{B}$ such that $x_1-x_2$ is parallel to $Dv(x_1)$, and $|x_1-x_2|=\e\,d(x_1)$, that is
\[
x_1-x_2=\e\,d(x_1)\frac{Dv(x_1)}{|Dv(x_1)|}.
\]
Therefore, as $\e\in (0,1/2)$, we have that
\[
x_2\in \overline B_{\e d(x_1)}(x_1)\subset B_{d(x_1)/2}(x_1)\subset \mathcal{B}
\]
By the mean value theorem, we may find $x_3$ belonging to the line segment generated by $x_1,x_2$, and denoted by $[x_1,x_2]$, such that
\[v(x_1)-v(x_2)=Dv(x_3)\cdot (x_1-x_2),
\] 
so that we have
\begin{equation}\label{tempo:interineq}
\begin{split}
|Dv(x_1)|&=Dv(x_1)\cdot\frac{x_1-x_2}{|x_1-x_2|}\leq Dv(x_3)\cdot \frac{(x_1-x_2)}{|x_1-x_2|}+|Dv(x_3)-Dv(x_1)|
\\
&= \frac{v(x_1)-v(x_2)}{\e\,d(x_1)}+\bigg(\frac{|Dv(x_3)-Dv(x_1)|}{|x_3-x_1|^\alpha}\,d(x_1,x_3)^{1+\alpha} \bigg)\,\frac{|x_3-x_1|^{\alpha}}{d(x_1,x_3)^{1+\alpha}}
\\
&\leq \frac{2\,|v|_0}{\e\,d(x_1)}+[Dv]_\alpha^{(1)}\,\frac{|x_3-x_1|^{\alpha}}{d(x_1,x_3)^{1+\alpha}}.
\end{split}
\end{equation}
On the other hand, as $x_3$ lies in the segment $[x_1,x_2]$, we have $|x_1-x_2|\geq |x_1-x_3|$, so that
\[
d(x_1)=\frac{|x_1-x_2|}{\e}\geq 2|x_1-x_2|\geq 2|x_1-x_3|
\]
which together with the triangle inequality implies $d(x_3)\geq d(x_1)-|x_1-x_3|\geq d(x_1)/2$. Therefore, there holds
\[
\frac{d(x_1)}{2}\leq d(x_1,x_3)\leq d(x_1),
\]
and so we have
\[
\frac{|x_1-x_3|^\alpha}{d(x_3,x_1)^{1+\alpha}}\leq 2^{1+\alpha}\frac{|x_1-x_3|^\alpha}{d(x_1)^{1+\alpha}}\leq 2^{1+\alpha}\frac{|x_2-x_1|^\alpha}{d(x_1)^{1+\alpha}}=2^{1+\alpha}\frac{\e^\alpha}{d(x_1)}
\]
By inserting this estimate into \eqref{tempo:interineq}, and multiplying the resulting inequality with $d(x_1)$ we get
\[
d(x_1)\,|Dv(x_1)|\leq \frac{2\,|v|_0}{\e}+2^{1+\alpha}\e^\alpha\,[Dv]_\alpha^{(1)}
\]
and by the arbitrariness of $x_1\in \mathcal{B}$ this implies \eqref{0inter:ineq}.

Now let $w\in C^0(\bar{\mathcal{B}})\cap C^1(\mathcal{B})$ be a function satisfying \eqref{hp:interpol}, and for fixed $L\in \R^n$ and $U\in \R$, we define 
\begin{equation}\label{def:nuovav}
    v(x):=w(x)-L\cdot x-U.
\end{equation}
Clearly, \(v\) still satisfies \eqref{hp:interpol}. 
We now show that $v$ satisfies the following inequality
\begin{equation}\label{1inter:ineq}
    [Dv]_\alpha^{(1)}\leq C(n,\alpha,c_1)\Big(|Dv|_0^{(1)}+K\,R_0^{1+\alpha} \Big).
\end{equation}
 To prove it, let us fix two points $x\neq y$ in $\mathcal{B}$, and we distinguish two cases. 
First suppose that  $|x-y|\geq \frac{d(x,y)}{2}$, in which case we have
\[
\begin{split}
d(x,y)^{1+\alpha}&\frac{|Dv(x)-Dv(y)|}{|x-y|^\alpha}\leq 2^\alpha\,d(x,y)\,|Dv(x)-Dv(y)|
\\
&\leq 2^\alpha\, d(x)|Dv(x)|+2^\alpha \,d(y)|Dv(y)|\leq 2^{\alpha+1}|Dv|_0^{(1)},
\end{split}
\]
so \eqref{1inter:ineq} holds in this case.

Let us now assume that  $|x-y|<\frac{d(x,y)}{2}$ and, without loss of generality, that $d(x,y)=d(x)$. So, we have that
\[
y\in \overline B_{|x-y|}(x)\subset B_{\frac{d(x)}{2}}(x),
\]
and thus, noting that $\operatorname{osc} Dv=\operatorname{osc} Dw$, by using \eqref{hp:interpol} with $v$ in place of $w$, and radii $r=|x-y|$ and $\vrho=d(x)/2$, we obtain
\[\begin{split}
    |Dv(x)-Dv(y)|\leq C(n)\,\operatorname*{osc}_{B_{|x-y|}(x)} Dv\leq C'\,\frac{|x-y|^\alpha}{d(x)^\alpha}\bigg\{\operatorname*{osc}_{B_{\frac{d(x)}{2}}(x)} Dv+K\,d(x)^\alpha \bigg\},
\end{split}\]
with $C'=C'(n,\alpha,c_1)>0$. Multiplying the above inequality by $\frac{d(x,y)^{1+\alpha}}{|x-y|^\alpha}=\frac{d(x)^{1+\alpha}}{|x-y|^\alpha}$, and using that $d(x)=R_0-|x|\leq R_0$, we obtain
\begin{equation}\label{tempo:inteineq2}
\begin{split}
d(x,y)^{1+\alpha}\frac{|Dv(x)-Dv(y)|}{|x-y|^\alpha}&\leq C\,d(x)\,\operatorname*{osc}_{B_{\frac{d(x)}{2}}(x)} Dv+C\,K\,d(x)^{1+\alpha}
\\
&\leq C'\,d(x)\,\sup_{B_{\frac{d(x)}{2}}(x)} |Dv|+C\,K\,R_0^{1+\alpha}.
\end{split}
\end{equation}
On the other hand, if $z\in B_{\frac{d(x)}{2}}(x)$, we have 
\[
d(x)=R_0-|x|\geq R_0-|x-z|-|z|\geq d(z)-\frac{d(x)}{2},
\]
so that $d(z)\geq d(x)/2$ for all $z\in B_{\frac{d(x)}{2}}(x)$. This in turn implies
\[
d(x)|Dv(z)|\leq 2d(z)|Dv(z)|\leq 2\,|Dv|_0^{(1)},
\]
which together with \eqref{tempo:inteineq2} and the arbitrariness of $x\neq y\in \mathcal{B}$ yields \eqref{1inter:ineq}.

Now, by \eqref{0inter:ineq}  and \eqref{1inter:ineq}, we infer
\begin{equation*}
    \begin{split}
 |Dv|_0^{(1)}\leq \frac{2|v|_0}{\e}+2^{1+\alpha}\e^\alpha\,[Dv]_\alpha^{(1)}\leq \frac{2|v|_0}{\e}+C\,\e^\alpha\Big(|Dv|_0^{1}+K\,R_0^{1+\alpha} \Big),
    \end{split}
\end{equation*}
with $C=C(n,\alpha,c_1)$, so by choosing $\e=(2C)^{-1/\alpha}$  and reabsorbing terms, we obtain
\begin{equation}\label{aaaalmost}
    |Dv|_0^{(1)}\leq C(n,\alpha,c_1)\,\Big\{|v|_0+K\,R_0^{1+\alpha}\Big\}.
\end{equation}
In particular, if $x\in B_{R_0/2}$, we have that $d(x)=R_0-|x|\geq R_0/2$, so that 
\[
|Dv(x)|=d(x)\,\frac{|Dv(x)|}{d(x)}\leq \frac{2}{R_0}\,|Dv|_0^{(1)},\quad x\in B_{R_0/2}.
\]
By coupling this inequality with \eqref{aaaalmost} and rewriting the resulting estimate in terms of \(w\), \(L\), and \(U\) via \eqref{def:nuovav}, we finally obtain \eqref{thesis:interpol}, which is the desired conclusion.

\bigskip{}{}

 \par\noindent {\bf Acknowledgments.} 
  The author is supported by the Italian Ministry of University and
Research (MUR) through the FIS 2 project "SiGmA - Singularities in Geometric Analysis: Minimal Surfaces
and Mean Curvature Flows", project code FIS-2023-02962, CUP G53C25000120001.

 The paper was completed while he was a postdoctoral fellow of the National Institute for Advanced Mathematics (INdAM) at the University of Florence.

\bigskip{}{}

 \par\noindent {\bf Data availability statement.} Data sharing not applicable to this article as no datasets were generated or analyzed during the current study. 

\section*{Compliance with Ethical Standards}\label{conflicts}

\par\noindent
{\bf Funding}. This research was partly funded by GNAMPA   of the Italian INdAM - National Institute of High Mathematics (grant number not available).

\bigskip
\par\noindent
{\bf Conflict of Interest}. The author declares that there is no conflict of interest.


\begin{thebibliography}{99}


\bibitem{AC251} C.A. Antonini, A. Cianchi, \textit{Global {Lipschitz} regularity in anisotropic elliptic problems with natural gradient growth}, ArXiv preprint, (2025), arXiv:2507.14606

\bibitem{ACCFM25} C.A. Antonini, A. Cianchi, G. Ciraolo, A. Farina, V.G. Maz'ya, \textit{Global second order estimates in anisotropic elliptic problems}, Proc. Lond. Math. Soc.  130,  (2025), no. 3, Paper No. e70034, 60 pp.,  doi = 10.1112/plms.70034

\bibitem{ACF23} C.A. Antonini, G. Ciraolo, A. Farina, \textit{Interior regularity results for inhomogeneous anisotropic
              quasilinear equations},
Math. Ann., vol. 387, no. 3-4, pp. 1745-1776, (2023), DOI = 10.1007/s00208-022-02500-x

\bibitem{AC25} C.A. Antonini, M. Cozzi, \textit{Global gradient regularity and a {Hopf} lemma for quasilinear operators of mixed local-nonlocal type}, J. Differ. Equations, vol. 425, pp. 342-382, (2025), doi = 10.1016/j.jde.2025.01.030


\bibitem{ATU17} J.D. Ara\'ujo, E.V. Texeira, J.M.  Urbano, \textit{A proof of the {{\(C^{p^\prime}\)}}-regularity conjecture in the plane},
Adv. Math., vol. 316, pp. 541-553, (2017),
 doi = 10.1016/j.aim.2017.06.027

\bibitem{ATU18} J.D. Ara\'ujo, E.V. Texeira, J.M.  Urbano, \textit {Towards the {{\(C^{p'}\)}}-regularity conjecture in higher dimensions},
Int. Math. Res. Not., vol=2018, pp. 6481-6495, (2018), 
 doi = 10.1093/imrn/rnx068

\bibitem{BaB25} S. Baasandorj, S.S. Byun, \textit{Regularity for {Orlicz} phase problems}, Mem. Am. Math. Soc., vol. 1556, (2025), Providence, RI: American Mathematical Society (AMS),
 doi = 10.1090/memo/1556

\bibitem{BBDS26} A. Kh. Balci, L. Behn, L. Diening, J. Storn, \textit{Examples of $p$-harmonic maps}, SIAM Journal on Mathematical Analysis, vol. 58, no. 1, pp. 260-275, (2026), 
doi = 10.1137/25M1755539

\bibitem{BMV25} D. Baratta, L. Muglia, D. Vuono, \textit{Second order regularity for solutions to anisotropic degenerate elliptic equations}, J. Differ. Equations, vol. 435, 29 pages, (2025), 
 doi = 10.1016/j.jde.2025.113250

\bibitem{BCM23} G. Barletta, A. Cianchi, G. Marino, \textit{Boundedness of solutions to {Dirichlet}, {Neumann} and {Robin} problems for elliptic equations in {Orlicz} spaces},
Calc. Var. Partial Differ. Equ., vol. 62, no. 2, 42 pages, (2023), 
 doi = 10.1007/s00526-022-02393-3

\bibitem{B15} P. Baroni, \textit{Riesz potential estimates for a general class of quasilinear equations},
Calc. Var. Partial Differ. Equ., vol. 53, no. 3-4, pp. 803-846, (2015), doi = 10.1007/s00526-014-0768-z

\bibitem{B25} P. Baroni,  \textit{Gradient continuity for {{\(p(x)\)}}-{Laplacian} systems under minimal conditions on the exponent}, J. Differ. Equations, vol. 367, pp. 415-450, (2023), 
 doi = 10.1016/j.jde.2023.04.043

\bibitem{BCM18} P. Baroni, M. Colombo, G. Mingione, \textit{Regularity for general functionals with double phase}, Calc. Var. Partial Differ. Equ., vol. 57, no. 2, 48 pages, (2018), 
 doi = 10.1007/s00526-018-1332-z

\bibitem{BL17} P. Baroni, C. Lindfors, \textit{The {Cauchy}-{Dirichlet} problem for a general class of parabolic equations}, Ann. Inst. Henri Poincar{\'e}, Anal. Non Lin{\'e}aire, vol. 32, no. 3, pp. 594-624, (2017), 
 doi = 10.1016/j.anihpc.2016.03.003

\bibitem{BS24} P. Bella, M. Sch{\"a}ffner, \textit{Lipschitz bounds for integral functionals with {{\((p,q)\)}}-growth conditions}, Adv. Calc. Var., vol. 17, no. 2, pp. 373-390, (2024), 
 doi = 10.1515/acv-2022-0016

\bibitem{BF02} A. Bensoussan, J. Frehse, \textit{Regularity results for nonlinear elliptic systems and applications}, Appl. Math. Sci., vol. 151, Berlin: Springer, (2002).

\bibitem{BERV25} S. Biagi, F. Esposito, A. Roncoroni, E. Vecchi, \textit{Brezis-Nirenberg type results for the anisotropic {{\(p\)}}-{Laplacian}},
J. Lond. Math. Soc., II. Ser., vol. 112, no. 4, 31 pages, (2025), 
 doi = 10.1112/jlms.70331

\bibitem{BO15}  V. B{\"o}gelein,  \textit{Global gradient bounds for the parabolic {{\(p\)}}-{Laplacian} system}, Proc. Lond. Math. Soc. (3), vol. 111, no. 3, pp. 633-680, (2015),
 doi = 10.1112/plms/pdv027

\bibitem{BDMS22} V. B{\"o}gelein, F. Duzaar, P. Marcellini, C. Scheven, \textit{Boundary regularity for elliptic systems with {{\(p,q\)}}-growth}, J. Math. Pures Appl. (9),  vol. 159, pp. 250-293, (2022),
 doi = 10.1016/j.matpur.2021.12.004

\bibitem{BDNS22} V. B{\"o}gelein, F. Duzaar, N. Liao, C. Scheven,
\textit{Boundary regularity for parabolic systems in convex domains},
J. Lond. Math. Soc., II. Ser., vol. 105, no. 3, pp. 1702-1751, (2022), 
 doi = 10.1112/jlms.12545

\bibitem{BB18} P. Bousquet, L. Brasco, \textit{{{\(C^1\)}} regularity of orthotropic {{\(p\)}}-harmonic functions in the plane}, Anal. PDE, vol. 11, no. 4, pp. 813-854, (2018), 
 doi = 10.2140/apde.2018.11.813

\bibitem{BB20} P. Bousquet, L. Brasco, \textit {Lipschitz regularity for orthotropic functionals with nonstandard growth conditions}, Rev. Mat. Iberoam., vol. 36, no. 7, pp. 1989-2032, (2020),
 doi = 10.4171/rmi/1189

\bibitem{BBCV18} P. Bousquet, L. Brasco, C. Leone, A. Verde, \textit{On the {Lipschitz} character of orthotropic {{\(p\)}}-harmonic functions}, Calc. Var. Partial Differ. Equ., vol. 57, no. 3, pp. 1-33, (2018), 
 doi = 10.1007/s00526-018-1349-3

\bibitem{BBC24} P. Bousquet, L. Brasco, C. Leone, \textit{Singular orthotropic functionals with nonstandard growth conditions}, Rev. Mat. Iberoam., vol. 40, no. 2, pp. 753-802, (2024), 
 doi = 10.4171/RMI/1446



\bibitem{BK17} S.S. Byun, E. Ko, \textit {Global {{\(C^{1,\alpha}\)}} regularity and existence of multiple solutions for singular {{\(p(x)\)}}-{Laplacian} equations}, Calc. Var. Partial Differ. Equ., vol. 56, no. 3, 29 pages, (2017),
 doi = 10.1007/s00526-017-1152-6

\bibitem{CS07} L. Caffarelli, L. Silvestre, \textit{An extension problem related to the fractional {Laplacian}}, Commun. Partial Differ. Equations, vol. 32, no. 8, pp 1245-1260, (2007), doi = 10.1080/03605300600987306

\bibitem{CRS19} D. Castorina, G. Riey, B. Sciunzi, \textit{Hopf lemma and regularity results for quasilinear anisotropic elliptic equations}, Calc. Var. Partial Differ. Equ., vol. 58, no. 3, 18 pages, (2019), 
 doi = 10.1007/s00526-019-1528-x

\bibitem{CKW23} I. Chlebicka, M. Kim, M. Weidner,  \textit{Gradient {Riesz} potential estimates for a general class of measure data quasilinear systems}, Adv. Calc. Var., vol. 199, no. 2, pp. 237-269, (2026),
       doi = 10.1515/acv-2025-0080

\bibitem{C90} A. Cianchi, \textit{Elliptic equations on manifolds and isoperimetric inequalities}, Proc. R. Soc. Edinb., Sect. A, Math., vol. 114, no. 3-4, pp. 213-227, (1990), 
 doi = 10.1017/S0308210500024392

\bibitem{C97} A. Cianchi, \textit{Boundedness of solutions to variational problems under general growth conditions}, Commun. Partial Differ. Equations, vol. 22, no. 9-10, pp. 1629-1646, (1997), 
 doi = 10.1080/03605309708821313

\bibitem{C00} A. Cianchi, \textit {Local boundedness of minimizers of anisotropic functionals}, Ann. Inst. Henri Poincar{\'e}, Anal. Non Lin{\'e}aire, vol. 17, no. 2, pp. 147-168, (2000), 
 doi = 10.1016/S0294-1449(99)00107-9

\bibitem{CS09} A. Cianchi, P. Salani, \textit{Overdetermined anisotropic elliptic problems}, Math. Ann.,  vol. 345, no. 4, pp. 859-881, (2009), 
 doi = 10.1007/s00208-009-0386-9

\bibitem{CM11} A. Cianchi, V.G. Maz'ya, \textit{Global Lipschitz regularity for a class of quasilinear elliptic equations}, Commun. Partial Differ. Equations, vol. 36, pp. 100-133 (2011), doi=10.1080/03605301003657843.

\bibitem{CM14}  A. Cianchi, V.G. Maz'ya, \textit{Global boundedness of the gradient for a class of nonlinear elliptic systems}, Arch. Ration. Mech. Anal., vol. 212, no. 1, pp. 129-177, (2014), doi = 10.1007/s00205-013-0705-x.

\bibitem{CM141} A. Cianchi, V.G. Maz'ya, \textit{Gradient regularity via rearrangements for {{\(p\)}}-{Laplacian} type elliptic boundary value problems},
J. Eur. Math. Soc. (JEMS), vol. 16, no. 3, pp. 571-595, (2014),
10.4171/JEMS/440,

\bibitem{CVM15} A. Cianchi, V.G. Maz'ya, \textit{Global gradient estimates in elliptic problems under minimal data and domain regularity},
Commun. Pure Appl. Anal., vol. 14, no. 1, pp. 285-311, (2015), 
 doi = 10.3934/cpaa.2015.14.285


\bibitem{CFR20} G. Ciraolo, A. Figalli, A. Roncoroni, \textit{Symmetry results for critical anisotropic {{\(p\)}}-{Laplacian} equations in convex cones}, Geom. Funct. Anal., vol. 30, no. 3, pp. 770-803, (2020), 
 doi = 10.1007/s00039-020-00535-3

\bibitem{CL22} G. Ciraolo, X. Li, \textit{An exterior overdetermined problem for {Finsler} {{\(N\)}}-{Laplacian} in convex cones}, Calc. Var. Partial Differ. Equ., vol. 61, no. 4, 27 pages, (2022), 
 doi = 10.1007/s00526-022-02235-2

\bibitem{CL24} G. Ciraolo, X. Li, \textit{Classification of solutions to the anisotropic {{\(N\)}}-{Liouville} equation in {{\(\mathbb{R}^N\)}}}, Int. Math. Res. Not., vol=2024, no. 19, pp. 12824-12856, (2024), 
 doi = 10.1093/imrn/rnae181

 
\bibitem{CM15} M. Colombo, G. Mingione, \textit{Bounded minimisers of double phase variational integrals}, Arch. Ration. Mech. Anal., vol. 218, no. 1, pp. 219-273, (2015), 
 doi = 10.1007/s00205-015-0859-9

\bibitem{CM16} M. Colombo, G. Mingione, \textit{Calder{\'o}n-Zygmund estimates and non-uniformly elliptic operators}, J. Funct. Anal., vol. 270, no. 4, pp. 1416-1478, (2016), 
 doi = 10.1016/j.jfa.2015.06.022

\bibitem{CFV16} M. Cozzi, A. Farina, E. Valdinoci, \textit{Monotonicity formulae and classification results for singular, degenerate, anisotropic {PDEs}}, Adv. Math., vol. 293, pp. 343-381, (2016), 
 doi = 10.1016/j.aim.2016.02.014

\bibitem{CM26} G. Cupini, P. Marcellini, \textit{Global boundedness of weak solutions to a class of
              nonuniformly elliptic equations}, Math. Ann., vol. 392, no. 2, pp. 1519-1539, (2025), doi = 10.1007/s00208-025-03126-5


\bibitem{CMM23} G. Cupini, P. Marcellini, E. Mascolo, \textit{Local boundedness of weak solutions to elliptic equations with {$p,q$}-growth}, {Math. Eng.}, vol. 5, (2023), no. 3, Paper No. 065, 28,
       doi = 10.3934/mine.2023065

\bibitem{CMM24} G. Cupini, P. Marcellini, E. Mascolo, \textit{Regularity for nonuniformly elliptic equations with {$p,q$}-growth and explicit {$x,u$}-dependence},
Arch. Ration. Mech. Anal., vol. 248, no. 4, 45 pages, (2024),   doi = 10.1007/s00205-024-01982-0

\bibitem{CMMP23} G. Cupini, P. Marcellini, E. Mascolo, A. Passerelli di Napoli, \textit{Lipschitz regularity for degenerate elliptic integrals with
              {$p,q$}-growth},
Adv. Calc. Var., vol. 16, no. 2, pp. 443- 465, (2023), 
       doi = 10.1515/acv-2020-0120


\bibitem{D98} L. Damascelli, \textit{Comparison theorems for some quasilinear degenerate elliptic operators and applications to symmetry and monotonicity results}, Ann. Inst. Henri Poincar{\'e}, Anal. Non Lin{\'e}aire, vol. 15, no. 4, pp. 493-516, (1998), 
doi = 10.1016/S0294-1449(98)80032-2.

\bibitem{DB83} E. DiBenedetto,  \textit{{{\(C^{1+\alpha}\)}} local regularity of weak solutions of degenerate elliptic equations}, Nonlinear Anal., Theory Methods Appl., vol. 7, pp. 827-850, (1983),
doi= 10.1016/0362-546X(83)90061-5

\bibitem{DB94} E. DiBenedetto, \textit{Degenerate parabolic equations}, Universitext, New York, NY: Springer-Verlag, (1993).

\bibitem{DB10} E. DiBenedetto, \textit{Partial Differential Equations}, 2nd ed.,
Cornerstones, (2010), Boston, MA: Birkh{\"a}user,
10.1007/978-0-8176-4552-6

\bibitem{DPV22} S. Dipierro, G. Poggesi, E. Valdinoci, \textit{Radial symmetry of solutions to anisotropic and weighted diffusion equations with discontinuous nonlinearities}, Calc. Var. Partial Differ. Equ., vol. 61, no. 2, 31 pages, (2022), 
 doi = 10.1007/s00526-021-02157-5

\bibitem{DfDfP25} C. De Filippis, F. De Filippis, M. Piccinini, \textit{Bounded minimizers of double phase problems at nearly linear growth}, ArXiv preprint, (2024), arXiv:2411.14325, accepted for publication in J. Eur. Math. Soc. (JEMS)

\bibitem{DfM23} C. De Filippis, G. Mingione, \textit{Nonuniformly elliptic {Schauder} theory}, Invent. Math., vol. 234, no. 3, pp. 1109-1196, (2023), 10.1007/s00222-023-01216-2

\bibitem{DfM231} C. De Filippis, G. Mingione, \textit{Regularity for double phase problems at nearly linear growth}, Arch. Ration. Mech. Anal., vol. 247, no. 5, 50 pages, (2023), 
 doi = 10.1007/s00205-023-01907-3

\bibitem{DfM25} C. De Filippis, G. Mingione, \textit{The sharp growth rate in nonuniformly elliptic {Schauder} theory}, Duke Math. J., vol. 174, no. 9, pp. 1775-1848, (2025),  doi = 10.1215/00127094-2024-0075

\bibitem{DfP23} C. De Filippis, M. Piccinini, \textit{Borderline global regularity for nonuniformly elliptic systems}, Int. Math. Res. Not., vol. 20, pp. 17324-17376, (2023), doi = 10.1093/imrn/rnac283

\bibitem{DfP24} F. De Filippis, M. Piccinini, \textit {Regularity for multi-phase problems at nearly linear growth}, J. Differ. Equations, vol. 410, pp. 832-868, (2024), 
 doi = 10.1016/j.jde.2024.08.023

\bibitem{DG56} E. De Giorgi, \textit{Sull'analiticita delle estremali degli integrali multipli}, Atti Accad. Naz. Lincei, VIII. Ser., Rend., Cl. Sci. Fis. Mat. Nat., vol. 20, pp. 438-441, (1956).


\bibitem{DT71} T.K. Donaldson, N.S. Trudinger,
\textit{Orlicz-Sobolev spaces and imbedding theorems}, J. Funct. Anal., vol. 8, pag. 52-75, (1971),
doi=10.1016/0022-1236(71)90018-8


\bibitem{DM10} F. Duzaar, G. Mingione, \textit{Gradient estimates via linear and nonlinear potentials}, J. Funct. Anal., vol. 259, no. 11, pp 2961-2998, (2010), doi = 10.1016/j.jfa.2010.08.006.

\bibitem{DM100} F. Duzaar, G. Mingione, \textit{Gradient continuity estimates}, Calc. Var. Partial Differ. Equ., vol. 39, no. 3-4, pp. 379-418, (2010),
 doi = 10.1007/s00526-010-0314-6

\bibitem{EMSV24} F. Esposito, L. Montoro, B. Sciunzi, D. Vuono, \textit{Asymptotic behaviour of solutions to the anisotropic doubly critical equation}, Calc. Var. Partial Differ. Equ., vol. 63, no. 3, 44 pages, (2024), 
 doi = 10.1007/s00526-024-02682-z

\bibitem{E82} L.C. Evans, \textit{A new proof of local {{\(C^{1,\alpha}\)}} regularity for solutions of certain degenerate elliptic {P}.{D}.{E}}, J. Differ. Equations, vol. 45, pp. 356-373
 doi = 10.1016/0022-0396(82)90033-X

\bibitem{F07} X. Fan, \textit{Global {{\(C^{1,\alpha}\)}} regularity for variable exponent elliptic equations in divergence form}, J. Differ. Equations, vol. 235, no. 2, pp. 397-417, (2007),
 doi = 10.1016/j.jde.2007.01.008

\bibitem{FSV24} A. Farina, B. Sciunzi, D. Vuono, \textit{Liouville-type results for some quasilinear anisotropic elliptic equations}, Nonlinear Anal., Theory Methods Appl., Ser. A, Theory Methods, vol. 238, 15 pages, (2024), 
 doi = 10.1016/j.na.2023.113402

\bibitem{FL21} F. Ferrari, C. Lederman, \textit{Regularity of flat free boundaries for a {$p(x)$}-{L}aplacian
              problem with right hand side}, Nonlinear Anal., vol. 212, 25 pages, (2021),  doi = 10.1016/j.na.2021.112444

\bibitem{FL24} F. Ferrari, C. Lederman, \textit{Regularity of {L}ipschitz free boundaries for a {$p (
              x$})-{L}aplacian problem with right hand side}, J. Math. Pures Appl. (9), vol. 171, pp. 26-74, (2023), 
       doi = 10.1016/j.matpur.2022.12.009

\bibitem{FL26} F. Ferrari, C. Lederman, \textit{Lipschitz regularity of solutions to two-phase $p$-Laplacian free boundary problems with right hand side}, Comm. Anal. Geom., vol. 34, no. 2, pp. 465-503, (2026), doi = 10.4310/cag.260528231018

\bibitem{GYZ26} Z. Gao, Y. Yan, Y. Zhou, \textit{On the Viscosity Solutions of Parabolic p-Laplacian Equations with Capillary-Type Boundary Conditions}, ArXiv preprint (2026), 	arXiv:2604.04391

\bibitem{G26} G. Giannone, \textit{Local boundedness of weak solutions to elliptic equations under unbalanced Orlicz growth conditions}, ArXiv Preprint (2025), 	arXiv:2510.18979 

\bibitem{G83}  M. Giaquinta, \textit{Multiple integrals in the calculus of variations and nonlinear elliptic systems}, Annals of
 Mathematics Studies, vol. 105, Princeton University Press, Princeton, NJ, 1983, vii+297.

\bibitem{GG83} M. Giaquinta, E. Giusti, \textit{Differentiability of minima of non-differentiable functionals}, Invent. Math., vol. 72, pp. 286-298, (1983), doi = 10.1007/BF01389324

\bibitem{GG84} M. Giaquinta, E. Giusti, \textit{Global {{\(C^{1,\alpha}\)}}-regularity for second order quasilinear elliptic equations in divergence form}, J. Reine Angew. Math., vol. 351, pp. 55-65, 1984

\bibitem{GM12} M. Giaquinta, L. Martinazzi, \textit{An introduction to the regularity theory for elliptic systems, harmonic maps and minimal graphs}, 2nd ed., Appunti, Sc. Norm. Super. Pisa (N.S.), vol. 11, (2012), doi=10.1007/978-88-7642-443-4.

\bibitem{GH80} D. Gilbarg, L. H\"ormander, \textit{Intermediate {Schauder} estimates}, Arch. Ration. Mech. Anal., vol. 74, pp. 297-318, (1980),
 doi = 10.1007/BF00249677.

\bibitem{GT} D. Gilbarg, N.S. Trudinger, \textit{Elliptic partial differential equations of second order}, Classics in Mathematics, Berlin: Springer, Reprint of the 1998 ed., (2001).

\bibitem{G03} E. Giusti, \textit{Direct methods in the calculus of variations}, Singapore: World Scientific, (2003).

\bibitem{grisvard} P. Grisvard, \textit{Elliptic problems in nonsmooth domains}, Monogr. Stud. Math., vol. 24, Pitman, Boston, MA, (1985).

\bibitem{GM23} U. Guarnotta, S. Mosconi, \textit {A general notion of uniform ellipticity and the regularity of the stress field for elliptic equations in divergence form}, Anal. PDE, vol. 16, no. 9, pp. 1955-1988, (2023),  doi = 10.2140/apde.2023.16.1955

\bibitem{HL11} Q. Han, F. Lin, \textit{Elliptic partial differential equations}, 2nd edition, Courant Lect. Notes Math., vol. 1, (2011).

\bibitem{HHbook}  P.Harjulehto, Petteri, P. H{\"a}st{\"o},
\textit{Orlicz spaces and generalized {Orlicz} spaces},
Lect. Notes Math., vol. 2236, (2019)
Cham: Springer,
 doi = 10.1007/978-3-030-15100-3

\bibitem{HLO25} P. H{\"a}st{\"o}, P. Lee, J. Ok, \textit{Mean oscillation conditions for nonlinear equation and regularity results}, ArXiv preprint (2025), arXiv:2504.02159

\bibitem{HO221} P. H{\"a}st{\"o}, J. Ok, \textit{Maximal regularity for local minimizers of non-autonomous functionals}, J. Eur. Math. Soc. (JEMS), vol. 24, no. 4, pp. 1285-1334,
 doi = 10.4171/JEMS/1118

\bibitem{HO22} P. H{\"a}st{\"o}, J. Ok, \textit{Regularity theory for non-autonomous partial differential equations without {Uhlenbeck} structure}, Arch. Ration. Mech. Anal., vol. 245, no. 3, pp. 1401-1436, (2022), 10.1007/s00205-022-01807-y

\bibitem{HO23} P. H{\"a}st{\"o}, J. Ok, \textit{Regularity theory for non-autonomous problems with a priori assumptions}, Calc. Var. Partial Differ. Equ., vol. 62, no. 9, 28 pages, (2023), doi = 10.1007/s00526-023-02587-3

\bibitem{H26} M. Hlel, \textit{Global $C^{1,\alpha}$-Regularity for Musielak-Orlicz Equations in Divergence Form}, ArXiv preprint, (2026), 	arXiv:2601.11495


\bibitem{IM89}  T. Iwaniec, J.J. Manfredi, \textit{Regularity of {{\(p\)}}-harmonic functions on the plane}, Rev. Mat. Iberoam., vol. 5, no. 1-2, pp. 1-19, (1989),
 doi = 10.4171/RMI/82

\bibitem{K90} A.G. Korolev, \textit{On boundedness of generalized solutions of elliptic differential equations with non-power nonlinearities},
Math. USSR, Sb., vol. 66, no. 1, pp. 83-106, (1990), 
 doi = 10.1070/SM1990v066n01ABEH001166

\bibitem{K84} N.V. Krylov,  \textit{Boundedly nonhomogeneous elliptic and parabolic equations in a domain}, Math. USSR, Izv., vol. 22, pp. 67-97, 1984, doi = 10.1070/IM1984v022n01ABEH001434

\bibitem{KM12} T. Kuusi, G. Mingione, \textit{Universal potential estimates}, J. Funct. Anal., vol. 262, no. 10, pp. 4205-4269, (2012),
 doi = 10.1016/j.jfa.2012.02.018

\bibitem{KM13} T. Kuusi, G. Mingione,  \textit{Linear potentials in nonlinear potential theory}, Arch. Ration. Mech. Anal., vol. 207, no.1, pp. 215-246, (2013),
 doi = 10.1007/s00205-012-0562-z

\bibitem{KM14} T. Kuusi, G. Mingione, {A nonlinear {Stein} theorem}, Calc. Var. Partial Differ. Equ., vol. 51, no. 1-2, pp. 45-86, (2014),
 doi = 10.1007/s00526-013-0666-9

\bibitem{KM141} T. Kuusi, G. Mingione, \textit{Guide to nonlinear potential estimates}, Bull. Math. Sci., vol. 4, no. 1, pp. 1-82, (2014),
 doi = 10.1007/s13373-013-0048-9

 \bibitem{KM18}  T. Kuusi, G. Mingione, \textit{Vectorial nonlinear potential theory}, J. Eur. Math. Soc. (JEMS), vol. 20, no. 4, pp. 929-1004, (2018), 
 doi = 10.4171/JEMS/780

\bibitem{LU68} O.A. Ladyzhenskaya, N.N. Uralt'seva, \textit{Linear and quasilinear equations}, Math. Sci. Eng., vol. 64, (1968)

\bibitem{Leo17} G. Leoni, \textit{A first course in {Sobolev} spaces}, 2nd edition, Graduate Studies in Mathematics, vol. 181, (2017), doi = 10.1090/gsm/181

\bibitem{L83} J.L. Lewis, \textit{Regularity of the derivatives of solutions to certain degenerate elliptic equations}, Indiana Univ. Math. J., vol. 32, pp. 849-858, (1983), 
 doi = 10.1512/iumj.1983.32.32058


\bibitem{L86} G.M. Lieberman, \textit{The {Dirichlet} problem for quasilinear elliptic equations with continuously differentiable boundary data}, Commun. Partial Differ. Equations, vol. 11, pp. 167-229, 1986,
 doi = 10.1080/03605308608820422

\bibitem{L88} G.M. Lieberman, \textit{Boundary regularity for solutions of degenerate elliptic equations}, Nonlinear Anal., Theory Methods Appl., vol. 12, no. 11, pp. 1203-1219, 1988,
 doi = 10.1016/0362-546X(88)90053-3

\bibitem{L90} G. M. Lieberman, \textit{Boundary regularity for solutions of degenerate parabolic equations}, Nonlinear Anal., Theory Methods Appl., vol. 14, no. 6, pp. 501-524, (1990) doi = 10.1016/0362-546X(90)90038-I

\bibitem{L91} G.M. Lieberman, \textit{The natural generalization of the natural conditions of {Ladyzhenskaya} and {Ural}'tseva for elliptic equations}, Commun. Partial Differ. Equations, vol. 16, no. 2-3, pp. 311-361, (1991), doi=10.1080/03605309108820761.

\bibitem{L931} G. M. Lieberman, \textit{Boundary and initial regularity for solutions of degenerate parabolic equations},
  Nonlinear Anal., Theory Methods Appl., vol. 20, no. 5, pp. 551-569, (1993), 
 doi = 10.1016/0362-546X(93)90038-T

\bibitem{L93} G.M. Lieberman,  \textit{Sharp forms of estimates for subsolutions and supersolutions of quasilinear elliptic equations involving measures}, Commun. Partial Differ. Equations, vol. 18, no. 7-8, pp. 1191-1212, (1993),
 doi = 10.1080/03605309308820969

\bibitem{L96} G.M. Lieberman, \textit{Second order parabolic differential equations}, 1996, Singapore: World Scientific.

\bibitem{L13} G.M. Lieberman, \textit{Oblique derivative problems for elliptic equations}, 2013
Hackensack, NJ: World Scientific, 2013,
 doi = 10.1142/8679

\bibitem{LL91} F.H. Lin, Y. Li, \textit{Boundary {{\(C^{1,\alpha{}}\)}} regularity for variational inequalities},
Commun. Pure Appl. Math., vol. 44, no. 6, pp. 715-732, (1991), doi = 10.1002/cpa.3160440605

\bibitem{M86} J.J. Manfredi, \textit{Regularity of the gradient for a class of nonlinear possibly degenerate elliptic equation}, Phd thesis, (1986).

\bibitem{M88} J.J. Manfredi,  \textit{Regularity for minima of functionals with p-growth}, J. Differ. Equations, vol. 76, no. 2, pp. 203-212, (1988), 
 doi = 10.1016/0022-0396(88)90070-8

\bibitem{M89} P. Marcellini, \textit{Regularity of minimizers of integrals of the calculus of variations with non-standard growth conditions}, Arch. Ration. Mech. Anal., vol. 105, no. 3, pp. 267-284, (1989),
 doi = 10.1007/BF00251503

\bibitem{M91} P. Marcellini, \textit{Regularity and existence of solutions of elliptic equations with p,q- growth conditions}, J. Differ. Equations, vol. 90, no. 1, pp. 1-30, (1991), 
 doi = 10.1016/0022-0396(91)90158-6

\bibitem{M93} P. Marcellini, \textit{Regularity for elliptic equations with general growth
              conditions}, J. Differential Equations, vol. 105, no. 2, pp. 296-333, (1993), 
       DOI = 10.1006/jdeq.1993.1091

\bibitem{M23} P. Marcellini, \textit{Local {L}ipschitz continuity for {$p,q$}-{PDE}s with explicit
              {$u$}-dependence},
Nonlinear Anal., vol. 226, 26 pages, (2023), 
       DOI = 10.1016/j.na.2022.113066

\bibitem{MNP26} P. Marcellini, A. Nastasi, C.C. Pacchiano, \textit{Unified a-priori estimates for minimizers under {{\(p, q\)}}-growth and exponential growth}, Nonlinear Anal., Theory Methods Appl., Ser. A, Theory Methods, vol. 264, pp. 21, (2026), 
 doi = 10.1016/j.na.2025.113982

\bibitem{MM24} G. Marino, S. Mosconi, \textit{Lipschitz regularity for solutions of a general class of elliptic equations}, Calc. Var. Partial Differ. Equ., vol. 63, no. 1, 40 pages, (2024), 10.1007/s00526-023-02632-1

\bibitem{MP96} E. Mascolo, G. Papi, \textit{Harnack inequality for minimizers of integral functionals with general growth conditions}, NoDEA, Nonlinear Differ. Equ. Appl., vol. 3, no. 2, pp. 231-244, 
 doi = 10.1007/BF01195916

\bibitem{M11} G. Mingione, \textit{Gradient potential estimates},
J. Eur. Math. Soc. (JEMS), vol. 13, no. 2, pp. 459-486, (2011), 
 doi = 10.4171/JEMS/258


\bibitem{M60} J. Moser,  \textit{A new proof of de {Giorgi}'s theorem concerning the regularity problem for elliptic differential equations},  Commun. Pure Appl. Math., vol. 13, pp. 457-468, (1960),
 doi = 10.1002/cpa.3160130308

\bibitem{N58} J. F. Nash, \textit{Continuity of solutions of parabolic and elliptic equations}, Am. J. Math., vol. 80, pp. 931-954, (1958).

\bibitem{Ok16} J. Ok, \textit{Gradient continuity for $p(\cdot)$-Laplace systems}, Nonlinear Anal. 141 (2016) 139--166

\bibitem{Ok17} J. Ok, {{{\(C^1\)}}-regularity for minima of functionals with {{\(p(x)\)}}-growth}, J. Fixed Point Theory Appl., vol. 19, no. 4, pp. 2697-2731, (2017),  doi= 10.1007/s11784-017-0446-9

 \bibitem{funcsp} L. Pick, A. Kufner, O. John, S. Fu\v{c}ik, \textit{Function spaces Vol. 1}, De Gruyter Series in Nonlinear Analysis and Applications 14, Walter de Gruyter \& Co., Berlin, 2013, xvi+479.

 \bibitem{RR91} M.M. Rao, Z.D. Ren,
 \textit{Theory of {Orlicz} spaces},
Pure Appl. Math., Marcel Dekker, vol. 146 (1991).

\bibitem{S24} M. Sch{\"a}ffner, \textit{Lipschitz bounds for nonuniformly elliptic integral functionals in the plane}, Proc. Am. Math. Soc., vol. 152, no. 11, pp. 4717-4727, (2024), 
 doi = 10.1090/proc/16878

\bibitem{S64} J. Serrin, \textit{Local behavior of solutions of quasi-linear equations},
Acta Math., vol. 111, pp. 247-302, (1964), 
 doi = 10.1007/BF02391014

\bibitem{simon76} L. Simon, \textit{Interior gradient bounds for non-uniformly elliptic equations}, Indiana Univ. Math. J., vol. 25, no. 9, pp. 821-855, (1976),  doi = 10.1512/iumj.1976.25.25066
       
\bibitem{SY26} K. Song,  Y. Youn, \textit{Gradient estimates for degenerate elliptic measure data problems with double phase}, ArXiv preprint (2026), 	arXiv:2605.02470

\bibitem{SYZ26} K. Song,  Y. Youn, A. Zatorska-Goldstein, \textit{Gradient estimates for singular elliptic measure data problems with double phase}, ArXiv preprint (2026), 	arXiv:2605.18235

\bibitem{T79} G. Talenti, \textit{Nonlinear elliptic equations, rearrangements of functions and {Orlicz} spaces},
Ann. Mat. Pura Appl. (4), vol. 120, pp. 159-184, (1979), 
 doi = 10.1007/BF02411942

\bibitem{Ta91} G. Talenti, \textit{Boundedness of minimizers}, Hokkaido Math. J. 19 (1990), no. 2, 259-279

\bibitem{T84} P. Tolksdorf, \textit{Regularity for a more general class of quasilinear equations}, J. Differ. Equations, vol. 51, pp. 126-150, (1984), doi = 10.1016/0022-0396(84)90105-0

\bibitem{T67} N.S. Trudinger, \textit{On {Harnack} type inequalities and their application to quasilinear elliptic equations}, Commun. Pure Appl. Math., vol. 20, pp. 721-747, (1967), 
 doi = 10.1002/cpa.3160200406

\bibitem{U77} K. Uhlenbeck, \textit{Regularity for a class of nonlinear elliptic systems}, Acta Math., vol. 138, pp. 219-240, (1977), 
 doi = 10.1007/BF02392316

\bibitem{U68} N.N. Ural'tseva, \textit{Degenerate quasilinear elliptic systems}, Zap. Nauk. Sem. 
Leningrad Otdel. Math. Inst. Steklov  7 (1968), pp. 184-222 (Russian).

\bibitem{V25} D. Vuono, \textit{Harnack inequalities for quasilinear anisotropic elliptic equations with a first order term}, NoDEA, Nonlinear Differ. Equ. Appl., vol. 32, no. 4, 32 pages, (2025),
 doi = 10.1007/s00030-025-01071-5


\bibitem{WY25} X. Wang, F. Yao, \textit{Optimal {{\({C}^{1, {{\alpha}}}\)}} regularity for quasilinear elliptic equations with {Orlicz} growth}, Electron. J. Qual. Theory Differ. Equ., vol. 2025, 18 pages, (2025), 10.14232/ejqtde.2025.1.57

\bibitem{Z90} E. Zeidler, \textit{Nonlinear Functional Analysis and its  Applications. II/B. Nonlinear Monotone Operators}, Springer Verlag, New York, 1990, pp. i-xvi and 469-1202.
\end{thebibliography}
\end{document}